\documentclass[10pt,twoside,reqno,a4paper]{amsart}
\usepackage[utf8]{inputenc}
\usepackage{titlesec}
\usepackage{topformat}
\usepackage[driver=pdftex,margin=3cm,heightrounded=true,centering]{geometry}
\usepackage{hyperref}
\usepackage{enumerate}
\usepackage{amssymb}
\usepackage{amsopn}
\usepackage{amsmath}
\usepackage{mathtools}
\usepackage{tikz}
\usepackage{tikz-cd}
\newsavebox{\wideeqbox}
\usepackage{amsthm}
\usepackage{topthm}
\author{Thorben Kastenholz}
\thanks{The author was supported by the SPP2026 
"Geometry at Infinity" of the DFG, and European Research Council Advanced Grant 
"Moduli" [Grant Number 290399]}
\date{\today}
\title{Homological Stability for Spaces of Subsurfaces 
with Tangential Structure}
\newcommand{\introduce}[1]
  {\textbf{#1}}
\newcommand{\comment}[1]
  {TODO:\@{\LARGE{#1}}}

\newcommand{\apply}[2]
  {{#1}\!\left({#2}\right)}
\newcommand{\at}[2]
  {\left.{#1}\right\rvert_{#2}}
\newcommand{\subsetv}%
  {\rotatebox[origin=c]{-90}{$\subseteq$}}
\newcommand{\EuclideanNorm}[1]
  {\left\lVert{#1}\right\rVert}
\newcommand{\Reals}%
  {\mathbb{R}}
\newcommand{\Integers}%
  {\mathbb{Z}}
\newcommand{\NaturalNumbers}%
  {\mathbb{N}}
\newcommand{\Identity}%
  {\mathrm{Id}}
\newcommand{\Projection}[1]
  {\pi_{#1}}
\newcommand{\DimensionIndex}%
  {n}
\newcommand{\BigDimensionIndex}%
  {N}
\newcommand{\Complement}[1]
  {\left(#1\right)^{c}}

\newcommand{\TopologicalSpace}%
  {X}
\newcommand{\Point}%
  {x}
\newcommand{\Neighbourhood}[1]
  {U_{#1}}
\newcommand{\ContinuousMap}%
  {g}
\newcommand{\Homotopy}%
  {H}
\newcommand{\Fiber}[2]
  {\mathrm{Fib}_{#1}\!\left(#2\right)}
\newcommand{\HomotopyFiber}[2]
  {\mathrm{HoFib}_{#1}\!\left(#2\right)}
\newcommand{\FiberSpace}%
  {F}
\newcommand{\ClassifyingSpace}[1]
  {B#1}
\newcommand{\MappingPair}[1]
  {C\!\left(#1\right)}
\newcommand{\MappingCone}[1]
  {\bar{C}\!\left(#1\right)}
\newcommand{\Suspension}[1]
  {\Sigma #1}
\newcommand{\Path}%
  {\gamma}

\newcommand{\Manifold}%
  {M}
\newcommand{\Ball}[1]
  {D^{#1}}
\newcommand{\Sphere}[1]
  {S^{#1}}
\newcommand{\BallRadius}[2]
  {B_{#1}\!\left(#2\right)}
\newcommand{\CutOut}[2]
  {\apply{#1}{#2}}
\newcommand{\BoundaryManifold}[1] 
  {\partial #1}
\newcommand{\DistBoundaryManifold}[2]
  {\partial^{#1}#2}
\newcommand{\TubularNeighbourhood}[2]
  {N_{#1,#2}}
\newcommand{\TubularNeighbourhoodMap}[2]
  {\Embedding^{T}_{#1,#2}}
\newcommand{\Submanifold}%
  {A}
\newcommand{\Collar}%
  {c}
\newcommand{\Elongation}[2]
  {#1^{#2}}
\newcommand{\SmoothMap}%
  {f}
\newcommand{\Embedding}%
  {e}

\newcommand{\Genus}%
  {g}
\newcommand{\BoundaryComponents}%
  {b}
\newcommand{\SurfaceGB}[2]
  {\Sigma_{#1,#2}}
\newcommand{\SurfaceG}[1]
  {\Sigma_{#1}}
\newcommand{\Subsurface}%
  {W}
\newcommand{\TwoSphere}%
  {S^2}
\newcommand{\Circle}%
  {S^1}
\newcommand{\SurfaceGBOneSkeleton}[2]
  {\Sigma^{1}_{#1,#2}}
\newcommand{\HalfDisk}%
  {D^{2}_{+}}
\newcommand{\ThickenedHalfDisk}[1]
  {\mathbf{#1}}
\newcommand{\CellBoundaryCondition}%
  {\mathfrak{d}}

\newcommand{\HomologyOfSpace}[2]
  {H_{#1}\!\left(#2\right)}
\newcommand{\HomologyOfMap}[1]
  {\left(#1\right)_{*}}
\newcommand{\HomotopyGroup}[2]
  {\pi_{#1}\!\left(#2\right)}
\newcommand{\HomotopyGroupMap}[1]
  {\left(#1\right)_{*}}
\newcommand{\Surgery}[2]
  {#1\natural#2}
\newcommand{\SurgeryEmbedding}%
  {e_{\HalfDisk}}

\DeclareMathOperator{\Map}
  {Map}
\newcommand{\MappingSpace}[2]
  {\Map\!\left(#1,#2\right)}
\newcommand{\Evaluation}%
  {\mathrm{ev}}
\newcommand{\ModuliSpace}[2]
  {\mathcal{M}^{#1}\left(#2\right)}
\newcommand{\SmoothMappingSpace}[2]
  {C^{\infty}\!\left(#1,#2\right)}
\DeclareMathOperator{\Tub}%
  {Tub}
\newcommand{\TubularNeighbourhoodClosedSpace}[2]
  {\overline{\Tub}\!\left(#1,#2\right)}
\DeclareMathOperator{\Emb}%
  {Emb}
\newcommand{\EmbeddingSpace}[2]
  {\Emb\!\left(#1,#2\right)}
\newcommand{\EmbeddingSpaceBoundaryCondition}[3]
  {{\Emb}\!\left(#1,#2;#3\right)}
\newcommand{\BoundaryConditionPar}%
  {d}
\DeclareMathOperator{\Diff}%
  {Diff}
\newcommand{\DiffeomorphismGroup}[1]
  {{\Diff^{+}}\!\left(#1\right)}
\newcommand{\DiffeomorphismGroupBoundary}[1]
  {{\Diff^{+}_{\partial}}\!\left(#1\right)}
\newcommand{\DiffeomorphismGroupCompactSupport}[1]
  {{\Diff^{+}_{c}}\!\left(#1\right)}
\newcommand{\DiffeomorphismGroupArcs}[2]
  {{\Diff^{+}_{#1}}\!\left(#2\right)}
\newcommand{\SpaceSubsurfaces}[3]
  {{\mathcal{E}^{+}_{#1,#2}}\!\left(#3\right)}
\newcommand{\SpaceSubsurfaceBoundary}[4]
  {{\mathcal{E}^{+}_{#1,#2}}\!\left(#3;#4\right)}
\newcommand{\BoundaryConditionSpace}[2]
  {{\mathcal{J}_{\partial}}\!\left(#1,#2\right)}
\newcommand{\BoundaryConditionSpacePar}[2]
  {J_{\partial}\!\left(#1,#2\right)}
\newcommand{\BoundaryConditionSpaceTangential}[3]
  {{\mathcal{J}^{#1}_{\partial}}\!\left(#2,#3\right)}
\newcommand{\BoundaryConditionMap}%
  {\mathcal{J}}
\newcommand{\BoundaryCondition}%
  {\delta}
\DeclareMathOperator{\TEmb}%
  {TEmb}
\newcommand{\ThickenedEmbeddingSpaceBoundaryCondition}[3]
  {\overline{\TEmb}\!\left( #1 , #2 ; #3 \right)}
\newcommand{\StabilizingBordism}[4]
  {{\mathcal{C}^{#1}}\!\left(#2;#3,#4\right)}
\newcommand{\StabilizingBordismAlt}[3]
  {\mathcal{C}^{#1}\!\left(#2,#3\right)}
\newcommand{\StabilizingManifold}%
  {W^{s}}
\newcommand{\SpaceAllSubsurfaceBoundary}[2]
  {{\mathcal{E}^{+}}\!\left(#1;#2\right)}
\newcommand{\EmbeddingSpaceTangential}[3]
  {\Emb^{#1}\!\left(#2,#3\right)}
\newcommand{\SpaceSubsurfacesTangential}[4]
  {\mathcal{E}^{#1}_{#2,#3}\!\left(#4\right)}
\newcommand{\SpaceSubsurfaceBoundaryTangential}[5]
  {{\mathcal{E}^{#1}_{#2,#3}}\!\left(#4;#5\right)}
\newcommand{\SpaceAllSubsurfaceBoundaryTangential}[3]
  {{\mathcal{E}^{#1}}\!\left(#2;#3\right)}
\newcommand{\BoundaryConditionTangential}[1]
  {#1^{T}}
\newcommand{\EmbeddingSpaceBoundaryConditionTangential}[4]
  {{\Emb}^{#1}\!\left(#2,#3;#4\right)}
\newcommand{\ThickenedEmbeddingSpaceBoundaryConditionTangential}[4]
  {\overline{\TEmb}^{#1}\!\left(#2,#3;#4\right)}
\newcommand{\SpaceOfSections}[1]
  {\Gamma\!\left(#1\right)}
\newcommand{\HomeoTangential}[2]
  {\mathrm{Homeo}_{#1}\!\left(#2\right)}

\newcommand{\StabilizationAlpha}[2]
  {\alpha_{#1,#2}}
\newcommand{\StabilizationBeta}[2]
  {\beta_{#1,#2}}
\newcommand{\StabilizationGamma}[2]
  {\gamma_{#1,#2}}
\newcommand{\StabilizationMap}[2]
  {\sigma_{#1}}
\newcommand{\StabilizationBordism}%
  {P}
\newcommand{\SubsetBordism}%
  {U}
\newcommand{\BoundaryBordism}%
  {P}
\newcommand{\AuxBordism}%
  {Z}
\newcommand{\TrivialityIndex}%
  {k}
\newcommand{\PiZeroStabilization}%
  {m}
\newcommand{\InductionIndex}%
  {h}
\newcommand{\AlphaBound}%
  {A}
\newcommand{\BetaBound}%
  {B}
\newcommand{\AlphaAuxBound}%
  {SA}
\newcommand{\BetaAuxBound}%
  {SB}
\newcommand{\kAugmentation}[3]
  {\sigma^{#1}_{#2,#3}}
\newcommand{\AlphakAugmentation}[3]
  {\mathfrak{a}^{#1}_{#2,#3}}
\newcommand{\BetakAugmentation}[3]
  {\mathfrak{b}^{#1}_{#2,#3}}

\newcommand{\DistBall}%
  {\ell}
\newcommand{\DistIntervall}[1]
  {\DistBall_{#1}}
\newcommand{\DiskEmbedding}%
  {a}
\newcommand{\ArcResolution}[7]
  {\mathcal{O}^{#1}_{#2,#3}\!\left(#4,#5;#6\right)_{#7}}
\newcommand{\ArcResolutionNoDisk}[7]
  {\overline{\mathcal{O}}^{#1}_{#2,#3}\!\left(#4,#5;#6\right)_{#7}}
\newcommand{\ArcResolutionEmbedding}[7]
  {O^{#1}_{#2,#3}\!\left(#4,#5;#6\right)_{#7}}
\newcommand{\ArcResolutionSingle}[7]
  {\mathcal{O}'^{#1}_{#2,#3}\!\left(#4,#5;#6\right)_{#7}}
\newcommand{\ArcResolutionTwo}[7]
  {\mathcal{O}''^{#1}_{#2,#3}\!\left(#4,#5;#6\right)_{#7}}
\newcommand{\SpaceOfStrips}[5]
  {A^{#1}\!\left(#2;#3,#4\right)_{#5}}
\newcommand{\BoundaryPath}%
  {y}
\newcommand{\ThickenedBoundaryPath}%
  {\mathbf{\BoundaryPath}}
\newcommand{\BoundaryPathResolution}[7]
  {\mathcal{Q}^{#1}_{#2,#3}\!\left(#4;#5,#6\right)_{#7}}
\newcommand{\SpaceOfBoundaryPaths}[4]
  {R^{#1}\!\left(#2,#3 \right)_{#4}}

\newcommand{\OrthogonalComplement}[1]
  {#1^{\perp}}
\newcommand{\Coordinate}%
  {x}
\newcommand{\TwoPlane}%
  {B}
\newcommand{\SOrthogonalGroup}[1]
  {SO\left(#1\right)}
\newcommand{\NormalBundle}[2]
  {N\!\left(#1,#2\right)}
\DeclareMathOperator{\Gr}%
  {Gr}
\newcommand{\Grassmannian}[2]
  {\Gr_{#1}\!\left(#2\right)}
\newcommand{\ProjectionGrassmannian}%
  {\Projection{\Gr}}
\newcommand{\TautologicalBundle}%
  {\gamma}
\newcommand{\TangentBundle}[1]
  {T#1}
\newcommand{\TangentSpace}[2]
  {T_{#1}#2}
\newcommand{\SphereBundle}[1]
  {S\!\left(#1\right)}
\newcommand{\TangentialFibration}%
  {\theta}
\newcommand{\TangentialSpace}[1]
  {B_{2}\!\left(#1\right)}
\newcommand{\Differential}[1]
  {D#1}
\newcommand{\GrassmannianDifferential}[1]
  {\Gr\!\left(\Differential{#1}\right)}
\newcommand{\TangentialStructure}[1]
  {T_{#1}}
\newcommand{\LineBundle}%
  {L}

\newcommand{\Group}%
  {G}
\newcommand{\LocalRetraction}%
  {\xi}

\newcommand{\SemiSimplicialCategory}%
  {\Delta^{op}_{inj}}
\newcommand{\SemiSimplicialSpace}%
  {X_{\bullet}}
\newcommand{\SemiSimplicialSpaceIndex}[1]
  {X_{#1}}
\newcommand{\SemiSimplicialIndex}%
  {i}
\newcommand{\Simplex}[1]
  {\left[#1\right]}
\newcommand{\IndexSimplex}[1]
  {\sigma_{#1}}
\newcommand{\FaceIndex}%
  {j}
\newcommand{\FaceMap}[2]
  {\partial_{#1,#2}}
\newcommand{\Augmentation}[1]
  {\epsilon_{#1}}
\newcommand{\GeometricRealization}[1]
  {\left\vert{#1}\right\vert}

\newcommand{\DiffeoMorphismGroupPointed}[2]
  {\Diff_{#1}\!\left(#2\right)}
\newcommand{\EmbeddingSpacePointedTangential}[4]
  {\Emb^{#1}_{#2}\!\left(#3,#4\right)}
\newcommand{\SpaceSubsurfacesPointedTangential}[4]
  {\mathcal{E}^{#1}_{#2}\!\left(#3,#4\right)}
\newcommand{\SpaceSubsurfacesPointedFlatTangential}[4]
  {\mathcal{E}^{#1}_{#2,\text{flat}}\!\left(#3,#4\right)}
\newcommand{\EmbeddingSpacePointedRadiusTangenital}[5]
  {\Emb^{#1}_{#2,#3}\!\left(#4,#5\right)}
\newcommand{\SpaceSubsurfacesPointedRadiusTangential}[5]
  {\mathcal{E}^{#1}_{#2,#3}\!\left(#4,#5\right)}
\newcommand{\SpaceAlmostFlatSubsurfaces}[4]
  {\mathcal{E}^{#1}_{#2,\text{almost}}\!\left(#3,#4\right)}
\newcommand{\SpaceAlmostFlatEmbeddings}[4]
  {\Emb^{#1}_{#2,\text{almost}}\!\left(#3,#4\right)}
\newcommand{\PointedStabilizationMap}[1]
  {\sigma_{#1}}
\newcommand{\AuxiliaryFunction}%
  {\varphi}
\newcommand{\AuxiliaryFunctionDiffeomorphism}[1]
  {\psi_{#1}}
\newcommand{\HomotopyInverse}[1]
  {\Psi_{#1}}%
\newcommand{\CoordinateChart}%
  {V}
\newcommand{\FixedRadius}%
  {R}
\newcommand{\Radius}%
  {r}
\newcommand{\SubsurfaceDisk}[3]
  {#1_{#2,#3}}
\newcommand{\ProjectedPoint}[2]
  {\apply{#1}{#2}}
\newcommand{\DiskRadius}[1]
  {D^2_{#1}}
\newcommand{\BallRemovalMap}%
  {\Phi}

\newcommand{\SymplecticForm}%
  {\omega}
\newcommand{\SpaceSymplecticEmbeddings}[3]
  {\text{SEmb}_{#1}\!\left(#2,#3\right)}
\newcommand{\FormalSymplecticEmbedding}%
  {F}
\newcommand{\SolutionHomotopy}%
  {H}
\newcommand{\SymplecticTwoPlane}%
  {B_{\SymplecticForm}}
\newcommand{\SymplecticGrassmannian}[2]
  {\Gr^{\SymplecticForm}_{#1}\!\left(#2\right)}
\newcommand{\SymplecticPathSpace}[2]
  {\mathcal{P}_{#1}\!\left(#2\right)}
\newcommand{\SymplecticTangentialStructure}%
  {\TangentialFibration_{\SymplecticForm}}
\newcommand{\SpaceSymplecticSubsurfaces}[3]
  {\mathcal{S}_{#1}\!\left(#2,#3\right)}

\begin{document}
\begin{abstract}
  Given a manifold with boundary, one can consider the space of 
  subsurfaces of this manifold meeting the boundary in a prescribed fashion. 
  It is known that these spaces of subsurfaces satisfy homological stability if 
  the manifold has at least dimension five and is simply-connected.
  We introduce a notion of tangential structure for subsurfaces and give a 
  general criterion for when the space of subsurfaces with tangential 
  structure satisfies homological stability provided that the manifold is 
  simply-connected and has dimension $n\geq 5$. 
  Examples of tangential structures such that the spaces of subsurface with 
  that tangential structure satisfy homological stability are
  framings or spin structures of their tangent 
  bundle, or $k$-frames of the normal bundle provided that $k\leq n-2$.
  
  Furthermore we introduce spaces of pointedly embedded subsurfaces and 
  construct stabilization maps, as well as prove homological stability for 
  these. This is used to prove homological stability for spaces of symplectic 
  subsurfaces.
\end{abstract}

\maketitle
\section{Introduction}
    The term homological stability refers to the following phenomenon in 
  algebraic topology:
  Consider a sequence of spaces $\TopologicalSpace_{n}$ indexed by the natural 
  numbers together 
  with continuous maps 
  $
    \ContinuousMap_{n}
    \colon 
    \TopologicalSpace_{n}
    \to 
    \TopologicalSpace_{n+1}
  $%
  , called \introduce{stabilization maps}. 
  Such a sequence satisfies \emph{homological stability} if 
  $
    \HomologyOfMap{\ContinuousMap_{n}}
    \colon
    \HomologyOfSpace{\ast}{\TopologicalSpace_{n}}
    \to
    \HomologyOfSpace{\ast}{\TopologicalSpace_{n+1}}
  $
  is an isomorphism for $\ast\leq \apply{\AuxiliaryFunction}{n}$, where 
  $
    \AuxiliaryFunction
    \colon 
    \NaturalNumbers
    \to 
    \NaturalNumbers
  $
  is a non-decreasing function such that 
  $
    \lim_{n\to \infty} 
    \apply{\AuxiliaryFunction}{n}
    =
    \infty
  $%
  . This phenomenon has been studied intensely in the context of group 
  homology. 
  For example it was proven in \cite{Naka} that the classifying spaces of the 
  symmetric groups $\ClassifyingSpace{S_n}$ together with the 
  maps induced from the inclusion $\{0,\ldots,n\} \to \{0,\ldots,n,n+1\}$ 
  satisfy homological stability.
  
  Harer showed in \cite{Harer} that homological stability holds true for 
  the mapping class group of closed oriented surfaces of genus $\Genus$ denoted 
  by $\text{Mod}_{\Genus}$. 
  Even though in this situation we do not have stabilization maps 
  $\text{Mod}_{\Genus}\to \text{Mod}_{\Genus+1}$, there are geometrically 
  motivated maps 
  $
    \HomologyOfSpace{\ast}{\text{Mod}_{\Genus}}
    \to
    \HomologyOfSpace{\ast}{\text{Mod}_{\Genus+1}}
  $
  that are only defined in a range of degrees, where they turn out to be 
  isomorphisms. Actually, homological stability for the mapping class group of 
  surfaces is more naturally formulated for surfaces with boundary, since their 
  mapping class groups have stabilization maps stemming from the inclusions of 
  the surfaces. Then the homology of the mapping class group of closed surfaces 
  can be related to the homology of the mapping class group of surfaces with 
  boundary in a range of degrees increasing with the genus. 
  A similar situation will repeat itself in the present text as well.
  
  Similar results were obtained for the mapping class groups of non-orientable 
  surfaces (see \cite{WahlHomologicalStability}), the spin mapping class groups 
  (see 
  \cite{HarerSpin}) and the mapping class groups of $3$-manifolds (see 
  \cite{HatcherWahlHomologicalStability}).  In \cite{RWFramed} it was shown, 
  among other things, that homological stability holds for the framed mapping 
  class group and generalized spin mapping class groups of surfaces.
  
  Since models for the classifying space of the mapping class group resemble 
  the spaces studied in this paper, the example of homological stability of the 
  mapping class groups of a surface is of particular interest for this paper.
  Let $\SurfaceG{\Genus}$ denote a closed, orientable and connected 
  surface of genus $\Genus$, and $\DiffeomorphismGroup{\SurfaceG{\Genus}}$ 
  its orientation preserving diffeomorphism group. 
  We will denote the space of smooth embeddings of 
  $\SurfaceG{\Genus}$ into a possibly infinite-dimensional manifold $\Manifold$ 
  by $\EmbeddingSpace{\SurfaceG{\Genus}}{\Manifold}$. 
  A model for the classifying space $\ClassifyingSpace{\text{Mod}_{\Genus}}$ is 
  given as follows:
  \[
    \ClassifyingSpace{\text{Mod}_{\Genus}}
    \coloneqq
    \EmbeddingSpace{\SurfaceG{\Genus}}{\Reals^{\infty}}
    /
    \DiffeomorphismGroup{\SurfaceG{\Genus}}
  \]
  Here the diffeomorphism group of the surface acts via precomposition.
  Replacing $\Reals^\infty$ in this construction by a finite dimensional 
  manifold $\Manifold$ yields the moduli space of 
  submanifolds of $\Manifold$ diffeomorphic to $\SurfaceG{\Genus}$. 
  Hence points in this quotient will be refered to as \introduce{subsurfaces}.
  This space classifies $\SurfaceG{\Genus}$-bundles together with an embedding 
  in a trivial $\Manifold$-bundle.
  These spaces and their relatives will be the main focus of this paper.
  
  Analogously to homological stability of the mapping class group of 
  surfaces, it will be more natural to define stabilization maps for 
  subsurfaces with boundary:
  Let $\SurfaceGB{\Genus}{\BoundaryComponents}$ denote an orientable, compact, 
  connected surface of genus $\Genus$ with $\BoundaryComponents$ boundary 
  components.
  Suppose $\Manifold$ has non-empty boundary, then we can consider subsurfaces 
  diffeomorphic to $\SurfaceGB{\Genus}{\BoundaryComponents}$ that intersect 
  $\BoundaryManifold{\Manifold}$ in a fixed set of embedded circles 
  $\BoundaryCondition$. One can define stabilization 
  maps: 
  \[
    \EmbeddingSpaceBoundaryCondition
      {\SurfaceGB{\Genus}{\BoundaryComponents}}
      {\Manifold}
      {\BoundaryCondition}
    /
    \DiffeomorphismGroup{\SurfaceGB{\Genus}{\BoundaryComponents}}
    \to
    \EmbeddingSpaceBoundaryCondition
      {\SurfaceGB{\Genus'}{\BoundaryComponents'}}
      {\Manifold}
      {\BoundaryCondition}
    /
    \DiffeomorphismGroup{\SurfaceGB{\Genus'}{\BoundaryComponents'}}
  \] 
  These change the number of boundary components and the genus.
  It was shown in \cite{CRW16} that if $\Manifold$ is at least $5$-dimensional 
  and simply-connected, then these stabilization maps induce isomorphism 
  in integral homology in a range of degrees that depends on the genus 
  $\Genus$.
  
  In this paper we will combine the techniques of \cite{CRW16} with the methods 
  introduces in \cite{RW16}, a paper where a unified approach to homological 
  stability for mapping class groups of surfaces with tangential structures was 
  given, to prove homological stability for spaces of subsurfaces that are 
  equipped with certain relative tangential structures.
  Furthermore we will extend these results to yield homological stability 
  results for certain spaces of closed subsurfaces of closed manifolds. Let 
  us proceed by introducing these relative tangential structures.
  
\paragraph{Subsurfaces with Tangential Structure:} 
  Let $\Manifold$ denote a smooth manifold possibly with boundary, let 
  $\Grassmannian{2}{\TangentBundle{\Manifold}}$ denote the Grassmannian of 
  oriented $2$-bundles in $\TangentBundle{\Manifold}$ and note that an 
  embedding 
  $
    \Embedding
    \colon
    \SurfaceGB{\Genus}{\BoundaryComponents}
    \to 
    \Manifold
  $ 
  lifts to the Grassmannian as a map 
  \begin{align*}
    \GrassmannianDifferential{\Embedding}
     \colon
     \SurfaceGB{\Genus}{\BoundaryComponents}
     &
     \to
     \Grassmannian{2}{\TangentBundle{\Manifold}}
     \\
     \Point
     &
     \mapsto
     \apply
      {\Differential{\Embedding}}
      {\TangentSpace{\Point}{\SurfaceGB{\Genus}{\BoundaryComponents}}}
  \end{align*}
  Let $\TangentialSpace{\Manifold}$ denote a topological space together with a 
  continuous Hurewicz-fibration 
  $
    \TangentialFibration
    \colon
    \TangentialSpace{\Manifold}
    \to 
    \Grassmannian{2}{\TangentBundle{\Manifold}}
  $%
  . We will call a lift $\TangentialStructure{\Embedding}$ of 
  $\GrassmannianDifferential{\Embedding}$ to $\TangentialSpace{\Manifold}$ a 
  \introduce{$\TangentialFibration$-structure for $\Embedding$} and refer to 
  the pair $(\Embedding,\TangentialStructure{\Embedding})$ as an 
  \introduce{embedding with tangential structure}. This is encapsulated in the 
  following diagram:
  \begin{equation}
    \label{eqt:TangentialStructureDiagram}
    \begin{tikzcd}
      &
      \TangentialSpace{\Manifold}
        \ar[d,"\TangentialFibration"]
      \\
      &
      \Grassmannian{2}{\TangentBundle{\Manifold}}
        \ar[d]
      \\
      \SurfaceGB{\Genus}{\BoundaryComponents}
        \ar[uur,"\TangentialStructure{\Embedding}"]
        \ar[ur] 
        \ar[r,"\Embedding"'] 
      &
      \Manifold
    \end{tikzcd}
  \end{equation}
  
  We will refer to 
  $
    \TangentialFibration
    \colon 
    \TangentialSpace{\Manifold}
    \to
    \Grassmannian{2}{\TangentBundle{\Manifold}}
  $
  as a \introduce{space of $\TangentialFibration$-structures of subplanes of 
  $\TangentBundle{\Manifold}$}.
  \begin{example}
    Fix a metric and an orientation on $\Manifold$, then there is an 
    isomorphism 
    $
      \Grassmannian{2}{\TangentBundle{\Manifold}}
      \cong
      \Grassmannian{\DimensionIndex-2}{\TangentBundle{\Manifold}}
    $%
    . This isomorphism corresponds to taking the normal bundle of the 
    subsurfaces and a tangential structure corresponds to a tangential 
    structure on the normal bundle. 

    Particular examples of these tangential structures are provided by the 
    various Stiefel-manifolds associated to the tautological bundle over 
    $\Grassmannian{\DimensionIndex-2}{\TangentBundle{\Manifold}}$.
    In this case an embedding with tangential structure is an embedding 
    together with a $k$-frame of its normal bundle. 
  
    Other examples are given by
    $
      \TangentialSpace{\Manifold}
      =
      \Grassmannian{2}{\TangentBundle{\Manifold}}
      \times
      \TopologicalSpace
    $
    where $\TopologicalSpace$ is an arbitrary topological space.

    A different example will be introduced in
    Section~\ref{scn:SymplecticSubsurfaces}.
  \end{example}
  \begin{example}
  \label{exm:NonRelativeTangentialStructures}
    This definition of tangential structures should be thought of as a relative 
    tangential structure in contrast to the usual definition of a tangential 
    structure in terms of a fibration 
    $
      \TangentialFibration'
      \colon
      B
      \to 
      \Grassmannian{2}{\Reals^{\infty}}
    $%
    . Nevertheless one can recover the usual definition as a special 
    case of the present definition by pulling back 
    $
      \TangentialFibration'
      \colon
      B
      \to 
      \Grassmannian{2}{\Reals^{\infty}}
    $ 
    via a classifying map of the tautological bundle over 
    $\Grassmannian{2}{\TangentBundle{\Manifold}}$ to get a tangential structure 
    $\TangentialFibration$ in our sense. 
  \end{example}
  We define the \introduce{space of embeddings with 
  $\TangentialFibration$-structure} as follows:
  \[
    \EmbeddingSpaceTangential
      {\TangentialFibration}
      {\SurfaceGB{\Genus}{\BoundaryComponents}}
      {\Manifold}
    \coloneqq
    \left\{
      (\Embedding,\TangentialStructure{\Embedding})
    \middle|
      \begin{aligned}
        &
        \Embedding
        \colon 
        \SurfaceGB{\Genus}{\BoundaryComponents}
        \to
        \Manifold 
        \text{ is an embedding and }
        \\ 
        &
        \TangentialStructure{\Embedding}
        \colon 
        \SurfaceGB{\Genus}{\BoundaryComponents}
        \to 
        \TangentialSpace{\Manifold}
        \text{
           makes Diagram~(\ref{eqt:TangentialStructureDiagram}) commute
         }
      \end{aligned}
    \right\}
  \]
  The definition of the topology of this space will be given in 
  Section~\ref{scn:SpacesOfSubsurfaces}. 
  Let $\AuxiliaryFunction$ denote an orientation preserving diffeomorphism of 
  $\SurfaceGB{\Genus}{\BoundaryComponents}$, note
  $
    \GrassmannianDifferential{\Embedding\circ \AuxiliaryFunction}
    =
    \GrassmannianDifferential{\Embedding}
    \circ
    \AuxiliaryFunction
  $%
  . Therefore 
  $\DiffeomorphismGroup{\SurfaceGB{\Genus}{\BoundaryComponents}}$ has a natural 
  and free action on 
  $
    \EmbeddingSpaceTangential
      {\TangentialFibration}
      {\SurfaceGB{\Genus}{\BoundaryComponents}}
      {\Manifold}
  $
  by precomposition. 
  
  In order to have a reasonable notion of stabilization, we will need to 
  restrict 
  ourselves to embeddings that intersect the boundary of $\Manifold$ in a fixed 
  fashion: 
  
  Suppose that $\BoundaryManifold{\Manifold}$ is non-empty and that we have 
  specified a connected component $\DistBoundaryManifold{0}{\Manifold}$ of 
  $\BoundaryManifold{\Manifold}$.
  
  We say that two embeddings with tangential structure
  $
    \Embedding_{1},\Embedding_{2}
    \in
    \EmbeddingSpaceTangential
      {\TangentialFibration}
      {\SurfaceGB{\Genus}{\BoundaryComponents}}
      {\Manifold}
  $ 
  \introduce{have the same jet along $\BoundaryManifold{\Manifold}$} if the 
  $\infty$-jet at 
  $
    \BoundaryManifold{\SurfaceGB{\Genus}{\BoundaryComponents}}
  $ 
  of $\Embedding_{1}$ and $\Embedding_{2}$ agree, and 
  $
    \at
      {\TangentialStructure{\Embedding_{1}}}
      {\BoundaryManifold{\SurfaceGB{\Genus}{\BoundaryComponents}}}
    =
     \at
      {\TangentialStructure{\Embedding_{2}}}
      {\BoundaryManifold{\SurfaceGB{\Genus}{\BoundaryComponents}}}
  $%
  .
  
  Fix an $\infty$-jet at 
  $\BoundaryManifold{\SurfaceGB{\Genus}{\BoundaryComponents}}$ of embeddings 
  with tangential structure of 
  $\SurfaceGB{\Genus}{\BoundaryComponents}$ into $\Manifold$, denoted by 
  $\BoundaryConditionTangential{\BoundaryCondition}$. 
  Consider the space of embeddings, whose $\infty$-jet at 
  $
    \BoundaryManifold{\SurfaceGB{\Genus}{\BoundaryComponents}}
  $
  agrees with $\BoundaryConditionTangential{\BoundaryCondition}$.
  Note that the group of diffeomorphisms that fixes the $\infty$-jet of the 
  identity at $\BoundaryManifold{\SurfaceGB{\Genus}{\BoundaryComponents}}$, 
  denoted by 
  $
    \DiffeomorphismGroupBoundary{\SurfaceGB{\Genus}{\BoundaryComponents}}
  $%
  , acts freely on this subspace of the space of embeddings.
  We denote the quotient by 
  $
    \SpaceSubsurfaceBoundaryTangential
      {\TangentialFibration}
      {\Genus}
      {\BoundaryComponents}
      {\Manifold}
      {\BoundaryConditionTangential{\BoundaryCondition}}
  $
  and call it the \introduce{space of subsurfaces with 
  $\TangentialFibration$-structure}, elements of this space are called
  \introduce{subsurfaces with $\TangentialFibration$-structure}. 
  The fixed boundary condition enables us to stabilize these spaces in the 
  following way:
  
  We can embed $\DistBoundaryManifold{0}{\Manifold} \times [0,1]$ into 
  $\Manifold$ via a collar that identifies 
  $
    \DistBoundaryManifold{0}{\Manifold}
    \times
    \{1\}
  $ 
  with $\DistBoundaryManifold{0}{\Manifold}$ via the identity and we pullback 
  $
    \TangentialFibration
    \colon
    \TangentialSpace{\Manifold}
    \to
    \Grassmannian{2}{\TangentBundle{\Manifold}}
  $ 
  to $\Grassmannian{2}{\DistBoundaryManifold{0}{\Manifold}\times[0,1]}$. This 
  yields a space of $\TangentialFibration$-structures on subplanes in 
  $\DistBoundaryManifold{0}{\Manifold}\times[0,1]$.
  Let $\StabilizationBordism$ denote a subsurface with tangential structure of 
  $\DistBoundaryManifold{0}{\Manifold}\times[0,1]$ such that every connected 
  component of $\StabilizationBordism$ has a non-empty intersection with 
  $\DistBoundaryManifold{0}{\Manifold}\times\{0\}$. 
  If we assume that 
  $
    \StabilizationBordism
    \cap
    \DistBoundaryManifold{0}{\Manifold}
    \times
    \{0\}
  $
  agrees with the image of the jet 
  $\BoundaryConditionTangential{\BoundaryCondition}$ and that the higher jets
  are compatible as well, then the following map is well-defined:
  \begin{align*}
    -\cup \StabilizationBordism
    \colon 
    \SpaceSubsurfaceBoundaryTangential
      {\TangentialFibration}
      {\Genus}
      {\BoundaryComponents}
      {\Manifold}
      {\BoundaryConditionTangential{\BoundaryCondition}}
    &
    \to 
    \SpaceSubsurfaceBoundaryTangential
      {\TangentialFibration}
      {\Genus'}
      {\BoundaryComponents'}
      {
        \Manifold\
        \cup_{\DistBoundaryManifold{0}{\Manifold}\times\{0\}}
        \DistBoundaryManifold{0}{\Manifold}
        \times
        [0,1]
      }
      {\BoundaryConditionTangential{\BoundaryCondition'}}
    \\
    \Subsurface \subset \Manifold 
    &
    \mapsto 
    \Subsurface \cup \StabilizationBordism 
    \subset
    \Manifold\
    \cup_{\DistBoundaryManifold{0}{\Manifold}\times\{0\}}
    \DistBoundaryManifold{0}{\Manifold}
    \times
    [0,1]
  \end{align*}
  Here $\Genus',\BoundaryComponents'$ and 
  $\BoundaryConditionTangential{\BoundaryCondition'}$ depend on the topology of 
  $\StabilizationBordism$. 
  By identifying 
  $
    \Manifold\
    \cup_{\DistBoundaryManifold{0}{\Manifold}\times\{0\}}
    \DistBoundaryManifold{0}{\Manifold}
    \times
    [0,1]
  $ 
  with $\Manifold$, we get stabilization maps of spaces of subsurfaces with 
  tangential structure of $\Manifold$. 
  We say that a tangential structure $\TangentialFibration$ 
  \introduce{satisfies homological stability} if $-\cup \StabilizationBordism$ 
  induces isomorphisms in integral homology in a range of degrees that 
  increases with $\Genus$ and an epimorphism in the next degree i.e. the 
  reduced homology of the mapping cone of $-\cup \StabilizationBordism$ 
  vanishes in a range of degrees increasing with $\Genus$.
  
  With these notions we can paraphrase our main result, which can be understood 
  as an adaptation of Theorem~1.2 in \cite{RW16} to spaces of subsurfaces:
  \begin{theorem}
\label{thm:HomologicalStabilityParaphrased}
  Suppose $\Manifold$ is an at least $5$-dimensional, simply-connected 
  manifold with non-empty boundary and 
  $\DistBoundaryManifold{0}{\Manifold}$ is a codimension $0$ simply-connected  
  submanifold of $\BoundaryManifold{\Manifold}$ together with a space of 
  $\TangentialFibration$-structures on subplanes of 
  $\TangentBundle{\Manifold}$. 
  Suppose further that $\TangentialFibration$ stabilizes on connected 
  components, i.e. 
  \[
    \HomotopyGroupMap{- \cup \StabilizationBordism}
    \colon 
    \HomotopyGroup
      {0}
      {
        \SpaceSubsurfaceBoundaryTangential
        {\TangentialFibration}
        {\Genus}
        {\BoundaryComponents}
        {\Manifold}
        {\BoundaryConditionTangential{\BoundaryCondition}}
      }
    \to
    \HomotopyGroup
      {0}
      { 
        \SpaceSubsurfaceBoundaryTangential
        {\TangentialFibration}
        {\Genus'}
        {\BoundaryComponents'}
        {
          \Manifold\
          \cup_{\DistBoundaryManifold{0}{\Manifold}\times\{0\}}
          \DistBoundaryManifold{0}{\Manifold}
          \times
          [0,1]
        }
        {\BoundaryConditionTangential{\BoundaryCondition'}}
      }
  \]
  and 
  \[
    \HomotopyGroupMap{- \cup \StabilizationBordism}
    \colon 
    \HomotopyGroup
    {0}
    {
      \SpaceSubsurfaceBoundaryTangential
        {\TangentialFibration}
        {\Genus}
        {\BoundaryComponents}
        {\DistBoundaryManifold{0}{\Manifold}\times[-1,0]}
        {\BoundaryConditionTangential{\BoundaryCondition}}
    }
    \to
    \HomotopyGroup
    {0}
    { 
      \SpaceSubsurfaceBoundaryTangential
        {\TangentialFibration}
        {\Genus'}
        {\BoundaryComponents'}
        {\DistBoundaryManifold{0}{\Manifold}\times[-1,1]}
        {\BoundaryConditionTangential{\BoundaryCondition'}}
    }
  \]
  are isomorphisms if $\Genus$ is high enough, then 
  $\TangentialFibration$ satisfies homological stability.
\end{theorem}
  See Theorem~\ref{thm:HomologicalStability} for the most general statement 
  including the bounds and a wider class of tangential structures. Furthermore 
  Theorem~\ref{thm:HomologicalStability} does not need that 
  $\DistBoundaryManifold{0}{\Manifold}$ is simply-connected.
  Theorem~\ref{thm:HomologicalStabilityConnectedComponents} is the more precise 
  version of \ref{thm:HomologicalStabilityParaphrased}. 
  
  The following theorem (See
  Theorem~\ref{thm:HomologicalStabilityFiberSimplyConnected} for a more precise 
  statement, that also includes the bounds) 
  yields a large number of tangential structures which satisfy the assumptions 
  in Theorem~\ref{thm:HomologicalStabilityParaphrased}. 
  \begin{theorem}
\label{thm:TangentialStructure0TrivialParaphrased}
  In the setting of Theorem~\ref{thm:HomologicalStabilityParaphrased},
  a space of $\TangentialFibration$-structures of subplanes 
  $
    \TangentialFibration
    \colon
    \TangentialSpace{\Manifold}
    \to
    \Grassmannian{2}{\TangentBundle{\Manifold}}
  $
  fulfils the assumptions in 
  Theorem~\ref{thm:HomologicalStabilityParaphrased} i.e. 
  \[
    \HomotopyGroupMap{- \cup \StabilizationBordism}
    \colon 
    \HomotopyGroup
    {0}
    {
      \SpaceSubsurfaceBoundaryTangential
        {\TangentialFibration}
        {\Genus}
        {\BoundaryComponents}
        {\Manifold}
        {\BoundaryConditionTangential{\BoundaryCondition}}
    }
    \to
    \HomotopyGroup
    {0}
    { 
      \SpaceSubsurfaceBoundaryTangential
        {\TangentialFibration}
        {\Genus'}
        {\BoundaryComponents'}
        {
          \Manifold\
          \cup_{\DistBoundaryManifold{0}{\Manifold}\times\{0\}}
          \DistBoundaryManifold{0}{\Manifold}
          \times
          [0,1]
        }
        {\BoundaryConditionTangential{\BoundaryCondition'}}
    }
  \]
  and
  \[
    \HomotopyGroupMap{- \cup \StabilizationBordism}
    \colon 
    \HomotopyGroup
    {0}
    {
      \SpaceSubsurfaceBoundaryTangential
        {\TangentialFibration}
        {\Genus}
        {\BoundaryComponents}
        {\DistBoundaryManifold{0}{\Manifold}\times[-1,0]}
        {\BoundaryConditionTangential{\BoundaryCondition}}
    }
    \to
    \HomotopyGroup
    {0}
    { 
      \SpaceSubsurfaceBoundaryTangential
        {\TangentialFibration}
        {\Genus'}
        {\BoundaryComponents'}
        {\DistBoundaryManifold{0}{\Manifold}\times[-1,1]}
        {\BoundaryConditionTangential{\BoundaryCondition'}}
    }
  \]
  is an isomorphism if $\Genus$ is high enough, if the homotopy fiber of 
  $
    \TangentialFibration
    \colon
    \TangentialSpace{\Manifold}
    \to
    \Grassmannian{2}{\TangentBundle{\Manifold}}
  $
  is simply-connected. 
  
  In particular such a tangential structure satisfies homological stability.
\end{theorem}
  Examples of such spaces of tangential structures are given by $k$-framings of 
  the normal bundle, where $k\leq\DimensionIndex-2$. It would interesting to 
  see wether one can use 
  Theorem~\ref{thm:HomologicalStabilityConnectedComponents} to prove that 
  homological stability also holds for framings of the normal bundle, similar 
  to the situation in \cite{RWFramed}. Another example, related 
  to symplectic structures, will be introduced in 
  Section~\ref{scn:SymplecticSubsurfaces}.
  
  Furthermore we prove in 
  Theorem~\ref{thm:HomologicalStabilityNonRelativeTangentialStructures} that
  a tangential structure $\TangentialFibration$ as in   
  Example~\ref{exm:NonRelativeTangentialStructures} satisfies the assumptions 
  of Theorem~\ref{thm:HomologicalStabilityParaphrased}, if the corresponding 
  moduli spaces of abstract surfaces 
  $
  \ModuliSpace
    {\TangentialFibration'}
    {\SurfaceGB{\Genus}{\BoundaryComponents}}
  $
  introduced in Definition~1.1 in \cite{RW16} fulfil homological stability (see 
  Theorem~7.1 in 
  \cite{RW16} for assumptions implying homological stability). In particular it 
  is proven in \cite{RW16}, that homological stability for the abstract moduli 
  space holds, if the stabilization maps in \cite{RW16} yield 
  bijections
  \[
    \HomotopyGroup
    {0}
    {
      \ModuliSpace
        {\TangentialFibration'}
        {\SurfaceGB{\Genus}{\BoundaryComponents}}
    }
    \to
    \HomotopyGroup
    {0}
    { 
      \ModuliSpace
        {\TangentialFibration'}
        {\SurfaceGB{\Genus}{\BoundaryComponents}}
    }
  \]
  provided that $\Genus$ is large enough.
  
  Examples of such "non-relative" tangential structures are given by framings 
  of the tangent bundle, or spin structures of the subsurfaces. As was 
  mentioned before, homological stability for these was proven in 
  \cite{RWFramed}.
  
  Since some tangential structures are more naturally defined in the context of 
  closed subsurfaces of closed manifolds, we will prove a version of 
  Theorem~\ref{thm:HomologicalStabilityParaphrased} for closed subsurfaces:
  
  \paragraph{Homological Stability for Pointed Subsurfaces:}
  Let $\Manifold$ denote a closed manifold with space of 
  $\TangentialFibration$-structures of subplanes of $\TangentBundle{\Manifold}$,
  fix $\Point \in \SurfaceG{\Genus}$, $\Point_{\Manifold}\in \Manifold$ and 
  $
    \TwoPlane
    \in
    \apply
      {\TangentialFibration^{-1}}
      {\Grassmannian{2}{\TangentSpace{\Point_{\Manifold}}{\Manifold}}}
  $%
  , define
  \[
    \EmbeddingSpacePointedTangential
      {\TangentialFibration}
      {\TwoPlane}
      {\SurfaceG{\Genus}}
      {\Manifold}
    \coloneqq
    \left\{
      \left(\Embedding,\TangentialStructure{\Embedding}\right)
      \in 
      \EmbeddingSpaceTangential
        {\TangentialFibration}
        {\SurfaceG{\Genus}}
        {\Manifold}
    \middle| 
      \apply{\Embedding}{\Point}
      =
      \Point_{\Manifold}
      \text{, } 
      \apply{\TangentialStructure{\Embedding}}{\Point}
      =
      \TwoPlane
    \right\}
  \]
  As before, the group of orientation-preserving diffeomorphisms of 
  $\SurfaceG{\Genus}$ that fix $\Point$, denoted by 
  $\DiffeoMorphismGroupPointed{\Point}{\SurfaceG{\Genus}}$ admits a natural and 
  free action on this space of embeddings by precomposition. 
  We will denote the quotient of this action by 
  $
    \SpaceSubsurfacesPointedTangential
      {\TangentialFibration}
      {\TwoPlane}
      {\SurfaceG{\Genus}}
      {\Manifold}
  $%
  . This space is called the \introduce{space of pointedly embedded 
  $\TangentialFibration$-subsurfaces of $\Manifold$}. In 
  Section~\ref{scn:PointedStabilization} we will construct a 
  stabilization map 
  \[
    \PointedStabilizationMap{\Genus}
    \colon
    \SpaceSubsurfacesPointedTangential
      {\TangentialFibration}
      {\TwoPlane}
      {\SurfaceG{\Genus}}
      {\Manifold}
    \to
    \SpaceSubsurfacesPointedTangential
      {\TangentialFibration}
      {\TwoPlane}
      {\SurfaceG{\Genus+1}}
      {\Manifold}
  \]
  which is heuristically given by flattening the subsurfaces in a neighborhood 
  of $\Point_{\Manifold}$ and then taking a connected sum with a torus in this 
  neighborhood.
  We also prove that this stabilization map yields homological stability for 
  many kinds of tangential structures   
  (Theorem~\ref{thm:HomologicalStabilityPointed}). Parts of this theorem can be 
  summarized as follows:
  \begin{theorem}
\label{thm:HomologicalStabilityPointedParaphrased}
  Suppose $(\Manifold,\Point_{\Manifold})$ is an at least $5$-dimensional, 
  simply-connected, pointed manifold and suppose further that 
  $\TangentialFibration$ is a space of
  tangential structures such that the homotopy fiber of $\TangentialFibration$ 
  is simply-connected, then $\PointedStabilizationMap{\Genus}$ 
  induces an isomorphism in integral homology in degrees $\ast\leq 
  \frac{2}{3}\Genus-1$ and an epimorphism in degrees less than 
  $\frac{2}{3}\Genus$.
\end{theorem}
  
\paragraph{Symplectic Subsurfaces:}
  As an application of this, we prove in 
  Section~\ref{scn:SymplecticSubsurfaces} a homological stability result 
  for spaces of symplectic subsurfaces:
  Let $(\Manifold,\SymplecticForm)$ 
  denote a simply connected symplectic manifold of dimension at least $6$ with 
  a base point $\Point_{\Manifold}\in\Manifold$. 
  We will call an embedding 
  $
    \Embedding
    \colon 
    \SurfaceG{\Genus} 
    \to 
    \Manifold
  $
  an oriented symplectic embedding if $\Embedding^{\ast}\SymplecticForm$ is a 
  symplectic form on $\SurfaceG{\Genus}$ and 
  $\int_{\SurfaceG{\Genus}} \Embedding^*\SymplecticForm>0$. 
  Fixing a symplectic $2$-plane
  $
    \SymplecticTwoPlane
    \in 
    \Grassmannian{2}{\TangentSpace{\Point_{\Manifold}}{\Manifold}}
  $%
  , i.e. $\at{\SymplecticForm}{\SymplecticTwoPlane}$ is non-degenerate, we can 
  consider 
  $
    \SpaceSymplecticEmbeddings
      {\SymplecticTwoPlane}
      {\SurfaceG{\Genus}}
      {\Manifold}
  $%
  , the space of oriented symplectic embeddings that map the tangent space of 
  the base point of the surface to $\SymplecticTwoPlane$. 
  Being an oriented symplectic embedding is invariant under the action of 
  $\DiffeoMorphismGroupPointed{\Point}{\SurfaceG{\Genus}}$ by precomposition. 
  We denote the quotient by 
  $
    \SpaceSymplecticSubsurfaces
      {\SymplecticTwoPlane}
      {\SurfaceG{\Genus}}
      {\Manifold}
  $%
  and call it the \introduce{space of symplectic subsurfaces which are 
  tangential to $\TwoPlane$}. 
  
  A proposition in the context of the h-principle, namely a 
  variation of Theorem~12.1.1 in \cite{Eliashberg}, allows us to relate the 
  space of symplectic subsurfaces to spaces of subsurfaces with a certain 
  tangential structure (See 
  Proposition~\ref{prp:hPrincipleSpacesOfSubsurfaces}). 
  Using Theorem~\ref{thm:HomologicalStabilityPointedParaphrased}, one obtains:
  \begin{theorem}
\label{thm:SymplecticHomologicalStability}
  Let $(\Manifold,\SymplecticForm)$ denote a simply-connected symplectic 
  manifold of dimension at least $6$. 
  For every $\Point_{\Manifold}\in \Manifold$ and every symplectic $2$-plane 
  $\SymplecticTwoPlane$ in $\TangentSpace{\Point_{\Manifold}}{\Manifold}$. 
  There is a homomorphism of integral homology
  \[
    \PointedStabilizationMap{\Genus}
    \colon
    \HomologyOfSpace
      {\ast}
      {\SpaceSymplecticSubsurfaces
        {\SymplecticTwoPlane}
        {\SurfaceG{\Genus}}
        {\Manifold}
      }
      \to
      \HomologyOfSpace
        {\ast}
        {\SpaceSymplecticSubsurfaces
          {\SymplecticTwoPlane}
          {\SurfaceG{\Genus+1}}
          {\Manifold}
        }
  \]
  And this homomorphism induces an isomorphism for 
  $\ast \leq \frac{2}{3}\Genus-1$ and an epimorphism for 
  $\ast\leq \frac{2}{3}\Genus$.
\end{theorem}
  Even though the present stabilization map exists only on the level of 
  homology, the author is hopeful that with the presented methods, it is 
  possible to construct an actual map 
  $
    \SpaceSymplecticSubsurfaces
      {\SymplecticTwoPlane}
      {\SurfaceG{\Genus}}
      {\Manifold}
    \to 
    \SpaceSymplecticSubsurfaces
      {\SymplecticTwoPlane}
      {\SurfaceG{\Genus+1}}
      {\Manifold}
  $ that realizes the map on homology (see the end of 
  Section~\ref{scn:SymplecticSubsurfaces} for 
  a more detailed discussion of this).
  
  It was shown in \cite{LiSymplecticSurfaces} that many second homology classes 
  can be represented by symplectic subsurfaces. Furthermore essential tools 
  in the theory of symplectic manifolds are $J$-holomorphic curves and their 
  moduli spaces (See \cite{McDuffSalamon} for an introduction), embedded 
  $J$-holomorphic curves are essentially the same as symplectic subsurfaces and 
  their moduli spaces are connected to the space of symplectic subsurfaces. We 
  hope that the homological stability result might give some new 
  insight into these moduli spaces.
  
\paragraph{Overview:}
  The first two sections describe spaces of subsurfaces with boundary, 
  tangential structures and related concepts used throughout the paper 
  (Section~\ref{scn:SpacesOfSubsurfaces}) and their stabilization maps 
  (Section~\ref{scn:Stabilization}). Furthermore we show that every 
  stabilization map can be decomposed in to simpler maps, given by stabilizing 
  with a pair of pants or a disk) for which we will prove homological stability.
  
  Section~\ref{scn:SemiSimplicial} introduces the most important tool in 
  the proof of Theorem~\ref{thm:HomologicalStabilityParaphrased} and 
  Theorem~\ref{thm:HomologicalStability}: Semi-simplicial resolutions.
  In Section~\ref{scn:Resolution} a semi-simplicial resolution of the space of 
  subsurfaces is built and its behaviour with respect to stabilization maps is 
  studied.
  
  In Section~\ref{scn:kTriviality} we introduce the notion of 
  $\TrivialityIndex$-triviality, a concept developed in \cite{RW16}, and adapt 
  it to spaces of subsurfaces instead of abstract surfaces.
  Using this notion we introduce in Section~\ref{scn:Proof} the actual 
  stability bounds and prove most of Theorem~\ref{thm:HomologicalStability} the 
  stronger version of Theorem~\ref{thm:HomologicalStabilityParaphrased}. 
  Section~\ref{scn:CappingOff} finishes the proof by explaining how to remove 
  boundary components.
  
  In Section~\ref{scn:TangentialkTriviality} we establish criteria for spaces 
  of tangential structures to be suitable for application of 
  Theorem~\ref{thm:HomologicalStability} and in doing so proof among other 
  results the strengthened 
  version of Theorem~\ref{thm:TangentialStructure0TrivialParaphrased}.
  In Section~\ref{scn:PointedStabilization}, we extend the notions of spaces of 
  subsurfaces to spaces of pointedly embedded subsurfaces and construct a 
  stabilization map for these.
  
  Lastly in Section~\ref{scn:SymplecticSubsurfaces}, we apply the results of 
  Section~\ref{scn:TangentialkTriviality} and 
  Section~\ref{scn:PointedStabilization} to prove 
  Theorem~\ref{thm:SymplecticHomologicalStability}, the homological stability 
  result for symplectic subsurfaces.
  
  Most of the setting and the general proof strategy in this paper were 
  inspired by \cite{CRW16}, and represent an adaptation of their ideas to the 
  newly introduced setting of tangential structures for subsurfaces. However 
  the introduction of the concept of $\TrivialityIndex$-triviality (adapted 
  from \cite{RW16}) made the proof more intricate. 
  
  The results of Sections~\ref{scn:TangentialkTriviality} to 
  \ref{scn:SymplecticSubsurfaces} are new and allow a much broader application 
  as well as an extension of the results of \cite{CRW16}.
  
\paragraph{Acknowledgements:}
  I would like to thank my PhD advisor Ursula Hamenst\"adt for many 
  helpful discussions and suggestions and furthermore for introducing me to 
  the topic of homological stability. 
  I would also like to thank Mark Pedron for plenty of discussions about this 
  subject and for helping me making the present text much more accessible. 
  Frederico Cantero helped me understand parts of \cite{CRW16} and he helped me 
  to put things into the correct context, for which I am very grateful. 
  Furthermore I would like to express my gratitude for numerous corrections and 
  hints at possible extensions of the results of this paper by an anonymous 
  referee.
  Lastly I would like to thank Peter Michor for an answer on MathOverflow 
  (\cite{MichorMathoverflow}) without which I would not have been able to 
  finish the proof of Proposition~\ref{prp:FlatHomotopyEquivalence}.
\section{Spaces of Subsurfaces and Tangential Structures}
  \label{scn:SpacesOfSubsurfaces}
The following section will introduce spaces of subsurfaces, tangential 
structures and will also lay most of the technical foundation for the rest of 
this paper.
\subsection{Spaces of Embeddings and Embedded Submanifolds}
Let 
$
  \left(
    \Manifold,\BoundaryManifold{\Manifold}
  \right)
$
denote a compact smooth manifold with boundary and let 
$\DistBoundaryManifold{0}{\Manifold}$ 
denote a codimension $0$ submanifold (possibly with boundary) of 
$\BoundaryManifold$.
We call an embedding 
$
  \Collar
  \colon 
  \BoundaryManifold{\Manifold}
  \times 
  \left[0,1\right)
  \to
  \Manifold
$
such that 
$
  \apply{\Collar}
    {
      \left(\Point,0\right)
    }
  =
  \Point
$
a \introduce{collar of $\Manifold$}.

We say that $\Manifold$ is \introduce{collared} if there is a 
fixed collar $\Collar$ of $\Manifold$. 
Note that if $\Manifold$ is collared, then 
\[
  \Elongation{\Manifold}{k}
  \coloneqq
  \Manifold 
  \cup_{\DistBoundaryManifold{0}{\Manifold}} 
  \DistBoundaryManifold{0}{\Manifold}
  \times 
  \left[0,k\right] 
\]
has a canonical smooth structure.

\begin{definition}
  Let $\Manifold_{1}$ and $\Manifold_{2}$ denote smooth collared 
  manifolds.
  A \introduce{smooth map 
  $
    \SmoothMap
    \colon
    \Manifold_{1}
    \to
    \Manifold_{2}
  $
  between collared manifolds} is a smooth  map that maps 
  $\BoundaryManifold{\Manifold_{1}}$ to 
  $\BoundaryManifold{\Manifold_{2}}$ and there exist neighborhoods of 
  $\BoundaryManifold{\Manifold_{1}}$ and $\BoundaryManifold{\Manifold_{2}}$
  such that in these neighborhoods $\SmoothMap$ is a product of a smooth map 
  and the identity with respect to the collars of $\Manifold_{1}$ and 
  $\Manifold_{2}$.
  
  The definition of \introduce{collared embedding} and 
  \introduce{collared diffemorphism} is completely analogous.
\end{definition}

Since every occuring manifold and map will be assumed to be a collared manifold 
or a collared map we will drop the adjective.

A collared embedding from 
$\Manifold_{1}$ to $\Manifold_{2}$ maps the boundary of 
$\Manifold_{1}$ to the boundary of $\Manifold_{2}$ and the map is 
transverse to the boundary of $\Manifold_{2}$.
If $\Submanifold_{1}\subset \Manifold_{1}$ and 
$\Submanifold_{2}\subset \Manifold_{2}$ are submanifolds, 
we call an embedding 
$
  \Embedding
  \colon
  \left(\Manifold_{1},\Submanifold_{1}\right)
  \to
  \left(\Manifold_{2},\Submanifold_{2}\right)
$ an \introduce{embedding of pairs} if 
$\apply{\Embedding^{-1}}{\Submanifold_{2}}=\Submanifold_{1}$. 
We will be interested in spaces associated to the \introduce{space of 
(collared) embeddings 
$\EmbeddingSpace{\Manifold_{1}}{\Manifold_{2}}$}. This space is 
equipped with the $C^\infty$-topology.
\begin{definition}
  Let 
  $
    \Embedding_{1},\Embedding_{2}
    \colon
    \Manifold_{1}
    \to
    \Manifold_{2}
  $ denote embeddings.
  We say that $\Embedding_{1}$ and $\Embedding_{2}$
  \introduce{have the same jet along 
  $\BoundaryManifold{\Manifold_{2}}$} if 
  there exists an open neighbourhood 
  $\Neighbourhood{\BoundaryManifold{\Manifold_{1}}}$ of 
  $\BoundaryManifold{\Manifold_{1}}$ such that 
  $
    \at{
        \Embedding_{1}
    }
    {
      \Neighbourhood{
        \BoundaryManifold{\Manifold_{1}}
      }
    }
    =
    \at{
      \Embedding_{2}
    }
    {
       \Neighbourhood{\BoundaryManifold{\Manifold_{1}}}
    }
  $.
  This defines an equivalence relation on 
  $\EmbeddingSpace{\Manifold_{1}}{\Manifold_{2}}$ and we denote the 
  quotient 
  map by \[
    \BoundaryConditionMap
    \colon
    \EmbeddingSpace{\Manifold_{1}}{\Manifold_{2}}
    \to 
    \BoundaryConditionSpacePar{\Manifold_{1}}{\Manifold_{2}}
  \]
  We write 
  $
    \EmbeddingSpaceBoundaryCondition
      {\Manifold_{1}}
      {\Manifold_{2}}
      {\BoundaryConditionPar}
    \coloneqq
    \apply{\BoundaryConditionMap^{-1}}
      {\BoundaryConditionPar}$. 
\end{definition}

From here on forth let us specialize to the case where the domain of 
the embedding is a compact connected oriented surface of genus 
$\Genus$ with $\BoundaryComponents$ boundary components denoted by 
$\SurfaceGB{\Genus}{\BoundaryComponents}$. 
We denote by 
$\DiffeomorphismGroup{\SurfaceGB{\Genus}{\BoundaryComponents}}$ the 
\introduce{group of orientation-preserving diffeomorphisms of 
$\SurfaceGB{\Genus}{\BoundaryComponents}$}. 
This group acts freely on the space of embeddings and on $J_\partial$ 
via precomposition. We define 
\begin{align*}
  \SpaceSubsurfaces{\Genus}{\BoundaryComponents}{\Manifold}&
  \coloneqq 
  \EmbeddingSpace{\SurfaceGB{\Genus}{\BoundaryComponents}}
    {\Manifold}
  /
  \DiffeomorphismGroup{\SurfaceGB{\Genus}{\BoundaryComponents}}
  \\
  \BoundaryConditionSpace{
    \SurfaceGB{\Genus}{\BoundaryComponents}
  }{
    \Manifold
  }
  &\coloneqq 
  \BoundaryConditionSpacePar
    {\SurfaceGB{\Genus}{\BoundaryComponents}}
    {\Manifold}
  /
  \DiffeomorphismGroup{\SurfaceGB{\Genus}{\BoundaryComponents}}
\end{align*}
The first space will be called the \introduce{space of subsurfaces of genus 
$\Genus$ and $\BoundaryComponents$ boundary components}.
We also define 
\[
  \SpaceAllSubsurfaceBoundary{\Manifold}{\BoundaryComponents}
  \coloneqq
  \bigsqcup_{\Genus}
  \SpaceSubsurfaceBoundary
    {\Genus}
    {\BoundaryComponents}
    {\Manifold}
    {\BoundaryCondition}
\]
An element 
$
  \Subsurface
  \in
  \SpaceSubsurfaces{\Genus}{\BoundaryComponents}{\Manifold}
$
is an unparametrized embedding or in other words 
a submanifold $\Subsurface\subset\Manifold$ diffeomorphic to 
$\SurfaceGB{\Genus}{\BoundaryComponents}$. We 
will call such an unparametrized embedding a \introduce{subsurface of $M$}.

Since the map $\BoundaryConditionMap$ is equivariant with 
respect to the 
$\DiffeomorphismGroup{\SurfaceGB{\Genus}{\BoundaryComponents}}$
-actions, we get an induced map 
between the quotient spaces also denoted by $\BoundaryConditionMap$. 
We define 
$\SpaceSubsurfaceBoundary
  {\Genus}
  {\BoundaryComponents}
  {\Manifold}
  {\BoundaryCondition}
$
as 
$\apply
  {\BoundaryConditionMap^{-1}}
  {\BoundaryCondition}
$
for some 
$
  \BoundaryCondition
  \in \BoundaryConditionSpace
    {\SurfaceGB{\Genus}{\BoundaryComponents}}
    {\Manifold}
$%
.
Note that if $\BoundaryCondition=\left[\BoundaryConditionPar\right]$ 
then 
\[
  \SpaceSubsurfaceBoundary
    {\Genus}
    {\BoundaryComponents}
    {\Manifold}
    {\BoundaryCondition}
  \cong
  \EmbeddingSpaceBoundaryCondition
    {\SurfaceGB{\Genus}{\BoundaryComponents}}
    {\Manifold}
    {\BoundaryConditionPar}
  /
  \DiffeomorphismGroupBoundary
    {\SurfaceGB{\Genus}{\BoundaryComponents}}
\]
Here 
$
  \DiffeomorphismGroupBoundary
    {\SurfaceGB{\Genus}{\BoundaryComponents}}
$
denotes the group of diffeomorphisms of 
$\SurfaceGB{\Genus}{\BoundaryComponents}$ that agree with the 
identity on a neighbourhood of the boundary.

As a slight abuse of notation we will also write $\BoundaryCondition$ 
for the image of some representative of 
$
  \BoundaryCondition
  \in 
  \BoundaryConditionSpace
    {\SurfaceGB{\Genus}{\BoundaryComponents}}
    {\Manifold}
$
in 
$\BoundaryManifold{\Manifold}$. 
We will call $\BoundaryCondition$ a \introduce{boundary condition}.

\subsection{Tangential structures}
Let $\Manifold$ denote a manifold possibly with boundary. 
There is the fiber bundle 
$
  \Grassmannian{2}{\TangentBundle{\Manifold}}
  \to 
  \Manifold
$%
, where 
$\Grassmannian{2}{\TangentBundle{\Manifold}}$ denotes 
the space of oriented $2$-planes in $\TangentBundle{\Manifold}$.
For the purpose of this paper a \introduce{space of 
$\TangentialFibration$-structures of 
subplanes of $\TangentBundle{\Manifold}$} will denote a continuous 
Hurewicz-fibration 
$
  \TangentialFibration
  \colon
  \TangentialSpace{\Manifold}
  \to 
  \Grassmannian{2}{\TangentBundle{\Manifold}}
$%
. For a submanifold $\Submanifold$ of $\Manifold$ we have an induced map 
$
  \Grassmannian{2}{\TangentBundle{\Submanifold}}
  \to 
  \Grassmannian{2}{\TangentBundle{\Manifold}}
$
and we will denote the pullback of $\TangentialSpace{\Manifold}$
along this inclusion by $\TangentialSpace{\Submanifold}$. 

If $\Subsurface$ is an oriented subsurface, then 
$\Grassmannian{2}{\Subsurface}\to \Subsurface$ is 
a two-sheeted covering with a distinguished section. In this 
particular case we will identify the correctly oriented connected component of 
$\Grassmannian{2}{\Subsurface}$ with $\Subsurface$ and consider 
$\TangentialSpace{\Subsurface}$ as a fibration over $\Subsurface$.

If we have an embedding $\Embedding\colon 
\SurfaceGB{\Genus}{\BoundaryComponents}\to \Manifold$, then 
we get an induced map 
\[
  \GrassmannianDifferential{\Embedding}
  \colon 
  \SurfaceGB{\Genus}{\BoundaryComponents}
  \to
  \Grassmannian{2}{\TangentBundle{\Manifold}}
\]
by mapping $\Point\in\SurfaceGB{\Genus}{\BoundaryComponents}$ 
to 
$
  \apply
    {\Differential{\Embedding}}
    {\TangentSpace
      {\Point}
      {\SurfaceGB{\Genus}{\BoundaryComponents}}
    }
$
considered as an oriented $2$-plane.
We will call this map the \introduce{Grassmannian 
differential of $\Embedding$}.
\begin{definition}
  \label{dfn:TangentialStructure}
  Given an embedding 
  $
    \Embedding
    \colon 
    \SurfaceGB{\Genus}{\BoundaryComponents}
    \to
    \Manifold
  $%
  ,
  we call a map 
  $
    \TangentialStructure{\Embedding}
    \colon
    \SurfaceGB{\Genus}{\BoundaryComponents}
    \to
   \TangentialSpace{\Manifold}
  $
  a \introduce{$\TangentialFibration$-structure of 
  $\Embedding$} if the following diagram commutes:
  \[
    \begin{tikzcd}
      &
      &
      \TangentialSpace{\Manifold}
        \ar[d,"\TangentialFibration"]
      \\
      \SurfaceGB{\Genus}{\BoundaryComponents}
        \ar[urr,"\TangentialStructure{\Embedding}"]
        \ar[rr,"\GrassmannianDifferential{\Embedding}" near end]
      &
      &
      \Grassmannian{2}{\TangentBundle{\Manifold}}
    \end{tikzcd}
  \]
  We will denote by 
  $
  \EmbeddingSpaceTangential
    {\TangentialFibration}
    {\SurfaceGB{\Genus}{\BoundaryComponents}}
    {\Manifold}
  $
  the \introduce{space of embeddings with 
  $\TangentialFibration$-structure}. This space is 
  topologized as a subspace of 
  $
    \MappingSpace
      {\SurfaceGB{\Genus}{\BoundaryComponents}}
      {\TangentialSpace{\Manifold}}
    \times
    \EmbeddingSpace
      {\SurfaceGB{\Genus}{\BoundaryComponents}}
      {\Manifold}
  $, where the first factor is equipped with the compact-open topology and the 
  second factor with the $C^\infty$-topology. Even though elements of 
  $\EmbeddingSpaceTangential
    {\TangentialFibration}
    {\SurfaceGB{\Genus}{\BoundaryComponents}}
    {\Manifold}
  $
  consist of pairs of maps, we will usually say that $\Embedding$ is an 
  embedding with $\TangentialFibration$-structure. In 
  this case we will denote the underlying 
  $\TangentialFibration$-structure by 
  $\TangentialStructure{\Embedding}$.
  We will say that two embeddings with 
  $\TangentialFibration$-structure 
  $\Embedding_{1},\Embedding_{2}$ have \introduce{the 
  same 
  $\TangentialFibration$-jet along 
  $\BoundaryManifold{\Manifold}$} if they have the same jet along 
  $\BoundaryManifold{\Manifold}$ and 
  $
    \at
      {\TangentialStructure{\Embedding_{1}}}
      {\BoundaryManifold{\SurfaceGB{\Genus}{\BoundaryComponents}}}
    =
    \at
      {\TangentialStructure{\Embedding_{2}}}
      {\BoundaryManifold{\SurfaceGB{\Genus}{\BoundaryComponents}}}
  $.

  For a diffeomorphism 
  $
    \phi
    \in 
    \DiffeomorphismGroup{\SurfaceGB{\Genus}{\BoundaryComponents}}
  $%
, we have
  $
    \apply
      {\GrassmannianDifferential{\Embedding}}
      {\Point}
    =
    \apply
      {\GrassmannianDifferential{\Embedding\circ\phi^{-1}}}
      {\apply{\phi}{\Point}}
  $.
  Hence the group
  $\DiffeomorphismGroup{\SurfaceGB{\Genus}{\BoundaryComponents}}$ 
  acts on 
  $
    \EmbeddingSpaceTangential
      {\TangentialFibration}
      {\SurfaceGB{\Genus}{\BoundaryComponents}}
      {\Manifold}
  $ 
  via precomposition, this action preserves the 
  equivalence relation 
  of having the same $\TangentialFibration$-jet.
  If we restrict the embeddings to have some fixed 
  $\TangentialFibration$-jet, then we 
  denote the corresponding space of embeddings with 
  $\TangentialFibration$-structure by 
  $\BoundaryConditionTangential{\BoundaryCondition}$
  to obtain 
  $\EmbeddingSpaceBoundaryConditionTangential
    {\TangentialFibration}
    {\SurfaceGB{\Genus}{\BoundaryComponents}}
    {\Manifold}
    {\BoundaryConditionTangential{\BoundaryCondition}}
  $
  and we denote the quotient of this subspace by the action of
  $
    \DiffeomorphismGroup
      {\SurfaceGB{\Genus}{\BoundaryComponents}}
  $
  by
  $
    \SpaceSubsurfaceBoundaryTangential
      {\TangentialFibration}
      {\Genus}
      {\BoundaryComponents}
      {\Manifold}
      {\BoundaryConditionTangential{\BoundaryCondition}}
   $%
  . Similarly, we obtain 
  $
    \BoundaryConditionSpaceTangential
      {\TangentialFibration}
      {\SurfaceGB{\Genus}{\BoundaryComponents}}
      {\Manifold}
  $
  and
  $
    \SpaceAllSubsurfaceBoundaryTangential
      {\TangentialFibration}
      {\Manifold}
      {\BoundaryConditionTangential{\BoundaryCondition}}
  $%
  . We will call the second space the \introduce{space of 
  subsurfaces with $\TangentialFibration$-structure} 
  and an 
  element in it 
  will be 
  called a \introduce{subsurface with 
  $\TangentialFibration$-structure of 
  $\Manifold$}.  If 
  some $\TangentialFibration$-jet 
  $\BoundaryConditionTangential{\BoundaryCondition}$ of 
  an embedding with $\TangentialFibration$-structure is 
  given, we will 
  sometimes 
  write $\BoundaryCondition$ for the image of the 
  underlying 
  boundary condition without 
  $\TangentialFibration$-structure.
\end{definition}
If $\TangentialFibration$ is the identity map, we will write $+$ instead of 
$\Identity$.
Let us recall parts of an example of the introduction:
\begin{example}
\label{exm:NormalSection}
  Let $\OrthogonalComplement{\TautologicalBundle}$ denote the complement of the 
  tautological $2$-plane bundle on 
  $\Grassmannian{2}{\TangentBundle{\Manifold}}$ i.e. 
  $
    \Projection
      {\Grassmannian{2}{\TangentBundle{\Manifold}}}
    ^*
    \TangentBundle{\Manifold}
    /
    \TautologicalBundle
  $%
  , where 
  $
    \Projection{\Grassmannian{2}{\TangentBundle{\Manifold}}}
    \colon 
    \Grassmannian{2}{\TangentBundle{\Manifold}}
    \to
    \Manifold
  $
  denotes the projection and $\TautologicalBundle$ denotes the tautological 
  bundle which includes into 
  $
    \Projection
      {\Grassmannian{2}{\TangentBundle{\Manifold}}}
    ^*
    \TangentBundle{\Manifold}
  $%
  . Let 
  $
    \SphereBundle{\OrthogonalComplement{\TautologicalBundle}}
    \coloneqq
    \OrthogonalComplement{\TautologicalBundle}
    \setminus
    0
  $
  denote the complement of the image of the zero section of this bundle.
  Then 
  $
    \TangentialFibration 
    \colon 
    \SphereBundle{\OrthogonalComplement{\TautologicalBundle}}
    \to 
    \Grassmannian{2}{\TangentBundle{\Manifold}}
  $
  is a Hurewicz-fibration.
  For an embedding 
  $
    \Embedding
    \colon
    \SurfaceGB{\Genus}{\BoundaryComponents}
    \to 
    \Manifold
  $
  the space 
  $
    \TangentialSpace
      {\apply
        {\Embedding}
        {\SurfaceGB{\Genus}{\BoundaryComponents}}
      }
  $
  is the complement of the zero-section of the normal bundle and a 
  $\TangentialFibration$-structure corresponds to a continuous 
  section of this bundle.
  In this case 
  $
    \EmbeddingSpaceTangential
      {\TangentialFibration}
      {\SurfaceGB{\Genus}{\BoundaryComponents}}
      {\Manifold}
  $
  is the space of embeddings together with a continuous section of the 
  normal bundle.
\end{example}
\begin{remark}
  In some situations it makes sense to replace 
  $
    \MappingSpace
      {\SurfaceGB{\Genus}{\BoundaryComponents}}
      {\TangentialSpace{\Manifold}}
  $
  by
  $
    \SmoothMappingSpace
      {\SurfaceGB{\Genus}{\BoundaryComponents}}
      {\TangentialSpace{\Manifold}}
  $
  in Definition \ref{dfn:TangentialStructure}. Of course this requires some 
  kind of smooth structure on $\TangentialSpace{\Manifold}$. In Example 
  \ref{exm:NormalSection}
  this would make the sections of the normal bundle smooth sections. One also 
  has to alter the definition of having the 
  same $\TangentialFibration$-jet to include the jet of the 
  $\TangentialFibration$-structures.
  Everything proven in this paper works in that case as well.
\end{remark}

\begin{remark}
If we have a subsurface with a tangential structure $\Subsurface$ 
and some isotopy 
$
  \left(\left[0,1\right],0\right)
  \to 
  \left(
    \DiffeomorphismGroupBoundary{\Manifold},\Identity
  \right)
$%
, we can move $\Subsurface$ along $\Manifold$ using this isotopy 
to obtain $\Subsurface_{t}$.
Since $\TangentialFibration$ is a Hurewicz-fibration we can also equip 
$\Subsurface_{t}$ with a tangential structure and if we say nothing about 
tangential structures in this context, we will assume implicitly that a lift of 
the aforementioned path in 
$
  \SpaceSubsurfaceBoundary
    {\Genus}
    {\BoundaryComponents}
    {\Manifold}
    {\BoundaryCondition}
$
is chosen. 
Note that the tangential structure of $\Subsurface_{1}$ is not unique i.e. 
there might be many different tangential structures for $\Subsurface_{t}$.
\end{remark}

Lastly, using a collar $\Elongation{\Manifold}{k}$ and 
$\Manifold$ are diffeomorphic and since collars are unique up to isotopy, the 
associated diffeomorphism is also unique up to isotopy. 
This enables us to equip $\Elongation{\Manifold}{k}$ with a unique space of 
$\TangentialFibration$-structures and we will always assume that 
$\Elongation{\Manifold}{k}$ has this space of 
$\TangentialFibration$-structures.
A space of $\TangentialFibration$-structures 
on a manifold $\Manifold$ induces a space of $\TangentialFibration$-structures 
on every submanifold and we will always assume that submanifolds are equipped 
with the induced structure.
In particular $\BoundaryManifold{\Manifold}\times\left[0,k\right]$ is equipped 
with a space of $\TangentialFibration$-structures.

\subsection{Retractile Spaces and Fibrations}
The following arguments will be used frequently 
throughout this paper to prove that certain maps are locally 
trivial fibrations and the ideas date back to \cite{cerf}.
\begin{definition}
  Let $\Group$ denote a topological group and $\TopologicalSpace$ a 
  $\Group$-space.
  We say that $\TopologicalSpace$ is \introduce{$\Group$-locally retractile} if 
  every point $\Point\in \TopologicalSpace$ possesses an 
  open neighborhood $\Neighbourhood{\Point}$ and a continuous map
  $
    \LocalRetraction
    \colon 
    \left(\Neighbourhood{\Point},\Point\right)
    \to 
    \left(\Group,\Identity\right)
  $%
  , called the 
  \introduce{$\Group$-local retraction around $\Point$}
  , such that 
  $
    \apply{\LocalRetraction}{\Point'}
    \cdot 
    \Point
    =
    \Point'
  $
  for all $\Point\in \Neighbourhood{\Point}$. Note that this implies 
  that $\LocalRetraction$ is a homeomorphism onto its image.
\end{definition}

The following lemma will be used several times 
throughout this paper.
\begin{lemma}
  \label{lem:LocallyRetractile}
  If $\TopologicalSpace$ is $\Group$-locally retractile and 
  $
    \ContinuousMap
    \colon
    \TopologicalSpace'
    \to
    \TopologicalSpace
  $
  denotes a $\Group$-equivariant map, then $\ContinuousMap$ is a locally 
  trivial fibration.
\end{lemma}
\begin{proof}
  For $\Point\in \TopologicalSpace$ choose an $\Neighbourhood{\Point}$ and a 
  $\Group$-local retraction
  $
    \LocalRetraction
    \colon
    \Neighbourhood{\Point}
    \to
    \Group
  $
  around $\Point$.
  Then
  \begin{align*}
    \apply{\ContinuousMap^{-1}}{\{\Point\}}
    \times
    \Neighbourhood{\Point}
    &
    \to 
    \apply{\ContinuousMap^{-1}}{\Neighbourhood{\Point}}
    \\
    \left(\Point_{\TopologicalSpace'},\Point_{\TopologicalSpace}\right)
    &
    \mapsto 
    \apply{\LocalRetraction}{\Point_{\TopologicalSpace}}
    \cdot
    \Point_{\TopologicalSpace'}
  \end{align*}
gives the desired local trivialisation as is easily checked.
\end{proof}

As before let $\DiffeomorphismGroupBoundary{\Manifold}$ 
denote the group of diffeomorphisms of $\Manifold$ 
that agree with the identity on a neighbourhood of 
$\BoundaryManifold{\Manifold}$. 
The following proposition can be found in 
\cite{cerf} (Chapter~2.2, Theorem~5).
\begin{proposition}
\label{prp:EmbeddingLocallyRetractile}
  Let $\BoundaryConditionPar$ denote the jet of an embedding from 
  $\Submanifold$ into $\Manifold$, then
  $
    \EmbeddingSpaceBoundaryCondition
      {\Submanifold}
      {\Manifold}
      {\BoundaryConditionPar}
  $
  is $\DiffeomorphismGroupBoundary{\Manifold}$-locally retractile.
\end{proposition} 

The following proposition can be found in \cite{CRW16} 
with a proof that relies on a reference to \cite{B81} 
or \cite{M80}.
We will give a sketch of a proof of the proposition, 
because similar ideas will occur later on.
\begin{proposition}
\label{prp:SpaceSubsurfacesLocallyRetractile}
  The space 
  $
    \SpaceSubsurfaceBoundary
      {\Genus}
      {\BoundaryComponents}
      {\Manifold}
      {\BoundaryCondition}
  $
  is $\DiffeomorphismGroupBoundary{\Manifold}$-locally 
  retractile.
\end{proposition}
\begin{proof}
  Let $\Subsurface \subset \Manifold$ denote an element of 
  $
    \SpaceSubsurfaceBoundary
      {\Genus}
      {\BoundaryComponents}
      {\Manifold}
      {\BoundaryComponents}
  $
  and $\TubularNeighbourhood{\Subsurface}{\Manifold}$ a 
  tubular neighborhood of 
  $\Subsurface$ with corresponding projection
  $
    \Projection{\TubularNeighbourhood{\Subsurface}{\Manifold}}
    \colon
    \TubularNeighbourhood{\Subsurface}{\Manifold}
    \to
    \Subsurface
  $%
  . We define 
  \[
    \Neighbourhood{\Subsurface}
    =
    \left\{
      \Subsurface'\in 
      \SpaceSubsurfaceBoundary
        {\Genus}
        {\BoundaryComponents}
        {\Manifold}
        {\BoundaryCondition}
      \mid 
      \Subsurface'
      \subset
      \TubularNeighbourhood{\Subsurface}{\Manifold}
      \text{ and }
      \at
      {\Projection
        {\TubularNeighbourhood{\Subsurface}{\Manifold}}
      }
      {\Subsurface'}
      \text{ is a diffeomorphism from } 
      \Subsurface'
      \text{ to }
      \Subsurface
    \right\}
  \]
  which is an open subset of 
  $
    \SpaceSubsurfaceBoundary
      {\Genus}
      {\BoundaryComponents}
      {\Manifold}
      {\BoundaryCondition}
  $%
  , because the diffeomorphisms form an open subset of the maps from 
  $\Subsurface'$ to $\Subsurface$ (See Proposition~1.4.2 in Chapter~2 in 
  \cite{cerf}). 
  For $\Subsurface'$ in $\Neighbourhood{\Subsurface}$ we can 
  define 
  an isotopy of the embedded subsets by pushing 
  $\Subsurface'$ along 
  the fiber of 
  $
    \TubularNeighbourhood
      {\Subsurface}
      {\Manifold}
  $
  onto $\Subsurface$.
  Furthermore we can use the isotopy extension theorem (See 
  Theorem 6.1.1 in \cite{WallDifferentialTopology}), 
  to extend these isotopies continuously to isotopies of $\Manifold$.
  The resulting diffeomorphisms at time $1$ of these isotopies give us the 
  desired $\DiffeomorphismGroupBoundary{\Manifold}$-local retraction.
\end{proof} 

The last Proposition implies the following useful observation:
\begin{corollary}
  The two maps
  \begin{align*}
    \EmbeddingSpaceBoundaryCondition
      {\SurfaceGB{\Genus}{\BoundaryComponents}}
      {\Manifold}
      {\BoundaryConditionPar}
    &
    \to
    \SpaceSubsurfaceBoundary
      {\Genus}
      {\BoundaryComponents}
      {\Manifold}
      {\BoundaryCondition}
    \\
    \EmbeddingSpaceBoundaryConditionTangential
      {\TangentialFibration}
      {\SurfaceGB{\Genus}{\BoundaryComponents}}
      {\Manifold}
      {\BoundaryConditionTangential{\BoundaryConditionPar}}
    &
    \to
    \SpaceSubsurfaceBoundaryTangential
      {\TangentialFibration}
      {\Genus}
      {\BoundaryComponents}
      {\Manifold}
      {\BoundaryConditionTangential{\BoundaryCondition}}
  \end{align*}
  are locally trivial fibrations.
\end{corollary}
\begin{proof}
  The only thing that might need clarification is the second assertion, but to 
  prove this one only has to note that 
  \[
    \begin{tikzcd}
      \EmbeddingSpaceBoundaryConditionTangential
        {\TangentialFibration}
        {\SurfaceGB{\Genus}{\BoundaryComponents}}
        {\Manifold}
        {\BoundaryConditionTangential{\BoundaryConditionPar}}
        \ar[d]
        \ar[r]
      &
      \EmbeddingSpaceBoundaryCondition
        {\SurfaceGB{\Genus}{\BoundaryComponents}}
        {\Manifold}
        {\BoundaryConditionPar}
        \ar[d]
      \\
      \SpaceSubsurfaceBoundaryTangential
        {\TangentialFibration}
        {\Genus}
        {\BoundaryComponents}
        {\Manifold}
        {\BoundaryConditionTangential{\BoundaryCondition}}
        \ar[r]
      &
      \SpaceSubsurfaceBoundary
        {\Genus}
        {\BoundaryComponents}
        {\Manifold}
        {\BoundaryCondition}
    \end{tikzcd}
  \]
  is a pullback diagram.
\end{proof}

Moreover we have the following central lemma:
\begin{lemma}
\label{lem:ForgettingTangentialFibration}
  The forgetful map 
  $
    \Projection{\TangentialFibration}
    \colon 
    \SpaceSubsurfaceBoundaryTangential
      {\TangentialFibration}
      {\Genus}
      {\BoundaryComponents}
      {\Manifold}
      {\BoundaryConditionTangential{\BoundaryCondition}}
    \to 
    \SpaceSubsurfaceBoundary
      {\Genus}
      {\BoundaryComponents}
      {\Manifold}
      {\BoundaryCondition}
  $
  is a Hurewicz fibration.
  The fiber over a surface $\Subsurface$ is given by 
  $\SpaceOfSections{\TangentialSpace{\Subsurface}}$, the space of sections of 
  $\TangentialSpace{\Subsurface}$. 
  The same holds true for the forgetful map
  $
    \EmbeddingSpaceBoundaryConditionTangential
      {\TangentialFibration}
      {\SurfaceGB{\Genus}{\BoundaryComponents}}
      {\Manifold}
      {\BoundaryConditionTangential{\BoundaryConditionPar}}
    \to
    \EmbeddingSpaceBoundaryConditionTangential
      {+}
      {\SurfaceGB{\Genus}{\BoundaryComponents}}
      {\Manifold}
      {\BoundaryConditionTangential{\BoundaryConditionPar}}
  $
\end{lemma} 
\begin{proof}
  By Theorem~42.3 and the corollary in 27.4 in \cite{CSGA} (together with 
  the fact that the quotient of a nuclear spaces by a closed subspaces
  is a nuclear space) we conclude that 
  $
  \SpaceSubsurfaceBoundary
    {\Genus}
    {\BoundaryComponents}
    {\Manifold}
    {\BoundaryCondition}
  $
  is paracompact and Hausdorff, hence we 
  conclude that a locally trivial fiber bundle over this space is actually a 
  Hurewicz fibration by Theorem~13 in Section~2.7 of \cite{spanier}. Hence it 
  suffices to prove that $\Projection{\TangentialFibration}$ is a locally 
  trivial fiber bundle. In order to do this we want to apply 
  Lemma~\ref{lem:LocallyRetractile}.
  
  Let $\HomeoTangential{\Manifold}{\TangentialSpace{\Manifold}}$ denote the 
  group of homeomorphisms of $\TangentialSpace{\Manifold}$ 
  which cover a diffeomorphism of 
  $
    \Grassmannian{2}{\TangentBundle{\Manifold}}
  $
  induced from a diffeomorphism of $\Manifold$ that fixes a neighbourhood of 
  the boundary i.e. the following 
  diagram commutes and the lower map is a diffeomorphism 
  \[
    \begin{tikzcd}
      \TangentialSpace{\Manifold}
        \ar[d]
        \ar[r]
      &
      \TangentialSpace{\Manifold}
        \ar[d]
      \\
      \Grassmannian{2}{\TangentBundle{\Manifold}}
        \ar[d]
        \ar[r,"\GrassmannianDifferential{\SmoothMap}"]
      &
      \Grassmannian{2}{\TangentBundle{\Manifold}}
        \ar[d]
      \\
      \Manifold
        \ar[r,"\SmoothMap"]
      &
      \Manifold
    \end{tikzcd}
  \]
  This group acts on 
  $
    \SpaceSubsurfaceBoundaryTangential
      {\TangentialFibration}
      {\Genus}
      {\BoundaryComponents}
      {\Manifold}
      {\BoundaryConditionTangential{\BoundaryCondition}}
  $
  and
  $
    \SpaceSubsurfaceBoundary
      {\Genus}
      {\BoundaryComponents}
      {\Manifold}
      {\BoundaryCondition}
  $
  by post-composition. 
  Proposition~\ref{prp:SpaceSubsurfacesLocallyRetractile} states that 
  $
    \SpaceSubsurfaceBoundary
      {\Genus}
      {\BoundaryComponents}
      {\Manifold}
      {\BoundaryCondition}
  $ 
  is $\DiffeomorphismGroupBoundary{\Manifold}$-locally 
  retractile. We will show that it is also 
  $\HomeoTangential{\Manifold}{\TangentialSpace{\Manifold}}$-locally 
  retractile. 
  For a 
  $
    \Subsurface
    \in 
    \SpaceSubsurfaceBoundary
      {\Genus}
      {\BoundaryComponents}
      {\Manifold}
      {\BoundaryCondition}
  $%
  , let
  $
    \LocalRetraction
    \colon 
    \Neighbourhood{\Subsurface}
    \to
    \DiffeomorphismGroupBoundary{\Manifold}
  $ 
  denote a 
  $\DiffeomorphismGroupBoundary{\Manifold}$-local 
  retraction around $\Subsurface$.
  Note that $\DiffeomorphismGroupBoundary{\Manifold}$ is locally contractible 
  and it is paracompact (For paracompactness see Lemma 41.11 of \cite{CSGA} to 
  conclude that $\DiffeomorphismGroupBoundary{\Manifold}$ with the $C^\infty$ 
  topology is metrizable). Hence there is a contractible, closed 
  neighbourhood of the identity in $\DiffeomorphismGroupBoundary{\Manifold}$. 
  By shrinking $\Neighbourhood{\Subsurface}$ we can assume that 
  $\apply{\LocalRetraction}{\Neighbourhood{\Subsurface}}$ maps to such a 
  contractible neighbourhood of the identity.
  
  We get a map 
  $
    \GrassmannianDifferential{\mathrm{ev}}
    \colon
    \Grassmannian{2}{\TangentBundle{\Manifold}}
    \times 
    \Neighbourhood{\Subsurface}
    \to
    \Grassmannian{2}{\TangentBundle{\Manifold}}
  $
  which maps $(\TwoPlane,\Subsurface)$ to 
  $
    \apply
      {\GrassmannianDifferential
        {\apply{\LocalRetraction}{\Subsurface}}
      }
      {\TwoPlane}
  $
  and we can use the contractibility of $\Neighbourhood{\Subsurface}$ together 
  with its paracompactness and the compactness of $\Manifold$ to conclude that 
  the pullback
  $
    \GrassmannianDifferential{\mathrm{ev}}^*
    \TangentialSpace{\Manifold}
  $ 
  is isomorphic to the pullback
  $
    \Projection{\Neighbourhood{\Subsurface}}^* 
    \TangentialSpace{\Manifold}
  $%
  , where 
  $
    \Projection{\Neighbourhood{\Subsurface}}
    \colon 
    \Grassmannian{2}{\TangentBundle{\Manifold}}
    \times 
    \Neighbourhood{\Subsurface}
    \to 
    \Grassmannian{2}{\TangentBundle{\Manifold}}
  $
  denotes the projection.
  Using this isomorphism, one can construct a lift of $\LocalRetraction$ 
  to conclude that 
  $
    \SpaceSubsurfaceBoundary
      {\Genus}
      {\BoundaryComponents}
      {\Manifold}
      {\BoundaryCondition}
  $
  is 
  $
    \HomeoTangential{\Manifold}{\TangentialSpace{\Manifold}}
  $%
  -locally retractile.
  
  By Lemma~\ref{lem:LocallyRetractile} we conclude that the forgetful map 
  $\Projection{\TangentialFibration}$ is a locally trivial fiber bundle.
  
  The specification of the fibre is evident. The proof for the second map is 
  completely analogous.
\end{proof}
We will also need the following standard lemma from algebraic topology, whose 
proof is omitted:
\begin{lemma}
  \label{lem:FibrationComposition}
  Suppose we have the following diagram of topological spaces
  \[
    \begin{tikzcd}
      \TopologicalSpace_{1}
        \ar[r,"\ContinuousMap_{1}"]
      &
      \TopologicalSpace_{2}
        \ar[r,"\ContinuousMap_{2}"]
      &
      \TopologicalSpace_{3}
    \end{tikzcd}
  \]
  where $\ContinuousMap_{1}$ is surjective. If $\ContinuousMap_{1}$ and 
  $\ContinuousMap_{2}\circ\ContinuousMap_{1}$ are fibrations, then 
  $\ContinuousMap_{2}$ is a fibration as well.
\end{lemma}
\subsection{Tubular neighborhoods and Thickened Embeddings}
If $\Submanifold\subset \Manifold$ denotes a 
submanifold, we will write 
$\NormalBundle{\Submanifold}{\Manifold}$ for the 
normal bundle of $\Subsurface$ i.e.
$
  \at
    {\TangentBundle{\Manifold}}
    {\Submanifold}
  /
  \TangentBundle{\Submanifold}
$%
. A tubular neighbourhood of $\Submanifold$ is an 
embedding 
$
  \TubularNeighbourhoodMap{\Submanifold}{\Manifold}
  \colon 
  \NormalBundle{\Submanifold}{\Manifold}
  \to
  \Manifold
$
such that 
$
  \at
    {\TubularNeighbourhoodMap{\Submanifold}{\Manifold}}
    {\Submanifold}
$
is the identity and the composition 
  \[
    \begin{tikzcd}
      \TangentBundle{\Submanifold}
      \oplus 
      \NormalBundle{\Submanifold}{\Manifold}
        \ar[r,"\cong"]
      &
      \at
        {\TangentBundle
          {(\NormalBundle{\Submanifold}{\Manifold})}
        }
        {\Submanifold}
        \ar[
          r
          ,
          "
            \Differential
               {\TubularNeighbourhoodMap{\Submanifold}{\Manifold}}
          "
        ]
      &
      \at
        {\TangentBundle{\Manifold}}
        {\Submanifold}
        \ar[
          r,
          "\Projection
            {\NormalBundle{\Submanifold}{\Manifold}}
          "
        ]
      &
      \NormalBundle{\Submanifold}{\Manifold}
    \end{tikzcd}
  \]
agrees with the projection onto the second factor.
If 
$
  \left(
    \Submanifold,\Submanifold'
  \right)
  \subset 
  \left(
    \Manifold,\Submanifold
  \right)
$
denotes an embedded pair, we define a tubular 
neighborhood of 
$
  \left(
    \Submanifold,\Submanifold'
  \right)
$
in 
$
  \left(
    \Manifold,\Submanifold
  \right)
$
to be a tubular neighborhood
$
  \TubularNeighbourhoodMap{\Submanifold}{\Manifold}
  \colon
  \NormalBundle{\Submanifold}{\Manifold}
  \to 
  \Manifold
$
of $\Submanifold$ such that 
$
  \at
    {\TubularNeighbourhoodMap{\Submanifold}{\Manifold}}
    {\NormalBundle{\Submanifold'}{\Submanifold}}
$
is a tubular neighbourhood of $\Submanifold'$ in 
$\Submanifold$.
If $\Submanifold$ has a boundary we assume that every 
tubular neighbourhood of $\Submanifold$ is 
actually a tubular neighbourhood of 
$
  \left(
    \Submanifold,\BoundaryManifold{\Submanifold}
  \right)
$
in
$
  \left(
    \Manifold,
    \BoundaryManifold{\Manifold}
  \right)
$%
.

From the vector bundle
$\NormalBundle{\Submanifold}{\Manifold}$,
one can obtain a disk bundle by compactifying this 
fiberwise using a sphere at infinity to obtain 
$\overline{\NormalBundle{\Submanifold}{\Manifold}}$.
An embedding of 
$
  \overline{\NormalBundle{\Submanifold}{\Manifold}}
  \to
  M
$ 
is called a closed tubular neighborhood if its 
restriction to $\NormalBundle{\Submanifold}{\Manifold}$ 
is a tubular neighborhood. 
We denote by
\[
  \TubularNeighbourhoodClosedSpace
    {\Submanifold}
    {\Manifold}
  \subset
  \EmbeddingSpace
    {\overline{\NormalBundle{\Submanifold}{\Manifold}}}
    {\Manifold}
\]
the subspace of closed tubular neighborhoods.
Similar notation will occur if we consider tubular 
neighborhoods of pairs.

A proof of the following lemma is sketched in Section 
2.5 of \cite{CRW16}.
\begin{lemma}
\label{lem:TubularNeighbourhoodSpaceContractible}
  If $\left(\Submanifold,\Submanifold'\right)$ is an embedded pair of compact
  submanifolds in $\left(\Manifold,\Submanifold\right)$, then 
  $
    \TubularNeighbourhoodClosedSpace
      {\Submanifold}
      {\Manifold}
  $
  and 
  $
    \TubularNeighbourhoodClosedSpace
      {\left(\Submanifold,\Submanifold'\right)}
      {\left(\Manifold,\Submanifold\right)}
  $
  are contractible.
\end{lemma}

\begin{definition}
  We call an embedding, possibly equipped with a
  $\TangentialFibration$-structure, 
  $
    \Embedding
    \colon
    \SurfaceGB{\Genus}{\BoundaryComponents}
    \to
    \Manifold
  $
  together with a closed tubular neighbourhood of 
  $\apply{\Embedding}{\SurfaceGB{\Genus}{\BoundaryComponents}}$ a 
  \introduce{thickened embedding of $\SurfaceGB{\Genus}{\BoundaryComponents}$}.
  We will write 
  $
    \ThickenedEmbeddingSpaceBoundaryCondition
      {\SurfaceGB{\Genus}{\BoundaryComponents}}
      {\Manifold}
      {\BoundaryConditionPar}
  $ for the 
  \introduce{space of thickened embeddings} and 
  $
    \ThickenedEmbeddingSpaceBoundaryConditionTangential
      {\TangentialFibration}
      {\SurfaceGB{\Genus}{\BoundaryComponents}}
      {\Manifold}
      {\BoundaryConditionTangential{\BoundaryConditionPar}}
  $ 
  for the \introduce{space of thickened embeddings with 
  $\TangentialFibration$-structure}. 
  We could furthermore replace 
  $\SurfaceGB{\Genus}{\BoundaryComponents}$ by a pair 
  $
    \left(
      \SurfaceGB{\Genus}{\BoundaryComponents},\Submanifold'
    \right)
  $
  and $\Manifold$ by 
  $
    \left(
      \Manifold,\Submanifold
    \right)
  $ 
  to obtain \introduce{spaces of thickened embeddings of pairs}.
\end{definition}

The topology of the space of thickened embeddings is 
quite tricky and we refer the reader to Section 2.5 in 
\cite{CRW16} for details. The important thing to note 
about this space is the following proposition:
\begin{proposition}
  $\DiffeomorphismGroupBoundary{\Manifold}$ acts via post 
  composition on 
  $
    \ThickenedEmbeddingSpaceBoundaryCondition
      {\Submanifold}
      {\Manifold}
      {\BoundaryConditionPar}
  $
  and this space is 
  $\DiffeomorphismGroupBoundary{\Manifold}$-locally 
  retractile.
\end{proposition}

\begin{corollary}
  The forgetful map 
  $
    \ThickenedEmbeddingSpaceBoundaryCondition
      {\SurfaceGB{\Genus}{\BoundaryComponents}}
      {\Manifold}
      {\BoundaryConditionPar}
    \to 
    \EmbeddingSpaceBoundaryCondition
      {\SurfaceGB{\Genus}{\BoundaryComponents}}
      {\Manifold}
      {\BoundaryConditionPar}
  $
  is a locally trivial fibration with fiber over a basepoint 
  $
    \Embedding
    \in 
    \EmbeddingSpaceBoundaryCondition
      {\SurfaceGB{\Genus}{\BoundaryComponents}}
      {\Manifold}
      {\BoundaryConditionPar}
  $
  given by the space of tubular neighborhoods of 
  $
    \apply
      {\Embedding}
        {\SurfaceGB{\Genus}{\BoundaryComponents}}
  $%
  .
  
  Furthermore the forgetful map 
  $
    \ThickenedEmbeddingSpaceBoundaryConditionTangential
      {\TangentialFibration}
      {\SurfaceGB{\Genus}{\BoundaryComponents}}
      {\Manifold}
      {\BoundaryConditionPar}
    \to
    \EmbeddingSpaceBoundaryConditionTangential
      {\TangentialFibration}
      {\SurfaceGB{\Genus}{\BoundaryComponents}}
      {\Manifold}
      {\BoundaryConditionPar}
  $
  is also a locally trivial fibration with fiber the space of 
  tubular neighborhoods of some 
  $
    \apply
      {\Embedding}
      {\SurfaceGB{\Genus}{\BoundaryComponents}}
  $%
  .
\end{corollary}
\begin{proof}
  The first assertion follows as 
  $
    \EmbeddingSpaceBoundaryCondition
      {\SurfaceGB{\Genus}{\BoundaryComponents}}
      {\Manifold}
      {\BoundaryConditionPar}
  $
  is $\DiffeomorphismGroupBoundary{\Manifold}$-locally 
  retractile.
  The second assertion follows since the following diagram is 
  a pullback diagram:
  \[
    \begin{tikzcd}
      \ThickenedEmbeddingSpaceBoundaryConditionTangential
        {\TangentialFibration}
        {\SurfaceGB{\Genus}{\BoundaryCondition}}
        {\Manifold}
        {\BoundaryConditionTangential{\BoundaryConditionPar}}
        \ar[r]
        \ar[d]
      &
      \ThickenedEmbeddingSpaceBoundaryCondition
        {\SurfaceGB{\Genus}{\BoundaryComponents}}
        {\Manifold}
        {\BoundaryConditionPar}
      \ar[d]
      \\
      \EmbeddingSpaceBoundaryConditionTangential
        {\TangentialFibration}
        {\SurfaceGB{\Genus}{\BoundaryComponents}}
        {\Manifold}
        {\BoundaryConditionTangential{\BoundaryConditionPar}}
      \ar[r]
      &
      \EmbeddingSpaceBoundaryCondition
        {\SurfaceGB{\Genus}{\BoundaryComponents}}
        {\Manifold}
        {\BoundaryConditionPar}
    \end{tikzcd}
  \]
\end{proof}
We will also need the following lemma:
\begin{lemma}
  \label{lem:RestrictionFibration}
  For every submanifold $\Submanifold$ of 
  $\SurfaceGB{\Genus}{\BoundaryComponents}$, the restriction maps
  \begin{align*}
    \EmbeddingSpaceBoundaryConditionTangential
      {\TangentialFibration}
      {\SurfaceGB{\Genus}{\BoundaryComponents}}
      {\Manifold}
      {\BoundaryConditionTangential{\BoundaryConditionPar}}
    &
    \to
    \EmbeddingSpaceBoundaryConditionTangential
      {\TangentialFibration}
      {\Submanifold}
      {\Manifold}
      {\BoundaryConditionTangential{\BoundaryConditionPar'}}
    \\
    \ThickenedEmbeddingSpaceBoundaryConditionTangential
      {\TangentialFibration}
      {\SurfaceGB{\Genus}{\BoundaryComponents}}
      {\Manifold}
      {\BoundaryConditionTangential{\BoundaryConditionPar}}
    &
    \to
    \ThickenedEmbeddingSpaceBoundaryConditionTangential
      {\TangentialFibration}
      {\Submanifold}
      {\Manifold}
      {\BoundaryConditionTangential{\BoundaryConditionPar'}}
  \end{align*}
  are fibrations
\end{lemma}

The proof will use the following lemma:
\begin{lemma}
  \label{lem:MetricDeformation}
  Let $\TopologicalSpace$ denote a metric space and $\Submanifold\subset 
  \TopologicalSpace$ a closed subset, with a 
  closed neighborhood $\Neighbourhood{\Submanifold}$ such that $\Submanifold$ 
  is a strong deformation retract of $\Neighbourhood{\Submanifold}$ via a 
  strong deformation $\Phi(x,t)$. 
  Then 
  $
    \TopologicalSpace \times \{0\} 
    \cup 
    \Submanifold \times [0,1]
  $
  is a strong deformation retract of $\TopologicalSpace \times [0,1]$. 
\end{lemma}
The assumptions about $\TopologicalSpace$ are quite restricting and the lemma 
should be true in a more general setting, but for the present context it 
certainly suffices.
\begin{proof}
  Define 
  $
    d
    \colon 
    \Neighbourhood{\Submanifold}
    \to 
    \Reals
  $
  as 
  \[
    \apply{d}{x}
    =
    \frac
      {d(x,\Submanifold)}
      {d(x,\Submanifold)+d(x,\partial \Neighbourhood{\Submanifold})}
  \] 
  As assumed above let $\Phi(-,t)$ denote a homotopy between the 
  identity on $\Neighbourhood{\Submanifold}$ and a retraction of the 
  neighborhood. 
  Then the deformation retract in the product space and the corresponding 
  homotopy is given by the following:
  \[
    \apply{\Psi}{(x,s),t}=
    \begin{cases}
      (\apply{\Phi}{x,0},st) 
      &
      \text{ for }\apply{d}{x}-s\geq 0
      \\
      \left(
        \apply
          {\Phi}
          {
            x
            ,
            (1-t)
            \left(
              \frac
                {s-\apply{d}{x}}
                {s}
            \right)
          }
        ,
        s-(1-t)\apply{d}{x}
      \right) 
      &
      \text{ for } 
      \apply{d}{x}-s\leq 0
      \\
      (x,st) 
      &
      \text{ for } 
      x\notin \Neighbourhood{\Submanifold}
    \end{cases}
  \]
  \end{proof}

\begin{proof}[Proof of Lemma~\ref{lem:RestrictionFibration}]
  Note that both statements follow from Lemma~\ref{lem:LocallyRetractile} if 
  $\TangentialFibration=\Identity$. 
  So consider the following lifting problem:
  \[
    \begin{tikzcd}
      \Ball{N}\times\{0\}
        \ar[r,"\bar{\ContinuousMap}"]
        \ar[d,hook]
      &
      \EmbeddingSpaceBoundaryConditionTangential
        {\TangentialFibration}
        {\SurfaceGB{\Genus}{\BoundaryComponents}}
        {\Manifold}
        {\BoundaryConditionTangential{\BoundaryConditionPar}}
        \ar[r]
        \ar[d]
      &
      \EmbeddingSpaceBoundaryConditionTangential
        {+}
        {\SurfaceGB{\Genus}{\BoundaryComponents}}
        {\Manifold}
        {\BoundaryConditionTangential{\BoundaryConditionPar}}
        \ar[d,"\text{fib}"]
      \\
      \Ball{N}\times[0,1]
        \ar[r,"\ContinuousMap"]
        \ar[ru,dashed]
      &
      \EmbeddingSpaceBoundaryConditionTangential
        {\TangentialFibration}
        {\Submanifold}
        {\Manifold}
        {\BoundaryConditionTangential{\BoundaryConditionPar'}}
        \ar[r]
      &
      \EmbeddingSpaceBoundaryConditionTangential
        {+}
        {\Submanifold}
        {\Manifold}
        {\BoundaryConditionTangential{\BoundaryConditionPar'}}
    \end{tikzcd}
  \]
  Since the right vertical map is a fibration, we can lift $\ContinuousMap$ to 
  $
    \EmbeddingSpaceBoundaryConditionTangential
      {+}
      {\SurfaceGB{\Genus}{\BoundaryComponents}}
      {\Manifold}
      {\BoundaryConditionTangential{\BoundaryConditionPar}}
  $%
  . Such a map corresponds to an embedding
  $
    \Embedding
    \colon
    \SurfaceGB{\Genus}{\BoundaryComponents}\times \Ball{N} \times [0,1]
    \to
    \Manifold \times \Ball{N} \times [0,1]
  $
  which is the identity on the $\Ball{N}\times[0,1]$ component. 
  Furthermore this embedding is equipped with a tangential structure on
  $
    \SurfaceGB{\Genus}{\BoundaryComponents}\times \Ball{N}\times\{0\}
    \cup
    \Submanifold \times [0,1]
  $
  which stems from $\ContinuousMap$ and $\bar{\ContinuousMap}$. Lifting 
  $\ContinuousMap$ corresponds to an extension of this tangential structure. 
  By Lemma~\ref{lem:MetricDeformation}, 
  $
    \SurfaceGB{\Genus}{\BoundaryComponents}\times \Ball{N}\times\{0\}
    \cup
    \Submanifold \times [0,1]
  $
  is a deformation retract of 
  $
    \SurfaceGB{\Genus}{\BoundaryComponents}\times \Ball{N}\times [0,1]
  $
  , hence 
  $
    \Embedding^{\ast}\TangentialSpace{\Manifold}
  $
  is isomorphic to 
  $
    \at
      {\Embedding^{\ast}}
      {
        \SurfaceGB{\Genus}{\BoundaryComponents}\times \Ball{N}\times\{0\}
        \cup
        \Submanifold \times [0,1]
      }
    \TangentialSpace{\Manifold}
  $%
  . Composition with such an isomorphism provides the required extension of the 
  tangential structure. The proof for the spaces of thickened embeddings is 
  completely analogous.
\end{proof}
\section{Stabilization Maps}
  \label{scn:Stabilization}
We will be interested in two kinds of stabilization maps. Up to homotopy 
equivalences, they are the same, but they will play different roles in the 
proof and therefore will be distinguished. They are motivated by the definition 
of inner and outer cobordisms in \cite{RW16}.

\subsection{Boundary Bordisms and Stabilization Maps}
Let $\Manifold$ denote an at least $5$-dimensional manifold and 
$\DistBoundaryManifold{0}{\Manifold}$ a codimension $0$ submanifold of 
$\BoundaryManifold{\Manifold}$.
By definition, a boundary condition 
$
  \BoundaryConditionTangential{\BoundaryCondition}
  \in
  \BoundaryConditionSpaceTangential
    {\TangentialFibration}
    {\SurfaceGB{\Genus}{\BoundaryComponents}}
    {\Manifold}
$
that intersects the boundary of $\DistBoundaryManifold{0}{\Manifold}$ 
transversely yields a submanifold with 
$\TangentialFibration$-structure in 
$\DistBoundaryManifold{0}{\Manifold}\times \left[0,1\right]$ by taking a
product with $\left[0,1\right]$. This specifies a 
$\TangentialFibration$-jet along 
$\DistBoundaryManifold{0}{\Manifold}\times \{0\}$. 
Together with a $\TangentialFibration$-jet 
$
  \BoundaryConditionTangential
    {\bar{\BoundaryCondition}}
$
along 
$
  \DistBoundaryManifold{0}{\Manifold}\times \{1\}
$
and using collars one obtains boundary conditions in
$
  \BoundaryConditionSpaceTangential
    {\TangentialFibration}
    {
      \SurfaceGB{\Genus'}{\BoundaryComponents+\BoundaryComponents'}
    }
    {\DistBoundaryManifold{0}{\Manifold}\times \left[0,1\right]}
$
for any $\Genus'$ and $\BoundaryComponents'$, which we will denote by 
$
  \BoundaryConditionTangential
    {\BoundaryCondition}
  \cup 
  \BoundaryConditionTangential
    {\bar{\BoundaryCondition}}
$%
. Let us denote 
$
  \bigsqcup_{\Genus'}
  \SpaceSubsurfaceBoundaryTangential
    {\TangentialFibration}
    {\Genus'}
    {\BoundaryComponents+\BoundaryComponents'}
    {\DistBoundaryManifold{0}{\Manifold}
      \times
      \left[0,1\right]
    }
    {
      \BoundaryConditionTangential{\BoundaryCondition}
      \cup 
      \BoundaryConditionTangential{\bar{\BoundaryCondition}}
    }
$
by
$
  \StabilizingBordism
    {\TangentialFibration}
    {\DistBoundaryManifold{0}{\Manifold}}
    {\BoundaryConditionTangential{\BoundaryCondition}}
    {\BoundaryConditionTangential{\bar{\BoundaryCondition}}}
$
and call it the space of \introduce{boundary bordisms}.

Given an element
$
  \BoundaryBordism
  \in
  \StabilizingBordism
    {\TangentialFibration}
    {\DistBoundaryManifold{0}{\Manifold}}
    {\BoundaryConditionTangential{\BoundaryCondition}}
    {\BoundaryConditionTangential{\bar{\BoundaryCondition}}}
$
one obtains a continuous
\begin{align*}
  -\cup \BoundaryBordism
  \colon 
  \SpaceSubsurfaceBoundaryTangential
    {\TangentialFibration}
    {\Genus}
    {\BoundaryComponents}
    {\Manifold}
    {\BoundaryConditionTangential{\BoundaryCondition}}
  &
  \to 
  \SpaceSubsurfaceBoundaryTangential
    {\TangentialFibration}
    {\Genus'}
    {\BoundaryComponents'}
    {\Elongation{\Manifold}{1}}
    {\BoundaryConditionTangential{\bar{\BoundaryCondition}}}
  \\
  \Subsurface
  &
  \mapsto 
  \Subsurface
  \cup 
  \BoundaryBordism
  \cup
  \BoundaryConditionTangential{\BoundaryCondition'}
  \times
  [0,1]
\end{align*}
where $\BoundaryCondition'$ consists of those boundary conditions that are not 
met by the boundary bordism.
These maps are called \introduce{stabilization maps}. 

We will call 
$
  \BoundaryBordism
  \cap 
  \DistBoundaryManifold{0}{\Manifold}\times \{0\}
$ 
the \introduce{incoming boundary of $\BoundaryBordism$}
and 
$
  \BoundaryBordism
  \cap 
  \DistBoundaryManifold{0}{\Manifold}\times \{1\}
$ 
the \introduce{outgoing boundary of $\BoundaryBordism$}. 
We have the following lemma, which says that $\BoundaryBordism$ decomposes into 
elementary pieces, for which it will be easier to prove homological stability. 

\begin{proposition}
  \label{prp:StabilizationMorse}
  Let $\Manifold$ be a manifold with space of $\TangentialFibration$-structures 
  of subplanes of $\TangentBundle{\Manifold}$. 
  If $\BoundaryBordism$ is a boundary bordism as above then there exists 
  $\BoundaryBordism_1,\ldots,\BoundaryBordism_k$ such that:
  \begin{enumerate}[(i)]
    \item 
      $\BoundaryBordism_i$ is a boundary bordism sitting inside of 
      $\DistBoundaryManifold{0}{\Manifold}\times[i-1,i]$ and all of 
      its connected components but one are homeomorphic to a cylinder.
    \item 
      $
        \BoundaryBordism_1
        \cup 
        \ldots 
        \cup 
        \BoundaryBordism_k 
        \simeq 
        k
        \cdot 
        \BoundaryBordism
        \subset 
        \DistBoundaryManifold{0}{\Manifold}
        \times 
        [0,k]
      $ 
      (The product means that 
      $
        (x,t)
        \in 
        \DistBoundaryManifold{0}{\Manifold}
        \times 
        [0,1]
      $ 
      is mapped to 
      $
        (x,kt)
        \in 
        \DistBoundaryManifold{0}{\Manifold} 
        \times
        [0,k]
      $%
      ) and $\simeq$ means in this case that the underlying submanifolds are 
      isotopic relative to their boundary and their tangential structures are 
      homotopic
    \item 
      the projection onto the second coordinate of 
      $
        \DistBoundaryManifold{0}{\Manifold}
        \times 
        [i-1,i]
      $ 
      restricted to $\BoundaryBordism_i$ is a Morse function with at most one 
      critical point.
    \item 
      If any such critical point happens to be a minimum or a maximum it has 
      to be a global maximum or minimum of 
      $
        \BoundaryBordism_1
        \cup 
        \ldots
        \cup 
        \BoundaryBordism_k
      $%
      .
  \end{enumerate}
\end{proposition}
\begin{remark}
  Since we implicitly required every connected component of $\BoundaryBordism$ 
  to meet 
  $
    \DistBoundaryManifold{0}{\Manifold}
    \times 
    \{0\}
  $
  the last part of the lemma says among other things that there is no 
  minimum. 
\end{remark}

To prove this proposition we will need to use surgery along a half-disk, a 
process we 
will explain now. We define 
\begin{enumerate}[(i)]
  \item 
    $
    \HalfDisk
    \coloneqq
    \left\{
      \left(
        \Coordinate_{1},\Coordinate_{2}
      \right)
      \in
      \Reals^{2}
      \middle|
      \Coordinate_{2} \geq 0 
      \text{, } 
      \EuclideanNorm{
        \left(
          \Coordinate_{1},\Coordinate_{2}
        \right)
      }
      \leq 1
    \right\}
    $
  \item
    $
    \DistBoundaryManifold{0}{\HalfDisk}
    \coloneqq
    \left\{
      \left(
        \Coordinate_{1},\Coordinate_{2}
      \right)
      \in 
      \BoundaryManifold{\HalfDisk}
    \middle| 
      \Coordinate_{2}=0
    \right\}
    $
  \item 
    $
    \DistBoundaryManifold{1}{\HalfDisk}
    \coloneqq
    \left\{
      \left(
        \Coordinate_{1},\Coordinate_{2}
      \right)
      \in 
      \BoundaryManifold{\HalfDisk}
    \middle|
      \EuclideanNorm{
        \left(
          \Coordinate_{1},\Coordinate_{2}
        \right)
        =
        1
      }
    \right\}
  $
\end{enumerate}

Let $\Subsurface\subset \Manifold$ denote a subsurface of $\Manifold$. 
We fix an embedded circle 
$
  c
  \subset 
  \DistBoundaryManifold{1}{\HalfDisk}
  \times 
  [0,1]
$ 
which contains 
$
  \left(
    \DistBoundaryManifold{1}{\HalfDisk}
    \setminus 
    (
      \BallRadius{\epsilon}{1}
      \cup
      \BallRadius{\epsilon}{-1}
    )
  \right)
  \times 
  \{0,1\}
$
and connects these two connected arcs via two small circle segments. 
Fix once and for all an embedding 
$
  a
  \colon 
  \Ball{2} 
  \to 
  \HalfDisk
  \times 
  [0,1]
$
such that 
$
  \apply
    {a}
    {\BoundaryManifold{\Ball{2}}}
  =
  c
$ 
and 
$
  \apply
    {a}
    {\Ball{2}\setminus \frac{1}{2}\Ball{2}}
  \subset
  \DistBoundaryManifold{1}{\HalfDisk}
  \times
  [0,1]
$
and such that 
$
  \apply
    {a}
    {\Reals\cdot (1,0) \cap \frac{1}{2}\Ball{2}}
$ 
is very close to 
$
  \DistBoundaryManifold{0}{\HalfDisk}
  \times
  \{\frac{1}{2}\}
$
(see Figure~\ref{fig:surgery} to clarify the definition of $c$ and $a$). 
Note that 
$
  \apply{a}{\Ball{2}}
$
cuts 
$
  \DistBoundaryManifold{1}{\HalfDisk}
  \times
  [0,1]
$ 
into two connected components and let $X$ denote the "inner one" i.e. the one 
that does not contain 
$\DistBoundaryManifold{1}{\HalfDisk}\times \{0,1\}$.
\begin{figure}[htb]           
  \centering                  
  \def\svgwidth{200pt}    
  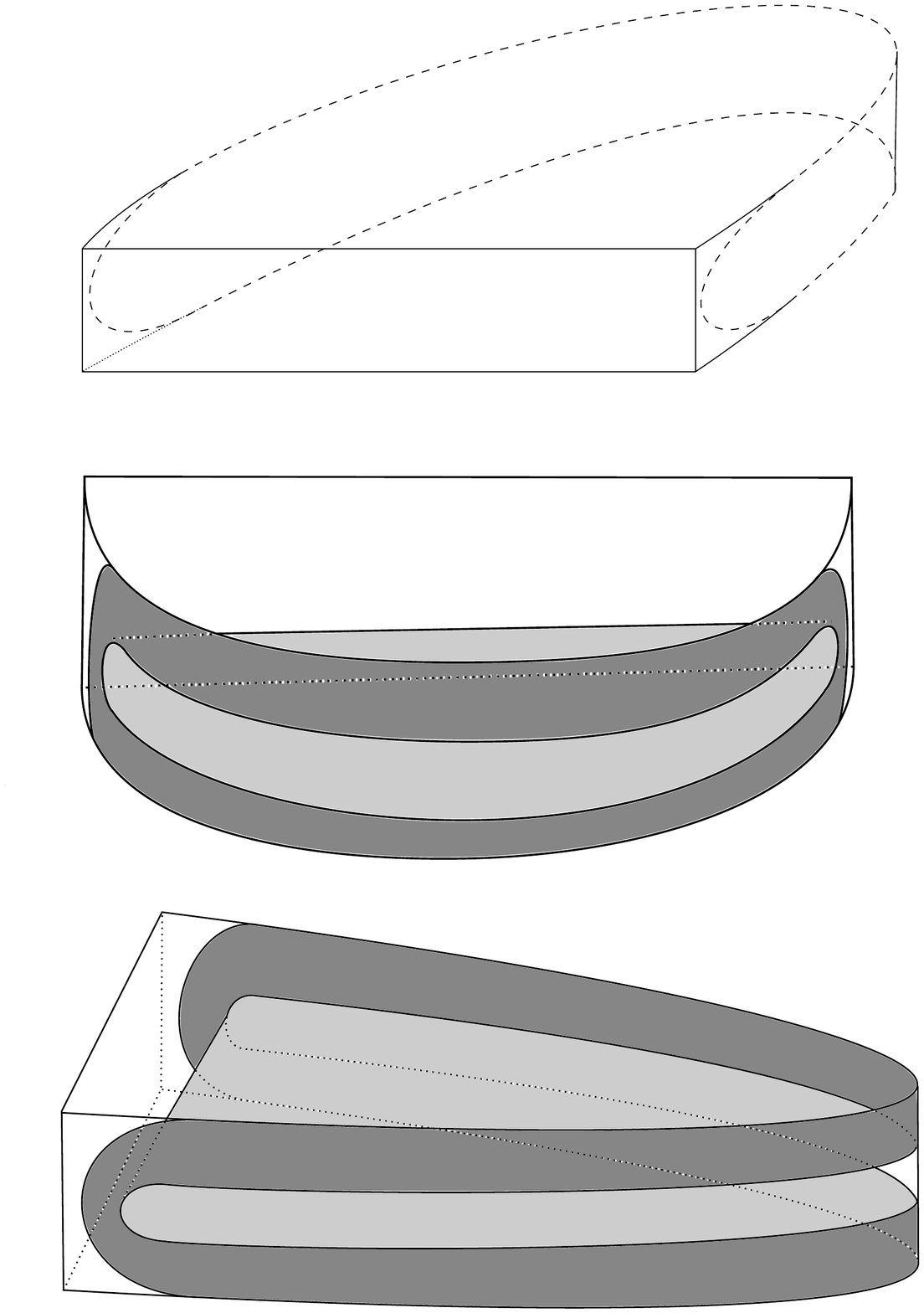
  \caption{The dashed circle in the first picture represents $c$ and the union 
  of the gray areas in the other two pictures represents the image of $a$ (of 
  course the boundary between the two areas has to be smoothed, but this 
  exceeds the drawing skills of the author)}
  \label{fig:surgery}          
\end{figure}
\begin{definition}
  Let $\Manifold$ denote a manifold as above and $\Subsurface$ a subsurface of 
  $\Manifold$.
  Let 
  $
    \SurgeryEmbedding
    \colon
    (
      \HalfDisk\times [0,1]
      ,
      \DistBoundaryManifold{1}{\HalfDisk}\times[0,1]
    )
    \to
    (\Manifold,\Subsurface)
  $
  denote an embedding, then we will denote 
  $
    (
      \Subsurface
      \setminus 
      \apply{\SurgeryEmbedding}{X}
    )
    \cup 
    \apply
      {\SurgeryEmbedding}
      {
        \apply
          {a}
          {\Ball{2}}
      }
  $ 
  by $\Surgery{\Subsurface}{\SurgeryEmbedding}$. 
  The above requirements for $a$ ensure that this is a smooth manifold. 
  
  Suppose that the dimension of $\Manifold$ is at least $5$ and that 
  $\Subsurface$ is a subsurface with tangential structure. 
  In this case $\Subsurface$ and $\Surgery{\Subsurface}{\SurgeryEmbedding}$ 
  are isotopic and we can use this to give 
  $\Surgery{\Subsurface}{\SurgeryEmbedding}$ 
  some tangential structure. 
  Even though this is not well-defined we will denote any such subsurface 
  with tangential structure by $\Surgery{\Subsurface}{\SurgeryEmbedding}$.
\end{definition}

Assume that the dimension of $\Manifold$ is at least $5$ and $\Subsurface$ 
is a subsurface, given a map of a half-disk 
$
  \SurgeryEmbedding
  \colon
  (\HalfDisk,\DistBoundaryManifold{1}{\HalfDisk})
  \to 
  (\Manifold,\Subsurface)
$%
, which is an embedding at every point except at $(\pm1,0)$, where 
$
  \apply
    {\Differential{\SurgeryEmbedding}}
    {\TangentSpace{(\pm1,0)}{\DistBoundaryManifold{1}{\HalfDisk}}}
$
agrees with 
$
  \apply
    {\Differential{\SurgeryEmbedding}}
    {\TangentSpace{(\pm1,0)}{\DistBoundaryManifold{0}{\HalfDisk}}}
$%
. We can use the sufficiently high codimension to ensure that 
$\SurgeryEmbedding$ 
extends to an embedding 
$
  \SurgeryEmbedding'
  \colon
  (
    \HalfDisk\times [0,1]
    ,
    \DistBoundaryManifold{1}{\HalfDisk}
  )
  \to
  (\Manifold,\Subsurface)
$
such that 
$
  \apply{\SurgeryEmbedding}{\DistBoundaryManifold{0}{\HalfDisk}}
$
agrees with 
$
  \apply
    {\SurgeryEmbedding'}
    {\apply
      {a}
      {\Reals\cdot(1,0) \cap \frac{1}{2}\Ball{2}}
    }
$%
. In this case we will still write 
$\Surgery{\Subsurface}{\SurgeryEmbedding}$ for the surgery along this extension 
of 
$\SurgeryEmbedding$.
Armed with this definition we can start proving the previous lemma. 
\begin{proof}[Proof of Proposition \ref{prp:StabilizationMorse}]
  We start by showing that there is an isotopy of $\BoundaryBordism$ such that 
  $
    \Projection{[0,1]}
    \colon
    \DistBoundaryManifold{0}{\Manifold}
    \times
    [0,1]
    \to
    [0,1]
  $
  restricted to $\BoundaryBordism$ is a Morse function with the 
  property that every critical value corresponds to a unique critical point, 
  which proves the first three parts of the lemma. 
  Then we proceed by explaining how to remove local minima. 
  Using the symmetry between minima and maxima given by the flip map 
  $(x,t)\mapsto (x,1-t)$ this also explains how to eliminate local maxima.
  
  We have an open inclusion
  \[
    \EmbeddingSpaceTangential
      {\TangentialFibration}
      {\BoundaryBordism}
      {\DistBoundaryManifold{0}{\Manifold}\times[0,1]}
    \to 
    \SmoothMappingSpace
      {\BoundaryBordism}
      {\DistBoundaryManifold{0}{\Manifold}\times[0,1]}
    \cong 
    \SmoothMappingSpace{\BoundaryBordism}{[0,1]}
    \times
    \SmoothMappingSpace{\BoundaryBordism}{\DistBoundaryManifold{0}{\Manifold}}
  \]
  The isomorphism stems from the identification of maps into products as 
  products of maps and the described map has the following form 
  $
    \Embedding
    \mapsto
    (
      \Projection{\DistBoundaryManifold{0}{\Manifold}}
      \circ
      \Embedding
      ,
      \Projection{[0,1]}
      \circ
      \Embedding
    )
  $. 
  That this map is indeed open is proven in \cite{cerf} Chapter~2, 
  Section~1.2. 
  Since Morse functions form an open and dense subset of all real 
  functions we get that in 
  $
    \SmoothMappingSpace{\BoundaryBordism}{[0,1]}
    \times
    \SmoothMappingSpace{\BoundaryBordism}{\DistBoundaryManifold{0}{\Manifold}}
  $
  the subset
  $
    \text{Morse}\!(\BoundaryBordism,[0,1])
    \times
    \SmoothMappingSpace{\BoundaryBordism}{\DistBoundaryManifold{0}{\Manifold}}
  $
  is open and dense as well. 
  Since the space of embeddings is also open there exists an embedding, such 
  that the projection onto the second component is a Morse function, such that 
  the embedding is arbitrarily close to the inclusion from $\BoundaryBordism$ 
  into $\DistBoundaryManifold{0}{\Manifold}$. 
  Since the space of embeddings is locally path-connected (see Theorem~44.1 of 
  \cite{CSGA}), we conclude that $\BoundaryBordism$ is isotopic to a subsurface
  such that the projection restricted to this subsurface is a Morse 
  function.
  It is easy to see that we can arrange the critical points to have
  different values by moving them up or down a little.
  
  For the fourth part of the lemma chose a riemannian metric $g$ on 
  $
    \DistBoundaryManifold{0}{\Manifold}
    \times
    [0,1]
  $ 
  such that $\partial_t$ (the directional derivative of the interval 
  coordinate) is orthogonal to 
  $
    \TangentSpace{(x,t)}{\DistBoundaryManifold{0}{\Manifold}\times\{t\}}
  $
  for every $(x,t)$ and assume that 
  $
    \at
      {\Projection{[0,1]}}
      {\BoundaryBordism}
  $
  is a Morse function with distinct critical values. 
  
  Before we can continue we have to fix some notation:
  Consider a critical point $x$ of 
  $
    \at
      {\Projection{[0,1]}}
      {\BoundaryBordism}
  $%
  , which is not a maximum.
  Consider the negative gradient flow of 
  $
    \at
      {\Projection{[0,1]}}
      {\BoundaryBordism}
  $
  in a neighborhood of $x$. 
  Then there have to be some flow lines $\gamma$ which converge to $x$ 
  meaning that 
  $
    \lim_{T\to-\infty}
    \apply{\gamma}{T}
    =
    x
  $
  but since $\BoundaryBordism$ is compact this 
  gradient flow line has to either meet a boundary component or it is an 
  embedding of $\Reals$ into $\BoundaryBordism$ in which case the compactness 
  of $\BoundaryBordism$ ensures 
  that 
  $
    \lim_{T\to\infty}
    \apply{\gamma}{T}
  $
  exists and has to be a critical value as well. 
  In this case we call $x$ the starting point of $\gamma$ and the other 
  limit point the endpoint of $\gamma$. 
  Furthermore note that if $\gamma$ ends at a point $x$, there also has to be 
  flow line of the gradient flow, which 
  starts at $x$ and goes in the opposite direction as $\gamma$, by which we 
  mean that in a Morse chart centered around $x$ $\gamma$ is given by 
  $e^{1/2T}v$, then the other curve is given by $-e^{1/2T}v$.
  
  Now let us proceed with the proof of the theorem. 
  Let $p$ denote a minimum for 
  $
    \at
      {\Projection{[0,1]}}
      {\BoundaryBordism}
  $
  or in other words a $0$-handle with respect to the handle-decomposition 
  imposed by 
  $
    \at
      {\Projection{[0,1]}}
      {\BoundaryBordism}
  $%
  . Since every connected component of 
  $\BoundaryBordism$ 
  meets 
  $
    \DistBoundaryManifold{0}{\Manifold}
    \times
    \{0\}
  $
  there has to be a $1$-handle, which cancels the $0$-handle 
  given by $p$. 
  In other words there has to exist a gradient flow line 
  $\gamma_1$ for 
  $
    \at
      {\Projection{[0,1]}}
      {\BoundaryBordism}
  $
  starting at $p$, which ends at an index $1$ 
  critical point $q$. 
  Following the negative gradient flow along the opposite 
  direction defined by $\gamma_1$ at $q$ gives us a gradient flow line 
  $\gamma_2$ that ends at some critical point $p'$ which is lower with respect 
  to its $t$-component than $q$. 
  By reparametrizing and including their endpoint and starting point we can 
  consider $\gamma_1$ and $\gamma_2$ as arcs starting and ending at critical 
  points. 
  We denote by $\gamma$ the concatenation of $\gamma_1$ and $\gamma_2$. 
  By changing $\BoundaryBordism$ a little bit but fixing 
  $
    \at
      {\Projection{[0,1]}}
      {\BoundaryBordism}
  $
  we can arrange that there is no open segment of 
  $\gamma$, where 
  $
    \Projection{\DistBoundaryManifold{0}{\Manifold}}
    \Circle
    \gamma
  $
  is constant. 
  Fix a strictly increasing or strictly decreasing function 
  $
    \SmoothMap
    \colon
    [0,1]
    \to 
    [0,1]
  $
  such that 
  $
    \apply{\SmoothMap}{0}
    =
    \apply{\Projection{[0,1]}}{p'}
  $
  and 
  $
    \apply{\SmoothMap}{1}
    =
    \apply{\Projection{[0,1]}}{p}
  $%
  . Then define 
  \begin{align*}
    H''
    \colon 
    [0,1]
    \times 
    [0,1]
    &
    \to 
    \DistBoundaryManifold{0}{\Manifold} 
    \times 
    [0,1]
    \\
    (t',s)
    &
    \mapsto 
    \left(
      \apply
        {\Projection{\DistBoundaryManifold{0}{\Manifold}}}
        {\apply{\gamma}{s}}
      ,
      (1-t')
      \apply
        {\Projection{[0,1]}}
        {\apply{\gamma}{s}}
      +
      t'\apply{\SmoothMap}{s}
    \right)
  \end{align*}
  Using the main theorem of \cite{Whitney}, which says that embeddings are 
  dense in the space of mappings, we can replace $H''$ by a $C^\infty$ close 
  map $H'$ such that $\bar{H}'$ factors through an embedding of a half disc $H'$
  (this is possible since $H''$ is constant on $[0,1]\times\{0\}$ and 
  $[0,1]\times\{1\}$)
  such that 
  $
    \apply
      {\Differential{\Projection{[0,1]}\circ H'}}
      {\TangentBundle{\HalfDisk}}
    =<\partial_t>
  $%
  , 
  $
    \apply
      {\Differential{\Projection{[0,1]}\circ H'}}
      {\TangentBundle{\DistBoundaryManifold{0}{\HalfDisk}}}
    =
    \lambda \partial_{t}
  $
  for $\lambda$ a negative function or $\lambda$ a positive function depending 
  on the height of $p$ and $p'$, and that 
  $
    \at
      {H'}
      {\DistBoundaryManifold{1}{\HalfDisk}}
    =
    \gamma
  $%
  . We can further replace $H'$ by $H$, using the denseness of transversal 
  maps, such that $H$ still fulfils the conditions above and that 
  $
    \apply
      {H}
      {\HalfDisk}
    \cap
    \BoundaryBordism
    =
    \gamma
  $%
  .
  
  Now the idea is to do surgery along $H$ to cancel pairs of critical points to 
  get rid of minima (In fact we have to change $H$ a little bit for technical 
  reasons, but this doesn't change the idea). 
  This strategy works since $\BoundaryBordism$ and 
  $\Surgery{\BoundaryBordism}{H}$ are isotopic by using a deformation of $H$ 
  to $\gamma$.
  
  We have to consider three and a half cases: 
  First consider the case, where $p'$ is also a minimum. 
  Without loss of generality we assume that $p'$ is lower than $p$. 
  In this case we replace $\gamma$ by the curve that ends a 
  little bit before $p'$ with respect to the flow time and alter 
  $
    \at
      {H}
      {\DistBoundaryManifold{0}{\HalfDisk}}
  $
  a little bit such that it is tangential to 
  $\gamma$ at 
  $\apply{H}{\pm1,0}$. 
  We alter $\gamma$ a bit further such that it starts a little bit before $p$ 
  in the sense that we go a short time in the 
  opposite direction of $\gamma$ at $p$ and proceed in the same fashion as 
  before to produce an $\bar{H}$ with the aforementioned properties. 
  Then $\Surgery{\BoundaryBordism}{\bar{H}}$ is isotopic to $\BoundaryBordism$ 
  , 
  $
    \at
      {\Projection{[0,1]}}
      {\Surgery{\BoundaryBordism}{\bar{H}}}
  $
  is still a Morse function with the same critical points as 
  $
    \at{\Projection{[0,1]}}{\BoundaryBordism}
  $ 
  except for $p$ and $q$ (this is possible by choosing a very small extension 
  of $\bar{H}$ in the definition of surgery). 
  In this case we have eliminated a minimum and start this process anew with 
  another minimum.
  
  The next case is the case where $p'$ is an index $1$ critical point which is 
  lower than $p$. 
  In this case we alter $\gamma$ exactly as in the previous 
  case and then 
  $
    \at{\Projection{[0,1]}}{\Surgery{\BoundaryBordism}{H}}
  $ 
  is still a Morse function with the same 
  critical points as 
  $
    \at{\Projection{[0,1]}}{\BoundaryBordism}
  $
  except for $p$ and $q$.
  In this case we have again eliminated a minimum without producing any new 
  critical points.
  
  The last full case, where $p'$ is an index $1$ critical point, which is 
  higher than $p$, is a little bit different compared to the previous two quite 
  similar cases and we will have to alter $H$ and $\gamma$ a little bit. 
  In this case we let $\gamma$ start not at $p$ but at 
  $\apply{\gamma}{\epsilon}$ for some small epsilon and choose 
  $
    \apply
      {H}
      {\DistBoundaryManifold{0}{\HalfDisk}}
  $
  to be tangential to $\gamma$ at its starting point. 
  Furthermore we alter $\gamma$ at its endpoint by stopping a short time before 
  arriving at $p'$ and then going down with respect to $\Projection{[0,1]}$ 
  avoiding $p'$ and points that lie on the gradient flow line which ends at 
  $p'$ and stopping at a point which is lower than $p'$. 
  As before 
  $
    \at
      {\Projection{[0,1]}}
      {\Surgery{\BoundaryBordism}{H}}
  $
  is a Morse function but the only critical point that could have vanished on 
  $\Surgery{\BoundaryBordism}{H}$ is $q$ as can be seen by considering 
  the derivative of $H$. 
  Since deleting only one critical point changes the Euler characteristic of 
  $\BoundaryBordism$, we know that this procedure has to create an index $1$ 
  critical point somewhere, but the only place where 
  this can occur is in a small neighborhood of 
  $
    \apply{H}{-1,0}
  $
  because at the other boundary point of $\gamma$ we can arrange 
  $
    \at
      {H}
      {\DistBoundaryManifold{0}{\HalfDisk}}
  $
  to be tangential to $\gamma$. 
  Another way to see this, is to note that 
  $
    \at
      {H}
      {\DistBoundaryManifold{0}{\HalfDisk}}
  $
  becomes part of a new gradient flow line of 
  $
    \at{\Projection{[0,1]}}{\Surgery{\BoundaryBordism}{H}}
  $%
  , which has to flow down to $p$ after passing a 
  neighborhood of 
  $
    \apply{H}{-1,0}
  $%
  , where it has to flow up. 
  Because every point close to 
  $
    \apply{H}{-1,0}
  $
  is a point that can flow down along 
  $
    \apply
      {\text{grad}}
      {\at{\Projection{[0,1]}}{\Surgery{\BoundaryBordism}{H}}}
  $
  to a critical point that is not $p$, we 
  conclude that the gradient flow line 
  $
    \at{H}{\DistBoundaryManifold{0}{\HalfDisk}}
  $
  left the open unstable manifold of $p$. 
  This procedure moves the peak of $\gamma$ below $p'$ so that $p'$ does not 
  play a role anymore in the canceling of $p$.
  
  All in all these considerations imply that we can proceed again as before but 
  we will not run into $p'$ again essentially reducing the number of critical 
  points that come into question for this construction.
  
  The last "half" case is the case, where $\gamma$ does not run into a critical 
  point after passing the index $1$-critical point, but rather runs into the 
  boundary 
  $
    \BoundaryBordism
    \cap
    \DistBoundaryManifold{0}{\Manifold}
    \times
    \{0\}
  $%
  . This case is handled completely analogous to the first two cases.
\end{proof}

This lemma allows us to split every stabilization map into a composition of 
maps which are given by taking the union with some $\BoundaryBordism_i$. 
The third and fourth condition limit the topology of $\BoundaryBordism_i$ to 
be either a pair of pants with incoming boundary consisting of one or two 
circles (a single index $1$ critical point) and corresponding outgoing boundary 
or a disk with incoming boundary a circle (a single index $2$ critical point). 
Accordingly we will write
\begin{align*}
  \StabilizationAlpha{\Genus}{\BoundaryCondition}
  \colon
  \SpaceSubsurfaceBoundaryTangential
    {\TangentialFibration}
    {\Genus}
    {\BoundaryComponents}
    {\Manifold}
    {\BoundaryConditionTangential{\BoundaryCondition}}
  &
  \to
  \SpaceSubsurfaceBoundaryTangential
    {\TangentialFibration}
    {\Genus+1}
    {\BoundaryComponents-1}
    {\Elongation{\Manifold}{1}}
    {\BoundaryConditionTangential{\bar{\BoundaryCondition}}}
  \\
  \StabilizationBeta{\Genus}{\BoundaryCondition}
  \colon
  \SpaceSubsurfaceBoundaryTangential
    {\TangentialFibration}
    {\Genus}
    {\BoundaryComponents}
    {\Manifold}
    {\BoundaryConditionTangential{\BoundaryCondition}}
  &
  \to
  \SpaceSubsurfaceBoundaryTangential
    {\TangentialFibration}
    {\Genus}
    {\BoundaryComponents+1}
    {\Elongation{\Manifold}{1}}
    {\BoundaryConditionTangential{\bar{\BoundaryCondition}}}
  \\
  \StabilizationGamma{\Genus}{\BoundaryCondition}
  \colon
  \SpaceSubsurfaceBoundaryTangential
    {\TangentialFibration}
    {\Genus}
    {\BoundaryComponents}
    {\Manifold}
    {\BoundaryConditionTangential{\BoundaryCondition}}
  &
  \to
  \SpaceSubsurfaceBoundaryTangential
    {\TangentialFibration}
    {\Genus}
    {\BoundaryComponents-1}
    {\Elongation{\Manifold}{1}}
    {\BoundaryConditionTangential{\bar{\BoundaryCondition}}}
\end{align*}
for the corresponding stabilization maps i.e. 
$\StabilizationAlpha{\Genus}{\BoundaryComponents}$ corresponds to 
a map which is given by stabilization with a pair of pants with incoming 
boundary two circles, $\StabilizationBeta{\Genus}{\BoundaryComponents}$ is a 
map which is given by stabilization with a pair of pants with incoming boundary 
a single circle and 
$\StabilizationGamma{\Genus}{\BoundaryComponents}$ is given by taking the union 
with a disk.
We will say that such a map is a \introduce{stabilization map of type 
$\alpha$,$\beta$ or $\gamma$} respectively.
Even though the homotopy class of these stabilization maps depend on the 
isotopy type of the boundary bordism, we will suppress $\BoundaryBordism$ from 
the notation. 
We will call $\BoundaryBordism$ the \introduce{corresponding boundary bordism} 
and whenever we need it, we will clarify that it corresponds to some fixed 
stabilization map.
In summary, one has:
\begin{corollary}
  Let $\Manifold$ be a manifold with a space of 
  $\TangentialFibration$-structures 
  of subplanes of $\TangentBundle{\Manifold}$. Then every stabilization map can 
  be written as a composition of maps of type $\alpha$, $\beta$ and $\gamma$.
\end{corollary}

\subsection{Subset Bordisms}
  Suppose $\Manifold''$ and $\Manifold'$ are closed (as a subset) 
  codimension-$0$-submanifolds, possibly with corners, of $\Manifold$ such that 
  $\Manifold'\subset\Manifold''$. Here we do not require that the boundaries of 
  these submanifolds lie in the boundary of $\Manifold$. 
  We deliberately do not go into detail about manifolds with corners as in the 
  application later on the situation will be more transparent. 
  Suppose we have a $2$-dimensional submanifold with 
  $\TangentialFibration$-strucutre
  $
    \SubsetBordism
    \subset 
    \overline{\Manifold''\setminus\Manifold'}
  $
  , then we have the following map:
  \begin{align*}
    -\cup \SubsetBordism
    \colon
    \SpaceSubsurfaceBoundaryTangential
      {\TangentialFibration}
      {\Genus}
      {\BoundaryComponents}
      {\Manifold'}
      {\BoundaryConditionTangential{\BoundaryCondition}}
    &
    \to
    \SpaceSubsurfaceBoundaryTangential
      {\TangentialFibration}
      {\Genus'}
      {\BoundaryComponents'}
      {\Manifold''}
      {\BoundaryConditionTangential{\BoundaryCondition'}}
    \\
    \Subsurface
    &
    \mapsto
    \Subsurface
    \cup 
    \SubsetBordism
  \end{align*}
  Here $\BoundaryConditionTangential{\BoundaryCondition}$ is a boundary 
  condition that contains (but does not necessarily equal!)
  $
    \SubsetBordism\cap \BoundaryManifold{\Manifold'}
  $
  and $\Genus'$ and $\BoundaryComponents'$ depend on the topology of 
  $\SubsetBordism$.
  We call $\SubsetBordism$ a \introduce{subset bordism}. 
  \begin{remark}
  \label{rmk:SmoothingTheAngleStabilization}
    Note that if $\overline{\Manifold''\setminus\Manifold'}$ is homeomorphic to 
    a ball and either $\Manifold''$ or $\Manifold'$ has no corners, then one 
    can find homeomorphisms that "smooth the corners" 
    $
      \Phi'
      \colon 
      \Manifold'
      \to 
      \tilde{\Manifold}'
    $
    and
    $
      \Phi''
      \colon 
      \Manifold''
      \to 
      \Elongation{\tilde{\Manifold}'}{1}
    $%
    , where $\tilde{\Manifold}'$ does not have corners such that
    \[
      \begin{tikzcd}
        \SpaceSubsurfaceBoundaryTangential
          {\TangentialFibration}
          {\Genus}
          {\BoundaryComponents}
          {\Manifold'}
          {\BoundaryConditionTangential{\BoundaryCondition}}
          \ar[r,"-\cup \SubsetBordism"]
          \ar[d,"\Phi'_{\ast}"]
        &
        \SpaceSubsurfaceBoundaryTangential
          {\TangentialFibration}
          {\Genus'}
          {\BoundaryComponents'}
          {\Manifold''}
          {\BoundaryConditionTangential{\BoundaryCondition'}}
          \ar[d,"\Phi''_{\ast}"]
        \\
        \SpaceSubsurfaceBoundaryTangential
          {\tilde{\TangentialFibration}}
          {\Genus}
          {\BoundaryComponents}
          {\tilde{\Manifold}'}
          {\BoundaryConditionTangential{\tilde{\BoundaryCondition}}}
          \ar[r,"\varphi"]
        &
        \SpaceSubsurfaceBoundaryTangential
          {\tilde{\TangentialFibration}}
          {\Genus'}
          {\BoundaryComponents'}
          {\Elongation{\tilde{\Manifold}'}{1}}
          {\BoundaryConditionTangential{\tilde{\BoundaryCondition}'}}
      \end{tikzcd}
    \]
    where $\Phi'_{\ast}$ and $\Phi''_{\ast}$ are induced by the aforementioned 
    homeomorphisms and such that $\varphi$ is a stabilization map and the 
    diagram commutes up to homotopy. 
    In other words, in this case 
    $
      -\cup\SubsetBordism
    $
    behaves up to homotopy like a stabilization map. In particular their 
    homological properties (being an isomorphism or epimorphism etc.) agree. 
  \end{remark}
  Examples of subset bordisms will be the approximate augmentations,
  which will be introduced in Section~\ref{scn:Resolution}.
\section{Resolutions via Semi-Simplicial Spaces}
  \label{scn:SemiSimplicial}
The following section is based entirely on the fourth 
section of \cite{CRW16} and all its statements and proofs or references to 
proofs 
can be found there. Let $\SemiSimplicialCategory$ 
denote 
the category, whose objects are non-empty finite 
ordinals and 
whose morphisms are injective order preserving maps.
A \introduce{semi-simplicial space} is a contravariant 
functor 
$
  \SemiSimplicialSpace
  \colon
  \SemiSimplicialCategory
  \to 
  \textbf{Top}
$%
, and $\SemiSimplicialSpaceIndex{\SemiSimplicialIndex}$ 
will 
denote the image of 
$
  \Simplex{\SemiSimplicialIndex}
  \coloneqq
  \left\{
    0,\ldots,\SemiSimplicialIndex
  \right\}
$%
.
We denote by 
$\FaceMap{\FaceIndex}{\SemiSimplicialIndex}$ the 
face maps that stem from 
the inclusion 
$
  \left\{
    0,\ldots,\SemiSimplicialIndex-1
  \right\}
  \to
  \left\{
    0,\ldots \SemiSimplicialIndex
  \right\}
$
that misses 
$
  \FaceIndex
  \in
  \left\{
    0,\ldots \SemiSimplicialIndex
  \right\}
$%
.

A semi-simplicial space $\SemiSimplicialSpace$ together 
with 
a continuous map 
$
  \Augmentation{\SemiSimplicialSpaceIndex{0}}
  \colon
  \SemiSimplicialSpaceIndex{0}
  \to
  \TopologicalSpace
$
is called a 
\introduce{semi-simplicial space augmented over 
$\TopologicalSpace$} if
\begin{equation}
\label{eqt:Augmentation}
  \Augmentation{\SemiSimplicialSpaceIndex{0}}
  \circ
  \FaceMap{0}{1}
  = 
  \Augmentation{\SemiSimplicialSpaceIndex{0}}
  \circ
  \FaceMap{1}{1}
\end{equation}%
Using the face maps and 
(\ref{eqt:Augmentation}) one obtains one map for every level of 
the semi-simplicial space, the 
\introduce{$\SemiSimplicialIndex$-th augmentation map},
$\SemiSimplicialSpaceIndex{\SemiSimplicialIndex}\to 
\TopologicalSpace$, that we will denote by 
$
  \Augmentation
    {\SemiSimplicialSpaceIndex{\SemiSimplicialIndex}}
$%
. We will usually write
$
  \Augmentation{\SemiSimplicialSpace}
  \colon
  \SemiSimplicialSpace
  \to
  \TopologicalSpace
$
for the collection of all augmentation maps and call 
this the 
\introduce{augmentation}.
If $\SemiSimplicialSpace$ and $\SemiSimplicialSpace'$ 
denote two possibly augmented semi-simplicial spaces, 
we call 
a natural transformation between the functors a 
\introduce{semi-simplicial map}, if the 
semi-simplicial spaces are augmented we furthermore 
require 
the maps given by the natural transformation to commute 
with the augmentations. In that case we will denote the map by 
$\ContinuousMap_{\bullet}$, the maps on the levels of the simplicial set by 
$\ContinuousMap_{\SemiSimplicialIndex}$ and the map between the topological 
spaces by $\ContinuousMap$.

There is a geometric realization functor (compare 
\cite{ERW}) 
\[
  \GeometricRealization{\cdot} 
  \colon
  \textbf{Semi-simplicial spaces} 
  \to
  \textbf{Top}
\]
and we call a semi-simplicial space augmented over 
$\TopologicalSpace$ an 
\introduce{$\DimensionIndex$-resolution} if the induced 
map between the geometric realization of the 
semi-simplicial 
space and $\TopologicalSpace$ is 
$\DimensionIndex$-connected. 
If the induced map is a weak equivalence, we will call 
the 
augmented semi-simplicial space a \introduce{resolution 
of 
$\TopologicalSpace$}.

Furthermore we call an augmented semi-simplicial space 
$
  \Augmentation{\SemiSimplicialSpace}
  \colon
  \SemiSimplicialSpace
  \to
  \TopologicalSpace
$
an 
\introduce{augmented topological flag complex} if 
\begin{enumerate}[(i)]
  \item
  the product map 
  $
    \SemiSimplicialSpaceIndex{\SemiSimplicialIndex}
    \to 
    \SemiSimplicialSpaceIndex{0}
    \times_X
    \ldots
    \times_X 
    \SemiSimplicialSpaceIndex{0}
  $%
  , given by the product of the face maps, is an open 
  embedding
  \item 
  a tuple 
  $
    \left(
      \Point_{0},
      \ldots, 
      \Point_{\SemiSimplicialIndex}
    \right)
    \in
    \SemiSimplicialSpaceIndex{0}
    \times_X
    \ldots
    \times_X 
    \SemiSimplicialSpaceIndex{0}
  $
  is in 
  $\SemiSimplicialSpaceIndex{\SemiSimplicialIndex}$ if 
  and only if for each 
  $
    0
    \leq
    \FaceIndex
    < 
    \FaceIndex'
    \leq 
    \SemiSimplicialIndex
  $
  we have 
  $
    \left(
      \Point_{\FaceIndex},\Point_\FaceIndex'
    \right) 
    \in 
    \SemiSimplicialSpaceIndex{0}
    \times_X 
    \SemiSimplicialSpaceIndex{0}
  $
  lies in $\SemiSimplicialSpaceIndex{1}$.
\end{enumerate}
The following three lemmas will be crucial for the
proofs of the following sections and can be found in 
Section~4 of \cite{CRW16}.
\begin{lemma}
\label{lem:FlagComplex}
  Let 
  $
    \Augmentation{\SemiSimplicialSpace} 
    \colon 
    \SemiSimplicialSpace 
    \to 
    \TopologicalSpace
  $
  be an augmented topological flag complex. 
  Suppose that
  \begin{enumerate}[(i)]
  \item 
    $
      \Augmentation{\SemiSimplicialSpaceIndex{0}}
      \colon 
      \SemiSimplicialSpaceIndex{0} 
      \to 
      \TopologicalSpace
    $
    has local sections that is 
    $\Augmentation{\SemiSimplicialSpaceIndex{0}}$ is surjective and for 
    each $\Point_{0} \in \SemiSimplicialSpaceIndex{0}$ such that 
    $
      \apply
        {\Augmentation{\SemiSimplicialSpaceIndex{0}}}
        {\Point_{0}}
      =
      \Point
      \in
      \TopologicalSpace
    $ 
    there is a neighbourhood $\Neighbourhood{\Point}$ of $\Point$ and a 
    map 
    $
      s
      \colon
      \Neighbourhood{\Point}
      \to
      \SemiSimplicialSpaceIndex{0}
    $
    such that 
    $
      \apply
        {\Augmentation
          {\SemiSimplicialSpaceIndex{0}}
          \circ
          s
        }
        {\Point'}
      =
      \Point'
    $
    for all $\Point'\in \Neighbourhood{\Point}$ 
    and $\apply{s}{\Point}=\Point_{0}$.
  \item 
    given any finite collection 
    $
      \left\{
        \Point_{0}^{1},
        \ldots,
        \Point_{0}^{n}
      \right\}
      \subset 
      \SemiSimplicialSpaceIndex{0}
    $ 
    in a single fiber of $\Augmentation{\SemiSimplicialSpaceIndex{0}}$ 
    over some $\Point \in \TopologicalSpace$, there is an 
    $
      \Point_{0}^{\infty}
      \in
      \Fiber{\Point}{\Augmentation{\SemiSimplicialSpaceIndex{0}}}
    $
    such that each 
    $
      \left(
        \Point_{0}^{1},\Point_{0}^{\infty}
      \right)
    $
    is a $1$-simplex
  \end{enumerate}
  then 
  $
    \GeometricRealization{\Augmentation{\SemiSimplicialSpace}}
    \colon
    \GeometricRealization{\SemiSimplicialSpace}
    \to
    \TopologicalSpace
  $
  is a weak equivalence.
\end{lemma}

\begin{lemma}
\label{lem:FibrationAugmentation}
  Let 
  $
    \Augmentation{\SemiSimplicialSpace}
    \colon
    \SemiSimplicialSpace 
    \to 
    \TopologicalSpace
  $ 
  denote an augmented semi-simplicial space. 
  If each 
  $
    \Augmentation{\SemiSimplicialSpaceIndex{\SemiSimplicialIndex}}
    \colon
    \SemiSimplicialSpaceIndex{\SemiSimplicialIndex}
    \to
    \TopologicalSpace
  $ 
  is a fibration and 
  $
    \Fiber
      {\Point}
      {\Augmentation{\SemiSimplicialSpaceIndex{\SemiSimplicialIndex}}}
  $
  denotes its fiber at $\Point\in\TopologicalSpace$, then the realization of 
  the semi-simplicial space 
  $
    \Fiber
      {\Point}
      {\Augmentation{\SemiSimplicialSpaceIndex{\SemiSimplicialIndex}}}
  $
  is weakly homotopy equivalent to the homotopy fiber of 
  $\GeometricRealization{\Augmentation{\SemiSimplicialSpace}}$ at $\Point$.
\end{lemma}

\begin{definition}
  \label{dfn:MappingPair}
  If 
  $
    \ContinuousMap
    \colon
    \TopologicalSpace_{1}
    \to
    \TopologicalSpace_{2}
  $
  denotes a continuous map, we will write $\MappingPair{\ContinuousMap}$ for 
  the pair 
  $
    \left(
      M_{\ContinuousMap},
      \TopologicalSpace_{1}
    \right)
  $%
  , where $M_{\ContinuousMap}$ denotes the mapping cylinder of $\ContinuousMap$.
\end{definition}
The following lemma hides an occurring spectral sequence argument in the proof. 
The proof of the lemma can be found in Criterion~4.4. in \cite{CRW16} and it is 
an abstraction of the first part of the proof of Theorem~9.3 in \cite{RW16}.
\begin{lemma}
\label{lem:StabilityCriterion}
  Let
  $
    \ContinuousMap_{\bullet}
    \colon
    \SemiSimplicialSpace
    \to 
    \SemiSimplicialSpace'
  $
  be a map of augmented semi-simplicial spaces such that 
  $\SemiSimplicialSpace \to \TopologicalSpace$ is an
  $(\DimensionIndex-1)$-resolution and 
  $\SemiSimplicialSpace' \to \TopologicalSpace'$ is an 
  $\DimensionIndex$-resolution.
  
  Suppose there is a sequence of path connected based spaces $(B_i,b_i)$ and 
  maps 
  $
    p_i
    \colon 
    \SemiSimplicialSpaceIndex{i}'
    \to 
    B_i
  $%
  , and form the map 
  \[
    g_i
    \colon
    \HomotopyFiber{b_i}{p_i\circ \ContinuousMap_{i}}
    \to
    \HomotopyFiber{b_i}{p_i}
  \]
  induced by the composition with $\ContinuousMap_{i}$. 
  \begin{center}
    \begin{tikzcd}
        \HomotopyFiber{b_i}{p_i\circ \ContinuousMap_{i}}
          \ar[r,"g_i"]
          \ar[d]
        & 
        \HomotopyFiber{b_i}{p_i}
          \ar[d] 
        \\
        \SemiSimplicialSpaceIndex{i}
          \ar[r,"\ContinuousMap_{i}"] 
          \ar[d] 
        & 
        \SemiSimplicialSpaceIndex{i}'
          \ar[r]
          \ar[d]
        &
        B_i
        \\
        \TopologicalSpace
          \ar[r,"\ContinuousMap"] 
        &
        \TopologicalSpace'
    \end{tikzcd}
  \end{center}
  Suppose that there is a $c\leq \DimensionIndex+1$ such that 
  \[
    \HomologyOfSpace{q}{\MappingPair{g_i}}
    =
    0
    \text{ when } 
    q+i\leq c, 
    \text{ except if } 
    (q,i)=(c,0)
  \]
  Then the map induced in homology by the composition of 
  the inclusion of the fiber and the augmentation map
  \[
    \begin{tikzcd}
      \HomologyOfSpace{q}{\MappingPair{g_0}}
        \ar[r] 
      &
      \HomologyOfSpace{q}{\MappingPair{f_0}}
        \ar[r,"\Augmentation{\bullet}"] 
      &
      \HomologyOfSpace{q}{\MappingPair{\ContinuousMap}}
    \end{tikzcd}
  \]
  is an epimorphism in degrees $q\leq c$.

  If in addition 
  $
    \HomologyOfSpace{c}{\MappingPair{g_0}}
    =
    0
  $%
  , then 
  $
    \HomologyOfSpace{q}{\MappingPair{\ContinuousMap}}
    =
    0
  $ 
  in degrees $q\leq c$.
\end{lemma}

\section{Resolutions of the Space of Subsurfaces}
  \label{scn:Resolution}
  The goal of this section is to establish a resolution of 
  $
    \SpaceSubsurfaceBoundaryTangential
      {\TangentialFibration}
      {\Genus}
      {\BoundaryComponents}
      {\Manifold}
      {\BoundaryConditionTangential{\BoundaryCondition}}
  $
  and to understand how this resolution behaves with respect to stabilization 
  maps. This will enable us in Section~\ref{scn:Proof} to relate spaces of 
  subsurfaces and homological properties of maps between them to homological 
  properties of maps between spaces of subsurfaces with smaller genera.
\subsection{Constructing a Resolution for the Space of Subsurfaces}
  Recall the following notation introduced in Section~\ref{scn:Stabilization}:
  \begin{enumerate}[(i)]
    \item 
      $
        \HalfDisk
        \coloneqq
        \left\{
          \left(
            \Coordinate_{1},\Coordinate_{2}
          \right)
          \in
          \Reals^{2}
        \middle|
          \Coordinate_{2} \geq 0 
          \text{, } 
          \EuclideanNorm{
            \left(
              \Coordinate_{1},\Coordinate_{2}
            \right)
          }
          \leq 1
        \right\}
      $
    \item
      $
        \DistBoundaryManifold{0}{\HalfDisk}
        \coloneqq
        \left\{
          \left(
            \Coordinate_{1},\Coordinate_{2}
          \right)
          \in 
          \BoundaryManifold{\HalfDisk}
        \middle|
          \Coordinate_{2}=0
        \right\}
      $
    \item 
      $
        \DistBoundaryManifold{1}{\HalfDisk}
        \coloneqq
        \left\{
          \left(
            \Coordinate_{1},\Coordinate_{2}
          \right)
          \in 
          \BoundaryManifold{\HalfDisk}
        \middle|
          \EuclideanNorm{
            \left(
              \Coordinate_{1},\Coordinate_{2}
            \right)
            =
            1
          }
        \right\}
      $
  \end{enumerate}

  From here on forth let us fix
  an at least $5$-dimensional simply-connected manifold $\Manifold$ together 
  with a non-empty codimension $0$ submanifold of the boundary denoted 
  by $\DistBoundaryManifold{0}{\Manifold}$ and a space of tangential structures
  $
    \TangentialFibration
    \colon
    \TangentialSpace{\Manifold}
    \to 
    \Grassmannian{2}{\TangentBundle{\Manifold}}
  $%
  . Furthermore let $\BoundaryConditionTangential{\BoundaryCondition}$ denote 
  some fixed boundary condition for subsurfaces with tangential structure in 
  $\Manifold$. 
  Lastly, let $\DistBall$ denote a codimension $0$ ball in 
  $\DistBoundaryManifold{0}{\Manifold}$ that intersects $\BoundaryCondition$ in 
  two intervals $\DistIntervall{0}$ and $\DistIntervall{1}$, which we label and 
  orientate according to the orientation of $\BoundaryCondition$ once and for 
  all. 
  
  To define a resolution of the space of subsurfaces, we will need the 
  following definition:
  \begin{definition}
  Given
  $
    \Subsurface
    \in
    \SpaceAllSubsurfaceBoundaryTangential
      {\TangentialFibration}
      {\Manifold}
      {\BoundaryConditionTangential{\BoundaryCondition}}
  $%
  , we call an embedding 
  \[
    \DiskEmbedding
    \colon
    \left(
      \HalfDisk
      ,
      \DistBoundaryManifold{1}{\HalfDisk}
    \right)
    \to
    \left(
      \Manifold
      ,
      \Subsurface
    \right)
  \] 
  which maps $\DistBoundaryManifold{0}{\HalfDisk}$ to $\DistBall$ and 
  $\left(1,0\right)$ to $\DistIntervall{0}$ and 
  $\left(-1,0\right)$ to $\DistIntervall{1}$ an \introduce{arc in 
  $\Subsurface$ with embedded boundary isotopy}. 
  We will call a thickening of an arc in $\Subsurface$ with embedded 
  boundary isotopy
  $
    \left(
      \DiskEmbedding,
      \ThickenedHalfDisk{\DiskEmbedding}
    \right)
  $
  a \introduce{thickened arc in $\Subsurface$ with embedded boundary isotopy}
  if the image of $\ThickenedHalfDisk{\DiskEmbedding}$ restricted 
  to the normal bundle of $\DistBoundaryManifold{0}{\HalfDisk}$
  lies in $\DistBall$.
  In this case we will call the image of
  $
    \at
      {\DiskEmbedding}
      {\DistBoundaryManifold{1}{\HalfDisk}}
  $ 
  the \introduce{underlying arc of $\DiskEmbedding$} and the image of 
  of $\ThickenedHalfDisk{\DiskEmbedding}$ restricted 
  to the normal bundle of 
  $
    \apply
      {\DiskEmbedding}
      {\DistBoundaryManifold{1}{\HalfDisk}}
  $
  the \introduce{thickened underlying arc of $\DiskEmbedding$} and we will 
  write 
  $\ThickenedHalfDisk{\DiskEmbedding}_{\Subsurface}$ for it.
  For notational reasons we will usually just write $\DiskEmbedding$ for the 
  thickened arc in $\Subsurface$ with embedded boundary isotopy even tough the 
  correct notation would include $\ThickenedHalfDisk{\DiskEmbedding}$.
  
  An embedding (without the presence of a subsurface)
  $
    \DiskEmbedding
    \colon 
    \HalfDisk
    \to
    \Manifold
  $
  such that 
  $
    \apply{\DiskEmbedding}{\DistBoundaryManifold{1}{\HalfDisk}}
    \subset
    \Manifold
    \setminus 
    \BoundaryManifold{\Manifold}
  $, 
  $
    \apply{\DiskEmbedding}{\DistBoundaryManifold{0}{\HalfDisk}}
    \subset
    \DistBall
  $ 
  and 
  $
    \apply{\DiskEmbedding}{\left((-1)^k,0\right)}
    \in
    \DistIntervall{k}
  $
  , will be called an \introduce{arc with embedded boundary isotopy} i.e. we 
  drop $\Subsurface$ from the previous notation.
  Similarly for a \introduce{thickened arc with embedded boundary isotopy}.
\end{definition}

  \begin{definition}
  \label{dfn:ArcResolution}
  Let 
  $\ArcResolution
    {\TangentialFibration}
    {\Genus}
    {\BoundaryComponents}
    {\Manifold}
    {\BoundaryConditionTangential{\BoundaryCondition}}
    {\DistBall}
    {\bullet}
  $
  denote the following semi-simplicial space:
  The space of $\SemiSimplicialIndex$-simplices consists of tuples
  $
    \left(
      \Subsurface,
      \left(
        \DiskEmbedding^0,\ThickenedHalfDisk{\DiskEmbedding}^{0}
      \right)
      ,\ldots,
      \left(
        \DiskEmbedding^{\SemiSimplicialIndex}
        ,
        \ThickenedHalfDisk{\DiskEmbedding}^{\SemiSimplicialIndex}
      \right)
    \right)
  $ 
  such that:
  \begin{enumerate}[(i)]
    \item 
      $
        \Subsurface
        \in
        \SpaceSubsurfaceBoundaryTangential
          {\TangentialFibration}
          {\Genus}
          {\BoundaryComponents}
          {\Manifold}
          {\BoundaryConditionTangential{\BoundaryCondition}}
      $
      is a surface with tangential structure in $\Manifold$.
    \item 
      All the 
      $
        \left(
          \DiskEmbedding^{\FaceIndex}
          ,
          \ThickenedHalfDisk{\DiskEmbedding}^{\FaceIndex}
        \right)
      $ 
      are thickened arcs in $\Subsurface$ with embedded boundary isotopy.
    \item
      The images of all $\ThickenedHalfDisk{\DiskEmbedding}^{\FaceIndex}$ are 
      disjoint.
    \item
      $\Subsurface$ without all underlying thickened arcs is connected, i.e. 
      the arc system consisting of the underlying arcs is coconnected.
    \item 
      The starting and endpoints of the underlying arcs are ordered from 
      $0$ to $\SemiSimplicialIndex$ in $\DistIntervall{0}$ and ordered from 
      $\SemiSimplicialIndex$ to $0$ in $\DistIntervall{1}$ (Note that 
      this makes sense as $\DistIntervall{k}$ is oriented).
      In this case, we will say that the arc system consisting of the 
      underlying arcs is \introduce{ordered}.
  \end{enumerate}
  The $\FaceIndex$-th face map forgets the $\FaceIndex$-th embedding and we 
  topologize the set of $\SemiSimplicialIndex$-simplices as a subspace of 
  \[
    \SpaceSubsurfaceBoundaryTangential
      {\TangentialFibration}
      {\Genus}
      {\BoundaryComponents}
      {\Manifold}
      {\BoundaryConditionTangential{\BoundaryCondition}}
    \times 
    \ThickenedEmbeddingSpaceBoundaryCondition
      {(\HalfDisk)^{\SemiSimplicialIndex+1}}
      {\Manifold}
      {\DistBall}
  \]
  where in
  $
    \ThickenedEmbeddingSpaceBoundaryCondition
      {(\HalfDisk)^{\SemiSimplicialIndex+1}}
      {\Manifold}
      {\DistBall}
  $
  the boundary condition means just that $\DistBoundaryManifold{0}{\HalfDisk}$ 
  maps to $\DistBall$.
  This semi-simplicial space possesses an augmentation map 
  $\Augmentation{\bullet}$ to 
  $
    \SpaceSubsurfaceBoundaryTangential
      {\TangentialFibration}
      {\Genus}
      {\BoundaryComponents}
      {\Manifold}
      {\BoundaryConditionTangential{\BoundaryCondition}}
  $ 
  which forgets the thickened arcs with embedded boundary isotopy. 
\end{definition}
\begin{notation}
  Sometimes we want to distinguish between the cases, where the intersection of 
  $\BoundaryCondition$ and $\DistBall$ meets a single connected component of 
  $\delta$ or two different components. 
  To emphasize this we will sometimes write 
  $
    \ArcResolutionSingle
      {\TangentialFibration}
      {\Genus}
      {\BoundaryCondition}
      {\Manifold}
      {\BoundaryConditionTangential{\BoundaryCondition}}
      {\DistBall}
      {\bullet}
  $
  for the single component case and 
  $
    \ArcResolutionTwo
      {\TangentialFibration}
      {\Genus}
      {\BoundaryCondition}
      {\Manifold}
      {\BoundaryConditionTangential{\BoundaryCondition}}
      {\DistBall}
      {\bullet}
  $ 
  for the different components case.
\end{notation}

  The general proof strategy of the following proposition was communicated to 
  me by Frederico Cantero as a proposed fix to some errors in the proof of 
  Proposition 5.3 in 
  \cite{CRW16}.
  \begin{proposition}
\label{prp:ArcResolution}
  Suppose that $\Manifold$ is an at least $5$-dimensional and simply-connected 
  manifold, then
  $\ArcResolution
    {\TangentialFibration}
    {\Genus}
    {\BoundaryComponents}
    {\Manifold}
    {\BoundaryConditionTangential{\BoundaryCondition}}
    {\DistBall}
    {\bullet}
  $
  is a $\left(\Genus-1\right)$-resolution of 
  $
    \SpaceSubsurfaceBoundaryTangential
      {\TangentialFibration}
      {\Genus}
      {\BoundaryComponents}
      {\Manifold}
      {\BoundaryConditionTangential{\BoundaryCondition}}
  $
  i.e. 
  \[
    \Augmentation{\bullet}
    \colon
    \GeometricRealization
      {\ArcResolution
        {\TangentialFibration}
        {\Genus}
        {\BoundaryComponents}
        {\Manifold}
        {\BoundaryConditionTangential{\BoundaryCondition}}
        {\DistBall}
        {\bullet}
      }
    \to
    \SpaceSubsurfaceBoundaryTangential
      {\TangentialFibration}
      {\Genus}
      {\BoundaryComponents}
      {\Manifold}
      {\BoundaryConditionTangential{\BoundaryCondition}}
  \]
  is $\left(\Genus-1\right)$-connected or in other words the homotopy fiber 
  of this map is $\left(\Genus-2\right)$-connected.
\end{proposition}
\begin{proof}
  Since 
  $
    \SpaceSubsurfaceBoundaryTangential
      {+}
      {\Genus}
      {\BoundaryComponents}
      {\Manifold}
      {\BoundaryConditionTangential{\BoundaryCondition}}
  $
  is 
  $\DiffeomorphismGroupBoundary{\Manifold}$-locally 
  retractile and the augmentation maps
  \[
    \ArcResolution
      {+}
      {\Genus}
      {\BoundaryComponents}
      {\Manifold}
      {\BoundaryConditionTangential{\BoundaryCondition}}
      {\DistBall}
      {\bullet}
    \to 
    \SpaceSubsurfaceBoundaryTangential
      {+}
      {\Genus}
      {\BoundaryComponents}
      {\Manifold}
      {\BoundaryConditionTangential{\BoundaryCondition}}
  \]
  are equivariant with respect to the natural action of 
  $\DiffeomorphismGroupBoundary{\Manifold}$ 
  via post composition, Lemma~\ref{lem:LocallyRetractile} implies that the 
  augmentation maps are locally trivial fibrations. 
  Using Lemma \ref{lem:FibrationAugmentation} we get that the 
  homotopy fiber 
  $
    \Fiber
      {\Subsurface}
      {\GeometricRealization{\Augmentation{\bullet}}}
  $
  is weakly homotopy equivalent to 
  $
    \GeometricRealization
      {\Fiber{\Subsurface}{\Augmentation{\bullet}}}
  $%
  . The $\SemiSimplicialIndex$-simplices of 
  $\Fiber{\Subsurface}{\Augmentation{\bullet}}$ are disjoint, ordered 
  thickened arcs in $\Subsurface$ with embedded boundary isotopy. 
  
  Fix $\Point_{k}\in \DistIntervall{k}$ and charts 
  $\Neighbourhood{\Point_{k}}$ centered at $\Point_{k}$. 
  We say that an embedding 
  $
    \Embedding
    \colon 
    \left[0,1\right]
    \to
    \Subsurface
  $ 
  meets $\Point_{k}$ in a radial fashion if its image in 
  $\Neighbourhood{\Point_{k}}$ is a straight ray meeting $0$.
  
  Let 
  $
    X
    \!
    \left(
      \Subsurface
      ;
      \Point_{0},\Point_{1}
    \right)
    _
    {\bullet}
  $
  denote the following semi-simplicial set:
  Its set of zero simplices consists of embeddings of an arc into 
  $\Subsurface$ that meets $\Point_{0}\in\DistIntervall{0}$ and 
  $\Point_{1}\in\DistIntervall{1}$ in a radial fashion together with a 
  tubular neighborhood of said arc and we furthermore require these arcs to 
  be non-isotopic to a part of the boundary.
  The set of $\SemiSimplicialIndex$-simplices is given by collections of 
  $0$-simplices such that the complement of the arcs is connected, and such 
  that the tubular neighborhoods are disjoint except for their intersection 
  with $\Neighbourhood{\Point_{k}}$,
  lastly we require that the ordering of the arcs at $\Point_{k}$ with 
  respect to the angle of the corresponding ray to be order-preserving at 
  $\Point_{0}$ and order-reversing at $\Point_{1}$. 
  
  There is a map 
  $
    \ContinuousMap
    \colon
    \Fiber{\Subsurface}{\Augmentation{\bullet}}
    \to
    X
    \!
    \left(\Subsurface,\Point_{0},\Point_{1}\right)
    _
    {\bullet}
  $
  given by sending an element in the fiber to the underlying thickened arcs 
  and then adding a collar of $\Subsurface$ to $\Subsurface$ and then joining 
  the arcs to $\Point_{k}$ in a controlled way to produce an element in 
  $
    X
    \!
    \left(\Subsurface,\Point_{0},\Point_{1}\right)
    _
    {\bullet}
  $%
  .
  
  Let $\Fiber{\Subsurface}{\Augmentation{\bullet}}^{\delta}$ denote the 
  semi-simplicial set (not space)
  $\Fiber{\Subsurface}{\Augmentation{\bullet}}$ i.e. the set of 
  $\SemiSimplicialIndex$-simplices 
  is given by the underlying set of 
  $\Fiber{\Subsurface}{\Augmentation{\SemiSimplicialIndex}}$.
  Now we want to apply Theorem A.7 of \cite{Kupers}, which says that, if
  \begin{enumerate}[(i)]
    \item 
      $
        X
        \!
        \left(\Subsurface,\Point_{0},\Point_{1}\right)
        _
        {\bullet}
      $
      is weakly Cohen-Macauley of dimension $(\Genus-1)$ i.e. it is 
      $(\Genus-2)$-connected and the link of every $p$-simplex is 
      $(\Genus-2-p-2)$-connected.
    \item 
      $\Fiber{\Subsurface}{\Augmentation{\bullet}}$ is Hausdorff and an 
      ordered flag space (See \cite{Kupers} for the definition of ordered 
      flag space)
    \item 
      $
        \GeometricRealization{\ContinuousMap}
        \colon
        \GeometricRealization
          {\Fiber{\Subsurface}{\Augmentation{\bullet}}^\delta}
        \to
        X
        \!
        \left(\Subsurface,\Point_{0},\Point_{1}\right)
        _
        {\bullet}
      $
      is simplexwise injective i.e. for every 
      $p$-simplex 
      $
        \left\{
          \DiskEmbedding_{0},
          \ldots,
          \DiskEmbedding_{p}
        \right\}
      $
      with $p\geq 1$ we have 
      $
        \apply{\ContinuousMap}{\DiskEmbedding_{i}}
        \neq
        \apply{\ContinuousMap}{\DiskEmbedding_{j}}
      $
      for all $i\neq j$.
    \item 
      For all finite collections 
      $
        \left\{
          \DiskEmbedding_{1},
          \ldots,
          \DiskEmbedding_{\FaceIndex}
        \right\}
        \subset 
        \Fiber{\Subsurface}{\Augmentation{0}}
      $
      and 
      $
        \Embedding_{0}
        \in
        X
        \!
        \left(\Subsurface,\Point_{0},\Point_{1}\right)
        _
        {\bullet}
      $ 
      such that 
      $
        \left(
          \Embedding_{0},
          \apply{\ContinuousMap}{\DiskEmbedding_{i}}
        \right)
      $
      is a $1$-simplex in 
      $
        X
        \!
        \left(\Subsurface,\Point_{0},\Point_{1}\right)
        _
        {\bullet}
      $ 
      then there exists an element 
      $
        \DiskEmbedding_{0}
        \in
        \Fiber{\Subsurface}{\Augmentation{0}}
      $ 
      such that $\apply{\ContinuousMap}{\DiskEmbedding_{0}}=\Embedding_{0}$ 
      and such that 
      $
        \left(
          \DiskEmbedding_{0},
          \DiskEmbedding_{i}
        \right)
      $
      is a $1$-simplex in 
      $\Fiber{\Subsurface}{\Augmentation{\bullet}}$.
  \end{enumerate}
  then 
  $
    \GeometricRealization{
      \Fiber{\Subsurface}{\Augmentation{\bullet}}
    }
  $
  is $(\Genus-2)$-connected.
  
  It is an easy observation that the second and third condition hold.
  For the first condition, note that Theorem 2.9 in \cite{nariman} 
  proves that 
  $
    \GeometricRealization{
      X
      \!
      \left(\Subsurface,\Point_{0},\Point_{1}\right)
      _
      {\bullet}
    }
  $ 
  is $(\Genus-2)$-connected, so we only have to prove that the link of a 
  $p$-simplex is $(\Genus-2-p-2)$-connected. Given a $p$-simplex 
  $\IndexSimplex{p}$, the link of this simplex consists of all arc systems 
  whose union with the arc system of $\IndexSimplex{p}$ is still a 
  coconnected arc system.
  Consider the  closed complement of the thickened arcs in $\IndexSimplex{p}$ 
  denoted by $\CutOut{\Subsurface}{\IndexSimplex{p}}$
  (where this means, that we take the complement of the tubular 
  neighborhoods, then take the closure in $\Subsurface$ and then resolve the 
  singularities at $\Point_{k}$). Evidently every simplex in the link of 
  $\IndexSimplex{p}$ consists of arcs that lie in 
  $\CutOut{\Subsurface}{\IndexSimplex{p}}$. 
  Furthermore the points 
  $\Point_{k}$ correspond to $p+2$ points in 
  $\CutOut{\Subsurface}{\IndexSimplex{p}}$ and each of these points 
  corresponds to an interval between two consecutive arcs in 
  $\IndexSimplex{p}$ in the ordering of arcs connecting the points 
  $\Point_{k}$.
  We will denote the subspace of the link of $\IndexSimplex{p}$ consisting of 
  those arcs that are bigger than all arcs in $\IndexSimplex{p}$ in the 
  ordering of the arcs connecting $\Point_{k}$ by $Y_{\bullet}$. 
  Unbending corners, $Y_{\bullet}$ is isomorphic to 
  $
    X
    \!
    \left(\CutOut{\Subsurface}{\IndexSimplex{p}},\Point'_{0},\Point'_{1}\right)
    _
    {\bullet}
  $
  which is at least $(g-p-3)$-connected by Theorem~2.9 in \cite{nariman} and 
  a simple Euler characteristic computation. 
  
  We will show, using the 
  techniques of Section~2.1 (i.e. bad and good simplices) in 
  \cite{HatcherVogtmann}, that the link has the 
  same connectivity as its subspace $Y_{\bullet}$. 
  We call a simplex in the link of $\IndexSimplex{p}$ bad if all its arcs
  are smaller than some arc in $\IndexSimplex{p}$. 
  These certainly satisfy the conditions for a set of bad simplices in 
  \cite{HatcherVogtmann}. 
  Furthermore given a bad simplex $\IndexSimplex{k}'$, then 
  $G_{\IndexSimplex{k}'}$ is given by 
  $
     X
    \!
    \left(
      \CutOut{\Subsurface}{\IndexSimplex{p}\cup\IndexSimplex{k}'},
      \Point'_{0},
      \Point'_{1}
    \right)
    _
    {\bullet}
  $%
  , which is $(\Genus-p-3-k-1)$-connected. 
  Thus by (b) of Corollary~2.2 in \cite{HatcherVogtmann} the link is at least 
  $(\Genus-2-p-1)$-connected.
  
  For the fourth condition, note that since
  $
    \left(
      \Embedding_{0},
      \apply{\ContinuousMap}{\DiskEmbedding_{i}}
    \right)
  $ is a $1$-simplex, in particular $\Embedding_0$ and the underlying arc of 
  $\apply{\ContinuousMap}{\DiskEmbedding_{i}}$ are coconnected and ordered 
  correctly, we can find a thickened embedded arc $\bar{\DiskEmbedding}_{0}$ 
  in $\Subsurface$ such that 
  $
  \apply
    {\ContinuousMap}
    {\bar{\DiskEmbedding}_{0}}
    =\Embedding_{0}
  $
  (this expression makes sense, since $\ContinuousMap$ only 
  takes the underlying thickened arcs of the elements in
  $\Fiber{\Subsurface}{\Augmentation{\bullet}}$) and such that the pairs
  consisting of $\bar{\DiskEmbedding}_{0}$ 
  and the thickened underlying arcs of $\DiskEmbedding_{i}$ form an 
  ordered and coconnected arc system. 
  Using that $\Manifold$ is simply-connected and 
  the main result of \cite{H61} we can find an extension of 
  $\bar{\DiskEmbedding}_{0}$ to an arc in $\Manifold$ with embedded boundary 
  isotopy.
  Since the dimension of $\Manifold$ is at 
  least $5$ we conclude that a small perturbation of this embedding yields an 
  embedding disjoint from all the other $\DiskEmbedding_{i}$ by 
  transversality. 
  Adding a sufficently small tubular neighborhood finishes the proof of the 
  fourth assumption and therefore implies the desired claim about the 
  connectivity of $\Fiber{\Subsurface}{\Augmentation{\bullet}}$.
  
  To finish the proof, just observe that the following diagram is a pullback 
  diagram, where the lower map denotes the map forgetting the tangential 
  structure.
  \begin{center}
    \begin{tikzcd}
      \GeometricRealization
        {\ArcResolution
          {\TangentialFibration}
          {\Genus}
          {\BoundaryComponents}
          {\Manifold}
          {\BoundaryConditionTangential{\BoundaryCondition}}
          {\DistBall}
          {\bullet}
        }
        \ar[r]
        \ar[d]
      &
      \GeometricRealization
        {\ArcResolution
          {+}
          {\Genus}
          {\BoundaryComponents}
          {\Manifold}
          {\BoundaryConditionTangential{\BoundaryCondition}}
          {\DistBall}
          {\bullet}
        }
        \ar[d]
      \\
      \SpaceSubsurfaceBoundaryTangential
        {\TangentialFibration}
        {\Genus}
        {\BoundaryComponents}
        {\Manifold}
        {\BoundaryConditionTangential{\BoundaryCondition}}
        \ar[r]
      &
      \SpaceSubsurfaceBoundary
        {\Genus}
        {\BoundaryComponents}
        {\Manifold}
        {\BoundaryCondition}
    \end{tikzcd}
  \end{center}
\end{proof}

We will need to establish some more notation.
\begin{definition}
  A \introduce{thickened strip with embedded boundary isotopy} consists of the 
  following data:
  \begin{enumerate}[(i)]
    \item
      A thickened arc with embedded boundary isotopy 
      $
        \DiskEmbedding
        \colon
        \HalfDisk
        \to 
        \Manifold
      $
    \item
      A subbundle $\LineBundle$ of 
      $
        \NormalBundle
          {\apply
            {\DiskEmbedding}
            {\DistBoundaryManifold{1}{\HalfDisk}}
          }
          {\Manifold}
      $
  \end{enumerate}
  such that the restriction 
  $
    \at
      {\ThickenedHalfDisk{\DiskEmbedding}}
      {\at{\LineBundle}{(1,0)}}
  $
  lies in $\DistIntervall{0}$ and 
  $
    \at
    {\ThickenedHalfDisk{\DiskEmbedding}}
    {\at{\LineBundle}{(-1,0)}}
  $
  lies in $\DistIntervall{1}$. We will denote the restriction of 
  $\ThickenedHalfDisk{\DiskEmbedding}$ to $\LineBundle$ by 
  $
    \ThickenedHalfDisk{\DiskEmbedding}
    _
    {\LineBundle}
  $%
  .
  If we add a tangential structure for the image of 
  $
    \ThickenedHalfDisk{\DiskEmbedding}
    _
    {\LineBundle}
  $%
  , which agrees with the tangential structure specified by 
  $\BoundaryConditionTangential{\BoundaryCondition}$, 
  wherever this makes sense, we will call this a \introduce{thickened strip 
  with tangential structure and embedded boundary isotopy}.
  We will call the image of 
  $
    \ThickenedHalfDisk{\DiskEmbedding}
    _
    {\LineBundle}
  $
  the \introduce{strip of $(\DiskEmbedding,\LineBundle)$}.
\end{definition}
\begin{notation}
  Similar as before we will usually suppress 
  $\ThickenedHalfDisk{\DiskEmbedding}$ and 
  $\ThickenedHalfDisk{\DiskEmbedding}_{\LineBundle}$ 
  from this notation and only write $\DiskEmbedding$ for the tuple 
  $(\DiskEmbedding,\ThickenedHalfDisk{\DiskEmbedding},\LineBundle)$.
\end{notation}

\begin{definition}
  Let 
  $
    \SpaceOfStrips
      {\TangentialFibration}
      {\Manifold}
      {\BoundaryConditionTangential{\BoundaryCondition}}
      {\DistBall}
      {\SemiSimplicialIndex}
  $
  denote the set of tuples 
  $
    \left(
      \DiskEmbedding^{0},
      \ldots,
      \DiskEmbedding^{\SemiSimplicialIndex}
    \right)
  $
  such that all the $\DiskEmbedding^k$ are thickened strips with tangential 
  structures and embedded boundary isotopies such that all the images of 
  $\ThickenedHalfDisk{\DiskEmbedding}^k$
  are disjoint and the starting and endpoints of the underlying arcs lie in 
  $\DistIntervall{0}$ and $\DistIntervall{1}$ and they are
  ordered from $0$ to $\SemiSimplicialIndex$ in $\DistIntervall{0}$ and ordered 
  from $\SemiSimplicialIndex$ to $0$ in $\DistIntervall{1}$.

  We topologize this as a subset of 
  $
    \ThickenedEmbeddingSpaceBoundaryCondition
      {\HalfDisk\times \Simplex{\SemiSimplicialIndex}}
      {\Manifold}
      {\DistBall}
    \times
    \EmbeddingSpaceTangential
      {\TangentialFibration}
      {I\times I \times \Simplex{\SemiSimplicialIndex}}
      {\Manifold}
  $%
  , where the thickened arcs with embedded boundary isotopies correspond to 
  elements in the first factor and the strips and their tangential structure 
  to the second factor.
  
  There is a continuous map
  $
    \ArcResolution
      {\TangentialFibration}
      {\Genus}
      {\BoundaryComponents}
      {\Manifold}
      {\BoundaryConditionTangential{\BoundaryCondition}}
      {\DistBall}
      {\SemiSimplicialIndex}
    \to
    \SpaceOfStrips
      {\TangentialFibration}
      {\Manifold}
      {\BoundaryConditionTangential{\BoundaryCondition}}
      {\DistBall}
      {\SemiSimplicialIndex}
  $
  which forgets the surface but keeps the tubular 
  neighbourhood in the surface and the tangential structure on it. 
\end{definition}

\begin{lemma}
\label{lem:HomotopyFiberArcResolution}
  The restriction map 
  $
    \ArcResolution
      {\TangentialFibration}
      {\Genus}
      {\BoundaryComponents}
      {\Manifold}
      {\BoundaryConditionTangential{\BoundaryCondition}}
      {\DistBall}
      {\SemiSimplicialIndex}
    \to
    \SpaceOfStrips
      {\TangentialFibration}
      {\Manifold}
      {\BoundaryConditionTangential{\BoundaryCondition}}
      {\DistBall}
      {\SemiSimplicialIndex}
  $
  is a Serre fibration and the fiber over a point 
  $
    \DiskEmbedding
    =
    \left(
      \DiskEmbedding_{0},\ldots,\DiskEmbedding_{\SemiSimplicialIndex}
    \right)
  $
  can be identified with 
  $
    \SpaceSubsurfaceBoundaryTangential
      {\TangentialFibration}
      {\Genus-\SemiSimplicialIndex-1}
      {\BoundaryComponents+\SemiSimplicialIndex+1}
      {\CutOut{\Manifold}{\DiskEmbedding}}
      {\BoundaryConditionTangential{\CutOut{\BoundaryCondition}{\DiskEmbedding}}}
  $ 
  if we are considering 
  $
    \ArcResolutionSingle
      {\TangentialFibration}
      {\Genus}
      {\BoundaryComponents}
      {\Manifold}
      {\BoundaryConditionTangential{\BoundaryCondition}}
      {\DistBall}
      {\SemiSimplicialIndex}
  $%
  . If we are considering 
  $
    \ArcResolutionTwo
      {\TangentialFibration}
      {\Genus}
      {\BoundaryComponents}
      {\Manifold}
      {\BoundaryConditionTangential{\BoundaryCondition}}
      {\DistBall}
      {\SemiSimplicialIndex}
  $ 
  instead, then the fiber is given by 
  $
    \SpaceSubsurfaceBoundaryTangential
      {\TangentialFibration}
      {\Genus-\SemiSimplicialIndex}
      {\BoundaryComponents+\SemiSimplicialIndex-1}
      {\CutOut{\Manifold}{\DiskEmbedding}}
      {\BoundaryConditionTangential{\CutOut{\BoundaryCondition}{\DiskEmbedding}}}
  $%
  .
\end{lemma}
Here we define 
$
  \CutOut{\Manifold}{\DiskEmbedding}
  =
  \overline{
    \Manifold
    \setminus 
    \cup_{\FaceIndex} 
    \ThickenedHalfDisk{\DiskEmbedding}_{\FaceIndex}
  }
$
and the boundary condition
$
  \BoundaryConditionTangential{\CutOut{\BoundaryCondition}{\DiskEmbedding}}
$ 
is given by the boundary of the image of the strips together with the 
intersection of the boundary condition $\BoundaryCondition$ and the submanifold
$\CutOut{\Manifold}{\DiskEmbedding}$.
\begin{remark}
\label{rmk:SmoothingTheAngleRemovingThings}
  Note that $\CutOut{\Manifold}{\DiskEmbedding}$ is a manifold with corners and 
  $\BoundaryConditionTangential{\CutOut{\BoundaryCondition}{\DiskEmbedding}}$ 
  is a boundary condition for a manifold with corners, but we can circumvent 
  this by fixing a homeomorphism from $\CutOut{\Manifold}{\DiskEmbedding}$ to 
  $\Manifold$ that is a diffeomorphism at all points except the corner points. 
  Heuristically such a homeomorphisms pushes the dent which came from 
  removing $\DiskEmbedding$ to the outside.
  
  Note that the boundary condition 
  $\BoundaryConditionTangential{\CutOut{\BoundaryCondition}{\DiskEmbedding}}$ 
  maps the corners of 
  $\CutOut{\SurfaceGB{\Genus}{\BoundaryCondition}}{\DiskEmbedding}$, i.e. the 
  surface with the corresponding arcs removed, to the corners of 
  $\CutOut{\Manifold}{\DiskEmbedding}$. 
  Therefore postcomposing a subsurface with the aforementioned homeomorphism 
  yields an embedded subsurface without corners in a manifold without corners. 
  Using this we can treat the occurring spaces of subsurfaces with corners 
  of manifolds with corners up to homeomorphism like ordinary spaces of 
  subsurfaces. 
\end{remark}
\begin{proof}
  We will construct the following diagram and prove that the maps labelled 
  $\text{fib}$ are fibrations to prove the claim:
  \begin{equation}
  \label{eqt:FibrationSpaceOfStrips}
    \begin{tikzcd}
      \EmbeddingSpaceBoundaryConditionTangential
        {\TangentialFibration}
        {\SurfaceGB{\Genus}{\BoundaryComponents}}
        {\Manifold}
        {\BoundaryConditionTangential{\BoundaryCondition}}
        \ar[r, "\text{fib}"]
      &
      \EmbeddingSpaceBoundaryConditionTangential
        {\TangentialFibration}
        {\sigma\times[-1,1]}
        {\Manifold}
        {\BoundaryCondition'}
        \\
        \ArcResolutionEmbedding
          {\TangentialFibration}
          {\Genus}
          {\BoundaryComponents}
          {\Manifold}
          {\BoundaryConditionTangential{\BoundaryCondition}}
          {\DistBall}
          {\SemiSimplicialIndex}
          \ar[u,"\text{fib}"]
          \ar[r,"\text{fib}"]
          \ar[d]
        &
        \SpaceOfStrips
          {\TangentialFibration}
          {\Manifold}
          {\BoundaryConditionTangential{\BoundaryCondition}}
          {\DistBall}
          {\SemiSimplicialIndex}
          \ar[u,"\text{fib}"]
        \\
        \ArcResolution
          {\TangentialFibration}
          {\Genus}
          {\BoundaryComponents}
          {\Manifold}
          {\BoundaryConditionTangential{\BoundaryCondition}}
          {\DistBall}
          {\SemiSimplicialIndex}
          \ar[ur]
    \end{tikzcd}
  \end{equation}
  Fix a set $\sigma$ of $\SemiSimplicialIndex+1$ disjoint arcs in 
  $\SurfaceGB{\Genus}{\BoundaryComponents}$ starting at 
  $
    \apply
      {\BoundaryConditionPar^{-1}}
      {\DistIntervall{0}}
  $
  and ending at 
  $
    \apply
      {\BoundaryConditionPar^{-1}}
      {\DistIntervall{1}}
  $%
  , and fulfilling the ordering condition of 
  Definition~\ref{dfn:ArcResolution}, and such that their complement is 
  connected. Here $\BoundaryConditionPar$ is a parametrization of 
  $\BoundaryCondition$.
  Then, the top map is the restriction map, which was proven to be a 
  fibration in Lemma~\ref{lem:RestrictionFibration}. 
  The right vertical map forgets the boundary isotopy. 
  Note that this is a fibration as the following diagram is a pullback and the 
  right hand map is a fibration by Lemma~\ref{lem:LocallyRetractile}:
  \[
    \begin{tikzcd}
      \SpaceOfStrips
        {\TangentialFibration}
        {\Manifold}
        {\BoundaryConditionTangential{\BoundaryCondition}}
        {\DistBall}
        {\SemiSimplicialIndex}
        \ar[r]
        \ar[d]
      &
       \SpaceOfStrips
        {\TangentialFibration}
        {\Manifold}
        {\BoundaryConditionTangential{\BoundaryCondition}}
        {\DistBall}
        {\SemiSimplicialIndex}
        \ar[d]
      \\
      \EmbeddingSpaceBoundaryConditionTangential
        {\TangentialFibration}
        {
          [0,1]
          \times
          [0,1]
          \times 
          \SemiSimplicialSpaceIndex{\SemiSimplicialIndex}
        }
        {\Manifold}
        {\BoundaryCondition'}
        \ar[r]
      &
      \EmbeddingSpaceBoundaryConditionTangential
        {+}
        {
          [0,1]
          \times
          [0,1]
          \times 
          \SemiSimplicialSpaceIndex{\SemiSimplicialIndex}
        }
        {\Manifold}
        {\BoundaryCondition'}
    \end{tikzcd}
  \]
  The definition of the space
  $
    \ArcResolutionEmbedding
      {\TangentialFibration}
      {\Genus}
      {\BoundaryComponents}
      {\Manifold}
      {\BoundaryConditionTangential{\BoundaryCondition}}
      {\DistBall}
      {\SemiSimplicialIndex}
  $
  is as follows:
  We define 
  $
    \ArcResolutionNoDisk
      {\TangentialFibration}
      {\Genus}
      {\BoundaryComponents}
      {\Manifold}
      {\BoundaryConditionTangential{\BoundaryCondition}}
      {\DistBall}
      {\SemiSimplicialIndex}
  $
  just like the arc resolution, but without the boundary isotopies (i.e. 
  subsurfaces with arc systems). 
  We define 
  $
    \ArcResolutionEmbedding
      {\TangentialFibration}
      {\Genus}
      {\BoundaryComponents}
      {\Manifold}
      {\BoundaryConditionTangential{\BoundaryCondition}}
      {\DistBall}
      {\SemiSimplicialIndex}
  $
  via the following diagram in which the space in the top left corner of both 
  squares are defined as pullbacks:
  \[
    \begin{tikzcd}
      \ArcResolutionEmbedding
        {\TangentialFibration}
        {\Genus}
        {\BoundaryComponents}
        {\Manifold}
        {\BoundaryConditionTangential{\BoundaryCondition}}
        {\DistBall}
        {\SemiSimplicialIndex}
        \ar[r]
        \ar[d]
      &
      \ArcResolutionEmbedding
        {+}
        {\Genus}
        {\BoundaryComponents}
        {\Manifold}
        {\BoundaryConditionTangential{\BoundaryCondition}}
        {\DistBall}
        {\SemiSimplicialIndex}
        \ar[d]
        \ar[r]
      &
      \EmbeddingSpaceBoundaryCondition
        {\SurfaceGB{\Genus}{\BoundaryComponents}}
        {\Manifold}
        {\BoundaryConditionTangential{\BoundaryConditionPar}}
        \ar[d]
      \\
      \ArcResolution
        {\TangentialFibration}
        {\Genus}
        {\BoundaryComponents}
        {\Manifold}
        {\BoundaryConditionTangential{\BoundaryCondition}}
        {\DistBall}
        {\SemiSimplicialIndex}
        \ar[r]
      &
      \ArcResolution
        {+}
        {\Genus}
        {\BoundaryComponents}
        {\Manifold}
        {\BoundaryConditionTangential{\BoundaryCondition}}
        {\DistBall}
        {\SemiSimplicialIndex}
        \ar[r]
      &
      \ArcResolutionNoDisk
        {+}
        {\Genus}
        {\BoundaryComponents}
        {\Manifold}
        {\BoundaryConditionTangential{\BoundaryCondition}}
        {\DistBall}
        {\SemiSimplicialIndex}
    \end{tikzcd}
  \]
  Note that 
  $\DiffeomorphismGroupArcs{\sigma}{\SurfaceGB{\Genus}{\BoundaryComponents}}$
  the group of diffeomorphisms that fix the boundary and $\sigma$ acts on the 
  space of embeddings in the diagram and the left most map is the quotient map 
  associated to this action. It is a fibration by 
  Lemma~\ref{lem:LocallyRetractile} and the fact that 
  $
    \ArcResolutionNoDisk
      {+}
      {\Genus}
      {\BoundaryComponents}
      {\Manifold}
      {\BoundaryConditionTangential{\BoundaryCondition}}
      {\DistBall}
      {\SemiSimplicialIndex}
  $
  is $\DiffeomorphismGroupBoundary{\Manifold}$-locally retractile. Since the 
  left vertical map is a fibration, all vertical maps in the diagram are 
  fibrations. Furthermore the top square in (\ref{eqt:FibrationSpaceOfStrips}) 
  is also easily seen to be a pullback diagram. 
  Hence all maps in that square are fibrations. 
  Since all maps in (\ref{eqt:FibrationSpaceOfStrips}), except for 
  the diagonal one, are fibrations
  Lemma~\ref{lem:FibrationComposition} implies that the diagonal map is a 
  fibration as well.
  
  Lastly we have to determine the fiber over $\DiskEmbedding$. 
  Note that removing $\SemiSimplicialIndex+1$ strips from a surface increases 
  its Euler characteristic by $\SemiSimplicialIndex+1$, since it corresponds to 
  taking out $\SemiSimplicialIndex+1$ one-cells. 
  To calculate the genus and the number of boundary components of the fiber 
  it is enough to specify its number of boundary components. 
  We will take out the $\SemiSimplicialIndex+1$ strips consecutively and there 
  are two cases we have to distinguish:
  Either all boundary points lie in the same connected component of the 
  boundary of the surface (Case~1) or they all lie in different connected 
  components (Case~2). 
  But note that removing the first arc in Case~2, reduces the calculation 
  to Case~1 for the rest of the arcs. 
  But in the first case we see that taking out an arc increases the number of 
  boundary components by one and the requirement for the ordering of the arcs 
  ensures that the consecutive arcs all connect the same connected component of 
  the boundary. 
  
  All in all we conclude that in Case~2 the number of boundary components 
  changes to $\BoundaryComponents+\SemiSimplicialIndex-1$ and in the first case 
  it changes to $\BoundaryComponents+\SemiSimplicialIndex+1$. 
  Using the Euler-characteristic formula for surfaces, we can 
  compute the corresponding genus to get the above specifications of the fiber. 
\end{proof}

\begin{definition}
  For some 
  $
    \left(
      \DiskEmbedding^{0},
      \ldots,
      \DiskEmbedding^{\SemiSimplicialIndex}
    \right)
    \in
    \SpaceOfStrips
      {\TangentialFibration}
      {\Manifold}
      {\BoundaryConditionTangential{\BoundaryCondition}}
      {\DistBall}
      {\SemiSimplicialIndex}
    $
  we call the composition of the inclusion of the fiber in Lemma 
  \ref{lem:HomotopyFiberArcResolution} with the 
  augmentation of 
  $
    \Augmentation{\SemiSimplicialIndex}
    \colon
    \ArcResolution
      {\TangentialFibration}
      {\Genus}
      {\BoundaryComponents}
      {\Manifold}
      {\BoundaryConditionTangential{\BoundaryCondition}}
      {\DistBall}
      {\SemiSimplicialIndex}
    \to
    \SpaceSubsurfaceBoundaryTangential
      {\TangentialFibration}
      {\Genus}
      {\BoundaryComponents}
      {\Manifold}
      {\BoundaryConditionTangential{\BoundaryCondition}}
  $
  the \introduce{$\SemiSimplicialIndex$-approximate augmentation} of the 
  resolution
  $
    \ArcResolution
      {\TangentialFibration}
      {\Genus}
      {\BoundaryComponents}
      {\Manifold}
      {\BoundaryConditionTangential{\BoundaryCondition}}
      {\DistBall}
      {\SemiSimplicialIndex}
  $ 
  over the $\SemiSimplicialIndex$-simplex
  $
    \left(
      \DiskEmbedding^{0},
      \ldots,
      \DiskEmbedding^{\SemiSimplicialIndex}
    \right)
  $%
  .
\end{definition}

\subsection{Stabilizing the arc resolution}
  We want to extend the maps 
  $
    \StabilizationAlpha{\Genus}{\BoundaryComponents}
  $
  and 
  $
    \StabilizationBeta{\Genus}{\BoundaryComponents}
  $
  to the aforementioned resolution as shown below: 
  \[
    \begin{tikzcd}
      \ArcResolution
        {\TangentialFibration}
        {\Genus}
        {\BoundaryComponents}
        {\Manifold}
        {\BoundaryConditionTangential{\BoundaryCondition}}
        {\DistBall}
        {\SemiSimplicialIndex}
        \arrow[dashrightarrow]{rr}
        \arrow{d}{\Augmentation{\SemiSimplicialIndex}}
      &
      & 
      \ArcResolution
        {\TangentialFibration}
        {\Genus+1}
        {\BoundaryComponents-1}
        {\Elongation{\Manifold}{1}}
        {\BoundaryConditionTangential{\bar{\BoundaryCondition}}}
        {\bar{\DistBall}}
        {\SemiSimplicialIndex}
        \arrow{d}{\Augmentation{\SemiSimplicialIndex}}
      \\
      \SpaceSubsurfaceBoundaryTangential
        {\TangentialFibration}
        {\Genus}
        {\BoundaryComponents}
        {\Manifold}
        {\BoundaryConditionTangential{\BoundaryCondition}}
        \arrow{rr}{\StabilizationAlpha{\Genus}{\BoundaryComponents}} 
      &
      &
      \SpaceSubsurfaceBoundaryTangential
        {\TangentialFibration}
        {\Genus+1}
        {\BoundaryComponents-1}
        {\Elongation{\Manifold}{1}}
        {\BoundaryConditionTangential{\bar{\BoundaryCondition}}}
    \end{tikzcd}
  \]
  and similarly for $\StabilizationBeta{\Genus}{\BoundaryComponents}$.
  Let $\StabilizationBordism$ denote a subsurface in 
  $
    \DistBoundaryManifold{0}{\Manifold}
    \times
    I
  $
  used in the definition of a stabilization map.
  We define $\bar{\DistBall}=\DistBall\times\{1\}$ and
  $\bar{\DistIntervall{i}}=\DistIntervall{i}\times\{1\}$ for $i\in 
  \{0,1\}$ and we assume without loss of generality that 
  $
    \StabilizationBordism
    \cap
    \left(
      \DistBall 
      \times 
      I
    \right)
    =
    \left(
      \DistBall
      \cap 
      \BoundaryCondition
    \right)
    \times 
    I
  $%
  , in particular
  $
    \bar{\DistBall}
    \cap
    \bar{\BoundaryCondition}
    =
    \left(
      \DistBall
      \cap
      \BoundaryCondition
    \right)
    \times
    \{1\}
  $%
  , where 
  $\BoundaryConditionTangential{\bar{\BoundaryCondition}}$ denotes the boundary 
  condition of the image of the stabilization map.
  (Here we isotope $\StabilizationBordism$ not relative the boundary to get 
  that $\DistBall \times I$ is contained in $\StabilizationBordism$).
  
  Define $\tilde{\DiskEmbedding}$ for 
  $
    \DiskEmbedding
    \in
    \SpaceOfStrips
      {\TangentialFibration}
      {\Manifold}
      {\BoundaryConditionTangential{\BoundaryCondition}}
      {\DistBall}
      {0}
  $
  as
  $
    \apply
      {\DiskEmbedding}
      {\DistBoundaryManifold{0}{\HalfDisk}}
    \times
    I
  $%
  . This allows us to extend the stabilization maps to
  $
  \ArcResolution
    {\TangentialFibration}
    {\Genus}
    {\BoundaryComponents}
    {\Manifold}
    {\BoundaryConditionTangential{\BoundaryCondition}}
    {\DistBall}
    {\SemiSimplicialIndex}
  $
  as follows 
  \[
    \left(
      \Subsurface,
      \left(
        \DiskEmbedding^{0},
        \ldots,
        \DiskEmbedding^{\SemiSimplicialIndex}
      \right)
    \right)
    \mapsto
    \left(
      \Subsurface
      \cup
      \StabilizationBordism
      \left(
        \DiskEmbedding^{0}
        \cup
        \tilde{\DiskEmbedding}^{0},
        \ldots,
        \DiskEmbedding^{\SemiSimplicialIndex}
        \cup
        \tilde{\DiskEmbedding}^{\SemiSimplicialIndex}
      \right)
    \right)
  \]%
  and we write $\bar{\DiskEmbedding}$ for 
  $\DiskEmbedding\cup\tilde{\DiskEmbedding}$. 
  This yields the dashed lifts. Since they commute with the face 
  maps and with the augmentation maps we get a map of augmented semi-simplicial 
  spaces 
  \[
    \StabilizationAlpha{\Genus}{\BoundaryComponents}^{\bullet} 
    \colon
    \ArcResolutionTwo
      {\TangentialFibration}
      {\Genus}
      {\BoundaryComponents}
      {\Manifold}
      {\BoundaryConditionTangential{\BoundaryCondition}}
      {\DistBall}
      {\bullet}
    \to
    \ArcResolutionSingle
      {\TangentialFibration}
      {\Genus+1}
      {\BoundaryComponents-1}
      {\Elongation{\Manifold}{1}}
      {\BoundaryConditionTangential{\bar{\BoundaryCondition}}}
      {\bar{\DistBall}}
      {\bullet}
    \]
    which is augmented over $\StabilizationAlpha{\Genus}{\BoundaryComponents}$. 
    If $\bar{\DistBall}$ intersects two different 
    components of the new boundary condition, then we obtain analogously
  \[
    \StabilizationBeta{\Genus}{\BoundaryComponents}^{\bullet} 
    \colon
    \ArcResolutionSingle
      {\TangentialFibration}
      {\Genus}
      {\BoundaryComponents}
      {\Manifold}
      {\BoundaryConditionTangential{\BoundaryCondition}}
      {\DistBall}
      {\bullet}
    \to
    \ArcResolutionTwo
      {\TangentialFibration}
      {\Genus}
      {\BoundaryComponents+1}
      {\Elongation{\Manifold}{1}}
      {\BoundaryConditionTangential{\bar{\BoundaryCondition}}}
      {\bar{\DistBall}}
      {\bullet}
  \]
  augmented over $\StabilizationBeta{\Genus}{\BoundaryComponents}$ .
  
  All of these considerations imply the following corollary:
  \begin{corollary}
  The semi-simplicial pair 
  $
    \MappingPair{\StabilizationAlpha{\Genus}{\BoundaryComponents}^{\bullet}}
  $
  together with the natural augmentation map to 
  $\MappingPair{\StabilizationAlpha{\Genus}{\BoundaryComponents}}$ 
  is a $\Genus$-resolution i.e. the map between pairs is $\Genus$-connected. 
  
  The semi-simplicial pair 
  $
    \MappingPair{\StabilizationBeta{\Genus}{\BoundaryComponents}^{\bullet}}
  $
  together with the natural augmentation map to 
  $\MappingPair{\StabilizationBeta{\Genus}{\BoundaryComponents}}$ 
  is a $(\Genus-1)$-resolution i.e. the map between pairs is 
  $(\Genus-1)$-connected. 
\end{corollary}

  There is a commutative square 
  \[
    \begin{tikzcd}
      \ArcResolutionTwo
        {\TangentialFibration}
        {\Genus}
        {\BoundaryComponents}
        {\Manifold}
        {\BoundaryConditionTangential{\BoundaryCondition}}
        {\DistBall}
        {\SemiSimplicialIndex}
        \ar[
          rr,
          "\StabilizationAlpha{\Genus}{\BoundaryComponents}^\SemiSimplicialIndex"
          ]
        \ar[d]
      &
      &
      \ArcResolutionSingle
        {\TangentialFibration}
        {\Genus+1}
        {\BoundaryComponents-1}
        {\Elongation{\Manifold}{1}}
        {\BoundaryConditionTangential{\bar{\BoundaryCondition}}}
        {\bar{\DistBall}}
        {\SemiSimplicialIndex}
      \ar[d]
      \\
      \SpaceOfStrips
        {\TangentialFibration}
        {\Manifold}
        {\BoundaryConditionTangential{\BoundaryCondition}}
        {\DistBall}
        {\SemiSimplicialIndex}
        \ar[rr,"\DiskEmbedding \mapsto \bar{\DiskEmbedding}"]
      &
      &
      \SpaceOfStrips
        {\TangentialFibration}
        {\Elongation{\Manifold}{1}}
        {\BoundaryConditionTangential{\bar{\BoundaryCondition}}}
        {\bar{\DistBall}}
        {\SemiSimplicialIndex}
    \end{tikzcd}
  \]
  using collars one sees that the lower map is a homotopy equivalence.
  By commutativity of the above square we get a map between the fibers over 
  the points $\DiskEmbedding$ and $\bar{\DiskEmbedding}$
  \[
    \SpaceSubsurfaceBoundaryTangential
      {\TangentialFibration}
      {\Genus-\SemiSimplicialIndex}
      {\BoundaryComponents+\SemiSimplicialIndex+1}
      {\CutOut{\Manifold}{\DiskEmbedding}}
      {\BoundaryConditionTangential{\CutOut{\BoundaryCondition}{\DiskEmbedding}}}
    \to
    \SpaceSubsurfaceBoundaryTangential
      {\TangentialFibration}
      {\Genus-\SemiSimplicialIndex}
      {\BoundaryComponents+i}
      {\CutOut{\Elongation{\Manifold}{1}}{\bar{\DiskEmbedding}}}
      {\CutOut
        {\BoundaryConditionTangential{\bar{\BoundaryCondition}}}
        {\bar{\DiskEmbedding}}
      }
  \]
  If $\StabilizationBordism$ denotes the bordism defining the map 
  $\StabilizationAlpha{\Genus}{\BoundaryComponents}$ in question, 
  then this map is given by taking the union with 
  $
    \CutOut
      {\StabilizationBordism}
      {\DiskEmbedding}
    \coloneqq
    \overline{
      \StabilizationBordism
      \setminus
      \cup_{\SemiSimplicialIndex}
      \DiskEmbedding_{\SemiSimplicialIndex}
    }
  $%
  . These are again manifolds with corners, but this is quite easily fixed 
  using similar techniques as in 
  Remark~\ref{rmk:SmoothingTheAngleStabilization} and 
  Remark~\ref{rmk:SmoothingTheAngleRemovingThings}.
  
  This map between the fibers is quite easily seen to be of the form
  $
    \StabilizationBeta
      {\Genus-\SemiSimplicialIndex}
      {\BoundaryComponents+\SemiSimplicialIndex-1}
  $%
  .
  As the map 
  $
    \SpaceOfStrips
      {\TangentialFibration}
      {\Manifold}
      {\BoundaryConditionTangential{\BoundaryCondition}}
      {\DistBall}
      {\SemiSimplicialIndex}
    \to
    \SpaceOfStrips
      {\TangentialFibration}
      {\Elongation{\Manifold}{1}}
      {\BoundaryConditionTangential{\bar{\BoundaryCondition}}}
      {\bar{\DistBall}}
      {\SemiSimplicialIndex}
  $
  is a homotopy equivalence, we conclude that the space of subsurfaces
  $
    \SpaceSubsurfaceBoundaryTangential
      {\TangentialFibration}
      {\Genus-\SemiSimplicialIndex}
      {\BoundaryComponents+\SemiSimplicialIndex-1}
      {\CutOut{\Manifold}{\DiskEmbedding}}
      {\CutOut{\BoundaryConditionTangential{\BoundaryCondition}}{\DiskEmbedding}}
  $
  is 
  homotopy equivalent to the homotopy fiber of the composition of the 
  map 
  $
    \ArcResolution
      {\TangentialFibration}
      {\Genus}
      {\BoundaryComponents}
      {\Manifold}
      {\BoundaryConditionTangential{\BoundaryCondition}}
      {\DistBall}
      {\SemiSimplicialIndex}
    \to
    \SpaceOfStrips
      {\TangentialFibration}
      {\Manifold}
      {\BoundaryConditionTangential{\BoundaryCondition}}
      {\DistBall}
      {\SemiSimplicialIndex}
  $
  with the aforementioned homotopy equivalence. Moreover we have shown that 
  the map between the fibers is a stabilization map of the form 
  $
    \StabilizationBeta
      {\Genus-\SemiSimplicialIndex}
      {\BoundaryComponents+\SemiSimplicialIndex-1}
  $%
  . This is expressed in the following diagram:
  \[
    \begin{tikzcd}
      \SpaceSubsurfaceBoundaryTangential
        {\TangentialFibration}
        {\Genus-\SemiSimplicialIndex}
        {\BoundaryComponents+\SemiSimplicialIndex-1}
        {\CutOut{\Manifold}{\DiskEmbedding}}
        {\CutOut
          {\BoundaryConditionTangential{\BoundaryCondition}}
          {\DiskEmbedding}
        }
        \ar[
          rr,
          "\StabilizationBeta
            {\Genus-\SemiSimplicialIndex}
            {\BoundaryComponents+\SemiSimplicialIndex-1}
          "]
        \ar[dd]
      &
      &
      \SpaceSubsurfaceBoundaryTangential
        {\TangentialFibration}
        {\Genus-\SemiSimplicialIndex}
        {\BoundaryComponents+i}
        {\CutOut{\Elongation{\Manifold}{1}}{\bar{\DiskEmbedding}}}
        {\CutOut
          {\BoundaryConditionTangential{\bar{\BoundaryCondition}}}
          {\bar{\DiskEmbedding}}
        }
        \ar[dd]
      \\
      \\
      \ArcResolutionTwo
        {\TangentialFibration}
        {\Genus}
        {\BoundaryComponents}
        {\Manifold}
        {\BoundaryConditionTangential{\BoundaryCondition}}
        {\DistBall}
        {\SemiSimplicialIndex}
        \ar[
          rr,
          "\StabilizationAlpha{\Genus}{\BoundaryComponents}^\SemiSimplicialIndex"
        ]
        \ar[rdd]
      &
      &
      \ArcResolutionSingle
        {\TangentialFibration}
        {\Genus+1}
        {\BoundaryComponents-1}
        {\Elongation{\Manifold}{1}}
        {\BoundaryConditionTangential{\bar{\BoundaryCondition}}}
        {\bar{\DistBall}}
        {\SemiSimplicialIndex}
        \ar[ldd]
      \\
      \\
      &
      \SpaceOfStrips
        {\TangentialFibration}
        {\Elongation{\Manifold}{1}}
        {\BoundaryConditionTangential{\bar{\BoundaryCondition}}}
        {\bar{\DistBall}}
        {\SemiSimplicialIndex}
    \end{tikzcd}
  \]

  Repeating the same procedure for maps of the form 
  $\StabilizationBeta{\Genus}{\BoundaryComponents}$, we obtain the following 
  corollary:
  \begin{corollary}
\label{crl:Fiber}
  The induced map between the homotopy fibers of
  \[
    \left(
      \StabilizationAlpha{\Genus}{\BoundaryComponents}^{\SemiSimplicialIndex}
    \right)
    \to
    \SpaceOfStrips
      {\TangentialFibration}
      {\Elongation{\Manifold}{1}}
      {\BoundaryConditionTangential{\bar{\BoundaryCondition}}}
      {\bar{\DistBall}}
      {\SemiSimplicialIndex}
  \]
  is given by 
  $
    \StabilizationBeta
      {\Genus-\SemiSimplicialIndex}
      {\BoundaryComponents+\SemiSimplicialIndex-1}
  $ 
  and analogously the induced map between the homotopy fibers of
  \[
    \left(
      \StabilizationBeta{\Genus}{\BoundaryComponents}^{\SemiSimplicialIndex}
    \right)
    \to
    \SpaceOfStrips
      {\TangentialFibration}
      {\Elongation{\Manifold}{1}}
      {\BoundaryConditionTangential{\bar{\BoundaryCondition}}}
      {\bar{\DistBall}}
      {\SemiSimplicialIndex}
  \] is given by
  $
    \StabilizationAlpha
      {\Genus-\SemiSimplicialIndex-1}
      {\BoundaryComponents+\SemiSimplicialIndex+1}
  $%
  .
\end{corollary}

  \begin{definition}
\label{def:RelativeApproximateAugmentation}
  We call the map 
  $
    \MappingPair
      {\StabilizationBeta
        {\Genus-\SemiSimplicialIndex}
        {\BoundaryComponents+\SemiSimplicialIndex-1}
      }
    \to
    \MappingPair{\StabilizationAlpha{\Genus}{\BoundaryComponents}}
  $
  given by the composition of the inclusion of the fiber in the previous 
  corollary into 
  $
    \MappingPair{
      \StabilizationAlpha{\Genus}{\BoundaryComponents}^{\SemiSimplicialIndex}
    }
  $
  with the projection onto 
  $
    \MappingPair{\StabilizationAlpha{\Genus}{\BoundaryComponents}}
  $ 
  the \introduce{relative $\SemiSimplicialIndex$-approximate augmentation} of 
  the resolution 
  $
    \left(
      \StabilizationAlpha{\Genus}{\BoundaryComponents}^{\bullet}
    \right)
  $%
  . 
  Analogously we call
  $
  \MappingPair
  {\StabilizationAlpha
    {\Genus-\SemiSimplicialIndex-1}
    {\BoundaryComponents+\SemiSimplicialIndex+1}
  }
  \to
  \MappingPair{\StabilizationBeta{\Genus}{\BoundaryComponents}}$
  defined as in the previous case the \introduce{relative 
  $\SemiSimplicialIndex$-approximate augmentation} of the resolution
  $
    \left(
      \StabilizationBeta{\Genus}{\BoundaryComponents}^{\bullet}
    \right)
  $%
  .
  If we want to emphasize the point 
  $
    \DiskEmbedding
    \in
    \SpaceOfStrips
      {\TangentialFibration}
      {\Manifold}
      {\BoundaryConditionTangential{\BoundaryCondition}}
      {\DistBall}
      {\SemiSimplicialIndex}
  $
  over which the fiber lives, then we call it the \introduce{relative 
  $\SemiSimplicialIndex$-approximate augmentation over $\DiskEmbedding$}.
\end{definition}

\section{$\TrivialityIndex$-Triviality and $\pi_0$-stabilization}
  \label{scn:kTriviality}
The following notions are inspired by the corresponding definitions in 
Section~6 of \cite{RW16} and will serve as the main requirements of a 
tangential structure in order to fulfil homological stability. The following 
definition is an adaptation of Definition~6.1 in \cite{RW16}.
\begin{definition}
\label{dfn:Absorbs}
  We say a subset bordism 
  $
    \SubsetBordism
    \subset
    \overline{
      \Manifold
      \setminus 
      \Manifold'
    }
  $
  absorbs a boundary bordism 
  $
    \BoundaryBordism
    \subset
    \DistBoundaryManifold{0}{\Manifold}
  $
  if
  $
    \SubsetBordism
    \cap
    \DistBoundaryManifold{0}{\Manifold}
    \times
    [0,1]
    \subset
    \BoundaryBordism
  $
  and there exists a 
  $
    \AuxBordism
    \subset 
    \Elongation{
      \left(
        M\setminus M'
      \right)
    }
    {1}
    \cup
    \DistBoundaryManifold{0}{\Manifold}
    \times
    [1,2]
  $
  , where 
  $
    \DistBoundaryManifold{0}{\Manifold'}
    =
    \DistBoundaryManifold{0}{\Manifold}
    \cap
    \Manifold'
  $
  (note that $\Elongation{\Manifold'}{1}$ can be embedded into 
  $\Elongation{\Manifold}{2}$ and $\AuxBordism$ is a subset of the difference 
  of these two)
  such that 
  \[
    \begin{tikzcd}
      \SpaceAllSubsurfaceBoundaryTangential
        {\TangentialFibration}
        {\Manifold'}
        {\CutOut
          {\BoundaryConditionTangential{\BoundaryCondition}}
          {\SubsetBordism}
        }
        \ar[r,"\CutOut{\BoundaryBordism}{\SubsetBordism}"]
        \ar[d,"\SubsetBordism"]
      &
      \SpaceAllSubsurfaceBoundaryTangential
        {\TangentialFibration}
        {\Elongation{\Manifold'}{1}}
        {\CutOut
           {\BoundaryConditionTangential{\BoundaryCondition'}}
           {\SubsetBordism}
        }
        \ar[d,"\Elongation{\SubsetBordism}{1}"]
        \ar[ldd,"\AuxBordism"' near start,dashed]
      \\
      \SpaceAllSubsurfaceBoundaryTangential
        {\TangentialFibration}
        {\Manifold}
        {\BoundaryConditionTangential{\BoundaryCondition}}
        \ar[r,"\BoundaryBordism" near start, crossing over]
        \ar[d,"\simeq"]
      &
      \SpaceAllSubsurfaceBoundaryTangential
        {\TangentialFibration}
        {\Elongation{\Manifold}{1}}
        {\BoundaryConditionTangential{\BoundaryCondition'}}
        \ar[d,"\simeq"]
      \\
      \SpaceAllSubsurfaceBoundaryTangential
        {\TangentialFibration}
        {\Elongation{\Manifold}{2}}
        {\BoundaryConditionTangential{\BoundaryCondition}}
        \ar[r,"\BoundaryBordism+1"]
      &
      \SpaceAllSubsurfaceBoundaryTangential
        {\TangentialFibration}
        {\Elongation{\Manifold}{3}}
        {\BoundaryConditionTangential{\BoundaryCondition'}}
    \end{tikzcd}
  \]
  the following diagram commutes up to homotopy, where the lower maps labeled 
  by $\simeq$ are 
  given gluing on a cylinder.
\end{definition}
\begin{remark}
  While this seems hard to check it is actually just a question about the path 
  components of the occuring bordism spaces that define the spaces. This will 
  be further emphasized in the proof of 
  Proposition~\ref{prp:ConnectedKTriviality}.
\end{remark}
The following definition is the analogue of the first part of Definition~6.2 
in \cite{RW16} of test pairs of height $\TrivialityIndex$:
\begin{definition}
  Given a boundary condition $\BoundaryCondition$, a sequence of pairs of 
  disjoint embedded intervals
  $ 
    \DistIntervall{0}^{i}
    \subset
    \BoundaryCondition
  $ 
  and 
  $
    \DistIntervall{1}^{i} 
    \subset
    \BoundaryCondition
  $
  for
  $1\leq i\leq \TrivialityIndex$, one can glue 
  $0$-handles along 
  $\DistIntervall{0}^{i}$ and 
  $\DistIntervall{1}^{i}$ to $\BoundaryCondition$. Doing this for the first 
  $i-1$ pairs of intervals yields an (abstract) $1$-manifold with 
  $\TrivialityIndex-i$ embedded pairs of intervals which we denote by 
  $\BoundaryCondition^{i}$. 
  
  We call such a sequence together with two more intervals 
  $\bar{\DistIntervall{0}}$ and $\bar{\DistIntervall{1}}$ \introduce{genus 
  maximizing of length $\TrivialityIndex$ for maps of type $\alpha$ and 
  boundary condition $\BoundaryCondition$}
  if:
  \begin{enumerate}[(i)]
    \item 
      All occurring embedded intervals are disjoint and 
      $\DistIntervall{0}^{i}\subset \bar{\DistIntervall{0}}$ and 
      $\DistIntervall{1}^{i}\subset \bar{\DistIntervall{1}}$
    \item 
      $\DistIntervall{0}^{1}$ and $\DistIntervall{1}^{1}$ lie in different 
      connected components of $\BoundaryCondition$
    \item
      $\DistIntervall{0}^{i}$ and $\DistIntervall{1}^{i}$ lie 
      in the same connected component of $\BoundaryCondition^{i}$ if $i$ is 
      even and in different components if $i$ is odd. 
  \end{enumerate}
  Such a sequence is \introduce{genus maximizing of length $\TrivialityIndex$ 
  for maps of type $\beta$ and boundary condition $\BoundaryCondition$} if:
  \begin{enumerate}[(i)]
  \item 
    All occurring embedded intervals are disjoint and 
    $\DistIntervall{0}^{i}\subset \bar{\DistIntervall{0}}$ and 
    $\DistIntervall{1}^{i}\subset \bar{\DistIntervall{1}}$
  \item 
    $\DistIntervall{0}^{1}$ and $\DistIntervall{1}^{1}$ lie in the same 
    connected components of $\BoundaryCondition$
  \item
    $\DistIntervall{0}^{i}$ and $\DistIntervall{1}^{i}$ lie 
    in the same connected component of $\BoundaryCondition^{i}$ if $i$ is 
    odd and in different components if $i$ is even. 
\end{enumerate}

  Furthermore given a genus maximizing sequence of length $\TrivialityIndex$, 
  we call a sequence of codimension $0$ balls $\DistBall^{i}$ all contained in 
  a ball $\overline{\DistBall}$ such that 
  $
    \DistBall^{i}
    \cap
    \BoundaryCondition 
    =
    \DistIntervall{0}^{i}
    \cup
    \DistIntervall{1}^{i}
  $
  and 
  $
    \overline{\DistBall}
    \cap
    \BoundaryCondition 
    =
    \overline{\DistIntervall{0}}
    \cup
    \overline{\DistIntervall{1}}
  $
  together with a sequence of
  $
    \DiskEmbedding^{i}
    \in
    \SpaceOfStrips
      {\TangentialFibration}
      {\CutOut{\Manifold}{\DiskEmbedding^{1},\ldots,\DiskEmbedding^{i-1}}}
      {\BoundaryConditionTangential{(\BoundaryCondition^{i-1})}}
      {\DistBall^{i}}
      {0}
  $
  for $1\leq i \leq \TrivialityIndex$
  \introduce{compatible with $(\DistIntervall{0}^{i},\DistIntervall{1}^{i})$}.
  We will usually suppress $\DistBall$ from this notation.
  In this case $\BoundaryCondition^{i}$ is naturally a boundary condition and 
  not just an abstract manifold.
\end{definition}
\begin{remark}
  Note that it is always possible to chose $\overline{\DistBall}$ and 
  $\DistBall^{i}$ 
  which intersect $\BoundaryCondition$ as required. 
  If $\BoundaryManifold{\Manifold}$ is simply-connected, then this ball is
  unique up to isotopy.
\end{remark}
We will be interested in tangential structures for which certain sequences of
approximate augmentations of the resolutions introduced in 
Section~\ref{scn:Resolution} absorb stabilization maps of type 
$\alpha$ and $\beta$.
\begin{definition}
  Given a genus maximizing sequence of length $\TrivialityIndex$ with 
  compatible $\DiskEmbedding^{i}$, we call the composition of all the $0$-th 
  approximate augmentations corresponding to the strips
  $
    \DiskEmbedding^{i}
    \in
    \SpaceOfStrips
      {\TangentialFibration}
      {\CutOut{\Manifold}{\bigcup_{j\in \{1,\ldots, i\}} \DiskEmbedding^{j}}}
      {\BoundaryConditionTangential{(\BoundaryCondition^{i})}}
      {\DistBall^{i}}
      {0}
  $
  the associated \introduce{augmentation composition of length 
  $\TrivialityIndex$} and denote it by
  \[
    \kAugmentation{\TrivialityIndex}{\Genus}{\BoundaryComponents}
    \colon
    \SpaceSubsurfaceBoundaryTangential
      {\TangentialFibration}
      {\Genus'}
      {\BoundaryComponents'}
      {
        \CutOut
          {\Manifold}
          {
            \bigcup_{j\in\{1,\ldots,\TrivialityIndex\}} 
            \DiskEmbedding^{j}
          }
      }
      {\BoundaryConditionTangential{(\BoundaryCondition^{\TrivialityIndex})}}
    \to
    \SpaceSubsurfaceBoundaryTangential
      {\TangentialFibration}
      {\Genus}
      {\BoundaryComponents}
      {\Manifold}
      {\BoundaryConditionTangential{\BoundaryCondition}}
  \]
  where 
  \begin{align*}
    \Genus'
    &=
    \begin{cases}
      \Genus-l \text{ if the sequence is genus maximizing for 
      maps of type $\alpha$ and $\TrivialityIndex$ is 2l+1}
    \\
      \Genus-l \text{ if the sequence is genus maximizing for 
        maps of type $\alpha$ and $\TrivialityIndex$ is 2l}
    \\
      \Genus-(l+1) \text{ if the sequence is genus maximizing for 
      maps of type $\beta$ and $\TrivialityIndex$ is 2l+1}
    \\
      \Genus-l \text{ if the sequence is genus maximizing for 
      maps of type $\beta$ and $\TrivialityIndex$ is 2l}
    \end{cases}
    \\
    \BoundaryComponents'
    &=
    \begin{cases}
    \BoundaryComponents-1 \text{ if the sequence is genus maximizing for 
      maps of type $\alpha$ and $\TrivialityIndex$ is 2l+1}
    \\
    \BoundaryComponents \text{ if the sequence is genus maximizing for 
      maps of type $\alpha$ and $\TrivialityIndex$ is 2l}
    \\
    \BoundaryComponents+l \text{ if the sequence is genus maximizing for 
      maps of type $\beta$ and $\TrivialityIndex$ is 2l+1}
    \\
    \BoundaryComponents \text{ if the sequence is genus maximizing for 
      maps of type $\beta$ and $\TrivialityIndex$ is 2l}
    \end{cases}
  \end{align*}
  Note that the augmentation compositions are subset bordisms.
  
  If we have a stabilization map defined via $\StabilizationBordism$, where the 
  defining boundary bordism $\StabilizationBordism$ fulfils 
  $
    \StabilizationBordism
    \cap 
    \left(
      \left(
        \bar{\DistIntervall{0}}\cup \bar{\DistIntervall{1}}
      \right)
      \times 
      [0,1]
    \right)
    =
    \left(
      \bar{\DistIntervall{0}}\cup \bar{\DistIntervall{1}}
    \right)
    \times 
    [0,1]
  $%
  , in this case we call $-\cup \StabilizationBordism$ \introduce{compatible 
  with the genus maximizing sequence}. Then $-\cup \StabilizationBordism$ lifts 
  to a map 
  \[
    -
    \cup 
    \CutOut
      {\StabilizationBordism}
      {
        \bigcup_{j\in\{1,\ldots,\TrivialityIndex\}} 
        \DiskEmbedding^{j}
      }
    \colon
    \SpaceSubsurfaceBoundaryTangential
      {\TangentialFibration}
      {\Genus'}
      {\BoundaryComponents'}
      {
        \CutOut
        {\Manifold}
        {
          \bigcup_{j\in\{1,\ldots,\TrivialityIndex\}} 
          \DiskEmbedding^{j}
        }
      }
      {\BoundaryConditionTangential{(\BoundaryCondition^{\TrivialityIndex})}}
    \to 
    \SpaceSubsurfaceBoundaryTangential
      {\TangentialFibration}
      {\Genus''}
      {\BoundaryComponents''}
      {\Elongation
        {\CutOut
          {\Manifold}
          {
            \bigcup_{j\in\{1,\ldots,\TrivialityIndex\}} 
            \bar{\DiskEmbedding}^{j}
          }
        }
        {1}
      }
      {\BoundaryConditionTangential{(\BoundaryCondition^{\TrivialityIndex})}}
  \]
  where $\Genus'$ and $\BoundaryComponents'$ are as before, and if 
  $\StabilizationBordism$ is of type $\alpha$ or $\beta$ and the sequence is 
  genus maximizing for maps of same type, then 
  $
    -
    \cup 
    \CutOut
    {\StabilizationBordism}
    {
      \bigcup_{j\in\{1,\ldots,\TrivialityIndex\}} 
      \DiskEmbedding^{j}
    }
  $
  is a map of the same type as before if $\TrivialityIndex$ is even, and of the 
  other type if $\TrivialityIndex$ is odd. We call this map a 
  \introduce{ 
    relative augmentation compositions of length $\TrivialityIndex$
  }
  and denote it by 
  $\AlphakAugmentation{\TrivialityIndex}{\Genus}{\BoundaryComponents}$ if 
  $-\cup \StabilizationBordism$ is of type alpha, and 
  $\BetakAugmentation{\TrivialityIndex}{\Genus}{\BoundaryComponents}$ in the 
  other case.
\end{definition}
\begin{remark}
  Note that by the dimension restriction on $\Manifold$, and 
  $\BoundaryManifold{\Manifold}$ being simply-connected, we can find an isotopy 
  that changes a $\StabilizationBordism$ into one that is compatible with the 
  genus maximizing sequence.
\end{remark}
Again this definition is inspired by the second part of Definition~6.2 in 
\cite{RW16}.
\begin{definition}
  \label{dfn:kTriviality}
  We call a space of $\TangentialFibration$-structures
  \introduce{$\TrivialityIndex$-trivial} if for every genus maximizing sequence 
  of length 
  $\TrivialityIndex$ and a compatible sequence $\DiskEmbedding^{i}$, and every 
  compatible stabilization map of type $\alpha$ or $\beta$, the
  boundary bordism $\StabilizationBordism$ is absorbed by the subset bordism 
  defining the relative augmentation compositions of length $\TrivialityIndex$ 
  associated to the genus maximizing sequence.
\end{definition}

In Section~\ref{scn:TangentialkTriviality} we will prove that it suffices to 
have stabilisation of connected components in order to obtain 
$\TrivialityIndex$-triviality for some $\TrivialityIndex$. To state this we 
need the following definition, which is the analogue of Definition~6.3 in 
\cite{RW16}:
\begin{definition}
  We say a tangential structure 
  $
    \TangentialFibration
    \colon 
    \TangentialSpace{\Manifold}
    \to
    \Grassmannian{2}{\TangentBundle{\Manifold}}
  $
  \introduce{$\pi_0$-stabilizes at genus $\Genus$} if all stabilization maps 
  \begin{align*}
    \StabilizationAlpha
      {\Genus'}
      {\BoundaryComponents}
    \colon
    \SpaceSubsurfaceBoundaryTangential
      {\TangentialFibration}
      {\Genus'}
      {\BoundaryComponents}
      {\Manifold}
      {\BoundaryConditionTangential{\BoundaryCondition}}
    &
    \to
    \SpaceSubsurfaceBoundaryTangential
      {\TangentialFibration}
      {\Genus'+1}
      {\BoundaryComponents-1}
      {\Elongation{\Manifold}{1}}
      {\BoundaryConditionTangential{\bar{\BoundaryCondition}}}
    \\
    \StabilizationBeta
      {\Genus'}
      {\BoundaryComponents}
    \colon
    \SpaceSubsurfaceBoundaryTangential
      {\TangentialFibration}
      {\Genus'}
      {\BoundaryComponents}
      {\Manifold}
      {\BoundaryConditionTangential{\BoundaryCondition}}
    &
    \to
    \SpaceSubsurfaceBoundaryTangential
      {\TangentialFibration}
      {\Genus'}
      {\BoundaryComponents+1}
      {\Elongation{\Manifold}{1}}
      {\BoundaryConditionTangential{\bar{\BoundaryCondition}}}
  \end{align*}
  induce bijections on $\pi_0$ for all $\Genus'\geq \Genus$ and they induce a 
  surjection on $\pi_0$ for all $\Genus'\geq \Genus-1$. 
  
  We say a tangential structure
  $
    \TangentialFibration
    \colon
    \TangentialSpace{\Manifold}
    \to
    \Grassmannian{2}{\TangentBundle{\Manifold}}
  $
  \introduce{$\pi_{0}$-stabilizes at the boundary at genus $\Genus$} if all 
  stabilization maps:
  \begin{align*}
    \StabilizationAlpha
      {\Genus'}
      {\BoundaryComponents}
    \colon
    \SpaceSubsurfaceBoundaryTangential
      {\TangentialFibration}
      {\Genus'}
      {\BoundaryComponents}
      {\DistBoundaryManifold{0}{\Manifold}\times [0,1]}
      {\BoundaryConditionTangential{\BoundaryCondition}}
    &
    \to
    \SpaceSubsurfaceBoundaryTangential
      {\TangentialFibration}
      {\Genus'+1}
      {\BoundaryComponents-1}
      {\DistBoundaryManifold{1}{\Manifold}\times [0,2]}
      {\BoundaryConditionTangential{\bar{\BoundaryCondition}}}
    \\
    \StabilizationBeta
      {\Genus'}
      {\BoundaryComponents}
    \colon
    \SpaceSubsurfaceBoundaryTangential
      {\TangentialFibration}
      {\Genus'}
      {\BoundaryComponents}
      {\DistBoundaryManifold{0}{\Manifold}\times[0,1]}
      {\BoundaryConditionTangential{\BoundaryCondition}}
    &
    \to
    \SpaceSubsurfaceBoundaryTangential
      {\TangentialFibration}
      {\Genus'}
      {\BoundaryComponents+1}
      {\DistBoundaryManifold{0}{\Manifold}\times [0,2]}
      {\BoundaryConditionTangential{\bar{\BoundaryCondition}}}
  \end{align*}
  induce bijections on $\pi_0$ for all $\Genus'\geq \Genus$ and they induce a 
  surjection on $\pi_0$ for all $\Genus'\geq \Genus-1$. 
\end{definition}
The rather abstract property of $\TrivialityIndex$-triviality is already 
implied by $\pi_0$-stabilization. This is encapsulated in the following 
proposition, which will be proven in Section~\ref{scn:TangentialkTriviality} 
and is an adaptation of Proposition~6.5 in \cite{RW16}.
\begin{proposition}
  \label{prp:ConnectedKTriviality}
  Suppose $\Manifold$ is an at least $5$-dimensional manifold which is simply 
  connected and $\DistBoundaryManifold{0}{\Manifold}$ is a simply-connected 
  codimension $0$ submanifold of $\BoundaryManifold{\Manifold}$.
  Suppose further that 
  $
    \TangentialFibration
    \colon
    \TangentialSpace{\Manifold}
    \to 
    \Grassmannian{2}{\TangentBundle{\Manifold}}
  $
  $\pi_0$-stabilizes at the boundary at $\Genus$, then it is 
  $2\Genus+1$-trivial.
\end{proposition}
\section{Proving Homological Stability}
  \label{scn:Proof}
In order to prove homological stability we will need the stability bounds. 
The following definition stems from Definition~6.9 in \cite{RW16}. 

\begin{definition}
  \label{dfn:StabilityBounds}
  Given two natural numbers $\TrivialityIndex\geq 1$ and 
  $\PiZeroStabilization\geq 0$, we define the following recursive functions 
  $\AlphaBound,\BetaBound,\AlphaAuxBound,\BetaAuxBound\colon \Integers\to 
  \Integers$ as follows:
  \begin{enumerate}[(i)]
    \item 
        \[
          \apply{\AlphaBound}{\Genus}
          =
          \apply{\BetaBound}{\Genus}
          =
          \apply{\AlphaAuxBound}{\Genus}
          =
          \apply{\BetaAuxBound}{\Genus}
          =
          -1
        \]
        for $\Genus\leq \PiZeroStabilization-2$ and 
        \[
          \apply{\AlphaBound}{\PiZeroStabilization-1}
          =
          \apply{\BetaBound}{\PiZeroStabilization}
          =
          \apply{\AlphaAuxBound}{\PiZeroStabilization}
          =
          \apply{\BetaAuxBound}{\PiZeroStabilization}
          =
          0
        \]
      \item
        We define
        \begin{equation*}
          \apply{\AlphaAuxBound}{\Genus}
          =
          \min
          \begin{cases}
            \apply{\AlphaBound}{\Genus-1}+1
            \\
            \apply{\BetaBound}{\Genus-1}+1
            \\
            \apply{\AlphaAuxBound}{\Genus-1}+1
            \\
            \apply{\BetaAuxBound}{\Genus-1}+1
            \\
            0\text{ if }\Genus\leq 0
          \end{cases}
          \phantom{asdasdasdasdasdas}
          \apply{\BetaAuxBound}{\Genus}
          =
          \begin{cases}
            \apply{\AlphaBound}{\Genus-2}+1
            \\
            \apply{\BetaBound}{\Genus-1}+1
            \\
            \apply{\AlphaAuxBound}{\Genus-2}+1
            \\
            \apply{\BetaAuxBound}{\Genus-1}+1
            \\
            0\text{ if }\Genus\leq 0
          \end{cases}
        \end{equation*}
      \item
        If $\TrivialityIndex=1$, we define:
        \begin{equation*}
          \apply{\AlphaBound}{\Genus}
          =
          \min
          \begin{cases}
            \apply{\AlphaBound}{\Genus-1}+1
            \\
            \apply{\AlphaAuxBound}{\Genus}
          \end{cases}
        \phantom{asdasdasdasdasdas}
        \apply{\BetaBound}{\Genus}
        =
        \begin{cases}
          \apply{\BetaBound}{\Genus-1}+1
          \\
          \apply{\AlphaAuxBound}{\Genus}
        \end{cases}
        \end{equation*}

      \item
        If $\TrivialityIndex=2l$ with $l>0$, we define:
        \begin{equation*}
          \apply{\AlphaBound}{\Genus}
          =
          \min
          \begin{cases}
            \apply{\AlphaBound}{\Genus-1}+1
            \\
            \apply{\BetaBound}{\Genus-l}+1
            \\
            \max\{\apply{\AlphaAuxBound}{\Genus+1-l},0\}
            \\
            \max \{\apply{\BetaAuxBound}{\Genus+1-l},0\}
          \end{cases}
          \phantom{asdasdasdasdasdas}
          \apply{\BetaBound}{\Genus}
          =
          \begin{cases}
            \apply{\AlphaBound}{\Genus-l-1}+1
            \\
            \apply{\BetaBound}{\Genus-1}+1
            \\
            \max\{\apply{\AlphaAuxBound}{\Genus-l},0\}
            \\
            \max\{\apply{\BetaAuxBound}{\Genus+1-l},0\}
          \end{cases}
        \end{equation*}
      \item
        If $\TrivialityIndex=2l+1$ with $l>0$, we define:
        \begin{equation*}
          \apply{\AlphaBound}{\Genus}
          =
          \min
          \begin{cases}
            \apply{\AlphaBound}{\Genus-l-1}+1
            \\
            \max\{\apply{\AlphaAuxBound}{\Genus-l},0\}
            \\
            \max \{\apply{\BetaAuxBound}{\Genus+1-l},0\}
          \end{cases}
          \phantom{asdasdasdasdasdas}
          \apply{\BetaBound}{\Genus}
          =
          \begin{cases}
            \apply{\BetaBound}{\Genus-l-1}+1
            \\
            \max\{\apply{\AlphaAuxBound}{\Genus-l},0\}
            \\
            \max\{\apply{\BetaAuxBound}{\Genus-l},0\}
          \end{cases}
        \end{equation*}
  \end{enumerate}
\end{definition}
\begin{remark}
\label{rmk:StabilityBounds}
  If $\TrivialityIndex=1$ and $\PiZeroStabilization=0$, then 
  $
    \apply{\AlphaBound}{\Genus}
    =
    \apply{\AlphaAuxBound}{\Genus}
    =
    \left\lfloor
      \frac{1}{3}(2\Genus + 1)
    \right\rfloor
  $
  and 
  $
    \apply{\BetaBound}{\Genus}
    =
    \apply{\BetaAuxBound}{\Genus}
    =
    \left\lfloor
      \frac{2}{3}\Genus
    \right\rfloor
  $
  gives the aforementioned functions.
\end{remark}
The goal of the rest of this section is to prove the most of the first two 
parts of the 
following theorem. 
The proof will be an adaptation of the proof of the main theorem of 
\cite{CRW16} to include $\TrivialityIndex$-triviality:
\begin{theorem}
  \label{thm:HomologicalStability}
  Suppose $\Manifold$ is an at least $5$-dimensional, simply-connected manifold 
  with non-empty boundary and $\DistBoundaryManifold{0}{\Manifold}$ is a 
  codimension $0$ submanifold of $\BoundaryManifold{\Manifold}$. Suppose 
  further that 
  $
    \TangentialFibration
    \colon
    \TangentialSpace{\Manifold}
    \to 
    \Grassmannian{2}{\TangentBundle{\Manifold}}
  $
  is a 
  $\TrivialityIndex$-trivial space of $\TangentialFibration$-structures on 
  subplanes of $\TangentBundle{\Manifold}$ which $\pi_0$-stabilizes at 
  $\PiZeroStabilization$, then
  \begin{itemize}
  \item
    the homology of 
    $\MappingPair{\StabilizationAlpha{\Genus}{\BoundaryComponents}}$ vanishes 
    in degrees $\ast\leq \apply{\AlphaBound}{\Genus}$
  \item 
    the homology of 
    $\MappingPair{\StabilizationBeta{\Genus}{\BoundaryComponents}}$ vanishes 
    in degrees $\ast\leq \apply{\BetaBound}{\Genus}$. If one of the newly 
    created boundary components is contractible in 
    $\DistBoundaryManifold{0}{\Manifold}$ and the tangential structure on that 
    boundary component vanishes in 
    $\HomotopyGroup{1}{\TangentialSpace{\Manifold}}$, then 
    $\StabilizationBeta{\Genus}{\BoundaryComponents}$ is a monomorphism in all 
    degrees
  \item
    the homology of 
    $\MappingPair{\StabilizationGamma{\Genus}{\BoundaryComponents}}$ vanishes
    in degrees $\ast\leq \apply{\BetaBound}{\Genus}$. If 
    $\BoundaryComponents>1$, then 
    $\StabilizationGamma{\Genus}{\BoundaryComponents}$ induces an epimorphism 
    in all degrees
  \end{itemize}
  where $\AlphaBound$ and $\BetaBound$ are defined as in 
  Definition~\ref{dfn:StabilityBounds} and $\MappingPair{-}$ denotes the 
  mapping pair as in Definition~\ref{dfn:MappingPair}.
\end{theorem}
The proof will proceed via induction. More precisely we will use induction for 
the following two statements:
\begin{itemize}
  \item
    $\AlphaBound_{\InductionIndex}$:
    For all $\Genus\leq \InductionIndex$ and
    for all $\StabilizationAlpha{\Genus}{\BoundaryComponents}$,
    the reduced homology of 
    $\MappingPair{\StabilizationAlpha{\Genus}{\BoundaryComponents}}$ vanishes 
    in degrees $\ast\leq \apply{\AlphaBound}{\Genus}$
  \item
    $\BetaBound_{\InductionIndex}$:
    For all $\Genus\leq \InductionIndex$ and
    for all $\StabilizationBeta{\Genus}{\BoundaryComponents}$, 
    the reduced homology of 
    $\MappingPair{\StabilizationBeta{\Genus}{\BoundaryComponents}}$ vanishes 
    in degrees $\ast\leq \apply{\BetaBound}{\Genus}$
\end{itemize}
Furthermore we will also need the following two auxiliary statements:
\begin{itemize}
\item
  $\AlphaAuxBound_{\InductionIndex}$: For all $\Genus\leq
  \InductionIndex$ and all strips with embedded boundary isotopy, the relative 
  $0$-approximate augmentation induces a surjection 
  $
    \HomologyOfSpace
      {\ast}
      {
        \MappingPair
        {\StabilizationBeta
          {\Genus}
          {\BoundaryComponents-1}
        }
      }
    \to
    \HomologyOfSpace
      {\ast}
      {
        \MappingPair{\StabilizationAlpha{\Genus}{\BoundaryComponents}}
      }
  $
  in 
  homology in degrees $\ast \leq \apply{\AlphaAuxBound}{\Genus}$.
\item
  $\BetaAuxBound_{\InductionIndex}$: For all $\Genus\leq \InductionIndex$ and 
  and all strip with embedded boundary isotopy, the relative $0$-approximate 
  augmentation induces a surjection
  $
    \HomologyOfSpace
      {\ast}
      {
        \MappingPair{\StabilizationAlpha{\Genus-1}{\BoundaryComponents+1}}
      }
    \to
    \HomologyOfSpace
      {\ast}
      {
        \MappingPair{\StabilizationBeta{\Genus}{\BoundaryComponents}}
      }
  $
  in degrees $\ast \leq \apply{\BetaAuxBound}{\Genus}$.
\end{itemize}
Note that the start of the induction is provided by the assumption that 
the tangential structure $\pi_0$-stabilizes at genus $\PiZeroStabilization$.
 
We will also denote the \introduce{mapping cone} of a map $\ContinuousMap$ by 
$\MappingCone{\ContinuousMap}$, if $\ContinuousMap$ is an inclusion 
$\Submanifold\to \TopologicalSpace$ we will write 
$\MappingCone{\TopologicalSpace,\Submanifold}$ instead.
The following technical lemma will be needed in the proof. It is Lemma~7.5 
in \cite{CRW16}.
  \begin{lemma}
  \label{lem:Nullhomotopic}
    Suppose we have a map of pairs \[
    \begin{tikzcd}
      A\arrow{r}{i}\arrow{d}{g}&X\arrow{d}{f}\\
      A'\arrow{r}{j}&X'
    \end{tikzcd}
    \] such that there exists a map $t\colon X\to A'$ making the bottom
    triangle commute up to a homotopy $H$.
    Then the induced maps between mapping cones 
    $
      (f,g)
      \colon 
      \bar{C}(X,A)
      \to 
      \bar{C}(X',A')
    $
    factors as 
    $
      \bar{C}(X,A)
      \xrightarrow{p}
      \bar{C}A
      \cup_i 
      \bar{C}X
      \xrightarrow{h}
      \bar{C}(X',A')
    $%
    , where $p$ comes from the Puppe sequence of the pair $(X,A)$. 
    
    Additionaly if there is a homotopy $G$, which makes the bottom triangle 
    commute, then the composite 
    $
      \bar{C}A
      \cup_i 
      \bar{C}X
      \xrightarrow{h}
      \bar{C}(X',A')
      \xrightarrow{p'}
      \bar{C}A'
      \cup_j 
      \bar{C}X'
    $
    is nullhomotopic.
  \end{lemma}

We have the following proposition of convoluted implications finishing the 
proof of Theorem~\ref{thm:HomologicalStability}:
\begin{proposition}
  With the assumptions of Theorem~\ref{thm:HomologicalStability} we have the 
  following implications:
  \begin{enumerate}[(i)]
  \item
  \label{itm:AlphaAuxBound}
    $\BetaBound_{\InductionIndex}$ implies
    $\AlphaAuxBound_{\InductionIndex}$.
  \item
  \label{itm:BetaAuxBound}
    $\AlphaBound_{\InductionIndex-1}$ implies $\BetaAuxBound_{\InductionIndex}$
  \item
  \label{itm:AlphaBound}
    If $\TrivialityIndex=1$, then $\AlphaAuxBound_{\InductionIndex}$ and 
    $\AlphaBound_{\InductionIndex-1}$ imply $\AlphaBound_{\InductionIndex}$
    
    If $\TrivialityIndex>1$ then $\AlphaAuxBound_{\InductionIndex}$, 
    $\BetaAuxBound_{\InductionIndex}$ and 
    $
      \begin{cases}
        \BetaBound_{\InductionIndex-l} \text{ if $\TrivialityIndex=2l$ }\\
        \AlphaBound_{\InductionIndex-l-1} \text{ if $\TrivialityIndex=2l+1$ }
      \end{cases}
    $
    imply $\AlphaBound_{\InductionIndex}$
  \item
  \label{itm:BetaBound}
    If $\TrivialityIndex=1$, then $\BetaAuxBound_{\InductionIndex}$ and 
    $\BetaBound_{\InductionIndex-1}$ imply $\BetaBound_{\InductionIndex}$
    
    If $\TrivialityIndex>1$ then $\BetaAuxBound_{\InductionIndex}$, 
    $\AlphaAuxBound_{\InductionIndex-1}$ and 
    $
    \begin{cases}
      \AlphaBound_{\InductionIndex-l-1} \text{ if $\TrivialityIndex=2l$ }\\
      \BetaBound_{\InductionIndex-l-1} \text{ if $\TrivialityIndex=2l+1$ }
    \end{cases}
    $
    imply $\BetaBound_{\InductionIndex}$
  \end{enumerate}
\end{proposition}
\begin{proof}[Proof of (\ref{itm:AlphaAuxBound}) and (\ref{itm:BetaAuxBound})]
  We will prove (\ref{itm:AlphaAuxBound}), the other part is completely 
  analogous.
  By corollary \ref{crl:Fiber} we get the following diagram:
  \[
    \begin{tikzcd}
      \SpaceSubsurfaceBoundaryTangential
        {\TangentialFibration}
        {\InductionIndex-\FaceIndex}
        {\BoundaryComponents+\FaceIndex-1}
        {\CutOut{\Manifold}{\DiskEmbedding}}
        {\BoundaryConditionTangential{\CutOut{\BoundaryCondition}{\DiskEmbedding}}}
        \ar[
          rr,
          "\StabilizationBeta
            {\InductionIndex-\FaceIndex}
            {\BoundaryComponents+\FaceIndex-1}
          "
        ]
        \ar[d]
      &
      &
      \SpaceSubsurfaceBoundaryTangential
        {\TangentialFibration}
        {\InductionIndex-\FaceIndex}
        {\BoundaryComponents+\FaceIndex}
        {\Elongation{\CutOut{\Manifold}{\bar{\DiskEmbedding}}}{1}}
        {
          \BoundaryConditionTangential
            {\CutOut{\BoundaryCondition'}{\bar{\DiskEmbedding}}}
        }
        \ar[d]
      \\
      \ArcResolutionTwo
        {\TangentialFibration}
        {\InductionIndex}
        {\BoundaryComponents}
        {\Manifold}
        {\BoundaryConditionTangential{\BoundaryCondition}}
        {\DistBall}
        {\FaceIndex}
        \ar[d] 
        \ar[rr,"(\StabilizationAlpha{\InductionIndex}{\BoundaryComponents})_{\bullet}"]
      &
      &
      \ArcResolutionSingle
        {\TangentialFibration}
        {\InductionIndex+1}
        {\BoundaryComponents-1}
        {\Elongation{\Manifold}{1}}
        {\BoundaryConditionTangential{\BoundaryCondition'}}
        {\bar{\DistBall}}
        {\FaceIndex}
      \ar[r]
      \ar[d]
      &
      \SpaceOfStrips
        {\TangentialFibration}
        {\Manifold}
        {\BoundaryConditionTangential{\BoundaryCondition}}
        {\DistBall}
        {\FaceIndex}
      \\
      \SpaceSubsurfaceBoundaryTangential
        {\TangentialFibration}
        {\InductionIndex}
        {\BoundaryComponents}
        {\Manifold}
        {\BoundaryConditionTangential{\BoundaryCondition}}
        \ar[rr,"\StabilizationAlpha{\InductionIndex}{\BoundaryComponents}"] 
      &
      &
      \SpaceSubsurfaceBoundaryTangential
        {\TangentialFibration}
        {\InductionIndex+1}
        {\BoundaryComponents-1}
        {\Elongation{\Manifold}{1}}
        {\BoundaryConditionTangential{\BoundaryCondition'}}
    \end{tikzcd}
  \]
  Furthermore the augmentation 
  $
    \ArcResolutionTwo
      {\TangentialFibration}
      {\InductionIndex}
      {\BoundaryComponents}
      {\Manifold}
      {\BoundaryConditionTangential{\BoundaryCondition}}
      {\DistBall}
      {\bullet} 
    \to 
    \SpaceSubsurfaceBoundaryTangential
      {\TangentialFibration}
      {\InductionIndex}
      {\BoundaryComponents}
      {\Manifold}
      {\BoundaryConditionTangential{\BoundaryCondition}}
  $
  is $(\InductionIndex-2)$-connected by Proposition 
  \ref{prp:ArcResolution}.
  Since $\CutOut{\Manifold}{\DiskEmbedding}$ is homeomorphic to
  $\Manifold$, $\BetaBound_{\InductionIndex}$ implies 
  that 
  $
    \HomologyOfSpace
      {q}
      {\StabilizationBeta{\InductionIndex-\FaceIndex}{\BoundaryComponents+\FaceIndex-1}}
    =
    0
  $
  for $q \leq \apply{\BetaBound}{\InductionIndex-\FaceIndex}$.
  This gives us all the ingredients to apply Lemma 
  \ref{lem:StabilityCriterion} with
  $
    c=\apply{\AlphaAuxBound}{\InductionIndex}
  $%
  , using that 
  $
    \apply{\AlphaAuxBound}{\InductionIndex} 
    \geq 
    \apply{\BetaBound}{\Genus-i}+i
  $%
  .
  The lemma implies that the induced map
  $
    \HomologyOfSpace
      {q}
      {\MappingPair{\StabilizationBeta{\InductionIndex}{\BoundaryComponents-1}}}
    \to 
    \HomologyOfSpace
      {q}
      {\MappingPair{\StabilizationAlpha{\InductionIndex}{\BoundaryComponents}}}
  $
  is an epimorphism for $q\leq \apply{\AlphaAuxBound}{\InductionIndex}$, which 
  is exactly the statement of $\AlphaAuxBound_{\InductionIndex}$. 
\end{proof}
\begin{proof}[Proof of (\ref{itm:AlphaBound}) and (\ref{itm:BetaBound})]
  We will prove (\ref{itm:AlphaBound}) for $\TrivialityIndex=2l$ here. The 
  proof for odd $\TrivialityIndex$ is completely analogous and we will indicate 
  the necessary minor changes for the proof of (\ref{itm:BetaBound}).
  
  Pick a genus maximizing sequence 
  $(\DistIntervall{j}^{i})_{1,\ldots,\TrivialityIndex+1}$ of length 
  $\TrivialityIndex+1$ for maps of type $\alpha$ with a compatible sequence
  $(\DiskEmbedding^{i})_{1,\ldots,\TrivialityIndex+1}$ of arcs with embedded 
  boundary isotopy such that 
  $(\DiskEmbedding^{i})_{2,\ldots,\TrivialityIndex,1}$ is also a genus 
  maximizing sequence of length $\TrivialityIndex+1$. 
  To get such a sequence, choose for example $\TrivialityIndex+1$ disjoint 
  intervals in the two 
  components of $\BoundaryCondition$ and then label them positively oriented 
  from $1$ to $\TrivialityIndex+1$, this gives the desired genus maximizing 
  sequence (For (\ref{itm:BetaBound}) choose $2\TrivialityIndex+2$ disjoint 
  intervals in the single component of $\BoundaryCondition$ and label them as 
  follows: $\DistIntervall{0}^{\TrivialityIndex+1},\ldots 
  \DistIntervall{0}^{1},\DistIntervall{1}^{\TrivialityIndex+1},\ldots 
  \DistIntervall{1}^{1}$; in this case $\DistIntervall{0}^{1}$ and 
  $\DistIntervall{1}^{1}$ switch if one moves $\DiskEmbedding^{1}$ to the end 
  of the sequence).
  
  Now form the following diagram, where we denote 
  $
    \DiskEmbedding^{1}
    \cup 
    \ldots 
    \cup
    \DiskEmbedding^{\TrivialityIndex+1}
  $ 
  by $d$, 
  $
    \bar{\DiskEmbedding}^{1}
    \cup
    \ldots
    \cup
    \bar{\DiskEmbedding}^{\TrivialityIndex+1}
  $
  by $\bar{d}$,
  $
    \DiskEmbedding^{2}
    \cup
    \ldots
    \cup 
    \DiskEmbedding^{\TrivialityIndex+1}
  $
  by $c$, and $\bar{c}$ analogously:
  \[
  \begin{lrbox}{\wideeqbox}
    $\begin{tikzcd}[cramped, column sep=small]
      \SpaceSubsurfaceBoundaryTangential
        {\TangentialFibration}
        {\InductionIndex-l}
        {\BoundaryComponents-1}
        {\CutOut
          {\Manifold}
          {d}
        }
        {\BoundaryConditionTangential
          {\CutOut
            {\BoundaryCondition}
            {d}
          }
        }
        \ar[r]
        \ar[d,"\kAugmentation{1}{\InductionIndex-l}{\BoundaryComponents}(\DiskEmbedding^{1})"]
      &
      \SpaceSubsurfaceBoundaryTangential
        {\TangentialFibration}
        {\InductionIndex-l}
        {\BoundaryComponents}
        {\CutOut
          {\Elongation{\Manifold}{1}}
          {\bar{d}}
        }
        {\BoundaryConditionTangential
          {\CutOut
            {\BoundaryCondition'}
            {\bar{d}}
          }
        }
        \ar[r]
        \ar[d,
          "
            \kAugmentation
              {1}
              {\InductionIndex-l}
              {\BoundaryComponents}
            (\bar{\DiskEmbedding}^{1})
          "
        ]
      &
      \MappingCone{\StabilizationBeta{\InductionIndex-l}{\BoundaryComponents-1}}
        \ar[d]
        \ar[r]
      &
      \Suspension{
        \SpaceSubsurfaceBoundaryTangential
          {\TangentialFibration}
          {\InductionIndex-l}
          {\BoundaryComponents-1}
          {\CutOut
            {\Manifold}
            {d}
          }
          {\BoundaryConditionTangential
            {\CutOut
              {\BoundaryCondition}
              {d}
            }
          }
      }
        \ar[d,"
          \Suspension
            {
              \kAugmentation{1}{\InductionIndex-l}{\BoundaryComponents}(\DiskEmbedding^{1})
            }
          "
        ]
      \\
      \SpaceSubsurfaceBoundaryTangential
        {\TangentialFibration}
        {\InductionIndex-l}
        {\BoundaryComponents}
        {\CutOut
          {\Manifold}
          {c}
        }
        {\BoundaryConditionTangential
          {\CutOut
            {\BoundaryCondition}
            {c}
          }
        }
        \ar[r]
        \ar[d,
          "
            \kAugmentation
              {\TrivialityIndex}
              {\InductionIndex}
              {\BoundaryComponents}
            (c)
          "
        ]
      &
      \SpaceSubsurfaceBoundaryTangential
        {\TangentialFibration}
        {\InductionIndex-l+1}
        {\BoundaryComponents-1}
        {\CutOut
          {\Elongation{\Manifold}{1}}
          {\bar{c}}
        }
        {\BoundaryConditionTangential
          {\CutOut
            {\BoundaryCondition'}
            {\bar{c}}
          }
        }
        \ar[r]
        \ar[d,
          "
            \kAugmentation
              {\TrivialityIndex}
              {\InductionIndex+1}
              {\BoundaryComponents-1}
            (\bar{c})
          "
        ]
      &
      \MappingCone{\StabilizationAlpha{\InductionIndex-l}{\BoundaryComponents}}
        \ar[d]
        \ar[r,"p"]
      &
      \Suspension{
        \SpaceSubsurfaceBoundaryTangential
          {\TangentialFibration}
          {\InductionIndex-l}
          {\BoundaryComponents}
          {\CutOut
            {\Manifold}
            {c}
          }
          {\BoundaryConditionTangential
            {\CutOut
              {\BoundaryCondition}
              {c}
            }
          }
      }
      \ar[dl,"h"]
      \\
      \SpaceSubsurfaceBoundaryTangential
        {\TangentialFibration}
        {\InductionIndex}
        {\BoundaryComponents}
        {\Manifold}
        {\BoundaryConditionTangential{\BoundaryCondition}}
        \ar[r]
      &
      \SpaceSubsurfaceBoundaryTangential
        {\TangentialFibration}
        {\InductionIndex+1}
        {\BoundaryComponents-1}
        {\Elongation{\Manifold}{1}}
        {\BoundaryConditionTangential{\BoundaryCondition'}}
        \ar[r]
      &
      \MappingCone{\StabilizationAlpha{\InductionIndex}{\BoundaryComponents}}
      \\
      \SpaceSubsurfaceBoundaryTangential
        {\TangentialFibration}
        {\InductionIndex}
        {\BoundaryComponents-1}
        {\CutOut
          {\Manifold}
          {\DiskEmbedding^{1}}
        }
        {\BoundaryConditionTangential
          {\CutOut
            {\BoundaryCondition}
            {\tilde{\DiskEmbedding}^{1}}
          }
        }
        \ar[r]
        \ar[u,
          "
            \kAugmentation
              {1}
              {\InductionIndex}
              {\BoundaryComponents}
            (\DiskEmbedding^{1})
          "
        ]
      &
      \SpaceSubsurfaceBoundaryTangential
        {\TangentialFibration}
        {\InductionIndex}
        {\BoundaryComponents}
        {\CutOut
          {\Elongation{\Manifold}{1}}
          {\bar{\DiskEmbedding}^{1}}
        }
        {\BoundaryConditionTangential
          {\CutOut
            {\BoundaryCondition'}
            {\tilde{\DiskEmbedding}^{1}}
          }
        }
        \ar[r]
        \ar[u,
          "
            \kAugmentation
              {1}
              {\InductionIndex+1}
              {\BoundaryComponents-1}
            (\bar{\DiskEmbedding}^{1})
          "
        ]
      &
      \MappingCone{\StabilizationBeta{\InductionIndex}{\BoundaryComponents-1}}
        \ar[u]
        \ar[r,"p'"]
      &
      \Suspension{
        \SpaceSubsurfaceBoundaryTangential
          {\TangentialFibration}
          {\InductionIndex}
          {\BoundaryComponents-1}
          {\CutOut
            {\Manifold}
            {\DiskEmbedding^{1}}
          }
          {\BoundaryConditionTangential
            {\CutOut
              {\BoundaryCondition}
              {\tilde{\DiskEmbedding}^{1}}
            }
          }
      }
        \ar[ul,"h'"]
      \\
      \SpaceSubsurfaceBoundaryTangential
        {\TangentialFibration}
        {\InductionIndex-l}
        {\BoundaryComponents-1}
        {\CutOut
          {\Manifold}
          {d}
        }
        {\BoundaryConditionTangential
          {\CutOut
            {\BoundaryCondition}
            {d}
          }
        }
        \ar[r]
        \ar[u,
          "
            \kAugmentation
              {\TrivialityIndex}
              {\InductionIndex}
              {\BoundaryComponents-1}
            (d)
          "
        ]
      &
      \SpaceSubsurfaceBoundaryTangential
        {\TangentialFibration}
        {\InductionIndex-l}
        {\BoundaryComponents}
        {\CutOut
          {\Elongation{\Manifold}{1}}
          {\bar{d}}
        }
        {\BoundaryConditionTangential
          {\CutOut
            {\BoundaryCondition'}
            {\bar{d}}
          }
        }
        \ar[r]
        \ar[u,
          "
            \kAugmentation
              {\TrivialityIndex}
              {\InductionIndex}
              {\BoundaryComponents}
            (\bar{d})
          "
        ]
      &
      \MappingCone{\StabilizationBeta{\InductionIndex-l}{\BoundaryComponents-1}}
      \ar[u]
      \ar[r]
      &
      \Suspension{
        \SpaceSubsurfaceBoundaryTangential
        {\TangentialFibration}
        {\InductionIndex-l}
        {\BoundaryComponents-1}
        {\CutOut
          {\Manifold}
          {d}
        }
        {\BoundaryConditionTangential
          {\CutOut
            {\BoundaryCondition}
            {d}
          }
        }
      }
      \ar[ul,"h''"]
    \end{tikzcd}$
    \end{lrbox}
    \makebox[0pt]{\usebox{\wideeqbox}}
  \]
  Here the rows are given by the corresponding Puppe sequences, note that the 
  first and last line agree. The vertical maps are augmentation compositions 
  and we have indicated in brackets which arcs with boundary isotopy are used 
  there. 
  
  Since $\TangentialFibration$ is $\TrivialityIndex$-trivial we can find a
  $
    \AuxBordism
    \subset
    c
    \cup 
    \DistBoundaryManifold
      {0}
      {\CutOut{\Manifold}{\DiskEmbedding^{1}}}
    \times 
    [1,2]
  $ 
  providing a diagonal map for the lower left square. 
  Note that 
  $
    \AuxBordism 
    \cup 
    (
      \bar{\DiskEmbedding}^{1}
      \setminus
      \DiskEmbedding^{1}
    )
    \subset
    c
    \cup
    \DistBoundaryManifold{0}{\Manifold}
    \times
    [1,2]
  $
  gives a bordism that provides a diagonal map for the second square. This 
  implies that the diagonal map of the lower square postcomposed with 
  $
    \kAugmentation
      {1}
      {\InductionIndex}
      {\BoundaryComponents}
    (\DiskEmbedding^{1})
  $
  agrees up to homotopy with the composition of 
  $
    \kAugmentation
      {1}
      {\InductionIndex-l}
      {\BoundaryComponents}
    (\bar{\DiskEmbedding}^{1})
  $
  and the diagonal map of the second square.
  Using this observation and Lemma~\ref{lem:Nullhomotopic}, one obtains the 
  following:
  \[
    h
    \circ 
    \Suspension
    {
      \kAugmentation{1}{\InductionIndex-l}{\BoundaryComponents}(\DiskEmbedding^{1})
    }
    \simeq
    (
      \kAugmentation
        {1}
        {\InductionIndex}
        {\BoundaryComponents}
      (\DiskEmbedding^{1})
      ,
      \kAugmentation
        {1}
        {\InductionIndex+1}
        {\BoundaryComponents-1}
      (\bar{\DiskEmbedding}^{1})
    )
    \circ 
    h''
    \simeq
    h'
    \circ
    p'
    \circ
    h''
    \simeq
    \ast
  \]
  This implies that $h$ induces the zero morphism in homology in those degrees, 
  where 
  $
    \Suspension
    {
      \kAugmentation{1}{\InductionIndex-l}{\BoundaryComponents}(\DiskEmbedding^{1})
    }
  $
  induces a surjective map in homology. 
  Note that by replacing the occurring manifolds with corners by manifolds with 
  boundary,
  $
    \kAugmentation{1}{\InductionIndex-l}{\BoundaryComponents}(\DiskEmbedding^{1})
  $
  becomes equivalent to a map of type 
  $\StabilizationBeta{\InductionIndex-l}{\BoundaryComponents}$, hence is 
  surjective in 
  degrees $\apply{\BetaBound}{\InductionIndex-l}$.
  Hence $h\circ p$, which is homotopic to 
  $ 
    \MappingCone{\StabilizationAlpha{\InductionIndex-l}{\BoundaryComponents}}
    \to
    \MappingCone{\StabilizationAlpha{\InductionIndex}{\BoundaryComponents}}
  $%
  , the relative augmentation composition on mapping cones instead of mapping 
  pairs,
  is the null homomorphism in those degrees as well.
  Since this map is defined as a composition of relative 
  $0$-approximate augmentations (on cones not pairs), which are epimorphisms in 
  a range of degrees (depending on the genus occurring in the approximate 
  augmentations) 
  by $\AlphaAuxBound_{\InductionIndex}$ and $\BetaAuxBound_{\InductionIndex}$, 
  one has
  $
    \HomologyOfSpace
      {\ast}
      {\MappingPair{\StabilizationAlpha{\InductionIndex}{\BoundaryComponents}}}
  $
  vanishes in all degrees 
  $
    \ast 
    \leq 
    \min
    \{
      \apply{\AlphaAuxBound}{\InductionIndex-l+1}
      ,
      \apply{\BetaAuxBound}{\InductionIndex-l+1}
      ,
      \apply{\BetaBound}{\InductionIndex-l}
    \}
  $%
  . This finishes the proof as $\apply{\AlphaBound}{\InductionIndex}$ is by 
  definition smaller or equal than the left hand side.
\end{proof}
\section{Homological Stability for Capping off Boundary Components}
  \label{scn:CappingOff}
  All that is left to do to finish the proof of 
  Theorem~\ref{thm:HomologicalStability} is to prove the 
  homological stability statement for maps of type $\gamma$ and prove 
  that 
  $\StabilizationBeta{\Genus}{\BoundaryComponents}$ induces a 
  monomorphism in 
  integral homology provided that one 
  of the newly created boundary components is contractible and the tangential 
  structure on it extends to a disk.
  To do this we will first observe that homological stability for maps 
  of 
  type $\StabilizationGamma{\Genus}{\BoundaryComponents}$ where 
  $\BoundaryComponents>1$ is an easy corollary of homological 
  stability of 
  maps of type $\StabilizationBeta{\Genus}{\BoundaryComponents}$ and 
  then we 
  will establish everything needed to
  use Lemma \ref{lem:StabilityCriterion} again to relate the 
  homological 
  stability for maps of type $\StabilizationGamma{\Genus}{1}$ to maps 
  of type 
  $\StabilizationGamma{\Genus}{\BoundaryComponents}$ 
  for $\BoundaryComponents>1$.
  
  So consider a subsurface 
  $
    \BoundaryBordism
    \subset
    \DistBoundaryManifold{0}{\Manifold}
    \times 
    [0,1]
  $%
  , which defines a 
  map of type $\StabilizationBeta{\Genus}{\BoundaryComponents}$. 
  Suppose that one of the boundary components of the outgoing boundary 
  of the 
  pair of pants component of $\BoundaryBordism$ is contractible 
  in $\DistBoundaryManifold{0}{\Manifold}\times [0,1]$. 
  Fix a contraction and, using the main result of \cite{H61}, realize 
  it up to homotopy as an embedding 
  $
    (D^2,\BoundaryManifold{D^2}) 
    \to 
    (
      \DistBoundaryManifold{0}{\Manifold}
      \times 
      [0,1]
      ,
      \DistBoundaryManifold{0}{\Manifold}
      \times 
      \{1\}
    )
  $%
  .
    
  By flipping this embedded contraction (i.e. flipping the interval 
  direction) and equipping it with a tangential 
  structure extending the one on the boundary given by 
  $\BoundaryBordism$, 
  yields a bordism defining a map of type 
  $\StabilizationGamma{\Genus}{\BoundaryComponents}$ and by 
  construction 
  $
    \StabilizationGamma{\Genus}{\BoundaryComponents+1}
    \circ 
    \StabilizationBeta{\Genus}{\BoundaryComponents}
  $
  is a union with a cylinder and hence the composition is a homotopy 
  equivalence and induces an isomorphism in homology. This implies 
  in particular that $\StabilizationBeta{\Genus}{\BoundaryComponents}$ 
  induces a monomorphism in integral homology in all degrees.
  
  Similarly given a disk defining a map of type 
  $\StabilizationGamma{\Genus}{\BoundaryComponents}$ such that 
  there exists another component of $\BoundaryCondition$ in the same 
  connected component of $\DistBoundaryManifold{0}{\Manifold}$, we can 
  find a 
  pair of pants defining a map 
  $\StabilizationBeta{\Genus}{\BoundaryComponents-1}$ 
  such that 
  $
    \StabilizationGamma{\Genus}{\BoundaryComponents}
    \circ
    \StabilizationBeta{\Genus}{\BoundaryComponents-1}
  $ 
  is a homotopy equivalence.
  We have already shown that 
  $\StabilizationBeta{\Genus}{\BoundaryComponents}$ induces an 
  isomorphism in homology in degrees $ \leq 
  \apply{\BetaBound}{\Genus}-1$ and 
  an epimorphism in the next degree. 
  Moreover in this case 
  $\StabilizationBeta{\Genus}{\BoundaryComponents-1}$ 
  induces a monomorphism in all degrees from which we conclude that it 
  actually induces an isomorphism in all degrees $\leq 
  \apply{\BetaBound}{\Genus}$. 
  This implies that 
  $\StabilizationGamma{\Genus}{\BoundaryComponents}$, 
  as a left inverse for this map, 
  induces an isomorphism in these degrees as well. All in all this 
  gives:
  \begin{lemma}
\label{lem:CappingOffBoundary}
  Suppose $\Manifold$ fulfils the assumptions of Theorem~\ref{thm:HomologicalStability}:
  If one of the new boundary components of the bordism that defines a map 
  of type $\StabilizationBeta{\Genus}{\BoundaryComponents}$ is contractible in 
  $\DistBoundaryManifold{0}{\Manifold}$, then 
  $\StabilizationBeta{\Genus}{\BoundaryComponents}$ induces a monomorphism in 
  homology in all degrees .
  
  Similarly if there exists another boundary component of $\BoundaryCondition$ 
  in the same connected component of $\DistBoundaryManifold{0}{\Manifold}$, 
  where the disk component of the bordism defining 
  $\StabilizationGamma{\Genus}{\BoundaryComponents}$ lies, then 
  $\StabilizationGamma{\Genus}{\BoundaryComponents}$ induces an 
  epimorphism in homology in all degrees and furthermore an isomorphism in 
  degrees $\leq \apply{\BetaBound}{\Genus}+1$.
\end{lemma}

  Therefore we only have to concern ourselves with the case where 
  there is no 
  more connected component of $\BoundaryCondition$ in the same 
  connected 
  component as the disk component of the bordism defining 
  $\StabilizationGamma{\Genus}{\BoundaryComponents}$. 
  We will tackle this case using a certain resolution, which 
  lets us relate 
  $
    \MappingPair{\StabilizationGamma{\Genus}{\BoundaryComponents}}
  $
  to 
  $
    \MappingPair{\StabilizationGamma{\Genus}{\BoundaryComponents+\FaceIndex}}
  $%
  . This section is very similar to Section~\ref{scn:Resolution} 
  albeit it is much easier and does not require any inductive 
  arguments.
  
  Let $\DistBall$ denote a codimension $0$ ball that is in the same connected 
  component of 
  $\DistBoundaryManifold{0}{\Manifold}$ as the disk component of the 
  bordism 
  defining some fixed map of type 
  $\StabilizationGamma{\Genus}{\BoundaryComponents}$ and that is 
  disjoint 
  from $\BoundaryCondition$.
  \begin{definition}
  Fix an embedded subsurface 
  $
    \Subsurface
    \in
    \SpaceSubsurfaceBoundaryTangential
      {\TangentialFibration}
      {\Genus}
      {\BoundaryComponents}
      {\Manifold}
      {\BoundaryConditionTangential{\BoundaryCondition}}
  $%
  .
  We call an embedding 
  \[
    \BoundaryPath
    \colon
    ([0,1],\frac{1}{2})
    \to
    (\Manifold,\Subsurface)
  \]
  such that 
  $
    \apply{\BoundaryPath}{0}\in\DistBall
  $
  and 
  $
    \apply{\BoundaryPath}{1}
    \in
    \Manifold
    \setminus
    \BoundaryManifold{\Manifold}
  $
  a \introduce{boundary path of $\Subsurface$}. 
  If we add a tubular neighborhood of $\BoundaryPath$ in the pair $(\Manifold,\Subsurface)$ 
  denoted by 
  $
    \ThickenedBoundaryPath
    =
    (\ThickenedBoundaryPath_{\Manifold},\ThickenedBoundaryPath_{\Subsurface})
  $
  such that 
  $
    \ThickenedBoundaryPath
    \cap 
    \BoundaryManifold{\Manifold}
    \subset 
    \DistBall
  $%
  , then we will call the pair 
  $
    (
      \BoundaryPath
      ,
      \ThickenedBoundaryPath
    )
  $
  a \introduce{thickened boundary path of $\Subsurface$}. 

  If we only consider embeddings as above 
  $
    \BoundaryPath
    \colon
    ([0,1],0)
    \to
    (\Manifold,\BoundaryManifold{\Manifold})
  $
  without a particular subsurface, we will drop the $\Subsurface$ from the notation 
  i.e. we will call these just \introduce{boundary paths}.
  
  We call a boundary path together with a $2$-plane $\TwoPlane$ in the normal 
  bundle at  $\apply{\BoundaryPath}{\frac{1}{2}}$ a \introduce{thickened 
  boundary path}. A 
  \introduce{thickened boundary path with tangential structure} is a thickenend 
  boundary path with a tangential structure on 
  $\apply{\ThickenedBoundaryPath}{\TwoPlane}$. 
\end{definition}  

  With this notation at hand we can proceed to define another 
  resolution of 
  $
    \SpaceSubsurfaceBoundaryTangential
      {\TangentialFibration}
      {\Genus}
      {\BoundaryComponents}
      {\Manifold}
      {\BoundaryConditionTangential{\BoundaryCondition}}
  $%
  .
  \begin{definition}
  Let 
  $
    \BoundaryPathResolution
      {\TangentialFibration}
      {\Genus}
      {\BoundaryComponents}
      {\Manifold}
      {\BoundaryConditionTangential{\BoundaryCondition}}
      {\DistBall}
      {\bullet}
  $
  denote the semi-simplicial space, whose $\SemiSimplicialIndex$-simplices are tuples 
  $
    (
      \Subsurface
      ,
      (\BoundaryPath^{0},\ThickenedBoundaryPath^{0})
      ,
      \ldots
      ,
      (\BoundaryPath^{\SemiSimplicialIndex},\ThickenedBoundaryPath^{\SemiSimplicialIndex})
    )
  $%
  , such that:
  \begin{enumerate}[(i)]
  \item 
    $
      \Subsurface
      \in 
      \SpaceSubsurfaceBoundaryTangential
        {\TangentialFibration}
        {\Genus}
        {\BoundaryComponents}
        {\Manifold}
        {\BoundaryConditionTangential{\BoundaryCondition}}
    $
  \item 
    $(\BoundaryPath^{\FaceIndex},\ThickenedBoundaryPath^{\FaceIndex})$ is a thickened 
    boundary path of $\Subsurface$.
  \item 
    The neighbourhoods 
    $
      \ThickenedBoundaryPath^{0}
      ,
      \ldots 
      ,
      \ThickenedBoundaryPath^{\SemiSimplicialIndex}
    $
    are pairwise disjoint.
  \end{enumerate}
  The $\FaceIndex$-th face map forgets the $\FaceIndex$-th boundary path and we topologize 
  this as a subspace of 
  $
    \SpaceSubsurfaceBoundaryTangential
      {\TangentialFibration}
      {\Genus}
      {\BoundaryComponents}
      {\Manifold}
      {\BoundaryConditionTangential{\BoundaryCondition}}
    \times
    \ThickenedEmbeddingSpaceBoundaryCondition
      {[0,1]\times \SemiSimplicialSpaceIndex{\SemiSimplicialIndex}}
      {\Manifold}
      {q}
  $%
  . There is also an augmentation map 
  $\Augmentation{\bullet}$ to 
  $
    \SpaceSubsurfaceBoundaryTangential
      {\TangentialFibration}
      {\Genus}
      {\BoundaryComponents}
      {\Manifold}
      {\BoundaryConditionTangential{\BoundaryCondition}}
  $%
  , which forgets the boundary paths.
\end{definition}

  \begin{lemma}
  If $\Manifold$ is connected and has dimension at least $3$, then the 
  semi-simplicial space 
  $
    \BoundaryPathResolution
      {\TangentialFibration}
      {\Genus}
      {\BoundaryComponents}
      {\Manifold}
      {\BoundaryConditionTangential{\BoundaryCondition}}
      {\DistBall}
      {\bullet}
  $
  is a resolution of 
  $
    \SpaceSubsurfaceBoundaryTangential
      {\TangentialFibration}
      {\Genus}
      {\BoundaryComponents}
      {\Manifold}
      {\BoundaryConditionTangential{\BoundaryCondition}}
  $%
  .
\end{lemma}
\begin{proof}
  There is the following commutative square:
  \[
    \begin{tikzcd}
      \BoundaryPathResolution
        {\TangentialFibration}
        {\Genus}
        {\BoundaryComponents}
        {\Manifold}
        {\BoundaryConditionTangential{\BoundaryCondition}}
        {\DistBall}
        {\bullet}
        \ar[r]
        \ar[d] 
      &
      \SpaceSubsurfaceBoundaryTangential
        {\TangentialFibration}
        {\Genus}
        {\BoundaryComponents}
        {\Manifold}
        {\BoundaryConditionTangential{\BoundaryCondition}}
        \ar[d]
      \\
      \BoundaryPathResolution
        {+}
        {\Genus}
        {\BoundaryComponents}
        {\Manifold}
        {\BoundaryCondition}
        {\DistBall}
        {\bullet}
        \ar[r]
      &
      \SpaceSubsurfaceBoundaryTangential
        {+}
        {\Genus}
        {\BoundaryComponents}
        {\Manifold}
        {\BoundaryCondition}
    \end{tikzcd}
  \]
  By Lemma \ref{lem:LocallyRetractile} 
  $
    \SpaceSubsurfaceBoundaryTangential
      {+}
      {\Genus}
      {\BoundaryComponents}
      {\Manifold}
      {\BoundaryCondition}
  $
  is 
  $\DiffeomorphismGroupBoundary{\Manifold}$-locally retractile.
  The group $\DiffeomorphismGroupBoundary{\Manifold}$ acts on 
  $
    \BoundaryPathResolution
      {+}
      {\Genus}
      {\BoundaryComponents}
      {\Manifold}
      {\BoundaryCondition}
      {\DistBall}
      {\bullet}
  $
  and the lower map is equivariant, hence the bottom map is a Hurewicz-fibration. 
  Furthermore it is obvious that the square is a pullback square, from 
  which we conclude that the top map is also a Hurewicz-fibration. 
  
  It is easy to see that     
  $
    \BoundaryPathResolution
      {+}
      {\Genus}
      {\BoundaryComponents}
      {\Manifold}
      {\BoundaryCondition}
      {\DistBall}
      {\bullet}
  $
  is a topological flag complex. We want to apply Lemma~\ref{lem:FlagComplex} to show that it is a 
  resolution. Since the augmentation is a fiber bundle (and the fiber is non-empty as $\Manifold$ is 
  connected) this flag complex has local sections. 
  All that is left to show to apply the lemma is that for 
  $
    \{\BoundaryPath^{1},\ldots \BoundaryPath^{n}\}
    \subset 
    \Fiber{\Subsurface}{\Augmentation{0}}
  $%
  , there exists an $\BoundaryPath^{\infty}$ such that 
  $(\BoundaryPath^{\SemiSimplicialIndex},\BoundaryPath^{\infty})$ is a $1$-simplex for every 
  $\SemiSimplicialIndex$.
  Since $\Manifold$ is at least $3$-dimensional we conclude that 
  $
    \Manifold 
    \setminus 
    (
      \ThickenedBoundaryPath^{0}
      \cup
      \ldots
      \cup 
      \ThickenedBoundaryPath^{n}
    )
  $
  is still path-connected, hence we can find a boundary path in this complement.
\end{proof}

  \begin{definition}
  Let $\SpaceOfBoundaryPaths{\TangentialFibration}{\Manifold}{\DistBall}{\SemiSimplicialIndex}$ denote 
  the set of tuples 
  $(\ThickenedBoundaryPath^{0},\ldots,\ThickenedBoundaryPath^{\SemiSimplicialIndex})$, where 
  $\ThickenedBoundaryPath^{\FaceIndex}$ are thickened boundary paths with 
  tangential structure such that all $\ThickenedBoundaryPath^{\FaceIndex}$ are disjoint.
  We topologize this as a subspace of 
  $
    \ThickenedEmbeddingSpaceBoundaryCondition
      {[0,1]\times\SemiSimplicialSpaceIndex{\SemiSimplicialIndex}}
      {\Manifold}
      {q}
    \times 
    \EmbeddingSpace{D^2\times \SemiSimplicialSpaceIndex{\SemiSimplicialIndex}}{\Manifold}
  $.
  
  There is a map from 
  $
    \BoundaryPathResolution
      {\TangentialFibration}
      {\Genus}
      {\BoundaryComponents}
      {\Manifold}
      {\BoundaryConditionTangential{\BoundaryCondition}}
      {\DistBall}
      {\bullet}
  $
  to 
  $\SpaceOfBoundaryPaths{\TangentialFibration}{\Manifold}{\DistBall}{\SemiSimplicialIndex}$, which 
  forgets the surface, but keeps the tubular neighbourhood in the surface and the corresponding 
  tangential structure on it.
\end{definition}

  \begin{lemma}
\label{lem:FiberSpecificationBoundaryPath}
  For a point 
  $
    \BoundaryPath
    =
    (
      \BoundaryPath_{0}
      ,
      \ldots
      ,
      \BoundaryPath_{\SemiSimplicialIndex}
    )
    \in 
    \SpaceOfBoundaryPaths
      {\TangentialFibration}
      {\Manifold}
      {\DistBall}
      {\SemiSimplicialIndex}
  $
  there is a homotopy fiber sequence:
  \[
    \begin{tikzcd}
      \SpaceSubsurfaceBoundaryTangential
        {\TangentialFibration}
        {\Genus}
        {\BoundaryComponents+\SemiSimplicialIndex+1}
        {\CutOut{\Manifold}{\BoundaryPath}}
        {\CutOut
          {\BoundaryConditionTangential{\BoundaryCondition}}
          {\BoundaryPath}
        }
        \ar[r]
      &
      \BoundaryPathResolution
        {\TangentialFibration}
        {\Genus}
        {\BoundaryComponents}
        {\Manifold}
        {\BoundaryConditionTangential{\BoundaryCondition}}
        {\DistBall}
        {\SemiSimplicialIndex}
        \ar[r]
      &
      \SpaceOfBoundaryPaths
        {\TangentialFibration}
        {\Manifold}
        {\DistBall}
        {\SemiSimplicialIndex}
    \end{tikzcd}
  \]
   where $\CutOut{\Manifold}{\BoundaryPath}$ denotes 
   $
    \overline{
      \Manifold
      \setminus 
      \left(
        \bigcup_{\FaceIndex} 
        \ThickenedBoundaryPath_{\FaceIndex}
      \right)
    }
  $
  and 
  $
    \CutOut
      {\BoundaryConditionTangential{\BoundaryCondition}}
      {\ThickenedBoundaryPath}
  $
  denotes 
  $
    \BoundaryConditionTangential{\BoundaryCondition} 
    \cup 
    \left(
      \bigcup_{\FaceIndex} 
      \BoundaryManifold{\ThickenedBoundaryPath_{\FaceIndex}^{\text{Disk}}}
    \right)
  $%
  where $\ThickenedBoundaryPath_{\FaceIndex}^{\text{Disk}}$ denotes the 
  embedding with $\TangentialFibration$-structure of $\Ball{2}$ into 
  $\Manifold$ associated to the two-plane in the definition of thickened 
  boundary path with tangential structure.
\end{lemma}
The proof of this lemma is completely analogous to the proof of 
Lemma~\ref{lem:HomotopyFiberArcResolution} and is therefore omitted.

The issues here that arise from the occurring manifolds with corners 
are dealt with in the exact same fashion as in the discussion after 
Lemma~\ref{lem:HomotopyFiberArcResolution}.

  Similar to Section~\ref{scn:Resolution}, we want to understand how 
  this resolution behaves under stabilization maps. 
  Let us take a look at the following diagram, in which we want to constructed 
  the dashed arrow:
  \[
    \begin{tikzcd}
      \BoundaryPathResolution
        {\TangentialFibration}
        {\Genus}
        {\BoundaryComponents}
        {\Manifold}
        {\BoundaryConditionTangential{\BoundaryCondition}}
        {\DistBall}
        {\bullet}
        \ar[r,dashed]
        \ar[d]
      &
      \BoundaryPathResolution
        {\TangentialFibration}
        {\Genus}
        {\BoundaryComponents-1}
        {\Elongation{\Manifold}{1}}
        {\BoundaryConditionTangential{\bar{\BoundaryCondition}}}
        {\bar{\DistBall}}
        {\bullet}
        \ar[d]
      \\
      \SpaceSubsurfaceBoundaryTangential
        {\TangentialFibration}
        {\Genus}
        {\BoundaryComponents}
        {\Manifold}
        {\BoundaryConditionTangential{\BoundaryCondition}}
        \ar[r,"\StabilizationGamma{\Genus}{\BoundaryComponents}"]
      &
      \SpaceSubsurfaceBoundaryTangential
        {\TangentialFibration}
        {\Genus}
        {\BoundaryComponents-1}
        {\Elongation{\Manifold}{1}}
        {\BoundaryConditionTangential{\bar{\BoundaryCondition}}}
    \end{tikzcd}
  \]
  Let us denote the surface defining 
  $\StabilizationGamma{\Genus}{\BoundaryComponents}$ by 
  $\StabilizationBordism$. 
  Replace $\StabilizationBordism$ by a different but isotopic bordism if 
  necessary to ensure that 
  $
    (\DistBall\times [0,1])
    \cap 
    \StabilizationBordism
    =
    \emptyset
  $%
  . We will write $\bar{\DistBall}$ for $\DistBall\times [0,1]$.
  For a 
  $
    \BoundaryPath
    \in 
    \SpaceOfBoundaryPaths
      {\TangentialFibration}
      {\Manifold}
      {\DistBall}
      {\SemiSimplicialIndex}
  $
  we define 
  $
    \tilde{\BoundaryPath}^{\FaceIndex}
    =
    \apply{\BoundaryPath^{\FaceIndex}}{0}
    \times 
    [0,1]
  $ 
  and 
  $
    \bar{\BoundaryPath}^{\FaceIndex}
    =
    \BoundaryPath^{\FaceIndex}
    \cup 
    \tilde{\BoundaryPath}^{\FaceIndex}
  $%
  . The dashed arrow is now given by mapping $\Subsurface$ to 
  $
    \Subsurface
    \cup 
    \StabilizationBordism
  $ 
  and 
  $\BoundaryPath^{\FaceIndex}$ to 
  $\bar{\BoundaryPath}^{\FaceIndex}$. 
  These maps commute with the face and augmentation maps and so they define a 
  map of augmented semi-simplicial spaces, 
  which we denote by 
  $\StabilizationGamma{\Genus}{\BoundaryComponents}^{\bullet}$. 
  
  There is a map
  $
    \SpaceOfBoundaryPaths
      {\TangentialFibration}
      {\Manifold}
      {\DistBall}
      {\SemiSimplicialIndex}
    \to
    \SpaceOfBoundaryPaths
      {\TangentialFibration}
      {\Elongation{\Manifold}{1}}
      {\bar{\DistBall}}
      {\SemiSimplicialIndex}
  $
  defined by $\BoundaryPath \mapsto \bar{\BoundaryPath}$. 
  This map is a homotopy equivalence. 
  We have the following diagram: 
  \[
    \begin{tikzcd}
      \SpaceSubsurfaceBoundaryTangential
        {\TangentialFibration}
        {\Genus}
        {\BoundaryComponents+\SemiSimplicialIndex+1}
        {\CutOut{\Manifold}{\BoundaryPath}}
        {\BoundaryConditionTangential{\CutOut{\BoundaryCondition}{\BoundaryPath}}}
        \ar[r]
        \ar[d,"\StabilizationGamma{\Genus}{\BoundaryComponents+\SemiSimplicialIndex+1}"]
      &
      \BoundaryPathResolution
        {\TangentialFibration}
        {\Genus}
        {\BoundaryComponents}
        {\Manifold}
        {\BoundaryConditionTangential{\BoundaryCondition}}
        {\DistBall}
        {\bullet}
        \ar[d,"\StabilizationGamma{\Genus}{\BoundaryComponents}^{\SemiSimplicialIndex}"]
        \ar[r]
      &
      \SpaceOfBoundaryPaths
        {\TangentialFibration}
        {\Manifold}
        {\DistBall}
        {\SemiSimplicialIndex}
        \ar[d,"\simeq"]
      \\
      \SpaceSubsurfaceBoundaryTangential
        {\TangentialFibration}
        {\Genus}
        {\BoundaryComponents+\SemiSimplicialIndex}
        {\CutOut{\Elongation{\Manifold}{1}}{\bar{\BoundaryPath}}}
        {\BoundaryConditionTangential
          {\CutOut{\BoundaryCondition'}{\bar{\BoundaryPath}}}
        }
        \ar[r]
      &
      \BoundaryPathResolution
        {\TangentialFibration}
        {\Genus}
        {\BoundaryComponents-1}
        {\Elongation{\Manifold}{1}}
        {\BoundaryConditionTangential{\BoundaryCondition'}}
        {\DistBall}
        {\bullet}
        \ar[r]
      &
      \SpaceOfBoundaryPaths
        {\TangentialFibration}
        {\Elongation{\Manifold}{1}}
        {\DistBall}
        {\SemiSimplicialIndex}
    \end{tikzcd}
  \]
  The commutativity of the right square gives us a map between the 
  fibers and it is easy to see that the left vertical arrow is given by 
  $\StabilizationGamma{\Genus}{\BoundaryComponents+\SemiSimplicialIndex+1}$. 
  Putting this together yields:
  \begin{corollary}
\label{crl:BoundaryPathsResolutionFibers}
  The induced map between the homotopy fibers of
  \[
    (\StabilizationGamma{\Genus}{\BoundaryComponents}^{\SemiSimplicialIndex})
    \to
     \SpaceOfBoundaryPaths
       {\TangentialFibration}
       {\Manifold}
       {\DistBall}
       {\SemiSimplicialIndex}
  \]
  is given by 
  $\StabilizationGamma{\Genus}{\BoundaryComponents+\SemiSimplicialIndex+1}$.
\end{corollary}

  Finally with all these tools at hand we are able to finish the proof of 
  Theorem~\ref{thm:HomologicalStability} by concluding:
  \begin{proposition}
  Suppose $\Manifold$ fulfils the requirements of 
  Theorem~\ref{thm:HomologicalStability},
  then 
  \[
    \HomologyOfSpace
      {\ast}
      {\MappingPair{\StabilizationGamma{\Genus}{\BoundaryComponents}}}
    =
    0
  \]
  for
  $\ast \leq \apply{\BetaBound}{\Genus}$.
\end{proposition}
\begin{proof}
  We want to apply Lemma~\ref{lem:StabilityCriterion}. 
  The role of $B_{\SemiSimplicialIndex}$ will be played by 
  $
    \SpaceOfBoundaryPaths
      {\TangentialFibration}
      {\Manifold}
      {\DistBall}
      {\SemiSimplicialIndex}
  $
  and
  $
    \BoundaryPathResolution
      {\TangentialFibration}
      {\Genus}
      {\BoundaryComponents}
      {\Manifold}
      {\BoundaryConditionTangential{\BoundaryCondition}}
      {\DistBall}
      {\bullet}
  $
  will be the two resolutions,
  $\StabilizationGamma{\Genus}{\BoundaryComponents}^{\bullet}$ will be the map 
  between them.
  Corollary~\ref{crl:BoundaryPathsResolutionFibers} specified the occurring 
  homotopy fibers. 
  
  We choose $\DimensionIndex=\infty$ and  $c=\apply{\BetaBound}{\Genus}+1$.
  By Lemma~\ref{lem:CappingOffBoundary} and 
  \ref{crl:BoundaryPathsResolutionFibers} 
  we conclude that the homology of the reduced mapping cone of the map between 
  the fibers equals zero in the desired range i.e. 
  \[
    \HomologyOfSpace
      {\ast}
      {\MappingPair
        {\StabilizationGamma{\Genus}{\BoundaryComponents+\SemiSimplicialIndex+1}}
      }
    =
    0
  \]
  for $\ast \leq \apply{\BetaBound}{\Genus}$, which finishes the proof of the 
  proposition.
\end{proof}
\section{Tangential Structures that fulfil Homological Stability}
  \label{scn:TangentialkTriviality}
The goal of this section is to establish criteria for when a 
space of tangential structures satisfies $\TrivialityIndex$-triviality, which 
was the key requirement for a tangential structure to fulfil homological 
stability, and prove Proposition~\ref{prp:ConnectedKTriviality}, which will be 
fundamental in the proofs of these criteria.
This will yield more easily applicable versions
of Theorem~\ref{thm:HomologicalStability}.

\subsection{Stability of Connected Components implies 
  $\TrivialityIndex$-triviality}
  \begin{proof}[Proof of Proposition~\ref{prp:ConnectedKTriviality}]
  We will prove $\TrivialityIndex$-triviality for maps of type $\alpha$, the 
  case of maps of type $\beta$ is completely analogous.
  Let $\left(\DistIntervall{0}^{i},\DistIntervall{1}^{i}\right)$ together with 
  $\bar{\DistIntervall{0}}$ and $\bar{\DistIntervall{1}}$ denote a genus 
  maximizing sequence of length $2 \Genus + 1$ for a map of type $\alpha$ and 
  compatible $\DiskEmbedding^{i}$ as in the definition of 
  $\TrivialityIndex$-triviality. 
  Let $\StabilizationBordism$ denote a stabilization bordism such that 
  $
    \StabilizationBordism
    \cap 
    \left(
      \left(
        \bar{\DistIntervall{0}}\cup \bar{\DistIntervall{1}}
      \right)
      \times 
      [0,1]
    \right)
    =
    \left(
      \bar{\DistIntervall{0}}\cup \bar{\DistIntervall{1}}
    \right)
    \times 
    [0,1]
  $%
  .
  Furthermore let us denote $\bigcup_{i} \DiskEmbedding_{\LineBundle}^{i}$ by
  $\DiskEmbedding_{\LineBundle}$ and $\bigcup_{i} 
  \bar{\DiskEmbedding}_{\LineBundle}^{i}$ by 
  $\bar{\DiskEmbedding}_{\LineBundle}$ and $\bigcup_{i} \DiskEmbedding^{i}$ by 
  $\DiskEmbedding$ and analogously for $\bar{\DiskEmbedding}$. 
  We want to prove that $\DiskEmbedding_{\LineBundle}$ absorbs 
  $\StabilizationBordism$ or in other words there exists a 
  $
    \AuxBordism 
    \subset 
    \bar{\DiskEmbedding} \cup \DistBoundaryManifold{0}{\Manifold}\times [1,2]
  $
  such that 
  \[
    \begin{tikzcd}
      \SpaceAllSubsurfaceBoundaryTangential
        {\TangentialFibration}
        {\CutOut{\Manifold}{\DiskEmbedding}}
        {\CutOut
          {\BoundaryConditionTangential{\BoundaryCondition}}
          {\SubsetBordism}
        }
        \ar[r,"-\cup\CutOut{\StabilizationBordism}{\DiskEmbedding}"]
        \ar[d,"-\cup\DiskEmbedding"]
      &
      \SpaceAllSubsurfaceBoundaryTangential
        {\TangentialFibration}
        {\Elongation{\CutOut{\Manifold}{\bar{\DiskEmbedding}}}{1}}
        {\CutOut
          {\BoundaryConditionTangential{\BoundaryCondition'}}
          {\bar{\DiskEmbedding}}
        }
        \ar[d,"-\cup\Elongation{\bar{\DiskEmbedding}}{1}"]
        \ar[ldd,"-\cup\AuxBordism"' near start,dashed]
      \\
      \SpaceAllSubsurfaceBoundaryTangential
        {\TangentialFibration}
        {\Manifold}
        {\BoundaryConditionTangential{\BoundaryCondition}}
        \ar[r,"-\cup\StabilizationBordism" near start, crossing over]
        \ar[d,"\simeq"]
      &
      \SpaceAllSubsurfaceBoundaryTangential
        {\TangentialFibration}
        {\Elongation{\Manifold}{1}}
        {\BoundaryConditionTangential{\BoundaryCondition'}}
        \ar[d,"\simeq"]
      \\
      \SpaceAllSubsurfaceBoundaryTangential
        {\TangentialFibration}
        {\Elongation{\Manifold}{2}}
        {\BoundaryConditionTangential{\BoundaryCondition}}
        \ar[r,"-\cup(\StabilizationBordism+2)"]
      &
      \SpaceAllSubsurfaceBoundaryTangential
        {\TangentialFibration}
        {\Elongation{\Manifold}{3}}
        {\BoundaryConditionTangential{\BoundaryCondition'}}
    \end{tikzcd}
  \]
  commutes up to homotopy. 
  
  Let us denote 
  $
    \overline{
      \BoundaryManifold{\DiskEmbedding_{\LineBundle}}
      \setminus 
      \BoundaryCondition
    }
  $
  by $\DistBoundaryManifold{0}{\DiskEmbedding_{\LineBundle}}$ and 
  $
      \overline{
        \BoundaryManifold{\overline{\DiskEmbedding}_{\LineBundle}}
        \setminus 
        \BoundaryCondition'
      }
  $
  by $\DistBoundaryManifold{0}{\overline{\DiskEmbedding}_{\LineBundle}}$. We 
  define 
  \begin{align}
    \BoundaryConditionTangential{\BoundaryCondition}_{\DiskEmbedding}
    &
    \coloneqq
    \DistBoundaryManifold{0}{\DiskEmbedding_{\LineBundle}}
    \cup
    \left(
      \BoundaryConditionTangential{\BoundaryCondition}
      \setminus
      \DiskEmbedding
    \right)
    \\
    \BoundaryConditionTangential{\BoundaryCondition'}_{\overline{\DiskEmbedding}}
    &
    \coloneqq
    \DistBoundaryManifold{0}{\overline{\DiskEmbedding}_{\LineBundle}}
    \cup
    \left(
      \BoundaryConditionTangential{\BoundaryCondition'}
      \setminus
      \overline{\DiskEmbedding}
    \right)
  \end{align}
  We have the following commutative diagram:
  \[
    \begin{tikzcd}
      \SpaceSubsurfaceBoundaryTangential
        {\TangentialFibration}
        {\Genus-1}
        {4}
        {
          \bar{\DiskEmbedding}
          \cup 
          \DistBoundaryManifold{0}{\Manifold}
          \times 
          [1,2]
        }
        {
          \BoundaryConditionTangential
            {\BoundaryCondition'}
          _{\overline{\DiskEmbedding}}
        }
        \ar[r,"-\cup \CutOut{\StabilizationBordism}{\bar{\DiskEmbedding}}"]
        \ar[d,"-\cup \StabilizationBordism+2"]
      &
      \SpaceSubsurfaceBoundaryTangential
      {\TangentialFibration}
      {\Genus}
      {3}
      {
        \DiskEmbedding 
        \cup 
        \DistBoundaryManifold{0}{\Manifold}
        \times 
        [0,2]
      }
      {\BoundaryConditionTangential{\BoundaryCondition}_{\DiskEmbedding}}
      \ar[d,"-\cup \StabilizationBordism+2"]
      \\
      \SpaceSubsurfaceBoundaryTangential
        {\TangentialFibration}
        {\Genus}
        {3}
        {
          \bar{\DiskEmbedding}
          \cup 
          \DistBoundaryManifold{0}{\Manifold}
          \times 
          [1,3]
        }
        {
          \BoundaryConditionTangential
            {\BoundaryCondition'}
          _{\overline{\DiskEmbedding}}
        }
        \ar[r,"-\cup \CutOut{\StabilizationBordism}{\bar{\DiskEmbedding}}"]
      &
      \SpaceSubsurfaceBoundaryTangential
        {\TangentialFibration}
        {\Genus+1}
        {2}
        {
          \DiskEmbedding 
          \cup 
          \DistBoundaryManifold{0}{\Manifold}
          \times 
          [0,3]
        }
        {\BoundaryConditionTangential{\BoundaryCondition}_{\DiskEmbedding}}
    \end{tikzcd}
  \]
  Note that the occuring spaces in this diagram are manifolds with corners, 
  this can be fixed by the usual smoothing the angle techniques explained in 
  Remark~\ref{rmk:SmoothingTheAngleRemovingThings} and 
  Remark~\ref{rmk:SmoothingTheAngleRemovingThings}, furthermore 
  note that in this diagram 
  $
    -
    \cup 
    \CutOut{\StabilizationBordism}{\DiskEmbedding}
  $ 
  is a complicated subset bordism and not a stabilization bordism. 
  
  Note that the non-cylcindrical parts of the occuring bordisms in the first 
  diagram occur as elements of the spaces in this second diagram. 
  For example the bordism for the composition of the two arrows on the right of 
  the first diagram is given by 
  $
    \bar{\DiskEmbedding}_{\LineBundle} 
    \cup 
    \BoundaryConditionTangential{\BoundaryCondition'}
    \times
    [1,3]
  $ which is an element of the lower left space. Analogously the bordism of the 
  composition on the left hand side 
  $
    \DiskEmbedding_{\LineBundle}
    \cup
    \BoundaryConditionTangential{\BoundaryCondition}
    \times 
    [0,2]
  $ 
  is an element of the top right space. 
  That the genus and the number of boundary conditions is as stated follows 
  from the requirement that the sequence of intervals is genus maximizing and 
  that $\StabilizationBordism$ contains 
  $
    \left(
      \bar{\DistIntervall{0}}
      \cup 
      \bar{\DistIntervall{1}}
    \right)
    \times 
    [0,1]
  $%
  .
  
  $\AuxBordism$ will be an element of the top left space and the 
  compositions with $-\cup \AuxBordism$ in the first diagram correspond to 
  composition in the second diagram.
  Furthermore note that 
  the left map is $\pi_0$-surjective and the right hand is $\pi_0$-injective. 
  Hence chose an element $\AuxBordism$ in 
  $
  \SpaceSubsurfaceBoundaryTangential
    {\TangentialFibration}
    {\Genus-1}
    {4}
    {
      \bar{\DiskEmbedding}
      \cup 
      \DistBoundaryManifold{0}{\Manifold}
      \times 
      [1,2]
    }
    {a}
  $
  such that $\AuxBordism\cup \StabilizationBordism + 2$ lies in the same 
  connected component as 
  $
    \bar{\DiskEmbedding_{\LineBundle}} 
    \cup
    \BoundaryConditionTangential{\BoundaryCondition'}
    \times
    [1,3]
  $%
  .
  Then the lower triangle in the first diagram commutes up to homotopy by 
  construction. 
  
  Since the second diagram is commutative the image of  
  $
    \AuxBordism 
    \cup 
    \CutOut{\StabilizationBordism}{\bar{\DiskEmbedding}}
  $
  under the right map is isotopic to 
  $
    \bar{\DiskEmbedding_{\LineBundle}} 
    \cup
    \BoundaryConditionTangential{\BoundaryCondition'}
    \times
    [1,3]
    \cup
    \CutOut{\StabilizationBordism}{\DiskEmbedding}
  $
  which in turn is isotopic to
  $
    \DiskEmbedding_{\LineBundle} 
    \cup
    \BoundaryConditionTangential{\BoundaryCondition}
    \times
    [0,2]
    \cup
    \StabilizationBordism + 2
  $%
  , since this is the image of 
  $
    \DiskEmbedding_{\LineBundle}
    \cup
    \BoundaryConditionTangential{\BoundaryCondition}
    \times
    [0,2]
  $
  under the right hand map, which is $\pi_0$-injective, we conclude that the 
  upper triangle in the first diagram commutes as well.
\end{proof}
  As a corollary to Proposition~\ref{prp:ConnectedKTriviality} and 
  Theorem~\ref{thm:HomologicalStability} one obtains:
  \begin{theorem}
\label{thm:HomologicalStabilityConnectedComponents}
  Suppose $\Manifold$ is an at least $5$-dimensional, simply-connected 
  manifold with non-empty boundary and 
  $\DistBoundaryManifold{0}{\Manifold}$ is a codimension $0$ simply-connected 
  submanifold of $\BoundaryManifold{\Manifold}$ with a space of 
  $\TangentialFibration$-structures of subplanes of 
  $\TangentBundle{\Manifold}$. 
  Suppose further that $\TangentialFibration$ 
  $\pi_0$-stabilizes at $\PiZeroStabilization_{1}$ and $\pi_0$-stabilizes 
  at the boundary at $\PiZeroStabilization_{2}$, then
  \begin{enumerate}[(i)]
  \item
    The homology of 
    $\MappingPair{\StabilizationAlpha{\Genus}{\BoundaryComponents}}$ 
    vanishes in degrees 
    $
      \ast
      \leq     
      \apply{\AlphaBound}{\Genus}
    $
  \item
    The homology of 
    $\MappingPair{\StabilizationBeta{\Genus}{\BoundaryComponents}}$ 
    vanishes in degrees 
    $
      \ast
      \leq 
      \apply{\BetaBound}{\Genus}
    $
  \item
    The homology of 
    $\MappingPair{\StabilizationGamma{\Genus}{\BoundaryComponents}}$ 
    vanishes in degrees 
    $
      \ast
      \leq 
      \apply{\BetaBound}{\Genus}+1
    $
  \end{enumerate}
  Where $\TrivialityIndex$ in the definition of $\AlphaBound$ and $\BetaBound$ 
  (Definition~\ref{dfn:StabilityBounds}) equals $2\PiZeroStabilization_{2}+1$.
\end{theorem}
\subsection{Criterions for Stabilization of Connected Components}
  The goal of this subsection is to prove the following theorem:
  \begin{theorem}
\label{thm:HomologicalStabilityFiberSimplyConnected}
  Suppose that $\Manifold$ is a simply connected manifold of 
  dimension at least 
  $5$ such that $\DistBoundaryManifold{0}{\Manifold}$ is simply-connected as 
  well and 
  $
    \TangentialFibration
    \colon
    \TangentialSpace{\Manifold}
    \to
    \Grassmannian{2}{\TangentBundle{\Manifold}}
  $
  is a space of $\TangentialFibration$-structures of 
  subplanes of $\TangentBundle{\Manifold}$ where
  $\Fiber{\TwoPlane}{\TangentialFibration}$ is simply-connected, then
  \begin{enumerate}[(i)]
      \item
        The homology of 
        $\MappingPair{\StabilizationAlpha{\Genus}{\BoundaryComponents}}$ 
        vanishes in degrees 
        $
          \ast
          \leq     
          \left\lfloor
            \frac{1}{3}(2\Genus + 1)
          \right\rfloor
        $
      \item
        The homology of 
        $\MappingPair{\StabilizationBeta{\Genus}{\BoundaryComponents}}$ 
        vanishes in degrees 
        $
          \ast
          \leq 
          \left\lfloor
            \frac{2}{3}\Genus
          \right\rfloor
        $
      \item
        The homology of 
        $\MappingPair{\StabilizationGamma{\Genus}{\BoundaryComponents}}$ 
        vanishes in degrees 
        $
          \ast
          \leq 
          \left\lfloor
            \frac{2}{3}\Genus+1
          \right\rfloor
        $
  \end{enumerate}
\end{theorem}
  The proof will use Theorem~\ref{thm:HomologicalStabilityConnectedComponents} 
  and Remark~\ref{rmk:StabilityBounds}.
  
  As was explained in the introduction, examples of spaces of 
  $\TangentialFibration$-structures with simply-connected 
  fiber are given by $k$ framings of the normal bundle of the subsurfaces, 
  where $k\leq\DimensionIndex-2$.
  
  In order to apply Theorem~\ref{thm:HomologicalStabilityConnectedComponents}, 
  we need a better understanding of the space of connected components of 
  $
    \SpaceSubsurfaceBoundaryTangential
      {\TangentialFibration}
      {\Genus}
      {\BoundaryComponents}
      {\Manifold}
      {\BoundaryConditionTangential{\BoundaryCondition}}
  $%
  if the fiber of the space of tangential structures is simply-connected.
  \begin{lemma}
  Suppose that $\Manifold$ is simply-connected and of dimension at least $5$, 
  and that the fiber
  $
    \Fiber{\TwoPlane}{\TangentialFibration}
  $
  is simply connected, then 
  $
    \HomotopyGroup
      {0}
      {\EmbeddingSpaceTangential{\TangentialFibration}{\TwoSphere}{\Manifold}}
  $
  carries a canonical group structure and there is an action of this group on
  $
    \HomotopyGroup
      {0}
      {\SpaceSubsurfaceBoundaryTangential
        {\TangentialFibration}
        {\Genus}
        {\BoundaryComponents}
        {\Manifold}
        {\BoundaryConditionTangential{\BoundaryCondition}
      }
    }
  $%
  , given by connected sum, such that it is an affine
  $
    \HomotopyGroup
      {0}
      {\EmbeddingSpaceTangential{\TangentialFibration}{\TwoSphere}{\Manifold}}
  $%
  -set.
\end{lemma}
\begin{proof}
  The proof is divided into three steps. We will first construct the group 
  structure on 
  $
    \HomotopyGroup
      {0}
      {\EmbeddingSpaceTangential{\TangentialFibration}{\TwoSphere}{\Manifold}}
  $
  by relating this space to the second homotopy groups of 
  $\TangentialSpace{\Manifold}$ and 
  $\Grassmannian{2}{\TangentBundle{\Manifold}}$, then we will relate this group 
  to 
  $
    \HomotopyGroup
      {0}
      {\EmbeddingSpaceBoundaryConditionTangential
        {\TangentialFibration}
        {\SurfaceGB{\Genus}{\BoundaryComponents}}
        {\Manifold}
        {\BoundaryConditionTangential{\BoundaryCondition}}
      }
  $
  using the simply-connectedness of $\Manifold$. Lastly, we will see that the 
  identifications we proved along the way behave very well with respect to the 
  action of the diffeomorphism group of the surface to get an understanding of
  $
    \HomotopyGroup
      {0}
      {\SpaceSubsurfaceBoundaryTangential
        {\TangentialFibration}
        {\Genus}
        {\BoundaryComponents}
        {\Manifold}
        {\BoundaryConditionTangential{\BoundaryCondition}
        }
      }
  $%
  . 
  
  For the first step, consider the following tower of fibrations:
  \[
    \begin{tikzcd}
      &
      \TangentialSpace{\Manifold}
        \ar[d,"\TangentialFibration"]
      \\
      &
      \Grassmannian{2}{\TangentBundle{\Manifold}}
        \ar[d]
      \\
      \SurfaceGB{\Genus}{\BoundaryComponents}
        \ar[r,"\Embedding"]
        \ar[ru,"\GrassmannianDifferential{\Embedding}"]
      &
      \Manifold
    \end{tikzcd}
  \]
  Let us calculate
  $
    \HomotopyGroup
      {2}
      {\Grassmannian{2}{\TangentBundle{\Manifold}}}
  $%
  . Consider the following diagram, which stems from the fact that the fibration
  $\Grassmannian{2}{\TangentBundle{\Manifold}}\to \Manifold$ 
  is the pullback of $\Grassmannian{2}{\TautologicalBundle}$, where 
  $\TautologicalBundle$ denotes the tautological bundle of some
  $\Grassmannian{\DimensionIndex}{\Reals^{\BigDimensionIndex}}$
  \[
    \begin{tikzcd}[cramped, column sep=small]
      \ldots
        \ar[r] 
      &
      \HomotopyGroup{3}{\Manifold}
        \ar[d]
        \ar[r] 
      &
      \HomotopyGroup{2}{\Grassmannian{2}{\Reals^{\DimensionIndex}}}
        \ar[d,"\Identity"]
        \ar[r]
      &
      \HomotopyGroup{2}{\Grassmannian{2}{\TangentBundle{\Manifold}}}
        \ar[d] 
        \ar[r]
      &
      \HomotopyGroup{2}{\Manifold}
        \ar[d]
        \ar[r] 
      &
      0
      \\
      \ldots
        \ar[r] 
      &
      0=\HomotopyGroup
        {3}
        {\Grassmannian{\DimensionIndex}{\Reals^{\BigDimensionIndex}}}
        \ar[r]
      & 
      \HomotopyGroup{2}{\Grassmannian{2}{\Reals^{\DimensionIndex}}}
        \ar[r]
      &
      \HomotopyGroup{2}{\Grassmannian{2}{\TautologicalBundle}}
        \ar[r]
      &
      \HomotopyGroup
        {2}
        {\Grassmannian{\DimensionIndex}{\Reals^{\BigDimensionIndex}}}
        \ar[r]
      &
      0 
    \end{tikzcd}
  \]
  Hence the upper long exact sequence breaks down to the following:
  \[
    \begin{tikzcd}
      0
        \ar[r]
      &
      \Integers
      \cong
      \HomotopyGroup{2}{\Grassmannian{2}{\Reals^{\DimensionIndex}}}
        \ar[r]
      &
      \HomotopyGroup{2}{\Grassmannian{2}{\TangentBundle{\Manifold}}}
        \ar[r]
      &
      \HomotopyGroup{2}{\Manifold}
        \ar[r]
        \ar[l,"\GrassmannianDifferential{-}"',bend right]
      &
      0 
    \end{tikzcd}
  \]
  Since $\Manifold$ is simply-connected, we have an isomorphism
  $
    \HomotopyGroup
      {2}
      {\Manifold}
    \simeq 
    \HomotopyGroup
      {0}
      {\MappingSpace{\TwoSphere}{\Manifold}}
  $ and this group is isomorphic to
  $\HomotopyGroup{0}{\EmbeddingSpace{\TwoSphere}{\Manifold}}$ by the main 
  result of \cite{H61}. Therefore the split in the aforementioned diagram is 
  given by taking an embedding as a representative and then applying 
  $\GrassmannianDifferential{-}$.
  From here on forth we consider 
  $\HomotopyGroup{0}{\EmbeddingSpace{\TwoSphere}{\Manifold}}$ as a subgroup of 
  $\HomotopyGroup{2}{\Grassmannian{2}{\TangentBundle{\Manifold}}}$.
  
  Consider the the following long exact sequence:
  \[
    \begin{tikzcd}
      \ldots
        \ar[r]
      &
      \HomotopyGroup{2}{\Fiber{\TwoPlane}{\TangentialFibration}}
        \ar[r]
      &
      \HomotopyGroup{2}{\TangentialSpace{\Manifold}}
        \ar[r,"\HomotopyGroupMap{\TangentialFibration}"]
      &
      \HomotopyGroup{2}{\Grassmannian{2}{\TangentBundle{\Manifold}}}
        \ar[r]
      &
      0
    \end{tikzcd}
  \]
  We claim that 
  $
    \apply
      {\HomotopyGroupMap{\TangentialFibration}^{-1}}
      {\HomotopyGroup{0}{\EmbeddingSpace{\TwoSphere}{\Manifold}}}
  $
  is isomorphic to 
  $
    \HomotopyGroup
      {0}
      {\EmbeddingSpaceTangential{\TangentialFibration}{\TwoSphere}{\Manifold}}
  $%
  .
  There is a canonical map 
  $
    \HomotopyGroup
      {0}
      {\EmbeddingSpaceTangential{\TangentialFibration}{\TwoSphere}{\Manifold}}
    \to
    \apply
      {\HomotopyGroupMap{\TangentialFibration}^{-1}}
      {\HomotopyGroup{0}{\EmbeddingSpace{\TwoSphere}{\Manifold}}}
  $%
  , which is evidently surjective. For injectivity of this map suppose two 
  elements 
  $
    \left[\Embedding\right]
    ,
    \left[\Embedding'\right]
    \in
    \HomotopyGroup
      {0}
      {\EmbeddingSpaceTangential{\TangentialFibration}{\TwoSphere}{\Manifold}}
  $
  map to the same element in
  $
    \apply
      {\HomotopyGroupMap{\TangentialFibration}^{-1}}
      {\HomotopyGroup{0}{\EmbeddingSpace{\TwoSphere}{\Manifold}}}
  $%
  . Hence there is a homotopy $\Homotopy$ from  
  $\TangentialStructure{\Embedding}$ to $\TangentialStructure{\Embedding'}$ as 
  maps into $\TangentialSpace{\Manifold}$.
  The main theorem of \cite{H61} implies that 
  $\TangentialFibration\circ \Homotopy$ is actually homotopic via 
  $\bar{\Homotopy}$ relative to $\Embedding$ and $\Embedding'$ to an isotopy. 
  We have the following commutative diagram, where the dashed arrows exists 
  since $\TangentialFibration$ is a fibration:
  \[
    \begin{tikzcd}
      \TwoSphere
      \times
      \left[0,1\right]
      \times 
      \left\{0\right\}
      \cup
      \TwoSphere
      \times 
      \left\{0,1\right\}
      \times 
      \left[0,1\right] 
        \ar[d,hookrightarrow] 
        \ar[rr,"\Homotopy \cup \Embedding \cup \Embedding'"] 
      &
      &
      \TangentialSpace{\Manifold}
        \ar[d]
      \\
      \TwoSphere
      \times
      \left[0,1\right]
      \times
      \left[0,1\right]
        \ar[urr,dashed] 
        \ar[rr,"\bar{\Homotopy}"] 
      &
      &
      \Grassmannian{2}{\TangentBundle{\Manifold}}
    \end{tikzcd}
  \]
  The indicated lift of $\bar{\Homotopy}$, when evaluated at 
  $\left(-,1\right)$, gives us the desired isotopy from 
  $\left[\Embedding\right]$ to $\left[\Embedding'\right]$ as embeddings with 
  $\TangentialFibration$-structure. 
  All in all this implies that 
  $
    \HomotopyGroup
      {0}
      {\EmbeddingSpaceTangential{\TangentialFibration}{\TwoSphere}{\Manifold}}
  $
  has a group structure, which stems from the isomorphism to 
  $
    \apply
      {\HomotopyGroupMap{\TangentialFibration}^{-1}}
      {\HomotopyGroup{2}{\Grassmannian{2}{\TangentBundle{\Manifold}}}}
  $%
  . For notational reasons, let us denote 
  $
    \HomotopyGroup
      {0}
      {\EmbeddingSpaceTangential{\TangentialFibration}{\TwoSphere}{\Manifold}}
  $
  by $\pi^{\TangentialFibration}$.
  
  For the second step, consider the cofibration 
  $
    \SurfaceGBOneSkeleton{\Genus}{\BoundaryComponents}
    \to
    \SurfaceGB{\Genus}{\BoundaryComponents}
  $%
  , where 
  $\SurfaceGBOneSkeleton{\Genus}{\BoundaryComponents}$ denotes a small tubular 
  neighbourhood of the $1$-skeleton of 
  $\SurfaceGB{\Genus}{\BoundaryComponents}$ to which only a single $2$-cell 
  is attached via $\sigma$.
  Since the inclusion $\SurfaceGBOneSkeleton{\Genus}{\BoundaryComponents}\to 
  \SurfaceGB{\Genus}{\BoundaryComponents}$ is a cofibration, 
  $
    \MappingSpace
      {\SurfaceGB{\Genus}{\BoundaryComponents}}
      {\Manifold}
    \to
    \MappingSpace
      {\SurfaceGBOneSkeleton{\Genus}{\BoundaryComponents}}
      {\Manifold}
  $
  is a fibration, with fiber over an embedding $\Embedding$ given by
  $
    \MappingSpace
      {\left(\Ball{2},\partial \Ball{2}\right)}
      {\left(\Manifold,\CellBoundaryCondition\right)}
  $%
  , where $\CellBoundaryCondition$ is the image of the boundary of the single 
  two cell, denoted by $\Ball{2}$, glued to
  $
    \SurfaceGBOneSkeleton{\Genus}{\BoundaryComponents}
  $%
  under $\Embedding$.
  Again by \cite{H61} and transversality, 
  $
    \HomotopyGroup{0}
      {\MappingSpace
        {\left(\Ball{2},\partial \Ball{2}\right)}
        {\left(\Manifold,\CellBoundaryCondition\right)}
      }
  $
  is isomorphic to 
  $
    \HomotopyGroup
      {0}
      {\EmbeddingSpaceBoundaryCondition
        {\Ball{2}}
        {\Manifold
          \setminus
          \apply
            {\Embedding}
            {\SurfaceGBOneSkeleton{\Genus}{\BoundaryComponents}}
        }
        {\CellBoundaryCondition}
      }
  $%
  .
  Similar considerations as before show that the a priori pointed set
  $
    \HomotopyGroup
      {0}
      {\EmbeddingSpaceBoundaryCondition
        {\Ball{2}}
        {\Manifold
          \setminus
          \apply
          {\Embedding}
          {\SurfaceGBOneSkeleton{\Genus}{\BoundaryComponents}}
        }
        {\CellBoundaryCondition}
      }
  $ 
  can be considered as a subgroup of 
  $
    \HomotopyGroup
      {2}
      {\Grassmannian{2}{\TangentBundle{\Manifold}},\CellBoundaryCondition}
  $%
  , where we consider $\CellBoundaryCondition$ (or its image to be more 
  precise) 
  as a subset of $\Grassmannian{2}{\TangentBundle{\Manifold}}$ via the 
  Grassmannian differential.
  
  Since $\Fiber{\TwoPlane}{\TangentialFibration}$ is simply-connected, there 
  exists an up to homotopy unique $\TangentialFibration$-structure on 
  $\CellBoundaryCondition$, which we 
  denote by $\BoundaryConditionTangential{\CellBoundaryCondition}$. 
  
  Let us introduce the shorthand notation 
  $\pi^{\TangentialFibration}_{\text{rel}}$ for
  $
    \apply
      {\HomotopyGroupMap{\TangentialFibration}^{-1}}
      {\HomotopyGroup
        {0}
        {\EmbeddingSpaceBoundaryCondition
          {\Ball{2}}
          {\Manifold
            \setminus
            \apply
            {\Embedding}
            {\SurfaceGBOneSkeleton{\Genus}{\BoundaryComponents}}
          }
          {\CellBoundaryCondition}
        }
      }
  $%
  , here
  $
    \TangentialFibration
    \colon
    \left(
      \TangentialSpace{\Manifold},
      \BoundaryConditionTangential{\CellBoundaryCondition}
    \right)
    \to
    \left(
      \Grassmannian{2}{\TangentBundle{\Manifold}},
      \CellBoundaryCondition
    \right)
  $%
  is considered as a map of pairs. We claim that 
  $
    \pi^{\TangentialFibration}
  $
  considered as a subgroup of $\HomotopyGroup{2}{\TangentialSpace{\Manifold}}$
  acts transitively and freely on
  $
    \pi^{\TangentialFibration}_{\text{rel}}
  $%
  .
  In order to prove this claim, consider the long exact sequence of the pair 
  $
    \left(
      \TangentialSpace{\Manifold}
      ,
      \BoundaryConditionTangential{\CellBoundaryCondition}
    \right)
  $%
  :
  \[
    \begin{tikzcd}
      0
        \ar[r]
      &
      \HomotopyGroup{2}{\TangentialSpace{\Manifold}}
        \ar[r] 
      &
        \HomotopyGroup
          {2}
          {\TangentialSpace{\Manifold},
            \BoundaryConditionTangential{\CellBoundaryCondition}
          }
          \ar[r,"\partial"]
        &
        \HomotopyGroup{1}{\BoundaryConditionTangential{\CellBoundaryCondition}}
    \end{tikzcd}
  \]
  That the action is free follows immediately from this long exact sequence.
  To prove that this action is transitive consider the following diagram given 
  by considering long exact sequences of pairs and the corresponding maps 
  between them:
  \[
    \begin{tikzcd}
      &
      \pi^{\TangentialFibration}
        \ar[d,phantom,"\subsetv"]
        \ar[r]
      &
      \pi^{\TangentialFibration}_{\text{rel}}
      \subset
      \apply
        {\partial^{-1}}
        {\left[\BoundaryConditionTangential{\CellBoundaryCondition}\right]}
        \ar[d,phantom,"\subsetv"]
      \\
      0 
        \ar[r] 
      &
      \HomotopyGroup{2}{\TangentialSpace{\Manifold}}
        \ar[r]
        \ar[d,"\HomotopyGroupMap{\TangentialFibration}"]
      &
      \HomotopyGroup
        {2}
        {\TangentialSpace{\Manifold},
          \BoundaryConditionTangential{\CellBoundaryCondition}
        }
        \ar[r]
        \ar[d,"\HomotopyGroupMap{\TangentialFibration}"] 
      &
      \HomotopyGroup{1}{\BoundaryConditionTangential{\CellBoundaryCondition}}
        \ar[d,"\cong"]
        \ar[r]
      &
      0
      \\
      0
        \ar[r]
      &
      \HomotopyGroup{2}{\Grassmannian{2}{\TangentBundle{\Manifold}}}
        \ar[r]
        \ar[d,
          "\HomotopyGroupMap
            {\Projection{\Grassmannian{2}{\TangentBundle{\Manifold}}}}
        "] 
      & 
      \HomotopyGroup
        {2}
        {\Grassmannian{2}{\TangentBundle{\Manifold}},\CellBoundaryCondition}
        \ar[r]
        \ar[d,
          "\HomotopyGroupMap
            {\Projection{\Grassmannian{2}{\TangentBundle{\Manifold}}}}
        "]
      &
      \HomotopyGroup{1}{\CellBoundaryCondition}
        \ar[d,"\cong"]
        \ar[r]
      &
      0
      \\
      0
        \ar[r]
      & 
      \HomotopyGroup{2}{\Manifold}
        \ar[r]
        \ar[u,"\GrassmannianDifferential{-}",bend left]
      & 
      \HomotopyGroup{2}{\Manifold,\CellBoundaryCondition}
        \ar[r]
        \ar[u,"\GrassmannianDifferential{-}",bend left]
      &
      \HomotopyGroup{1}{\CellBoundaryCondition}
        \ar[r]
      &
      0
    \end{tikzcd}
  \]
  The commutativity of the lower left square gives us that 
  $\HomotopyGroup{0}{\EmbeddingSpace{\TwoSphere}{\Manifold}}$ acts 
  freely and transitively on 
  $
  \HomotopyGroup
    {0}
    {\EmbeddingSpaceBoundaryCondition
      {\Ball{2}}
      {\Manifold
        \setminus
        \apply
        {\Embedding}
        {\SurfaceGBOneSkeleton{\Genus}{\BoundaryComponents}}
      }
      {\CellBoundaryCondition}
    }
  $%
  and taking preimages with respect to $\TangentialFibration$ does not change 
  the transitivity of the action.
  
  Furthermore the same arguments as for $\pi^{\TangentialFibration}$ before 
  show that 
  $
    \pi^{\TangentialFibration}_{\text{rel}}
  $
  is isomorphic to the pointed set
  $
    \HomotopyGroup
      {0}
      {\EmbeddingSpaceBoundaryConditionTangential
        {\TangentialFibration}
        {\Ball{2}}
        {\Manifold
          \setminus 
          \apply
            {\Embedding}
            {\SurfaceGBOneSkeleton{\Genus}{\BoundaryComponents}}
        }
        {\BoundaryConditionTangential{\CellBoundaryCondition}}
      }
  $%
  .
  
  We claim that 
  $
    \pi^{\TangentialFibration}_{\text{rel}}
  $
  is isomorphic to 
  $
    \HomotopyGroup
      {0}
      {\EmbeddingSpaceTangential
        {\TangentialFibration}
        {\SurfaceGB{\Genus}{\BoundaryComponents}}
        {\Manifold}
      }
  $%
  .
  Consider the long exact sequence in homology of the pair 
  $
    \left(
      \TangentialSpace{\Manifold}
      ,
      \BoundaryConditionTangential{\BoundaryCondition}
    \right)
  $%
  (As before $\BoundaryCondition$ is the boundary condition for the spaces of 
  subsurfaces)
  : 
  \[
    \begin{tikzcd}[cramped, column sep=small]
      0
        \ar[r]
      &
      \HomologyOfSpace{2}{\TangentialSpace{\Manifold}}
      \simeq
      \HomotopyGroup{2}{\TangentialSpace{\Manifold}}
        \ar[r]
      &
      \HomologyOfSpace
        {2}
        {\TangentialSpace{\Manifold},
          \BoundaryConditionTangential{\BoundaryCondition}
        }
      \simeq
      \HomotopyGroup
        {2}
        {\TangentialSpace{\Manifold}
          ,\BoundaryConditionTangential{\BoundaryCondition}
        }
        \ar[r,"\partial^{\prime}"]
      &
      \HomologyOfSpace{1}{\BoundaryConditionTangential{\BoundaryCondition}}
        \ar[r]
      &
      0
    \end{tikzcd}
  \]
  Note that $\TangentialSpace{\Manifold}$ is simply-connected since 
  $\Fiber{\TwoPlane}{\TangentialFibration}$ and 
  $\Grassmannian{2}{\TangentBundle{\Manifold}}$ are simply-connected. Likewise 
  the pair 
  $
    \left(
      \TangentialSpace{\Manifold},
      \BoundaryConditionTangential{\CellBoundaryCondition}
    \right)
  $
  is simply-connected. This yields the claimed isomorphisms by the 
  Hurewicz-Theorem.
  The long exact sequence implies that 
  $
    \pi^{\TangentialFibration}
    \subset
    \HomotopyGroup{2}{\TangentialSpace{\Manifold}}
    \cong
    \HomologyOfSpace{2}{\TangentialSpace{\Manifold}}
  $
  acts freely on 
  $
    \apply
      {\partial^{\prime-1}}
      {[\BoundaryConditionTangential{\BoundaryCondition}]}
  $%
  . We have a map from 
  $
    \pi^{\TangentialFibration}_{\text{rel}}
  $
  to 
  $
    \apply
      {\partial^{\prime-1}}
      {[\BoundaryConditionTangential{\BoundaryCondition}]}
  $
  by extending an embedding by 
  $\at{e}{\SurfaceGBOneSkeleton{\Genus}{\BoundaryComponents}}$ and
  considering the image of the fundamental class of 
  $\SurfaceGB{\Genus}{\BoundaryComponents}$. 
  By definition this map factorizes through 
  $
    \HomotopyGroup
      {0}
      {\EmbeddingSpaceBoundaryConditionTangential
        {\TangentialFibration}
        {\SurfaceGB{\Genus}{\BoundaryComponents}}
        {\Manifold}
        {\BoundaryConditionTangential{\BoundaryCondition}}
      }
  $ 
  and it is equivariant with respect to the
  $
    \pi^{\TangentialFibration}
  $%
  -action on $\HomologyOfSpace{2}{\TangentialSpace{\Manifold}}$ and 
  $
    \apply
      {\partial^{\prime-1}}
      {[\BoundaryConditionTangential{\BoundaryCondition}]}
  $%
  . This implies that the map from 
  $
    \pi^{\TangentialFibration}_{\text{rel}}
  $
  to 
  $
    \HomotopyGroup
      {0}
      {\EmbeddingSpaceBoundaryConditionTangential
        {\TangentialFibration}
        {\SurfaceGB{\Genus}{\BoundaryComponents}}
        {\Manifold}
        {\BoundaryConditionTangential{\BoundaryCondition}}
      }
  $ 
  is injective. 
  
  To see that this map is surjective consider the fibration 
  (proven to be a fibration in Lemma~\ref{lem:RestrictionFibration})
  $
    \EmbeddingSpaceBoundaryConditionTangential
      {\TangentialFibration}
      {\SurfaceGB{\Genus}{\BoundaryComponents}}
      {\Manifold}
      {\BoundaryConditionTangential{\BoundaryCondition}}
    \to
    \EmbeddingSpaceBoundaryConditionTangential
      {\TangentialFibration}
      {\SurfaceGBOneSkeleton{\Genus}{\BoundaryComponents}}
      {\Manifold}
      {\BoundaryConditionTangential{\BoundaryCondition}}
  $
  and consider its long exact sequence:
  \[
    \begin{tikzcd}[cramped, column sep=small]
      \pi^{\TangentialFibration}_{\text{rel}}
      =
      \HomotopyGroup
        {0}
        {
          \EmbeddingSpaceBoundaryConditionTangential
            {\TangentialFibration}
            {\Ball{2}}
            {
              \Manifold
              \setminus 
              \apply
                {\Embedding}
                {\SurfaceGBOneSkeleton{\Genus}{\BoundaryComponents}}
            }
            {\BoundaryConditionTangential{\bar{\BoundaryCondition}}}
        }
        \ar[r]
      &
      \HomotopyGroup
        {0}
        {
          \EmbeddingSpaceBoundaryConditionTangential
            {\TangentialFibration}
            {\SurfaceGB{\Genus}{\BoundaryComponents}}
            {\Manifold}
            {\BoundaryConditionTangential{\BoundaryCondition}}
        }
        \ar[r]
      &
      \HomotopyGroup
        {0}
        {
          \EmbeddingSpaceBoundaryConditionTangential
            {\TangentialFibration}
            {\SurfaceGBOneSkeleton{\Genus}{\BoundaryComponents}}
            {\Manifold}
            {\BoundaryConditionTangential{\BoundaryCondition}}
        }
    \end{tikzcd}
  \]
  Hence the map is surjective, as
  $
    \HomotopyGroup
      {0}
      {
        \EmbeddingSpaceBoundaryConditionTangential
          {+}
          {\SurfaceGBOneSkeleton{\Genus}{\BoundaryComponents}}
          {\Manifold}
          {\BoundaryCondition}
      }
  $
  has only one element by the main result of \cite{H61}, and 
  the simply-connectedness of $\Fiber{\TwoPlane}{\TangentialFibration}$.
  
  Let us proceed with the last step. We have proven so far that the 
  $
    \pi^{\TangentialFibration}
  $%
  -equivariant map from
  $
    \pi^{\TangentialFibration}_{\text{rel}}
  $
  to 
  $
    \HomotopyGroup
      {0}
      {
        \EmbeddingSpaceBoundaryConditionTangential
        {\TangentialFibration}
        {\SurfaceGB{\Genus}{\BoundaryComponents}}
        {\Manifold}
        {\BoundaryConditionTangential{\BoundaryCondition}}
      }
  $
  is an isomorphism and that the map from there to 
  $
  \apply
    {\partial^{\prime-1}}
    {[\BoundaryConditionTangential{\BoundaryCondition}]}
  $
  is an injection. But this map factorizes through 
  $
    \HomotopyGroup
      {0}
      {
        \SpaceSubsurfaceBoundaryTangential
          {\TangentialFibration}
          {\Genus}
          {\BoundaryComponents}
          {\Manifold}
          {\BoundaryConditionTangential{\BoundaryCondition}}
      }
  $%
  , since an orientation preserving diffeomorphism does not change the image of 
  the fundamental class. 
  Therefore we conclude that
  $
    \HomotopyGroup
      {0}
      {
        \SpaceSubsurfaceBoundaryTangential
        {\TangentialFibration}
        {\Genus}
        {\BoundaryComponents}
        {\Manifold}
        {\BoundaryConditionTangential{\BoundaryCondition}}
      }
  $ 
  is an affine 
  $
    \pi^{\TangentialFibration}
  $%
  -set. 
\end{proof}

  From this we obtain almost immediately the following 
  corollary.
  \begin{corollary}
\label{crl:PiZeroStabilizes}
  If $\Manifold$ is simply-connected and of dimension at 
  least $5$, $\DistBoundaryManifold{0}{\Manifold}$ is simply-connected as well 
  and the fiber of $\TangentialFibration$ is 
  simply-connected, then $\TangentialFibration$ $\pi_0$-stabilizes at $0$ and 
  $\pi_0$-stabilizes at the boundary at $0$ as well.
\end{corollary}
\begin{proof}
  The stabilization maps are given by gluing in surfaces, 
  which induce 
  $
    \HomotopyGroup
      {0}
      {\EmbeddingSpaceTangential
        {\TangentialFibration}
        {\TwoSphere}
        {\Manifold}
      }
  $%
  -equivariant maps between 
  $
    \HomotopyGroup
      {0}
      {\SpaceSubsurfaceBoundaryTangential
        {\TangentialFibration}
        {\Genus}
        {\BoundaryComponents}
        {\Manifold}
        {\BoundaryConditionTangential{\BoundaryCondition}}
      }
  $
  and
  $
    \HomotopyGroup
    {0}
    {\SpaceSubsurfaceBoundaryTangential
      {\TangentialFibration}
      {\Genus+1}
      {\BoundaryComponents-1}
      {\Manifold}
      {\BoundaryConditionTangential{\BoundaryCondition}}
    }
  $
  or
  $
    \HomotopyGroup
    {0}
    {\SpaceSubsurfaceBoundaryTangential
      {\TangentialFibration}
      {\Genus}
      {\BoundaryComponents+1}
      {\Manifold}
      {\BoundaryConditionTangential{\BoundaryCondition}}
    }
  $%
  . Since both of these spaces are affine 
  $
    \HomotopyGroup
    {0}
    {\EmbeddingSpaceTangential
      {\TangentialFibration}
      {\TwoSphere}
      {\Manifold}
    }
  $%
  -spaces, the map is a bijection. Repeating the same argument for 
  $\DistBoundaryManifold{0}{\Manifold}\times[0,1]$ proves the claim.
\end{proof}

  Applying this corollary to $\Manifold$ together with
  Theorem~\ref{thm:HomologicalStabilityConnectedComponents} for 
  $\PiZeroStabilization_{1}=\PiZeroStabilization_{2}=0$, one obtains 
  Theorem~\ref{thm:HomologicalStabilityFiberSimplyConnected}.
\subsection{Non-relative Tangential Structures}
  As was explained in the introduction, the notion of tangential structure in 
  this paper is a generalization of the usual notion to include relative 
  tangential structures. This subsection will deal with non-relative tangential 
  structures.
  
  There is a (up to homotopy) canonical map 
  $
    \Grassmannian{2}{\TangentBundle{\Manifold}}
    \to 
    \Grassmannian{2}{\Reals^{\infty}}
    \simeq
    \ClassifyingSpace{\SOrthogonalGroup{2}}
  $
  and recall that in the common definition of tangential structure one 
  considers a fibration 
  $
    \TangentialFibration
    \colon
    \TangentialSpace{\Grassmannian{2}{\Reals^{\infty}}}
    \to
    \Grassmannian{2}{\Reals^{\infty}}
  $%
  . In this case one obtains classifying spaces for surface bundles equipped 
  with a $\TangentialFibration$-structure, denoted by 
  $
    \ModuliSpace
      {\TangentialFibration}
      {\SurfaceGB{\Genus}{\BoundaryComponents}}
  $
  (we are a bit sloppy here, since there are also boundary conditions 
  involved). We will not define these classifying spaces here and the reader is 
  advised to consult \cite{RW16} in particular Definition~1.1 in that paper.
  \begin{definition}
  We say that a tangential structure 
  $
    \TangentialFibration
    \colon
    \TangentialSpace{\Manifold}
    \to
    \Grassmannian{2}{\TangentBundle{\Manifold}}
  $
  (in the sense of this paper) is a \introduce{non-relative tangential 
  structure stemming from $\TangentialFibration'$} if there exists a tangential 
  structure (in the common sense)
  \[
    \TangentialFibration'
    \colon
    \TangentialSpace{\Grassmannian{2}{\Reals^{\infty}}}
    \to
    \Grassmannian{2}{\Reals^{\infty}}
  \]
  such that the following diagram is a pullback diagram:
  \[
    \begin{tikzcd}
      \TangentialSpace{\Manifold}
        \ar[d,"\TangentialFibration"]
        \ar[r]
      &
      \TangentialSpace{\Grassmannian{2}{\Reals^{\infty}}}
        \ar[d,"\TangentialFibration'"]
      \\
      \Grassmannian{2}{\TangentBundle{\Manifold}}
        \ar[r]
      &
      \Grassmannian{2}{\Reals^{\infty}}
    \end{tikzcd}
  \]
\end{definition}
  The following theorem states that for the question of whether a non-relative 
  tangential structure fulfils homological stability, it does not matter, 
  whether one considers it as tangential structures of 
  abstract surfaces or subsurfaces of a (finite-dimensional) manifold.
  \begin{theorem}
\label{thm:HomologicalStabilityNonRelativeTangentialStructures}
  Suppose $\Manifold$ is a simply-connected at least $5$-dimensional manifold 
  and that $\DistBoundaryManifold{0}{\Manifold}$ is simply-connected.
  Let
  $
    \TangentialFibration
    \colon
    \TangentialSpace{\TangentBundle{\Manifold}}
    \to 
    \Grassmannian{2}{\TangentBundle{\Manifold}}
  $
  denote a non-relative tangential structure stemming from 
  $\TangentialFibration'$ such that 
  $
    \ModuliSpace
      {\TangentialFibration'}
      {\SurfaceGB{\Genus}{\BoundaryComponents}}
  $
  fulfils homological stability as stated in Theorem~7.1 in \cite{RW16} (Note 
  that the definitions of $F$ and $G$ in the theorem, agree with our definition 
  of $\AlphaBound$ and $\BetaBound$). 
  
  Then $\TangentialFibration$ $\pi_0$-stabilizes at genus $\Genus$, and also 
  $\pi_0$-stabilizes at the boundary at genus $\Genus$, where $\Genus$ is the 
  first degree in which the first homology of 
  $
    \MappingPair{\StabilizationAlpha{\Genus}{\BoundaryComponents}}
  $%
  , of
  $
    \MappingPair{\StabilizationBeta{\Genus}{\BoundaryComponents}}
  $%
  , and of 
  $
  \ MappingPair{\StabilizationGamma{\Genus}{\BoundaryComponents}}
  $
  vanishes, here these stabilization maps are stabilization maps of abstract 
  surfaces in the sense of Subsection~1.3 in \cite{RW16}.
  Hence $\TangentialFibration$ is $2\Genus+1$-trivial and 
  $
    \SpaceSubsurfaceBoundaryTangential
      {\TangentialFibration}
      {\Genus}
      {\BoundaryComponents}
      {\Manifold}
      {\BoundaryConditionTangential{\BoundaryCondition}}
  $
  fulfills homological stability as stated in 
  Theorem~\ref{thm:HomologicalStability}. 
  
  In particular $\TangentialFibration$ fulfils homological stability if 
  $\TangentialFibration'$ stabilizes on connected components of 
  $
    \ModuliSpace
      {\TangentialFibration'}
      {\SurfaceGB{\Genus}{\BoundaryComponents}}
  $
\end{theorem}
\begin{proof}
  By Theorem~\ref{thm:HomologicalStabilityConnectedComponents} it suffices to 
  prove that 
  $
    \SpaceSubsurfaceBoundaryTangential
      {\TangentialFibration}
      {\Genus}
      {\BoundaryComponents}
      {\Manifold}
      {\BoundaryConditionTangential{\BoundaryCondition}}
  $
  stabilizes on connected components. 
  There is the following pullback diagram, where the horizontal maps stems from 
  a fixed embedding of $M$ into $\Reals^{\infty}$,
  \[
    \begin{tikzcd}
      \SpaceSubsurfaceBoundaryTangential
        {\TangentialFibration}
        {\Genus}
        {\BoundaryComponents}
        {\Manifold}
        {\BoundaryConditionTangential{\BoundaryCondition}}
        \ar[d]
        \ar[r]
      &
      \ModuliSpace
        {\TangentialFibration'}
        {\SurfaceGB{\Genus}{\BoundaryComponents}}
        \ar[d]
      \\
      \SpaceSubsurfaceBoundary
        {\Genus}
        {\BoundaryComponents}
        {\Manifold}
        {\BoundaryCondition}
        \ar[r]
      &
      \ModuliSpace{+}{\SurfaceGB{\Genus}{\BoundaryComponents}}
    \end{tikzcd}
  \]
  Here $\ModuliSpace{+}{\SurfaceGB{\Genus}{\BoundaryComponents}}$ is the 
  classifying space of oriented surface bundles.
  Let us denote the fiber of the vertical maps of this pullback diagram by 
  $F_{\Genus,\BoundaryComponents}$.
  Let us prove stabilization of connected components for a stabilization map of 
  type $\StabilizationAlpha{\Genus}{\BoundaryComponents}$, the case of another 
  stabilization map is treated analogously.
  
  One obtains the following diagram, where the rows are exact sequences and the 
  vertical maps are steming from stabilization maps 
  $\StabilizationAlpha{\Genus}{\BoundaryComponents}$:
  \[
    \begin{lrbox}{\wideeqbox}
    $\begin{tikzcd}[cramped, column sep=small]
      \HomologyOfSpace
        {1}
        {\ModuliSpace
          {\TangentialFibration'}
          {\SurfaceGB{\Genus}{\BoundaryComponents}}
        }
        \ar[r]
        \ar[d,"\cong\text{$\Genus$ large}"]
      &
      \HomologyOfSpace
        {1}
        {\ModuliSpace{+}{\SurfaceGB{\Genus}{\BoundaryComponents}}}
        \ar[r]
        \ar[d,"\cong\text{$\Genus$ large}"]
      &
      \HomologyOfSpace{0}{F_{\Genus,\BoundaryComponents}}
        \ar[r]
        \ar[d,"\cong\text{ Five Lemma}"]
      &
      \HomologyOfSpace
        {0}
        {\ModuliSpace
          {\TangentialFibration'}
          {\SurfaceGB{\Genus}{\BoundaryComponents}}
        }
        \ar[r]
        \ar[d,"\cong\text{$\Genus$ large}"]
      &
      \HomologyOfSpace
        {0}
        {\ModuliSpace{+}{\SurfaceGB{\Genus}{\BoundaryComponents}}}
        \ar[d,"\cong\text{$\Genus$ large}"]
      \\
      \HomologyOfSpace
        {1}
        {\ModuliSpace
          {\TangentialFibration'}
          {\SurfaceGB{\Genus+1}{\BoundaryComponents-1}}
        }
        \ar[r]
      &
      \HomologyOfSpace
        {1}
        {\ModuliSpace{+}{\SurfaceGB{\Genus+1}{\BoundaryComponents-1}}}
        \ar[r]
      &
      \HomologyOfSpace{0}{F_{\Genus+1,\BoundaryComponents-1}}
        \ar[r]
      &
      \HomologyOfSpace
        {0}
        {\ModuliSpace
          {\TangentialFibration'}
          {\SurfaceGB{\Genus+1}{\BoundaryComponents-1}}
        }
        \ar[r]
      &
      \HomologyOfSpace
      {0}
        {\ModuliSpace{+}{\SurfaceGB{\Genus+1}{\BoundaryComponents-1}}}
    \end{tikzcd}$
    \end{lrbox}
    \makebox[0pt]{\usebox{\wideeqbox}}
  \]
  From this we can we conclude that if 
  $
    \HomologyOfMap
      {\StabilizationAlpha{\Genus}{\BoundaryComponents}}
    \colon
    \HomologyOfSpace
      {1}
      {\ModuliSpace
        {\TangentialFibration'}
        {\SurfaceGB{\Genus}{\BoundaryComponents}}
      }
    \to
    \HomologyOfSpace
      {1}
      {\ModuliSpace
        {\TangentialFibration'}
        {\SurfaceGB{\Genus+1}{\BoundaryComponents-1}}
      }
  $
  and
  $
    \HomologyOfMap
      {\StabilizationAlpha{\Genus}{\BoundaryComponents}}
    \colon
    \HomologyOfSpace
      {0}
      {\ModuliSpace
        {+}
        {\SurfaceGB{\Genus}{\BoundaryComponents}}
      }
    \to
    \HomologyOfSpace
      {0}
      {\ModuliSpace
        {+}
        {\SurfaceGB{\Genus+1}{\BoundaryComponents-1}}
      }
  $
  are isomorphisms, then the induced map between the fibers
  $
    \HomologyOfMap
      {\StabilizationAlpha{\Genus}{\BoundaryComponents}}
    \colon
    \HomologyOfSpace
      {0}
      {F_{\Genus,\BoundaryComponents}}
    \to
    \HomologyOfSpace
      {0}
      {F_{\Genus+1,\BoundaryComponents-1}}
  $
  is an isomorphism as well.
  
  Hence we obtain:
  \[
    \begin{tikzcd}
      \HomologyOfSpace
        {1}
        {\SpaceSubsurfaceBoundary
          {\Genus}
          {\BoundaryComponents}
          {\Manifold}
          {\BoundaryCondition}
        }
        \ar[r]
        \ar[d,"\cong\text{ if $\Genus$ is large enough}"]
      &
      \HomologyOfSpace
        {0}
        {F_{\Genus,\BoundaryComponents}}
        \ar[r]
        \ar[d,"\cong\text{ if $\Genus$ is large enough}"]
      &
      \HomologyOfSpace
        {0}
        {\SpaceSubsurfaceBoundaryTangential
            {\TangentialFibration}
            {\Genus}
            {\BoundaryComponents}
            {\Manifold}
            {\BoundaryConditionTangential{\BoundaryCondition}}
        }
        \ar[r]
        \ar[d,"\cong\text{ Five Lemma}"]
      &
      \HomologyOfSpace
        {0}
        {\SpaceSubsurfaceBoundary
          {\Genus}
          {\BoundaryComponents}
          {\Manifold}
          {\BoundaryCondition}
        }
        \ar[r]
        \ar[d,"\cong\text{ if $\Genus$ is large enough}"]
      &
      0
      \\
      \HomologyOfSpace
        {1}
        {\SpaceSubsurfaceBoundary
          {\Genus}
          {\BoundaryComponents}
          {\Manifold}
          {\BoundaryCondition}
        }
        \ar[r]
      &
      \HomologyOfSpace{0}{F_{\Genus+1,\BoundaryComponents-1}}
        \ar[r]
      &
      \HomologyOfSpace
        {0}
        {\SpaceSubsurfaceBoundaryTangential
          {\TangentialFibration'}
          {\Genus+1}
          {\BoundaryComponents-1}
          {\Manifold}
          {\BoundaryConditionTangential{\BoundaryCondition}}
        }
        \ar[r]
      &
      \HomologyOfSpace
        {0}
        {\SpaceSubsurfaceBoundary
          {\Genus}
          {\BoundaryComponents}
          {\Manifold}
          {\BoundaryCondition}
        }
        \ar[r]
      &
      0
    \end{tikzcd}
  \]
  Repeating the same argument for $\DistBoundaryManifold{0}{\Manifold}\times 
  [0,1]$, we conclude that $\TangentialFibration$ 
  $\pi_0$-stabilizes and $\pi_0$-stabilizes at the boundary. The rest of 
  the proof follows from 
  Theorem~\ref{thm:HomologicalStabilityConnectedComponents}. 
\end{proof}
  As was explained in the introduction, examples of such tangential structures 
  are given by framings of the tangent bundle of the subsurface or spin 
  structures on the subsurfaces (See \cite{RW16} for the corresponding 
  statement about the moduli spaces of abstract surfaces with these tangential 
  structures).
\section{Spaces of Pointed Subsurfaces and their Stabilization}
  \label{scn:PointedStabilization}
While stabilizing at the boundary, as we have done so far, is the most 
intuitive way to stabilize subsurfaces, it is not suitable for many kinds of 
tangential structures.
In particular it is not well suited for symplectic subsurfaces, which we will 
consider in Section~\ref{scn:SymplecticSubsurfaces}. 
In order to deal with this, we will introduce spaces of 
pointedly embedded subsurfaces and construct stabilization maps between 
them that are heuristically given by taking the connected sum with a fixed 
torus and then prove homological stability for these stabilization maps.
\subsection{Spaces of Pointedly Embedded Subsurfaces}
  Let $\Manifold$ denote a closed smooth $\DimensionIndex$-manifold, where 
  $\DimensionIndex\geq 5$ 
  and let us abbreviate $\SurfaceGB{\Genus}{0}$ by $\SurfaceG{\Genus}$. 
  Fix $\Point\in \SurfaceG{\Genus}$, 
  $\Point_{\Manifold}\in \Manifold$ and 
  an oriented $2$-plane with tangential structure 
  $
    \TwoPlane
    \in
    \apply
      {\TangentialFibration^{-1}}
      {\Grassmannian{2}{\TangentSpace{\Point_{\Manifold}}{\Manifold}}}
  $%
  , by abuse of notation we will also write $\TwoPlane$ for the underlying 
  $2$-plane. 
  We will consider the space
  \[
    \EmbeddingSpacePointedTangential
      {\TangentialFibration}
      {\TwoPlane}
      {\SurfaceG{\Genus}}
      {\Manifold}
    \coloneqq
    \left\{
      \Embedding \in 
      \EmbeddingSpaceTangential
        {\TangentialFibration}
        {\SurfaceG{\Genus}}
        {\Manifold}
    \middle| 
      \text{
        $\apply{\Embedding}{\Point}=\Point_{\Manifold}$ 
        and 
        $
          \apply
            {\TangentialStructure{\Embedding}}
            {\Point}
          =
          \TwoPlane
        $
      }
    \right\}
  \]
  equipped with the same topology as in the non-pointed case. 
  Let $\DiffeoMorphismGroupPointed{\Point}{\SurfaceG{\Genus}}$ denote the 
  subgroup of $\DiffeomorphismGroup{\SurfaceG{\Genus}}$ that 
  fixes $\Point$. 
  This group acts on 
  $
    \EmbeddingSpacePointedTangential
      {\TangentialFibration}
      {\TwoPlane}
      {\SurfaceG{\Genus}}
      {\Manifold}
  $
  freely via precomposition and we denote the quotient by 
  $
    \SpaceSubsurfacesPointedTangential
      {\TangentialFibration}
      {\TwoPlane}
      {\SurfaceG{\Genus}}
      {\Manifold}
  $%
  .
  Elements in this space will be called \introduce{pointed subsurfaces of 
  $\Manifold$}
  \begin{remark} 
    In this situation, one could replace $\Manifold$ by a manifold with 
    boundary and $\SurfaceG{\Genus}$ by 
    $\SurfaceGB{\Genus}{\BoundaryComponents}$ and include some 
    boundary condition $\BoundaryConditionTangential{\BoundaryCondition}$ 
    into this notation. 
    Since all proofs presented here, will work verbatim for the pointed case 
    with boundary, the results hold in that case as well.
  \end{remark}

As before the following lemma will play a central role.
\begin{lemma}
    \label{lem:ForgettingTangentialFibrationPointed}
    The forgetful map 
    $
    \Projection{\TangentialFibration}
    \colon 
    \SpaceSubsurfacesPointedTangential
      {\TangentialFibration}
      {\TwoPlane}
      {\SurfaceG{\Genus}}
      {\Manifold}
    \to 
    \SpaceSubsurfacesPointedTangential
      {+}
      {\TwoPlane}
      {\SurfaceG{\Genus}}
      {\Manifold}
    $
    is a Hurewicz fibration..
\end{lemma} 
The proof is verbatim the same as the proof of 
Lemma~\ref{lem:ForgettingTangentialFibration}. One only has to replace the 
occuring homeomorphism and diffeomorphism groups by their pointed analogues and 
note that     
$
  \SpaceSubsurfacesPointedTangential
    {+}
    {\TwoPlane}
    {\SurfaceG{\Genus}}
    {\Manifold}
$
is $\DiffeoMorphismGroupPointed{\Point_{\Manifold}}{\Manifold}$-locally 
retractile, where this is the group of diffeomorphisms fixing 
$\TangentSpace{\Point_{\Manifold}}{\Manifold}$ pointwise.
\subsection{Stabilization Maps for Pointedly Embedded Subsurfaces}
  The goal of this subsection is to construct stabilization maps
  \[
    \PointedStabilizationMap{\Genus}
    \colon
    \SpaceSubsurfacesPointedTangential
      {\TangentialFibration}
      {\TwoPlane}
      {\SurfaceG{\Genus}}
      {\Manifold}
    \to
    \SpaceSubsurfacesPointedTangential
      {\TangentialFibration}
      {\TwoPlane}
      {\SurfaceG{\Genus+1}}
      {\Manifold}
  \]%
  As the first and most complicated step, we will show that there are homotopy 
  equivalent spaces of subsurfaces where the subsurfaces meet a neighbourhood 
  of $\Point_{\Manifold}$ in a prescribed way. From there it will be quite easy 
  to define the stabilization maps.
  
  Let $\CoordinateChart$ denote a coordinate chart centered around 
  $\Point_{\Manifold}$ diffeomorphic to 
  $\Reals^{\DimensionIndex}$ such that the fixed tangent plane $\TwoPlane$ 
  corresponds to $\Reals^{2}$ in $\Reals^{\DimensionIndex}$, where this 
  inclusion is given by specifying the last $\DimensionIndex-2$ coordinates to 
  be zero.
  Fix once and for all $0<\FixedRadius \in \Reals$ and we write 
  $
    \DiskRadius{\FixedRadius}
    \coloneqq
    \Reals^{2}
    \cap 
    \BallRadius{\FixedRadius}{0}
  $%
  .
  \begin{definition}
  $
    \SpaceSubsurfacesPointedFlatTangential
      {\TangentialFibration}
      {\TwoPlane}
      {\SurfaceG{\Genus}}
      {\Manifold}
  $
  is the subspace of 
  $
    \SpaceSubsurfacesPointedTangential
      {\TangentialFibration}
      {\TwoPlane}
      {\SurfaceG{\Genus}}
      {\Manifold}
  $
  consisting of those pointed subsurfaces such that their intersection with the 
  open ball of a radius $\FixedRadius$ in $\CoordinateChart$ equals 
  $\DiskRadius{\FixedRadius}$, and such that the 
  tangential structure is constantly $\TwoPlane$ on 
  $\BallRadius{\frac{\FixedRadius}{2}}{0}$ with respect 
  to some fixed trivialization of 
  $\TangentialSpace{\DiskRadius{\FixedRadius}}$.
\end{definition}

  We will need the following two results to define pointed stabilization 
  maps. The first proposition will deal with the case without tangential 
  structure and its proof will occupy most of this section, the subsequent 
  corollary will extend this to the case with tangential structures.
  \begin{proposition}
  \label{prp:FlatHomotopyEquivalence}
  There is a coordinate chart centered at $\Point_{\Manifold}$ such that  the 
  inclusion 
  \[
    \SpaceSubsurfacesPointedFlatTangential
      {+}
      {\TwoPlane}
      {\SurfaceG{\Genus}}
      {\Manifold}
    \to
    \SpaceSubsurfacesPointedTangential
      {+}
      {\TwoPlane}
      {\SurfaceG{\Genus}}
      {\Manifold}
  \]
  is a homotopy equivalence with homotopy inverse denoted by 
  $\HomotopyInverse{\text{flat}}$.
\end{proposition}
\begin{remark}
  It is far easier to prove that this inclusion is a weak 
  homotopy equivalence, but we will need a homotopy inverse to the 
  inclusion (i.e. maps that flatten the subsurfaces) to construct the 
  stabilization maps.
\end{remark}

    \begin{corollary}
  \label{crl:FlatHomotopyEquivalenceTangential}
    The inclusion 
    $
      \SpaceSubsurfacesPointedFlatTangential
        {\TangentialFibration}
        {\TwoPlane}
        {\SurfaceG{\Genus}}
        {\Manifold}
      \to
      \SpaceSubsurfacesPointedTangential
        {\TangentialFibration}
        {\TwoPlane}
        {\SurfaceG{\Genus}}
        {\Manifold}
    $ 
    is a homotopy equivalence with homotopy inverse denoted by 
    $\HomotopyInverse{\text{flat}}$ as well.
  \end{corollary}
  \begin{proof}
    We define 
    $
      \SpaceSubsurfacesPointedFlatTangential
        {\TangentialFibration}
        {\TwoPlane}
        {\SurfaceG{\Genus}}
        {\Manifold}
      ^{\ast}
    $
    as 
    $
      \apply
        {\Projection{\TangentialFibration}^{-1}}
        {
          \SpaceSubsurfacesPointedFlatTangential
            {+}
            {\TwoPlane}
            {\SurfaceG{\Genus}}
            {\Manifold}
        }
    $%
    , where 
    \[
      \Projection{\TangentialFibration}
      \colon
      \SpaceSubsurfacesPointedTangential
        {\TangentialFibration}
        {\TwoPlane}
        {\SurfaceG{\Genus}}
        {\Manifold}
      \to
      \SpaceSubsurfacesPointedTangential
        {+}
        {\TwoPlane}
        {\SurfaceG{\Genus}}
        {\Manifold}
    \]
    denotes the projection. 
    Note that the following diagram is a pullback diagram:
    \[
      \begin{tikzcd}
        \SpaceSubsurfacesPointedFlatTangential
          {\TangentialFibration}
          {\TwoPlane}
          {\SurfaceG{\Genus}}
          {\Manifold}
        ^{\ast}
          \ar[r,hook]
          \ar[d]
        &
        \SpaceSubsurfacesPointedTangential
          {\TangentialFibration}
          {\TwoPlane}
          {\SurfaceG{\Genus}}
          {\Manifold}
          \ar[d]
        \\
        \SpaceSubsurfacesPointedFlatTangential
          {+}
          {\TwoPlane}
          {\SurfaceG{\Genus}}
          {\Manifold}
          \ar[r,hook]
        &
        \SpaceSubsurfacesPointedTangential
          {+}
          {\TwoPlane}
          {\SurfaceG{\Genus}}
          {\Manifold}
      \end{tikzcd}
    \]
    It is a standard fact that a pullback of a homotopy equivalence along a 
    Hurewicz-fibration produces an induced map, which is also a homotopy 
    equivalence (see for example Proposition~2.3 in \cite{HomPull}). 
    This fact together with 
    Lemma~\ref{lem:ForgettingTangentialFibrationPointed} 
    (which implies that the right arrow is a fibration) implies that the top 
    map is a homotopy equivalence.
    Furthermore the homotopy inverse is a map over the homotopy inverse of 
    the base map. 
    
    So all that is left to do is trivializing the tangential structure on the 
    disk but this is easily done by using a fixed trivialization of 
    $
      \TangentialSpace{\DiskRadius{\FixedRadius}}
      \cong
      \DiskRadius{\FixedRadius}
      \times
      \apply
        {\TangentialFibration^{-1}}
        {\Grassmannian{2}{\TangentSpace{0}{\DiskRadius{\FixedRadius}}}}
    $ 
    and just continuously extend the tangential structure of $\TwoPlane$ to 
    $\DiskRadius{\frac{\FixedRadius}{2}}$. 
  \end{proof}
  \paragraph{Proof of Proposition~\ref{prp:FlatHomotopyEquivalence}}
  In order to proof Proposition~\ref{prp:FlatHomotopyEquivalence}, we will 
  define two subspaces of 
  $
    \SpaceSubsurfacesPointedTangential
      {+}
      {\TwoPlane}
      {\SurfaceG{\Genus}}
      {\Manifold}
  $%
  , which both contain 
  $
    \SpaceSubsurfacesPointedFlatTangential
      {+}
      {\TwoPlane}
      {\SurfaceG{\Genus}}
      {\Manifold}
  $%
  , and fit in the following diagram 
  \[
    \begin{tikzcd}
      \SpaceSubsurfacesPointedTangential
        {+}
        {\TwoPlane}
        {\SurfaceG{\Genus}}
        {\Manifold}
        \ar[r,"\HomotopyInverse{1}",bend left=]
      &
      \SpaceSubsurfacesPointedRadiusTangential
        {+}
        {\TwoPlane}
        {\FixedRadius}
        {\SurfaceG{\Genus}}
        {\Manifold}
        \ar[l,hook']
        \ar[r,"\HomotopyInverse{2}",bend left=43]
      &
      \SpaceAlmostFlatSubsurfaces
        {+}
        {\TwoPlane}
        {\SurfaceG{\Genus}}
        {\Manifold}
        \ar[r,"\HomotopyInverse{3}",bend left=40]
        \ar[l,hook']
      &
      \SpaceSubsurfacesPointedFlatTangential
        {+}
        {\TwoPlane}
        {\SurfaceG{\Genus}}
        {\Manifold}
        \ar[l,hook']
    \end{tikzcd}
  \]
  The hooked arrows are inclusions and the $\HomotopyInverse{i}$ represent 
  their respective homotopy inverses. 
  We will denote 
  $\HomotopyInverse{3}\circ\HomotopyInverse{2}\circ\HomotopyInverse{1}$ by 
  $\HomotopyInverse{\text{flat}}$
  Let us define these spaces:
  \begin{definition}
  Let 
  $
    \Projection{2}
    \colon
    \Reals^{\DimensionIndex}
    \to 
    \Reals^{2}
  $
  denote the projection onto the first two coordinates i.e. the 
  map that sends 
  $
    (
      \Coordinate_{1},\ldots, \Coordinate_{\DimensionIndex}
    )
  $ 
  to 
  $
    (
      \Coordinate_{1},\Coordinate_{2}
    )
  $%
  . For $\Radius\in (0,\infty)$, we define 
  \[
    \EmbeddingSpacePointedRadiusTangenital
      {+}
      {\TwoPlane}
      {\Radius}
      {\SurfaceG{\Genus}}
      {\Manifold}
    \coloneqq 
    \left\{
      \Embedding
      \in
      \EmbeddingSpacePointedTangential
        {\TangentialFibration}
        {\TwoPlane}
        {\SurfaceG{\Genus}}
        {\Manifold}
    \middle|
      \begin{aligned}
        &
          \exists 
          \Neighbourhood{\Point} 
          \text{ a closed neighborhood of $\Point$ s.th.}
        \\
        &
          \at{\Projection{2}\circ \Embedding}{\Neighbourhood{\Point}} 
          \text{ is a diffeomorphism}
        \\
        & 
          \text{with image }\overline{\BallRadius{\Radius}{0}}\subset \Reals^{2}
      \end{aligned}
    \right\}
  \]
  The action of $\DiffeoMorphismGroupPointed{\Point}{\SurfaceG{\Genus}}$ via 
  precomposition on 
  $
    \EmbeddingSpacePointedTangential
      {\TangentialFibration}
      {\TwoPlane}
      {\SurfaceG{\Genus}}
      {\Manifold}
  $ 
  restricts to a free action on the subspace
  $
    \EmbeddingSpacePointedRadiusTangenital
      {+}
      {\TwoPlane}
      {\Radius}
      {\SurfaceG{\Genus}}
      {\Manifold}
  $
  and we denote the quotient by 
  $
    \SpaceSubsurfacesPointedRadiusTangential
      {+}
      {\TwoPlane}
      {\Radius}
      {\SurfaceG{\Genus}}
      {\Manifold}
  $%
  . 
  For a 
  $
    \Subsurface
    \in 
    \SpaceSubsurfacesPointedRadiusTangential
      {+}
      {\TwoPlane}
      {\Radius}
      {\SurfaceG{\Genus}}
      {\Manifold}
  $
  we will write $\SubsurfaceDisk{\Subsurface}{\Radius}{\Point}$ for the image 
  of $\Neighbourhood{\Point}$ in $\Manifold$.
  
  Let 
  $
    \SpaceAlmostFlatSubsurfaces
      {+}
      {\TwoPlane}
      {\SurfaceG{\Genus}}
      {\Manifold}
  $
  denote the subspace of 
  $
    \SpaceSubsurfacesPointedTangential
      {+}
      {\TwoPlane}
      {\SurfaceG{\Genus}}
      {\Manifold}
  $
  of such subsurfaces $\Subsurface$ such that 
  $
    \DiskRadius{\FixedRadius}
    \subset 
    \Subsurface
    \cap 
    \BallRadius{\FixedRadius}{0}
  $
\end{definition}
\begin{remark}
  Note that the subtle difference between 
  $
    \SpaceAlmostFlatSubsurfaces
      {+}
      {\TwoPlane}
      {\SurfaceG{\Genus}}
      {\Manifold}
  $
  and 
  $
    \SpaceSubsurfacesPointedFlatTangential
      {+}
      {\TwoPlane}
      {\SurfaceG{\Genus}}
      {\Manifold}
  $
  is that subsurfaces in the latter intersect $\BallRadius{\FixedRadius}{0}$ 
  exactly in $\DiskRadius{\FixedRadius}$, while surfaces in the former only 
  have to contain $\DiskRadius{\FixedRadius}$, but there can 
  be more parts of these surfaces that intersect 
  $\BallRadius{\FixedRadius}{0}$.
\end{remark}

  After the following three auxillary lemmas we will prove that all the 
  aforementioned inclusions are indeed homotopy 
  equivalences by constructing their inverse and the necessary homotopies. 
  This will finish the proof of Proposition \ref{prp:FlatHomotopyEquivalence}.
  \begin{lemma}
\label{lem:HomotopyInverseInclusion}
  Let $\iota \colon \TopologicalSpace'\to \TopologicalSpace''$ denote an 
  inclusion and let 
  $
    \HomotopyInverse{\iota}
    \colon 
    \TopologicalSpace''
    \to 
    \TopologicalSpace'
    \subset
    \TopologicalSpace''
  $ 
  be a map such that there exists a homotopy 
  $
    \Homotopy
    \colon
    \TopologicalSpace''
    \times 
    [0,1]
    \to 
    \TopologicalSpace''
  $ 
  such that 
  $
    \apply{\Homotopy}{-,0}
    =
    \Identity_{\TopologicalSpace''}
  $
  and 
  $
    \apply{\Homotopy}{-,1}
    =
    \HomotopyInverse{\iota}
  $
  and
  $
    \apply{\Homotopy}{\TopologicalSpace'\times[0,1]}
    \subset 
    \TopologicalSpace'
  $%
  , then $\iota$ is a homotopy equivalence with homotopy inverse 
  $\HomotopyInverse{\iota}$.
\end{lemma}
\begin{proof}
  $\Homotopy$ is already a homotopy between 
  $
    \iota 
    \circ 
    \HomotopyInverse{\iota}
  $
  and the identity of $\TopologicalSpace''$, 
  therefore we only need to construct a homotopy between 
  $\HomotopyInverse{\iota}\circ \iota$ and the identity on 
  $\TopologicalSpace'$. 
  But 
  $
    \at
      {\Identity_{\TopologicalSpace''}}
      {\TopologicalSpace'}
    =
    \Identity_{\TopologicalSpace'}
  $%
  , therefore the necessary homotopy is again given by $\Homotopy$. 
\end{proof}

  This lemma will be used to prove that the $\HomotopyInverse{i}$ are homotopy 
  inverse to the aforementioned inclusions. 
  The next rather technical lemma will be used towards the end of the proof of 
  Proposition~\ref{prp:FlatHomotopyEquivalence} to show that a  certain map is 
  continuous.
  \begin{lemma}
\label{lem:InfimumContinuous}
  Suppose we have a diagram of the following form:
  \[
    \begin{tikzcd}
      \TopologicalSpace_{1}
      \times 
      \TopologicalSpace_{2} 
        \ar[r,"\ContinuousMap"]
        \ar[d,"\Projection{1}"] 
      &
      \Reals
      \cup 
      \{\infty\}
      \\
      \TopologicalSpace_{1}
        \ar[r,"\ContinuousMap_{\mathrm{inf}}"]
      &
      \Reals
      \cup
      \{\infty\}
    \end{tikzcd}
  \]
  where $\TopologicalSpace_{2}$ is compact and 
  $
    \apply{\ContinuousMap_{\mathrm{inf}}}{\Point_{1}}
    =
    \inf_{\Point_{2} \in \TopologicalSpace_{2}} 
    \apply{\ContinuousMap}{\Point_{1},\Point_{2}}
  $%
  . Then $\ContinuousMap_{\mathrm{inf}}$ is continuous.
\end{lemma}
\begin{proof}
  Since intervals of the form $(C,\infty]$ and $(-\infty,C)$ form a subbasis of 
  the topology of $\Reals\cup\{\infty\}$ it suffices to show that their 
  preimages are open.
  We will first show that 
  $
    \apply{\ContinuousMap_{\mathrm{inf}}^{-1}}{(C,\infty]}
  $
  is open for every $C$. 
  For this one only has to note that 
  $
    \apply{\ContinuousMap_{\mathrm{inf}}^{-1}}{(C,\infty]}
    =
    \apply
      {\Projection{1}}
      {\Complement{\apply{\ContinuousMap^{-1}}{(-\infty,c]}}}
  $
  which is open since $\ContinuousMap$ is continuous and $\Projection{1}$ is 
  closed since $\TopologicalSpace_{2}$ is compact. 
    
  Now we will proof that 
  $\apply{\ContinuousMap_{\mathrm{inf}}^{-1}}{(-\infty,C]}$ is open,
  observe that 
  $
    \apply{\ContinuousMap_{\mathrm{inf}}^{-1}}{(-\infty,C]}
    =
    \apply
      {\Projection{1}}
      {\apply{\ContinuousMap^{-1}}{(-\infty,C)}}
  $%
  . This is open since $\Projection{1}$ is an open map.
\end{proof}
  Lastly we need the following lemma before we can start the proof of 
  Proposition~\ref{prp:FlatHomotopyEquivalence}.
  \begin{lemma}
\label{lem:SpacePointedSubsurfacesUnion}
  $
  \SpaceSubsurfacesPointedRadiusTangential
    {+}
    {\TwoPlane}
    {\Radius}
    {\SurfaceG{\Genus}}
    {\Manifold}
  $ 
  is an open subset of 
  $
    \SpaceSubsurfacesPointedTangential
      {+}
      {\TwoPlane}
      {\SurfaceG{\Genus}}
      {\Manifold}
  $
  for every $\Radius\in (0,\infty)$ and if
  $\epsilon_{k}$ denotes a sequence of positive real numbers that converges 
  to zero, then 
  \[
    \SpaceSubsurfacesPointedTangential
      {+}
      {\TwoPlane}
      {\SurfaceG{\Genus}}
      {\Manifold}
    =
    \bigcup_{k} 
    \SpaceSubsurfacesPointedRadiusTangential
      {+}
      {\TwoPlane}
      {\epsilon_{k}}
      {\SurfaceG{\Genus}}
      {\Manifold}
  \]
\end{lemma}
\begin{proof}
  By the inverse function theorem, for every 
  $
    \Subsurface 
    \in 
    \SpaceSubsurfacesPointedTangential
      {+}
      {\TwoPlane}
      {\SurfaceG{\Genus}}
      {\Manifold}
  $ 
  there is an $\Radius>0$ such that 
  $
    \at
      {\Projection{2}}
      {\apply
        {\Projection{2}^{-1}}
        {\overline{\BallRadius{\Radius}{0}}\cap\Subsurface}
      }
  $
  restricted to the connected component of 
  $
  \apply
    {\Projection{2}^{-1}}
    {\overline{\BallRadius{\Radius}{0}}\cap\Subsurface}
  $
  containing $\Point_{\Manifold}$ is a diffeomorphism. 
  This implies that every 
  $
    \Subsurface
    \in 
    \SpaceSubsurfacesPointedTangential
      {+}
      {\TwoPlane}
      {\SurfaceG{\Genus}}
      {\Manifold}
  $
  is contained in some 
  $
    \SpaceSubsurfacesPointedRadiusTangential
      {+}
      {\TwoPlane}
      {\Radius}
      {\SurfaceG{\Genus}}
      {\Manifold}
  $ 
  for $\Radius$ small enough.
  
  To finish the proof we have to show that 
  $
    \SpaceSubsurfacesPointedRadiusTangential
      {+}
      {\TwoPlane}
      {\Radius}
      {\SurfaceG{\Genus}}
      {\Manifold}
  $ 
  is an open subset of 
  $
    \SpaceSubsurfacesPointedTangential
      {+}
      {\TwoPlane}
      {\SurfaceG{\Genus}}
      {\Manifold}
  $%
  , which is equivalent to its preimage in 
  $
    \EmbeddingSpacePointedTangential
      {\TangentialFibration}
      {\TwoPlane}
      {\SurfaceG{\Genus}}
      {\Manifold}
  $
  being open. 
  Take some embedding $\Embedding$ in this set. 
  There is some closed neighborhood $\Neighbourhood{\Point}$ of
  $\Point$ in $\SurfaceG{\Genus}$ 
  such that 
  $
    \at
      {\Projection{2}}
      {\apply{\Embedding}{\Neighbourhood{\Point}}}
  $
  is a diffeomorphism onto $\overline{\BallRadius{\Radius}{0}}$. 
  Furthermore this implies that the differential of 
  $\Projection{2}\circ\Embedding$ is an isomorphism on these points. 
  Since the isomorphisms form an open subset of all linear maps we conclude 
  that there is actually an open neighborhood $\Neighbourhood{\Point}'$ of 
  $\Neighbourhood{\Point}$ such that on the closure of this 
  neighborhood the differential of $\Projection{2}\circ\Embedding$ is still an 
  isomorphism. 
  Since 
  $
    \at
      {\Projection{2}\circ\Embedding}
      {\overline{\Neighbourhood{\Point}'}}
  $ 
  is proper because $\overline{\Neighbourhood{\Point}'}$ is compact and a local 
  diffeomorphism onto $\Reals^2$ by the requirement that the differential is an 
  isomorphism, we conclude that it is a covering of its image. 
  Since a closed disk can only cover a closed disk via a diffeomorphism we see 
  that this map is actually a diffeomorphism onto some closed disk in 
  $\Reals^2$ containing $\overline{\BallRadius{\Radius}{0}}$. 
  Let $2\varepsilon$ denote 
  $
    \min_{\Point'\in \partial \Neighbourhood{\Point}'}
    \EuclideanNorm{\apply{\Projection{2}\circ\Embedding}{\Point'}}
    -
    \Radius
    >
    0
  $%
  . Consider the following neighbourhood of $\Embedding$:
  \begin{equation*}
    \Neighbourhood{\Embedding}
    \coloneqq 
    \left\{
      \Embedding' 
      \in
      \EmbeddingSpacePointedTangential
        {+}
        {\TwoPlane}
        {\SurfaceG{\Genus}}
        {\Manifold}
    \middle|
      \begin{aligned}
        &
          \sup_{\Neighbourhood{\Point}'}
          \EuclideanNorm{\Embedding-\Embedding'}
          <\varepsilon
        \\
        & 
          \at
            {\Differential{\Projection{2}\circ\Embedding'}}
            {\TangentBundle{\Neighbourhood{\Point}'}}
            \text{ is an isomorphism}
      \end{aligned}
    \right\}
  \end{equation*}
  With our previous considerations we conclude that 
  $\Neighbourhood{\Embedding}$ lies in the preimage of 
  $
    \SpaceSubsurfacesPointedRadiusTangential
      {+}
      {\TwoPlane}
      {\Radius}
      {\SurfaceG{\Genus}}
      {\Manifold}
  $
  in the embedding space
  $
    \EmbeddingSpacePointedTangential
      {+}
      {\TwoPlane}
      {\SurfaceG{\Genus}}
      {\Manifold}
  $%
  . As it is by definition an open neighbourhood in the $C^1$ topology and the 
  $C^\infty$ topology is finer than all the $C^k$ topologies we conclude that 
  it is in fact an open neighborhood of $\Embedding$.
\end{proof}

  \begin{lemma}
\label{lem:RadiusToPointedHomotopyEquivalence}
  If $\CoordinateChart$ sits inside another coordinate chart 
  $\CoordinateChart'$ and the coordinate change is given by 
  $
    (
      \Coordinate_{1}
      ,
      \ldots
      ,
      \Coordinate_{\DimensionIndex}
    )
    \mapsto 
    (
      \apply{\mathrm{arctan}}{\Coordinate_{1}}
      ,
      \ldots
      ,
      \apply{\mathrm{arctan}}{\Coordinate_{\DimensionIndex}}
    )
  $%
  , then the inclusion 
  $
    \SpaceSubsurfacesPointedRadiusTangential
      {+}
      {\TwoPlane}
      {\Radius}
      {\SurfaceG{\Genus}}
      {\Manifold}
    \to 
    \SpaceSubsurfacesPointedTangential
      {+}
      {\TwoPlane}
      {\SurfaceG{\Genus}}
      {\Manifold}
  $
  is a homotopy equivalence for any $\Radius\in(0,\infty)$ with homotopy 
  inverse $\HomotopyInverse{1}$.
\end{lemma}
\begin{proof}
  By Theorem 42.3 and the corollary in 27.4 in \cite{CSGA} (together with 
  the fact that the quotient of a nuclear spaces by a closed subspaces
  is a nuclear space) we conclude that 
  $
    \SpaceSubsurfacesPointedTangential
      {+}
      {\TwoPlane}
      {\SurfaceG{\Genus}}
      {\Manifold}
  $ 
  is paracompact and Hausdorff and by 
  Lemma~27.8 in the same source, we conclude that it is actually 
  metrizable (This line of argument was communicated to me in 
  \cite{MichorMathoverflow}). 
  Therefore we can find a continuous function 
  $
    \AuxiliaryFunction_{\text{radius}}
    \colon
    \SpaceSubsurfacesPointedTangential
      {+}
      {\TwoPlane}
      {\SurfaceG{\Genus}}
      {\Manifold}
    \to 
    (0,\Radius]
  $%
  , which is $\Radius$ on 
  $
    \SpaceSubsurfacesPointedFlatTangential
    {+}
    {\TwoPlane}
    {\SurfaceG{\Genus}}
    {\Manifold}
  $ 
  and for every $\Subsurface$ it is smaller or equal than the maximal 
  $\Radius'$ such that 
  $
    \SpaceSubsurfacesPointedRadiusTangential
      {+}
      {\TwoPlane}
      {\FixedRadius'}
      {\SurfaceG{\Genus}}
      {\Manifold}
  $ 
  contains $\Subsurface$ 
  (This uses the fact that 
  $
    \SpaceSubsurfacesPointedFlatTangential
      {+}
      {\TwoPlane}
      {\SurfaceG{\Genus}}
      {\Manifold}
  $ 
  is closed in 
  $
  \SpaceSubsurfacesPointedTangential
    {+}
    {\TwoPlane}
    {\SurfaceG{\Genus}}
    {\Manifold}
  $%
  ).
  
  Let us now construct the homotopy inverse to the inclusion from 
  $
    \SpaceSubsurfacesPointedRadiusTangential
      {+}
      {\TwoPlane}
      {\FixedRadius}
      {\SurfaceG{\Genus}}
      {\Manifold}
    \to
    \SpaceSubsurfacesPointedTangential
      {+}
      {\TwoPlane}
      {\SurfaceG{\Genus}}
      {\Manifold}
  $%
  . Let 
  $
    \AuxiliaryFunctionDiffeomorphism{1}
    \colon
    (0,\infty)
    \to
    \DiffeomorphismGroupCompactSupport{\Manifold}
  $
  denote the following continuous map into the compactly supported 
  diffeomorphisms: 
  Fix a smooth monotone function 
  $
    \AuxiliaryFunction
    \colon 
    \Reals 
    \to 
    [0,1]
  $ 
  such that 
  $
    \AuxiliaryFunction
    \equiv 
    0
  $
  on an open neighborhood of 
  $
    \left[
      -\frac{\pi}{2}
      ,
      \frac{\pi}{2}
    \right]
  $ 
  and 
  $
    \AuxiliaryFunction 
    \equiv
    1
  $ 
  outside of 
  $(-\pi,\pi)$. 
  We define $\AuxiliaryFunctionDiffeomorphism{1}$ via the following formula:
  \[
    \apply
      {\apply{\AuxiliaryFunctionDiffeomorphism{1}}{t}}
      {((\Coordinate_{1},\ldots,\Coordinate_{\DimensionIndex})}
    =
    \left(
      (
        \frac{\Radius}{t}
        \apply{\AuxiliaryFunction}{\Coordinate_{1}}
        +
        (
          1
          -
          \apply{\AuxiliaryFunction}{\Coordinate_{1}}
        )
      )
      \Coordinate_{1}
      ,
      (
        \frac{\Radius}{t}
        \apply{\AuxiliaryFunction}{\Coordinate_{2}}
        +
        (
          1
          -
          \apply{\AuxiliaryFunction}{\Coordinate_{2}}
        )
      )
      \Coordinate_{2}
      ,
      \Coordinate_{3}
      ,
      \ldots
      ,
      \Coordinate_{\DimensionIndex}
    \right)
  \]
  Since this is the identity outside the cube $(-\pi,\pi)^{\DimensionIndex}$ we 
  can consider this as a compactly supported diffeomorphism of $\Manifold$ by 
  extending it via the identity.
  Let $\HomotopyInverse{1}$ denote 
  \begin{align*}
    \HomotopyInverse{1}
    \colon
    \SpaceSubsurfacesPointedTangential
      {+}
      {\TwoPlane}
      {\SurfaceG{\Genus}}
      {\Manifold}
    &
    \to
    \SpaceSubsurfacesPointedTangential
      {+}
      {\TwoPlane}
      {\SurfaceG{\Genus}}
      {\Manifold}
    \\
    \Subsurface
    &
    \mapsto 
    \apply
      {\apply
        {\AuxiliaryFunctionDiffeomorphism{1}}
        {\apply
          {\AuxiliaryFunction_{\text{radius}}}
          {\Subsurface}  
        }
      }
      {\Subsurface}
  \end{align*}
  This maps 
  $
    \SpaceSubsurfacesPointedTangential
      {+}
      {\TwoPlane}
      {\SurfaceG{\Genus}}
      {\Manifold}
  $
  into
  $
    \SpaceSubsurfacesPointedRadiusTangential
      {+}
      {\TwoPlane}
      {\Radius}
      {\SurfaceG{\Genus}}
      {\Manifold}
  $%
  . Furthermore this gives us a homotopy equivalence between 
  $
    \SpaceSubsurfacesPointedRadiusTangential
      {+}
      {\TwoPlane}
      {\Radius}
      {\SurfaceG{\Genus}}
      {\Manifold}
  $ 
  and 
  $
    \SpaceSubsurfacesPointedTangential
      {+}
      {\TwoPlane}
      {\SurfaceG{\Genus}}
      {\Manifold}
  $
  since $\AuxiliaryFunction_{\text{radius}}$ can be homotoped to be the map 
  that is constantly $\Radius$ via linear interpolation and such a homotopy 
  induces maps 
  $
    \SpaceSubsurfacesPointedTangential
      {+}
      {\TwoPlane}
      {\SurfaceG{\Genus}}
      {\Manifold}
    \times 
    [0,1]
    \to
    \SpaceSubsurfacesPointedTangential
      {+}
      {\TwoPlane}
      {\SurfaceG{\Genus}}
      {\Manifold}
  $
  as in Lemma~\ref{lem:HomotopyInverseInclusion}. 
  Hence $\HomotopyInverse{1}$ is indeed a homotopy inverse to the inclusion.
\end{proof}

  \begin{lemma}
  If $\CoordinateChart$ sits inside another coordinate chart 
  $\CoordinateChart'$ and the coordinate change is given by 
  $
  (
    \Coordinate_{1}
    ,
    \ldots
    ,
    \Coordinate_{\DimensionIndex}
  )
    \mapsto 
  (
    \apply{\mathrm{arctan}}{\Coordinate_{1}}
    ,
    \ldots
    ,
    \apply{\mathrm{arctan}}{\Coordinate_{\DimensionIndex}}
  )
  $%
  , then
  the inclusion 
  $
    \SpaceAlmostFlatSubsurfaces
      {+}
      {\TwoPlane}
      {\SurfaceG{\Genus}}
      {\Manifold}
    \to
    \SpaceSubsurfacesPointedRadiusTangential
      {+}
      {\TwoPlane}
      {\Radius}
      {\SurfaceG{\Genus}}
      {\Manifold}
  $
  is a homotopy equivalence with homotopy inverse denoted by 
  $\HomotopyInverse{2}$.
\end{lemma}
\begin{proof}
  Fix a smooth function 
  $
    \AuxiliaryFunction_{1}
    \colon 
    \Reals^{2}
    \to 
    \Reals
  $ 
  that is invariant under rotation and $1$ on 
  $\overline{\BallRadius{\FixedRadius}{0}}$ and $0$ 
  outside of $\BallRadius{2\FixedRadius}{0}$. 
  We define a map 
  \[
    \Homotopy
    \colon 
    \SpaceSubsurfacesPointedRadiusTangential
      {+}
      {\TwoPlane}
      {2\FixedRadius}
      {\SurfaceG{\Genus}}
      {\Manifold}
    \times
    [0,1]
    \to
    \SpaceSubsurfacesPointedRadiusTangential
      {+}
      {\TwoPlane}
      {2\FixedRadius}
      {\SurfaceG{\Genus}}
      {\Manifold}
  \]
  as follows: 
  Note that $\Projection{2}$ extends to $\CoordinateChart'$. 
  For an 
  $
    \Point'
    \in 
    \BallRadius{2\FixedRadius}{0}
  $
  we will denote the corresponding point in 
  $
    \SubsurfaceDisk{\Subsurface}{2\FixedRadius}{\Point_{\Manifold}}
  $
  by 
  $
    \ProjectedPoint{\Subsurface}{\Point'}
  $%
  . Fix a smooth monotone function 
  $
    \AuxiliaryFunction_{2}
    \colon [
    [0,\infty) 
    \to
    \Reals
  $ 
  such that 
  $
    \apply{\AuxiliaryFunction_{2}^{-1}}{1}
    =
    [0,\sqrt{\DimensionIndex}\frac{\pi}{4}]
  $
  and 
  $
    \apply{\AuxiliaryFunction_{2}^{-1}}{0}
    =
    (2\sqrt{\DimensionIndex}\pi,\infty)
  $%
  . Now we define
  \begin{align*}
      \AuxiliaryFunctionDiffeomorphism{2}
      \colon
      \SpaceSubsurfacesPointedRadiusTangential
        {+}
        {\TwoPlane}
        {2\FixedRadius}
        {\SurfaceG{\Genus}}
        {\Manifold}
      \times 
      [0,1]
    &
      \to 
      \DiffeomorphismGroupCompactSupport{\Manifold}
    \\
      \apply
        {\apply{\AuxiliaryFunctionDiffeomorphism{2}}{\Subsurface,t}}
        {\Point'}
    &
      =
      \Point'
      +
      \apply{\AuxiliaryFunction_{2}}{\EuclideanNorm{\Point'}}
      \apply{\AuxiliaryFunction_{1}}{\apply{\Projection{2}}{\Point'}}
      \left(
        \apply{\Projection{2}{\Point'}}
        -
        \ProjectedPoint{\Subsurface}{\Point'}
      \right)
  \end{align*}
  Note that $\ProjectedPoint{\Subsurface}{\Point'}$ is only defined when 
  $
    \EuclideanNorm
      {\apply{\Projection{2}}{\Point'}}
    \leq
    2\FixedRadius
  $
  but in that case
  $
    \apply
      {\AuxiliaryFunction_{1}}
      {\apply{\Projection{2}}{\Point'}}
    =
    0
  $%
  . The map $\AuxiliaryFunctionDiffeomorphism{1}$ is well-defined i.e. the 
  image is indeed a diffeomorphism, since a map 
  \begin{align*}
      \Reals^{\DimensionIndex-2}
    &
      \to
      \Reals^{\DimensionIndex-2}
    \\
      \Point' 
    &
      \mapsto 
      \Point' 
      + 
      \apply{\ContinuousMap}{\EuclideanNorm{\Point'}}
      v
  \end{align*}
  for $v$ some vector and a positive, monotonously decreasing and bounded 
  function $\ContinuousMap$, is a diffeomorphism if 
  $
    \max \EuclideanNorm{\ContinuousMap}
    \EuclideanNorm{v}
    <
    1
  $
  because this ensure the derivative to be an isomorphism which implies that 
  the map is locally injective. 
  Furthermore such a map extends to a degree $1$ map from 
  $\Sphere{2\DimensionIndex-1}$ 
  into itself and therefore it also has to be surjective. Hence it is a 
  diffeomorphism, if the derivative of $\ContinuousMap$ is small enough, but we 
  can chose $\ContinuousMap$ a posteriori with a small enough derivative.
  
  Lastly we have to argue that $\AuxiliaryFunctionDiffeomorphism{2}$ is a 
  continuous map as can be seen since it can be written in a neighborhood of 
  $\Subsurface$ in terms of the image of $\Subsurface$ 
  and the sections of the "normal" bundle of $\Subsurface$ (where we require 
  the "normal" bundle to have $\{0\}\times \Reals^{\DimensionIndex-2}$ as fiber 
  over $\SubsurfaceDisk{\Subsurface}{2\FixedRadius}{\Point_{\Manifold}}$), 
  which is a neighborhood of $\Subsurface$ in 
  $
    \SpaceSubsurfacesPointedRadiusTangential
      {+}
      {\TwoPlane}
      {2\FixedRadius}
      {\SurfaceG{\Genus}}
      {\Manifold}
  $%
  .
  
  We define 
  $
    \apply{\Homotopy}{\Subsurface,t}
    =
    \apply
      {\apply{\AuxiliaryFunctionDiffeomorphism{2}}{\Subsurface,t}}
      {\Subsurface}
  $%
  . Note that 
  $
    \SpaceSubsurfacesPointedRadiusTangential
      {+}
      {\TwoPlane}
      {\FixedRadius}
      {\SurfaceG{\Genus}}
      {\Manifold}
  $ 
  includes into 
  $
    \SpaceSubsurfacesPointedRadiusTangential
      {+}
      {\TwoPlane}
      {\FixedRadius}
      {\SurfaceG{\Genus}}
      {\Manifold}
  $
  and with a similar construction as in 
  Lemma~\ref{lem:RadiusToPointedHomotopyEquivalence} we can see 
  that this inclusion is actually a homotopy equivalence with homotopy inverse, 
  which is given by postcomposing with a diffeomorphism of $\Manifold$,
  denoted by $\HomotopyInverse{\FixedRadius,2\FixedRadius}$. 
  Note that $\Homotopy$ fixes 
  $
    \apply
      {\HomotopyInverse{\FixedRadius,2\FixedRadius}}
      {
        \SpaceAlmostFlatSubsurfaces
          {+}
          {\TwoPlane}
          {\SurfaceG{\Genus}}
          {\Manifold}
      }
  $ 
  and hence by Lemma~\ref{lem:HomotopyInverseInclusion}, we conclude that the 
  inclusion
  $
    \SpaceAlmostFlatSubsurfaces
      {+}
      {\TwoPlane}
      {\SurfaceG{\Genus}}
      {\Manifold}
    \to
    \SpaceSubsurfacesPointedRadiusTangential
      {+}
      {\TwoPlane}
      {\FixedRadius}
      {\SurfaceG{\Genus}}
      {\Manifold}
  $ 
  followed by $\HomotopyInverse{\FixedRadius,2\FixedRadius}$ is a homotopy 
  equivalence with homotopy inverse 
  $\apply{\Homotopy}{-,1}$.
  We get the following diagram:
  \begin{center}
    \begin{tikzcd}
      &
      \SpaceSubsurfacesPointedRadiusTangential
        {+}
        {\TwoPlane}
        {2\FixedRadius}
        {\SurfaceG{\Genus}}
        {\Manifold}
        \ar[rd,"\apply{\Homotopy}{-,1}","\simeq"']
        \ar[ld,hook',"\simeq"]
      \\
      \SpaceSubsurfacesPointedRadiusTangential
        {+}
        {\TwoPlane}
        {\FixedRadius}
        {\SurfaceG{\Genus}}
        {\Manifold}
        \ar[ur,"\HomotopyInverse{\FixedRadius,2\FixedRadius}",bend left]
      &
      &
      \SpaceAlmostFlatSubsurfaces
        {+}
        {\TwoPlane}
        {\SurfaceG{\Genus}}
        {\Manifold}
      \ar[ll,hook']
    \end{tikzcd}
  \end{center}
  This diagram shows that 
  $
    \apply{\Homotopy}{-,1}
    \circ 
    \HomotopyInverse{\FixedRadius,2\FixedRadius}
  $
  is a homotopy inverse to the inclusion of 
  $
    \SpaceAlmostFlatSubsurfaces
      {+}
      {\TwoPlane}
      {\SurfaceG{\Genus}}
      {\Manifold}
  $
  into
  $
    \SpaceSubsurfacesPointedRadiusTangential
      {+}
      {\TwoPlane}
      {\FixedRadius}
      {\SurfaceG{\Genus}}
      {\Manifold}
  $
  and we denote it by $\HomotopyInverse{2}$.
\end{proof}

  Lastly we have the following lemma which finishes the proof of
  Proposition~\ref{prp:FlatHomotopyEquivalence}.
  \begin{lemma}
    The inclusion 
    $
      \SpaceSubsurfacesPointedFlatTangential
        {+}
        {\TwoPlane}
        {\SurfaceG{\Genus}}
        {\Manifold}
      \to
      \SpaceAlmostFlatSubsurfaces
        {+}
        {\TwoPlane}
        {\SurfaceG{\Genus}}
        {\Manifold}
    $
    is a homotopy equivalence with homotopy inverse denoted by 
    $\HomotopyInverse{3}$.
\end{lemma}
\begin{proof}
  Consider the function 
  $
    \AuxiliaryFunction 
    \colon
    \SpaceAlmostFlatSubsurfaces
      {+}
      {\TwoPlane}
      {\SurfaceG{\Genus}}
      {\Manifold}
    \to 
    \Reals
  $ 
  that sends $\Subsurface$ to 
  $
    \inf_{\Point'\in \Subsurface\setminus \DiskRadius{\FixedRadius}} 
    \EuclideanNorm{\Point'}
  $%
  . Note that $\AuxiliaryFunction$ is bigger than zero for every surface and 
  bounded by $\FixedRadius$. 
  Assume for now that this function is continuous. 
  We define 
  $
    \AuxiliaryFunctionDiffeomorphism{3}
    \colon 
    [0,\FixedRadius]
    \times 
    [0,1]
    \to
    \DiffeomorphismGroupCompactSupport{\Reals^{\DimensionIndex}}
  $ 
  as follows: 
  $
    \apply
      {\AuxiliaryFunctionDiffeomorphism{3}}
      {t_{1},t_{2}}
  $ 
  is the diffeomorphism that is the identity outside the ball of radius 
  $2\FixedRadius$ and the identity for 
  $t_2=0$, and 
  $
    \apply
      {\AuxiliaryFunctionDiffeomorphism{3}}
      {t_{1},t_{2}}
  $ 
  sends the ball of radius $t_1$ bijectively onto 
  the ball of radius $\FixedRadius t_2+(1-t_2)t_1$ and finally only acts on the 
  norm component if an element in $\Reals^{\DimensionIndex}$ is represented in 
  polar coordinates. 
  We then define 
  \begin{align*}
    \Homotopy
    \colon 
    \SpaceAlmostFlatSubsurfaces
      {+}
      {\TwoPlane}
      {\SurfaceG{\Genus}}
      {\Manifold}
    \times 
    [0,1]
    &
    \to 
    \SpaceAlmostFlatSubsurfaces
      {+}
      {\TwoPlane}
      {\SurfaceG{\Genus}}
      {\Manifold}
    \\
    (\Subsurface,t)
    &
    \mapsto
    \apply{
      \apply
        {\AuxiliaryFunctionDiffeomorphism{3}}
        {\frac{\apply{\AuxiliaryFunction}{\Subsurface}}{2},t}
    }
    {\Subsurface}
  \end{align*}
  Then $\apply{\Homotopy}{-,1}$ is a homotopy inverse to the inclusion of 
  $
    \SpaceSubsurfacesPointedFlatTangential
      {+}
      {\TwoPlane}
      {\SurfaceG{\Genus}}
      {\Manifold}
  $
  into 
  $
    \SpaceAlmostFlatSubsurfaces
      {+}
      {\TwoPlane}
      {\SurfaceG{\Genus}}
      {\Manifold}
  $ 
  by Lemma~\ref{lem:HomotopyInverseInclusion}. We denote it by 
  $\HomotopyInverse{3}$.
  
  So the only thing left to prove is that 
  $
    \AuxiliaryFunction
  $ is continuous. 
  Let 
  $
    \SpaceAlmostFlatEmbeddings
    {+}
    {\TwoPlane}
    {\SurfaceG{\Genus}}
    {\Manifold}
  $ 
  denote the subspace of those embeddings in 
  $
    \EmbeddingSpacePointedTangential
      {+}
      {\TwoPlane}
      {\SurfaceG{\Genus}}
      {\Manifold}
  $ 
  whose $\DiffeoMorphismGroupPointed{\Point}{\SurfaceG{\Genus}}$-equivalence 
  class lie in 
  $
    \SpaceAlmostFlatSubsurfaces
      {+}
      {\TwoPlane}
      {\SurfaceG{\Genus}}
      {\Manifold}
  $ 
  and we denote by $\Projection{\Emb}$ the 
  projection of this space to 
  $
    \SpaceAlmostFlatSubsurfaces
      {+}
      {\TwoPlane}
      {\SurfaceG{\Genus}}
      {\Manifold}
  $%
  . Let 
  $
    \Evaluation
    \colon 
    \MappingSpace{\TopologicalSpace_{1}}{\TopologicalSpace_{2}}
    \times 
    \TopologicalSpace_{1}
    \to 
    \TopologicalSpace_{2}
  $ 
  denote the canonical evaluation map. 
  Let 
  $
    \bar{\Submanifold}
    =
    \apply
      {\Evaluation^{-1}}
      {\Complement{\DiskRadius{\FixedRadius}}}
    \subset
    \SpaceAlmostFlatEmbeddings
      {+}
      {\TwoPlane}
      {\SurfaceG{\Genus}}
      {\Manifold}
    \times
    \SurfaceG{\Genus}
  $%
  . 
  $\DiffeoMorphismGroupPointed{\Point}{\SurfaceG{\Genus}}$ acts on 
  $
    \SpaceAlmostFlatEmbeddings
      {+}
      {\TwoPlane}
      {\SurfaceG{\Genus}}
      {\Manifold}
  $ 
  as well as $\bar{\Submanifold}$ (Here we let 
  $\DiffeoMorphismGroupPointed{\Point}{\SurfaceG{\Genus}}$ from the left on the 
  $\SurfaceG{\Genus}$ component of $\bar{\Submanifold}$). 
  We denote the quotient of this action of $\bar{\Submanifold}$ by 
  $\Submanifold$.
  
  The projection onto the first factor from $\bar{\Submanifold}$ to 
  $
    \SpaceAlmostFlatEmbeddings
      {+}
      {\TwoPlane}
      {\SurfaceG{\Genus}}
      {\Manifold}
  $
  descends to a map from $\Submanifold$ to 
  $
    \SpaceAlmostFlatSubsurfaces
      {+}
      {\TwoPlane}
      {\SurfaceG{\Genus}}
      {\Manifold}
  $ 
  which we denote by $\Projection{\Submanifold}$ and we claim that
  this map is a fiber bundle. 
  Let 
  $
    \Subsurface
    \in 
    \SpaceAlmostFlatSubsurfaces
      {+}
      {\TwoPlane}
      {\SurfaceG{\Genus}}
      {\Manifold}
  $
  denote an embedded subsurface and 
  $\TubularNeighbourhood{\Subsurface}{\Manifold}$ 
  a tubular neighborhood of $\Subsurface$ with corresponding projection 
  $\Projection{\Subsurface}$.
  Consider the following set 
  \[
    \Neighbourhood{\Subsurface}
    =
    \left\{
      \Subsurface'
      \in 
      \SpaceAlmostFlatSubsurfaces
        {+}
        {\TwoPlane}
        {\SurfaceG{\Genus}}
        {\Manifold}
    \middle| 
      \Subsurface'
      \subset
      \TubularNeighbourhood{\Subsurface}{\Manifold} 
      \text{ and } 
      \at{\Projection{\Subsurface}}{\Subsurface'} 
      \text{ is a diffeomorphism}
    \right\}
  \]
  which is open because $\TubularNeighbourhood{\Subsurface}{\Manifold}$ is open 
  and the space of diffeomorphisms is an open subset of 
  $\MappingSpace{\Subsurface'}{\Subsurface}$ (see  Section~1.2 of the second 
  chapter of \cite{cerf}). 
  We will proof that $\Projection{\Submanifold}$ is a locally trivial 
  fibration. 
  Note that
    \begin{align*}
    \apply{\Projection{\Submanifold}^{-1}}{\Neighbourhood{\Subsurface}}
    &
    \to 
    \apply{\Projection{\Submanifold}^{-1}}{\{\Subsurface\}}
    \times 
    \Neighbourhood{\Subsurface}
    &
    \text{ and \phantom{gur}} 
    &
    \apply{\Projection{\Submanifold}^{-1}}{\{\Subsurface\}}
    \times 
    \Neighbourhood{\Subsurface}
    &
    \to 
    \apply{\Projection{\Submanifold}^{-1}}{\Neighbourhood{\Subsurface}}
    \\
    [\Embedding,\Point']
    &
    \mapsto 
    \left(
      [\Embedding_{\Subsurface},\apply{\Projection{\Subsurface}}{\Point'}]
      ,
      [\Embedding]
    \right) 
    &
    &
    \left(
      [\Embedding_{\Subsurface},\Point']
      ,
      [\Embedding]
    \right)
    &
    \mapsto 
    [
      \Embedding
      ,
      \apply
        {\at{\Projection{\Subsurface}}{\Subsurface'}^{-1}}
        {\Point'}
    ]
  \end{align*}
  are inverses of another.
  Here $\Embedding_{\Subsurface}$ is any representative of $\Subsurface$. 
  Therefore we conclude that $\Projection{\Submanifold}$ is a locally trivial 
  fibration, whose fiber is given by $\SurfaceG{\Genus}$ with an 
  open disc removed.
  
  So we have the following diagram:
  \[
    \begin{tikzcd}
      \Submanifold
        \ar[d,"\Projection{\Submanifold}"]
        \ar[r,"\Evaluation"]
      &
      \Manifold
        \ar[r,"\EuclideanNorm{-}"]
      &
      \Reals\cup\{\infty\}
      \\
      \SpaceAlmostFlatSubsurfaces
        {+}
        {\TwoPlane}
        {\SurfaceG{\Genus}}
        {\Manifold}
        \ar[rr,"\AuxiliaryFunction"]
      &
      &
      \Reals\cup\{\infty\}
    \end{tikzcd}
  \]
  where $\Evaluation$ denotes again the evaluation map, which descends from 
  $
    \SpaceAlmostFlatEmbeddings
      {+}
      {\TwoPlane}
      {\SurfaceG{\Genus}}
      {\Manifold}
    \times 
    \SurfaceG{\Genus}
  $ 
  to $\Submanifold$ and $\EuclideanNorm{-}$ denotes the 
  norm mapping in the chart $\CoordinateChart$ extended by $\infty$ to the 
  complement of $\CoordinateChart$. 
  We give $\Reals\cup \{\infty\}$ the topology coming from open sets 
  in $\Reals$ and neighborhoods of $\infty$ have the form $(C,\infty]$. 
  Note that $\apply{\AuxiliaryFunction}{\Subsurface}$ is given by taking the 
  infimum over the composition of the horizontal maps for all elements in 
  $
    \apply{\Projection{\Submanifold}}{\Subsurface}
  $ 
  and this always lands in $(0,\FixedRadius]$. 
  Since $\Projection{\Submanifold}$ is locally a trivial fibration, 
  Lemma~\ref{lem:InfimumContinuous} shows that $\AuxiliaryFunction$ is locally 
  continuous and therefore continuous.
\end{proof}

  This finishes the proof of Proposition~\ref{prp:FlatHomotopyEquivalence}.
  \begin{remark}
    At no point in the proof did we use that the submanifolds were 
    two-dimensional. The corollary would also hold true for higher dimensional 
    submanifolds.
  \end{remark}
  To proceed with the construction of the stabilization maps we need the 
  following corollary of our previous considerations:
  \begin{corollary}
\label{crl:TakingOutBallHomotopyEquivalence}
  There is a homotopy equivalence
  \[
    \BallRemovalMap
    \colon
    \SpaceSubsurfacesPointedFlatTangential
      {\TangentialFibration}
      {\TwoPlane}
      {\SurfaceG{\Genus}}
      {\Manifold}
    \to
    \SpaceSubsurfaceBoundaryTangential
      {\TangentialFibration}
      {\Genus}
      {1}
      {
        \overline
          {
            \Manifold
            \setminus 
            \BallRadius{\frac{\FixedRadius}{2}}{\Point_{\Manifold}}
          }
      }
      {\BoundaryConditionTangential{\BoundaryCondition}}
  \]
  where $\BallRadius{\frac{\FixedRadius}{2}}{\Point_{\Manifold}}$ is the 
  preimage of
  $
    \BallRadius{\frac{\FixedRadius}{2}}{\Point_{\Manifold}}
    \subset
    \CoordinateChart
  $
  in $\Manifold$ and $\BoundaryConditionTangential{\BoundaryCondition}$ is the 
  preimage of the intersection of $\TwoPlane$
  with the sphere of radius $\frac{\FixedRadius}{2}$ in $\CoordinateChart$ with 
  the tangential structure stemming from the constant one. We denote its 
  homotopy inverse by $\HomotopyInverse{+\TwoPlane}$
\end{corollary}
\begin{proof}
  Removing $\BallRadius{\frac{\FixedRadius}{2}}{\Point_{\Manifold}}$ yields
  a map
  $
    \BallRemovalMap
    \colon
    \SpaceSubsurfacesPointedFlatTangential
      {\TangentialFibration}
      {\TwoPlane}
      {\SurfaceG{\Genus}}
      {\Manifold}
    \to
    \SpaceSubsurfaceBoundaryTangential
      {\TangentialFibration}
      {\Genus}
      {1}
      {
        \overline
        {
          \Manifold
          \setminus 
          \BallRadius{\frac{\FixedRadius}{2}}{\Point_{\Manifold}}
        }
      }
      {\BoundaryConditionTangential{\BoundaryCondition}}
  $
  its image is the subspace of
  $
    \SpaceSubsurfaceBoundaryTangential
      {\TangentialFibration}
      {\Genus}
      {1}
      {
        \overline
        {
          \Manifold
          \setminus 
          \BallRadius{\frac{\FixedRadius}{2}}{\Point_{\Manifold}}
        }
      }
      {\BoundaryConditionTangential{\BoundaryCondition}}
  $
  of those subsurfaces that intersect 
  $
    \CoordinateChart
    \setminus 
    \BallRadius{\frac{\FixedRadius}{2}}{0}
  $
  like 
  $
    \TwoPlane
    \cap 
    \CoordinateChart
    \setminus 
    \BallRadius{\frac{\FixedRadius}{2}}{0}
  $%
  , and we will denote it by 
  $
    \SpaceSubsurfaceBoundaryTangential
      {\TangentialFibration}
      {\Genus}
      {1}
      {
        \overline
        {
          \Manifold
          \setminus 
          \BallRadius{\frac{\FixedRadius}{2}}{\Point_{\Manifold}}
        }
      }
      {\BoundaryConditionTangential{\BoundaryCondition}}
    ^{\text{flat}}
  $%
  . Evidently $\BallRemovalMap$ is a homeomorphism onto its image. 
  Furthermore the inclusion 
  $
    \SpaceSubsurfaceBoundaryTangential
      {\TangentialFibration}
      {\Genus}
      {1}
      {
        \overline
        {
          \Manifold
          \setminus 
          \BallRadius{\frac{\FixedRadius}{2}}{\Point_{\Manifold}}
        }
      }
      {\BoundaryConditionTangential{\BoundaryCondition}}
    ^{\text{flat}}
    \to
    \SpaceSubsurfaceBoundaryTangential
      {\TangentialFibration}
      {\Genus}
      {1}
      {
        \overline
        {
          \Manifold
          \setminus 
          \BallRadius{\frac{\FixedRadius}{2}}{\Point_{\Manifold}}
        }
      }
      {\BoundaryConditionTangential{\BoundaryCondition}}
  $
  is a homotopy equivalence, as can be seen using collars.
\end{proof}

  We define the \introduce{pointed stabilization maps}
$\PointedStabilizationMap{\Genus}$ 
via the composition indicated in the following diagram:
\[
  \begin{tikzcd}
    \SpaceSubsurfacesPointedTangential
      {\TangentialFibration}
      {\TwoPlane}
      {\SurfaceG{\Genus}}
      {\Manifold}
      \ar[r,"\HomotopyInverse{\text{flat}}"]
      \ar[dd,"\PointedStabilizationMap{\Genus}"]
    &
    \SpaceSubsurfacesPointedFlatTangential
      {+}
      {\TwoPlane}
      {\SurfaceG{\Genus}}
      {\Manifold}
      \ar[r,"\BallRemovalMap"]
    &
    \SpaceSubsurfaceBoundaryTangential
      {\TangentialFibration}
      {\Genus}
      {1}
      {
        \overline
        {
          \Manifold
          \setminus 
          \BallRadius{\frac{\FixedRadius}{2}}{\Point_{\Manifold}}
        }
      }
      {\BoundaryConditionTangential{\BoundaryCondition}}
      \ar[d,"\StabilizationBeta{\Genus}{1}"]
    \\
    &
    &
    \SpaceSubsurfaceBoundaryTangential
      {\TangentialFibration}
      {\Genus}
      {2}
      {
        \Elongation
        {
          \overline
          {
            \Manifold
            \setminus 
            \BallRadius{\frac{\FixedRadius}{2}}{\Point_{\Manifold}}
          }
        }
        {1}
      }
      {\BoundaryConditionTangential{\BoundaryCondition'}}
      \ar[d,"\StabilizationAlpha{\Genus}{2}"]
    \\
    \SpaceSubsurfacesPointedTangential
      {\TangentialFibration}
      {\TwoPlane}
      {\SurfaceG{\Genus+1}}
      {\Manifold}
    &
    \SpaceSubsurfacesPointedFlatTangential
      {+}
      {\TwoPlane}
      {\SurfaceG{\Genus+1}}
      {\Manifold}
      \ar[l,hook]
    &
    \SpaceSubsurfaceBoundaryTangential
      {\TangentialFibration}
      {\Genus}
      {1}
      {
        \Elongation
        {
          \overline
          {
            \Manifold
            \setminus 
            \BallRadius{\frac{\FixedRadius}{2}}{\Point_{\Manifold}}
          }
        }
        {2}
      }
      {\BoundaryConditionTangential{\BoundaryCondition}}
      \ar[l,"\HomotopyInverse{+\TwoPlane}"]
  \end{tikzcd} 
\]
Note that we required the boundary condition of the lower right space of 
subsurfaces to agree with the boundary condition of the top right space.

  In the pointed case, we say that a space of $\TangentialFibration$-structures 
  $\pi_0$-stabilizes at $\Genus$ if $\PointedStabilizationMap{\Genus'}$ induces 
  an isomorphism on $\pi_0$ provided that $\Genus'\geq\Genus$ and a surjection 
  for all $\Genus'\geq \Genus-1$. 
  Since $\PointedStabilizationMap{\Genus}$ is defined in terms of homotopy 
  equivalences (Corollary~\ref{crl:FlatHomotopyEquivalenceTangential} and 
  Corollary~\ref{crl:TakingOutBallHomotopyEquivalence}) and stabilization maps 
  of type $\alpha$ and $\beta$, one immediately obtains (using 
  Theorem~\ref{thm:HomologicalStabilityFiberSimplyConnected} for the second 
  part):
  \begin{theorem}
  \label{thm:HomologicalStabilityPointed}
  Suppose $(\Manifold,\Point_{\Manifold})$ is an at least $5$-dimensional, 
  simply-connected, pointed manifold with a $\TrivialityIndex$-trivial space of 
  $\TangentialFibration$-structures of subplanes of $\TangentBundle{\Manifold}$ 
  $
    \TangentialFibration
    \colon
    \TangentialSpace{\Manifold}
    \to
    \Grassmannian{2}{\TangentBundle{\Manifold}}
  $
  which $\pi_0$ stabilizes at $\PiZeroStabilization$, then
  the homology of $\MappingPair{\PointedStabilizationMap{\Genus}}$ 
  vanishes in degrees 
  $
    \ast
    \leq 
    \min
      \left\{
        \apply{\AlphaBound}{\Genus}
        ,
        \apply{\BetaBound}{\Genus}
      \right\}
  $
  where $\AlphaBound$ and $\BetaBound$ are defined as in  
  Definition~\ref{dfn:StabilityBounds}.

  In particular if $\Fiber{\TwoPlane}{\TangentialFibration}$ is 
  simply-connected, then the homology of
  $\MappingPair{\PointedStabilizationMap{\Genus}}$ vanishes in degrees
  $
    \ast
    \leq
    \left\lfloor
      \frac{2}{3}\Genus
    \right\rfloor
  $%
  .
\end{theorem}
\section{Homological Stability for Symplectic Subsurfaces}
  \label{scn:SymplecticSubsurfaces}
In this last section we will explain how to use 
Theorem~\ref{thm:HomologicalStabilityPointed} to prove homological stability 
for pointed symplectic subsurfaces of a given closed simply-connected 
symplectic manifold $(\Manifold,\SymplecticForm)$ (i.e. $\SymplecticForm$ is a 
point-wise non-degenerate $2$-form on $\Manifold$) of dimension 
$2\DimensionIndex\geq 6$.
The proof will use a classical theorem, that is proven via the h-principle, to 
relate symplectic subsurfaces to tangential structures as they were defined in 
the previous parts of this paper. 

Let 
$
  \ProjectionGrassmannian
  \colon
  \Grassmannian{2}{\TangentBundle{\Manifold}}
  \to
  \Manifold
$
denote the Grassmannian of oriented two-planes in $\TangentBundle{\Manifold}$ 
and $\SurfaceG{\Genus}$ an oriented, connected, closed surface of genus 
$\Genus$.
Fix 
$\Point \in \SurfaceG{\Genus}$, $\Point_{\Manifold}\in \Manifold$ and 
$
  \SymplecticTwoPlane
  \in
  \Grassmannian{2}{\TangentSpace{\Point_{\Manifold}}{\Manifold}}
$
such that $\at{\SymplecticForm}{\SymplecticTwoPlane}$ is 
non-degenerate and the orientation of $\SymplecticTwoPlane$ agrees with the 
orientation 
induced by $\at{\SymplecticForm}{\SymplecticTwoPlane}$.
\begin{definition}
  \label{dfn:PointedSymplecticEmbedding}
  We say an embedding 
  $
    \Embedding
    \colon
    \SurfaceG{\Genus}
    \to
    \Manifold
  $
  is an \introduce{oriented, pointed symplectic embedding} if 
  \begin{enumerate}[(i)]
  \item 
    $
      \apply{\Embedding}{\Point}=\Point_{\Manifold}
    $
    and 
    $
      \apply
        {\Differential{\Embedding}}
        {\TangentSpace{\Point}{\SurfaceG{\Genus}}}
      =
      \SymplecticTwoPlane
    $
  \item 
    $
      \at
        {\SymplecticForm}
        {\apply
          {\Differential{\Embedding}}
          {\TangentBundle{\SurfaceG{\Genus}}}
        }
    $
    is non-degenerate.
  \item
  \label{itm:PositiveSymplecticForm}
    $
      \int_{\SurfaceG{\Genus}} 
      \Embedding^{*}\SymplecticForm
      >
      0
    $
  \end{enumerate}
  We equip the set of symplectic embeddings with the $C^\infty$-topology and we 
  denote this space by 
  $\SpaceSymplecticEmbeddings{\SymplecticTwoPlane}{\SurfaceG{\Genus}}{\Manifold}$.
  
  Similarly we say that a smooth map 
  $
    \FormalSymplecticEmbedding
    \colon
    \SurfaceG{\Genus}
    \to
    \Grassmannian{2}{\TangentBundle{\Manifold}}
  $ 
  is a \introduce{formal solution to the oriented symplectic embedding problem}
  if 
  \begin{enumerate}[(i)]
  \item 
    $\ProjectionGrassmannian \circ \FormalSymplecticEmbedding$ is an embedding 
    such that 
    $
      \apply{(\ProjectionGrassmannian\circ \FormalSymplecticEmbedding)}{\Point}
      =
      \Point_{\Manifold}
    $
    and
    $
      \apply
        {\Differential{(\ProjectionGrassmannian\circ\FormalSymplecticEmbedding)}}
        {\TangentSpace{\Point}{\SurfaceG{\Genus}}}
      =
      \apply{\FormalSymplecticEmbedding}{\Point}
      =
      \SymplecticTwoPlane
    $
  \item 
    $
      \at
        {\SymplecticForm}
        {\apply{\FormalSymplecticEmbedding}{\Point'}}$ is non-degenerate for 
        every point $\Point'\in\SurfaceG{\Genus}$
  \item 
    $
      \int_{\SurfaceG{\Genus}} 
      (
        \ProjectionGrassmannian
        \circ
        \FormalSymplecticEmbedding
      )^*
      \SymplecticForm
      >
      0
    $
  \end{enumerate}
\end{definition}

Note that the last condition ensures that the orientations of $\Sigma_g$ 
induced by $\omega$ and by the orientation of $\Sigma_g$ agree.
\begin{remark}
  For an oriented, pointed symplectic embedding 
  $
    \Embedding
    \colon 
    \SurfaceG{\Genus}
    \to 
    \Manifold
  $%
  , the Grassmannian differential 
  $
    \GrassmannianDifferential{\Embedding}
    \colon
    \SurfaceG{\Genus}
    \to
    \Grassmannian{2}{\TangentBundle{\Manifold}}
  $
  is a formal solution to the oriented symplectic embedding problem.
\end{remark}
\begin{definition}
  We call a continuous map 
  $
    \SolutionHomotopy
    \colon
    \SurfaceG{\Genus}
    \times 
    [0,1]
    \to 
    \Grassmannian{2}{\TangentBundle{\Manifold}}
  $ 
  a \introduce{solution homotopy} if 
  \begin{enumerate}[(i)]
  \item
    There exists an embedding 
    $
      \Embedding
      \colon
      \SurfaceG{\Genus}
      \to
      \Manifold
    $ 
    such that 
    $
      \apply{\SolutionHomotopy}{-,0}
    $
    agrees with $\GrassmannianDifferential{\Embedding}$, 
    $\apply{\Embedding}{\Point}=\Point_{\Manifold}$ and 
    $
      \apply
        {\Differential{\Embedding}}
        {\TangentSpace{\Point}{\Manifold}}
      =
      \SymplecticTwoPlane
    $ 
  \item 
    The following diagram commutes: 
    \[
      \begin{tikzcd}
        \SurfaceG{\Genus}
        \times
        [0,1]
          \ar[d,"\Projection{\SurfaceG{\Genus}}"]
          \ar[r,"\SolutionHomotopy"]
        &
        \Grassmannian{2}{\TangentBundle{\Manifold}}
          \ar[d,"\ProjectionGrassmannian"]
        \\
        \SurfaceG{\Genus}
          \ar[r,"\Embedding"]
        &
        \Manifold
      \end{tikzcd}
    \]
  \item
    For all $t\in[0,1]$, we have
    $\apply{\SolutionHomotopy}{\Point,t}=\SymplecticTwoPlane$
  \item
    $\apply{\SolutionHomotopy}{-,1}$ is a formal solution to the oriented 
    symplectic embedding problem
  \end{enumerate}
  The set of solution homotopies is a subspace of 
  $
    \EmbeddingSpace{\SurfaceG{\Genus}}{\Manifold}
    \times
    \MappingSpace
      {\SurfaceG{\Genus}}
      {\MappingSpace{[0,1]}{\Grassmannian{2}{\TangentBundle{\Manifold}}}}
  $
  and hence inherits a topology, where the first factor is equipped with the 
  $C^\infty$ topology and the second one with the compact-open topology. 
\end{definition}
Note that there is an inclusion from the space of symplectic embeddings into 
the space of solution homotopies by sending an embedding to the constant 
solution homotopy over that embedding. The proof of the following theorem can 
be found in Section~12 of \cite{Eliashberg}.
\begin{theorem}
\label{thm:FormalSymplectichPrinciple}
  If 
  $
    (\Manifold,\SymplecticForm)$ is an at least six-dimensional symplectic 
    manifold, then the inclusion of the symplectic embeddings into the space of 
    solution homotopies is a weak equivalence. 
\end{theorem}

\subsection{Interpreting Solution Homotopies as Tangential Structures:}
Inspired by Theorem \ref{thm:FormalSymplectichPrinciple}, we want to construct 
a space of $\TangentialFibration$-structures of subplanes of 
$\TangentBundle{\Manifold}$ such that the space of subsurfaces with that 
tangential structures is the space of solution homotopies.

We call an element 
$
  \SymplecticTwoPlane
  \in
  \Grassmannian{2}{\TangentBundle{\Manifold}}
$
a \introduce{symplectic $2$-plane} if 
$\at{\SymplecticForm}{\SymplecticTwoPlane}$ is non-degenerate and we denote by 
$\SymplecticGrassmannian{2}{\TangentBundle{\Manifold}}$ the (open) subspace of 
symplectic two-planes of $\TangentBundle{\Manifold}$. 
Analogously we define 
$\SymplecticGrassmannian{2}{\TangentSpace{\Point'}{\Manifold}}$.
For a $\SymplecticTwoPlane'\in 
\Grassmannian{2}{\TangentSpace{\Point'}{\Manifold}}$
we will denote the homotopy fiber of the inclusion 
$
  \SymplecticGrassmannian{2}{\TangentSpace{\Point'}{\Manifold}}
  \to
  \Grassmannian{2}{\TangentSpace{\Point'}{\Manifold}}
$
by 
$\SymplecticPathSpace{2}{\Manifold}_{\SymplecticTwoPlane'}$. 
More explicitly this is the space of paths in 
$\Grassmannian{2}{\TangentSpace{\Point'}{\Manifold}}$ starting at 
$\SymplecticTwoPlane'$ 
and ending in $\SymplecticGrassmannian{2}{\TangentSpace{\Point'}{\Manifold}}$. 
If we allow $\SymplecticTwoPlane'$ to vary, we obtain the mapping path space 
$\SymplecticPathSpace{2}{\Manifold}_{\Point'}$ of the inclusion 
$
  \SymplecticGrassmannian{2}{\TangentSpace{\Point'}{\Manifold}}
  \to 
  \Grassmannian{2}{\TangentSpace{\Point'}{\Manifold}}
$%
, i.e. 
\[
  \SymplecticPathSpace{2}{\Manifold}_{\Point'}
  \coloneqq
  \left\{
    \Path
    \colon 
    [0,1]
    \to 
    \Grassmannian{2}{\TangentSpace{\Point'}{\Manifold}}
  \middle| 
    \apply{\Path}{1}
    \in 
    \SymplecticGrassmannian{2}{\TangentSpace{\Point'}{\Manifold}}
  \right\}
\]
This space is a subset of the following space, which is equipped with the 
compact open topology
\[
  \SymplecticPathSpace{2}{\Manifold}
  \coloneqq
  \left\{
    \Path
    \colon
    [0,1]
    \to
    \Grassmannian{2}{\TangentBundle{\Manifold}}
  \middle|
    \apply{\Path}{1}
    \in
    \SymplecticGrassmannian{2}{\TangentBundle{\Manifold}}
    \text{ and }
    \ProjectionGrassmannian \circ \Path
    \text{ is constant}
  \right\}
\]
If we trivialize $\Grassmannian{2}{\TangentBundle{\Manifold}}\to \Manifold$ via 
a Darboux chart defined on $U\subset \Manifold$, we conclude that the inclusion 
$
  \SymplecticGrassmannian{2}{\TangentBundle{\Manifold}}
  \to 
  \Grassmannian{2}{\TangentBundle{\Manifold}}
$
is locally equivalent to 
$
  \SymplecticGrassmannian{2}{\TangentSpace{\Point'}{\Manifold}}
  \times
  U
  \to
  \Grassmannian{2}{\TangentSpace{\Point'}{\Manifold}}
  \times
  U
$%
. This implies that 
\begin{align*}
  \SymplecticTangentialStructure
  \colon
  \SymplecticPathSpace{2}{\Manifold}
  &
  \to 
  \Grassmannian{2}{\TangentBundle{\Manifold}}
  \\
  \Path
  &
  \mapsto 
  \apply{\Path}{0}
\end{align*}
is locally equivalent to 
$
  \SymplecticPathSpace{2}{\Manifold}_{\Point'}
  \times
  U
  \to
  \Grassmannian{2}{\TangentSpace{\Point'}{\Manifold}}
  \times
  U
$%
, which implies that this map is a Hurewicz-fibration as it is the product of 
two Hurewicz-fibrations, one being the projection of the mapping path space 
$
  \SymplecticPathSpace{2}{\Manifold}_{\Point'}
  \to 
  \Grassmannian{2}{\TangentSpace{\Point'}{\Manifold}}
$ 
and the other one being the identity.
Using Theorem~13 in Chapter~2.7 of \cite{spanier}, we conclude that it is 
indeed a Hurewicz-fibration as it is locally a Hurewicz-fibration and the 
basespace $\Grassmannian{2}{\TangentBundle{\Manifold}}$ is paracompact.
\begin{lemma}
  The fiber of 
  \begin{align*}
    \SymplecticTangentialStructure
    \colon
    \SymplecticPathSpace{2}{\Manifold}
    &
    \to 
    \Grassmannian{2}{\TangentBundle{\Manifold}}
    \\
    \Path
    &
    \mapsto 
    \apply{\Path}{0}
  \end{align*} 
  is simply-connected.
\end{lemma}
\begin{proof}
  The claim that $\SymplecticTangentialStructure$ is indeed a fibration and 
  hence a space of $\SymplecticTangentialStructure$-structures of subplanes was 
  just shown.
  Using Proposition~\ref{prp:ConnectedKTriviality} and 
  Corollary~\ref{crl:PiZeroStabilizes} it suffices to show that the fiber of 
  $\SymplecticTangentialStructure$ is simply-connected:
  
  The fiber over an oriented 2-plane $\SymplecticTwoPlane'$ in 
  $\TangentSpace{\Point'}{\Manifold}$ is given by the set 
  of paths in $\Grassmannian{2}{\TangentSpace{\Point'}{\Manifold}}$ from 
  $\SymplecticTwoPlane'$ to 
  $\SymplecticGrassmannian{2}{\TangentSpace{\Point'}{\Manifold}}$.
  As we have mentioned before, this is the homotopy fiber of 
  $
    \SymplecticGrassmannian{2}{\TangentSpace{\Point'}{\Manifold}}
    \to
    \Grassmannian{2}{\TangentSpace{\Point'}{\Manifold}}
  $%
  .
  Choose some almost complex structure $J$ on $\TangentBundle{\Manifold}$ 
  compatible with the symplectic form $\SymplecticForm$.
  Let 
  $V^{\Reals}_2(\TangentSpace{\Point'}{\Manifold})$ denote 
  the compact $2$-Stiefel manifold i.e. the space of real (orthogonal) 
  $2$-frames of $\TangentSpace{\Point'}{\Manifold}$ and let 
  $V^{symp}_2(\TangentSpace{\Point'}{\Manifold})$ denote the space of 
  $2$-frames $(v_1,v_2)$ such that 
  $
    \apply{\SymplecticForm}{v_1,v_2}
    =
    1
  $%
  . Furthermore define $V^{\mathbb{C}}_1(\TangentSpace{\Point'}{\Manifold})$ as 
  the compact Stiefel manifold of complex $1$-frames (i.e. just non-zero 
  vectors).
  Then we have an inclusion 
  \begin{align*}
    V^{\mathbb{C}}_1(\TangentSpace{\Point'}{\Manifold})
    &
    \to 
    V^{symp}_2(\TangentSpace{\Point'}{\Manifold})
    \\
    v
    &
    \mapsto 
    (v,\apply{J}{v})
  \end{align*}
  and we claim that this is a homotopy equivalence with homotopy inverse given 
  by the projection on the first vector.
  Indeed the homotopy is given by 
  \[
    ((v_1,v_2),t) \mapsto (v_1,(1-t)v_2+tJ(v_1))
  \] 
  which is easily seen to be well-defined. 
  Therefore we get the following maps between fiber sequences:
  \begin{center}
    \begin{tikzcd}
      S^1\simeq Sp(2,\Reals) 
        \ar[d]
        \ar[r] 
      &
      SO(2) \cong S^1
        \ar[d]
      \\
      V^\mathbb{C}_1(\TangentSpace{\Point'}{\Manifold})
      \simeq 
      V^{symp}_2(\TangentSpace{\Point'}{\Manifold}) 
        \ar[r]
        \ar[d]
      &
      V^{\Reals}_2(\TangentSpace{\Point'}{\Manifold})
        \ar[d]
      \\
      \SymplecticGrassmannian{2}{\TangentSpace{\Point'}{\Manifold}}
        \ar[r]
      &
      \Grassmannian{2}{\TangentSpace{\Point'}{\Manifold}}
    \end{tikzcd}
  \end{center}
  Note that 
  $
    V^\mathbb{C}_1(\TangentSpace{\Point'}{\Manifold})
    \cong 
    S^{2\DimensionIndex-1}
  $
  and 
  $
    V^{\Reals}_2(\TangentSpace{\Point'}{\Manifold})
    \cong 
    S\TangentBundle{S^{2\DimensionIndex-1}}
  $
  (the unit sphere bundle of the tangent bundle of $S^{2n-1}$). 
  Since both of these spaces are $(2\DimensionIndex-3)$-connected and the top 
  horizontal map is a homotopy equivalence, we conclude that the map from 
  $
    \SymplecticGrassmannian{2}{\TangentSpace{\Point'}{\Manifold}}
    \to
    \Grassmannian{2}{\TangentSpace{\Point'}{\Manifold}}
  $ 
  induces an isomorphism on $\pi_{i}$ for $i\leq 3$ or in other words 
  $
    \HomotopyGroup
      {i}
      {
        \Grassmannian{2}{\TangentSpace{\Point'}{\Manifold}}
        ,
        \SymplecticGrassmannian{2}{\TangentSpace{\Point'}{\Manifold}}
      }
  $
  for 
  $i\leq 2$, which implies that the homotopy fiber is two-connected, since
  \[
    \HomotopyGroup
      {i+1}
      {
        \Grassmannian{2}{\TangentSpace{\Point'}{\Manifold}}
        ,
        \SymplecticGrassmannian{2}{\TangentSpace{\Point'}{\Manifold}}
      }
    \cong
    \HomotopyGroup
      {i}
      {\HomotopyFiber
        {\SymplecticTwoPlane'}
        {
          \SymplecticGrassmannian{2}{\TangentSpace{\Point'}{\Manifold}}
          \to
          \Grassmannian{2}{\TangentSpace{\Point'}{\Manifold}}
        }
      }
  \]
\end{proof}

\subsection{Spaces of Symplectic Subsurfaces and Homological Stability:}
The group $\DiffeoMorphismGroupPointed{\Point}{\SurfaceG{\Genus}}$ acts freely 
on 
$\SpaceSymplecticEmbeddings{\SymplecticTwoPlane}{\SurfaceG{\Genus}}{\Manifold}$ 
via 
precomposition, since being a symplectic embedding is independent of the 
parametrization. 
We denote the quotient by this group action by 
$\SpaceSymplecticSubsurfaces{\SymplecticTwoPlane}{\SurfaceG{\Genus}}{\Manifold}$.

Elements of this space are referred to as \introduce{pointed symplectic 
subsurfaces}. 
By a slight abuse of notation, let $\SymplecticTwoPlane$ also denote the 
path in $\SymplecticPathSpace{2}{\Manifold}$ that is constantly 
$\SymplecticTwoPlane$. 
The inclusion of the space of symplectic embeddings into the space of solution 
homotopies descends to an inclusion 
\[
  \SpaceSymplecticSubsurfaces{\SymplecticTwoPlane}{\SurfaceG{\Genus}}{\Manifold}
  \to
  \SpaceSubsurfacesPointedTangential
    {\SymplecticTangentialStructure}
    {\SymplecticTwoPlane}
    {\SurfaceG{\Genus}}
    {\Manifold}
\]

Since Condition~(\ref{itm:PositiveSymplecticForm}) in 
Definition~\ref{dfn:PointedSymplecticEmbedding} implies that $\SymplecticForm$ 
evaluates positively on the fundamental class of pointed symplectic 
subsurfaces, this inclusion can not be $\pi_0$-surjective. In order 
to use this inclusion later on we will alter its image. 
We define:
\begin{align*}
  \EmbeddingSpacePointedTangential
    {\SymplecticTangentialStructure}
    {\SymplecticTwoPlane}
    {\SurfaceG{\Genus}}
    {\Manifold}
  ^{\SymplecticForm}
  &
  \coloneqq 
  \left\{ 
    \Embedding
    \in
    \EmbeddingSpacePointedTangential
      {\SymplecticTangentialStructure}
      {\SymplecticTwoPlane}
      {\SurfaceG{\Genus}}
      {\Manifold}
  \middle| 
    \int_{\SurfaceG{\Genus}} \Embedding^* \SymplecticForm >0 
  \right\}
  \\
  \SpaceSubsurfacesPointedTangential
    {\SymplecticTangentialStructure}
    {\SymplecticTwoPlane}
    {\SurfaceG{\Genus}}
    {\Manifold}
  ^{\SymplecticForm}
  &
  \coloneqq 
  \left\{
    \Subsurface
    \in 
    \SpaceSubsurfacesPointedTangential
      {\SymplecticTangentialStructure}
      {\SymplecticTwoPlane}
      {\SurfaceG{\Genus}}
      {\Manifold}
  \middle| 
    \int_{\Subsurface} \SymplecticForm >0
  \right\}
\end{align*}
With these definition the following corollary follows immediately:
\begin{corollary}
  $
    \EmbeddingSpacePointedTangential
      {\SymplecticTangentialStructure}
      {\SymplecticTwoPlane}
      {\SurfaceG{\Genus}}
      {\Manifold}
    ^{\SymplecticForm}
  $
  is the space of solution homotopies.
\end{corollary}

Observe that the stabilization map of 
$
\SpaceSubsurfacesPointedTangential
  {\TangentialFibration}
  {\SymplecticTwoPlane}
  {\SurfaceG{\Genus}}
  {\Manifold}
$
defined in Section \ref{scn:PointedStabilization} restricts to 
$
  \SpaceSubsurfacesPointedTangential
    {\TangentialFibration}
    {\SymplecticTwoPlane}
    {\SurfaceG{\Genus}}
    {\Manifold}
  ^{\SymplecticForm}
$ 
namely 
\[
  \PointedStabilizationMap{\Genus}
  \colon 
  \SpaceSubsurfacesPointedTangential
    {\TangentialFibration}
    {\SymplecticTwoPlane}
    {\SurfaceG{\Genus}}
    {\Manifold}
  ^{\SymplecticForm}
  \to 
  \SpaceSubsurfacesPointedTangential
    {\TangentialFibration}
    {\SymplecticTwoPlane}
    {\SurfaceG{\Genus+1}}
    {\Manifold}
  ^{\SymplecticForm}
\]
since for every 
$
  \Subsurface
  \in 
  \SpaceSubsurfacesPointedTangential
    {\TangentialFibration}
    {\SymplecticTwoPlane}
    {\SurfaceG{\Genus}}
    {\Manifold}
$ 
we have 
$
  [\Subsurface]
  =
  [\apply{\PointedStabilizationMap{\Genus}}{\Subsurface}]
  \in
  \HomologyOfSpace{2}{\Manifold}
$ and if $[\SymplecticForm]$ denotes the real cohomology 
class corresponding to $\SymplecticForm$, then 
$
  \int_{\Subsurface} 
  \SymplecticForm
  >
  0
$
is equivalent to 
$
  \apply{[\SymplecticForm]}{[\Subsurface]_{\Reals}}>0
$%
, where $[\Subsurface]_{\Reals}$ denotes the real fundamental class. 
Since 
$
  \SpaceSubsurfacesPointedTangential
    {\TangentialFibration}
    {\SymplecticTwoPlane}
    {\SurfaceG{\Genus}}
    {\Manifold}
  ^{\SymplecticForm}
$ 
is a union of connected components of 
$
\SpaceSubsurfacesPointedTangential
  {\TangentialFibration}
  {\SymplecticTwoPlane}
  {\SurfaceG{\Genus}}
  {\Manifold}
$
and
$
  (\PointedStabilizationMap{\Genus})_{*}
  \colon 
  \HomotopyGroup
    {0}
    {\SpaceSubsurfacesPointedTangential
      {\TangentialFibration}
      {\SymplecticTwoPlane}
      {\SurfaceG{\Genus}}
      {\Manifold}
    }
  \to
  \HomotopyGroup
  {0}
  {\SpaceSubsurfacesPointedTangential
    {\TangentialFibration}
    {\SymplecticTwoPlane}
    {\SurfaceG{\Genus+1}}
    {\Manifold}
  }
$
is an isomorphism, Theorem~\ref{thm:HomologicalStabilityPointed} implies that 
homological stability holds for 
$
  \SpaceSubsurfacesPointedTangential
    {\TangentialFibration}
    {\SymplecticTwoPlane}
    {\SurfaceG{\Genus}}
    {\Manifold}
  ^{\SymplecticForm}
$ 
as well i.e.
\[
  (\PointedStabilizationMap{\Genus})_{*}
  \colon 
  \HomologyOfSpace
  {\ast}
  {\SpaceSubsurfacesPointedTangential
    {\TangentialFibration}
    {\SymplecticTwoPlane}
    {\SurfaceG{\Genus}}
    {\Manifold}
  ^{\SymplecticForm}
  }
  \to
  \HomologyOfSpace
  {\ast}
  {\SpaceSubsurfacesPointedTangential
    {\TangentialFibration}
    {\SymplecticTwoPlane}
    {\SurfaceG{\Genus+1}}
    {\Manifold}
  ^{\SymplecticForm}
  }
\] 
is an isomorphism for $\ast\leq\frac{2}{3}\Genus-1$ and an epimorphism in the 
next degree.
\begin{remark}
  The same arguments show that even in the general case, homological stability 
  still holds for all kinds of homological constraints we could put on 
  $
  \SpaceSubsurfacesPointedTangential
    {\TangentialFibration}
    {\SymplecticTwoPlane}
    {\SurfaceG{\Genus+1}}
    {\Manifold}
  $ as long as these constraints depend only on 
  $\HomologyOfSpace{2}{\Manifold}$.
\end{remark}
The following proposition allows us to relate homological stability for the 
space of solution homotopies to 
$\SpaceSymplecticSubsurfaces{\SymplecticTwoPlane}{\SurfaceG{\Genus}}{\Manifold}$.
\begin{proposition}
  \label{prp:hPrincipleSpacesOfSubsurfaces}
  $
    \SpaceSymplecticSubsurfaces{\SymplecticTwoPlane}{\SurfaceG{\Genus}}{\Manifold}
    \to
    \SpaceSubsurfacesPointedTangential
      {\SymplecticTangentialStructure}
      {\SymplecticTwoPlane}
      {\SurfaceG{\Genus}}
      {\Manifold}
    ^{\SymplecticForm}
  $
  is a weak equivalence.
\end{proposition}
\begin{proof}
  We have the following map between fibrations:
  \[
    \begin{tikzcd}
      \DiffeoMorphismGroupPointed{\Point}{\SurfaceG{\Genus}}
        \ar[d]
        \ar[r,"\Identity"]
      &
      \DiffeoMorphismGroupPointed{\Point}{\SurfaceG{\Genus}}
        \ar[d]
      \\
      \SpaceSymplecticEmbeddings
        {\SymplecticTwoPlane}
        {\SurfaceG{\Genus}}
        {\Manifold}
        \ar[r]
        \ar[d]
      & 
      \EmbeddingSpacePointedTangential
        {\SymplecticTangentialStructure}
        {\SymplecticTwoPlane}
        {\SurfaceG{\Genus}}
        {\Manifold}
      ^{\SymplecticForm}
        \ar[d]
      \\
      \SpaceSymplecticSubsurfaces
        {\SymplecticTwoPlane}
        {\SurfaceG{\Genus}}
        {\Manifold}
        \ar[r]
      & 
      \SpaceSubsurfacesPointedTangential
        {\SymplecticTangentialStructure}
        {\SymplecticTwoPlane}
        {\SurfaceG{\Genus}}
        {\Manifold}
      ^{\SymplecticForm}
    \end{tikzcd}
  \]
  Note that the lower right map is a fibration by 
  \ref{lem:ForgettingTangentialFibrationPointed}, the lower map on the left 
  side is a fibration as a restriction of this fibration and the horizontal 
  arrow in the middle is a weak equivalence by 
  Theorem~\ref{thm:FormalSymplectichPrinciple}. 
  Therefore we conclude from the $5$-Lemma and the long exact sequence of 
  homotopy groups for a fibration that the inclusion is in fact a weak 
  equivalence.
\end{proof}

This proposition together with 
Theorem~\ref{thm:HomologicalStabilityPointed} 
and the previously discussed fact that 
$
  \SpaceSubsurfacesPointedTangential
    {\TangentialFibration}
    {\SymplecticTwoPlane}
    {\SurfaceG{\Genus}}
    {\Manifold}
  ^{\SymplecticForm}
$
also fulfils homological stability immediately implies:
\begin{theorem}
  Let $(\Manifold,\SymplecticForm)$ denote a simply-connected symplectic 
  manifold of dimension at least $6$. 
  For every $\Point_{\Manifold}\in \Manifold$ and every symplectic $2$-plane 
  $\SymplecticTwoPlane$ in $\TangentSpace{\Point_{\Manifold}}{\Manifold}$. 
  There is a homomorphism of integral homology
  \[
    \ContinuousMap_{\Genus}
    \colon
    \HomologyOfSpace
      {\ast}
      {\SpaceSymplecticSubsurfaces
        {\SymplecticTwoPlane}
        {\SurfaceG{\Genus}}
        {\Manifold}
      }
      \to
      \HomologyOfSpace
        {\ast}
        {\SpaceSymplecticSubsurfaces
          {\SymplecticTwoPlane}
          {\SurfaceG{\Genus+1}}
          {\Manifold}
        }
  \]
  And this homomorphism induces an isomorphism for 
  $\ast \leq \frac{2}{3}\Genus-1$ and an epimorphism for 
  $\ast\leq \frac{2}{3}\Genus$.
\end{theorem}

The main problem here is that the present text does not provide an actual 
stabilization map for symplectic subsurfaces in the spirit of the pointed 
stabilization maps $\PointedStabilizationMap{\Genus}$, but only an abstract one 
that uses the inverses of the aforementioned weak equivalence and therefore 
only exists on the level of homology. 

Nevertheless the author is confident that the construction in 
Section~\ref{scn:PointedStabilization} can be adapted to work in the symplectic 
setting. 
One would have to be more careful in choosing the deformations so 
that the embedded subsurfaces stay symplectic.

Furthermore there are nice symplectic embeddings of tori with a hole into 
$\Reals^{2\DimensionIndex}$, which should provide the necessary tori in 
$\BallRadius{\FixedRadius}{0}$ to provide an actual stabilization map in the 
symplectic case, which would be defined like $\PointedStabilizationMap{\Genus}$ 
i.e. it would be defined by taking the connected sum with one of these 
symplectic tori and map symplectic subsurfaces to symplectic 
subsurfaces.
\bibliographystyle{amsalpha}
\bibliography{sources}
\end{document}